\documentclass[11pt]{amsart}

\usepackage{amsmath}
\usepackage{amsfonts}
\usepackage{amssymb}
\usepackage{graphicx}
\usepackage{mathrsfs}
\usepackage{pb-diagram}
\usepackage{epstopdf}
\usepackage{amsmath,amsfonts,amsthm,enumerate,amscd,latexsym,curves}
\usepackage{bbm}
\usepackage[mathscr]{eucal}
\usepackage{epsfig,epsf,epic}
\usepackage{caption}
\usepackage[font=small]{subcaption}
\usepackage{multirow}
\usepackage{color}
\usepackage[all]{xy}
\usepackage{fourier}
\usepackage{tikz}
\usetikzlibrary{fit,matrix}
\usetikzlibrary{matrix,decorations.pathreplacing,calc,positioning}
\usepackage{tikz-cd}
\usepackage{calc}
\usepackage{float}
\usepackage{comment}
\usepackage{longtable}
\usepackage{adjustbox}
\usepackage{enumitem}
\usepackage{systeme}
\usepackage{mathtools}
\usepackage{array,booktabs,colortbl} 
\usepackage{collcell}
\usepackage{xcolor}
\usepackage{cancel}
\usepackage{lipsum}
\usepackage{arydshln}
\usepackage[linktocpage]{hyperref}
\hypersetup{
    colorlinks=true,
    linkcolor=blue,
    filecolor=magenta,      
    urlcolor=cyan,
}
\usepackage{rotating}
\usepackage{bm}
\usepackage{stmaryrd} 
\usepackage{blkarray} 
\usepackage{stackrel}

\graphicspath{ {images/} }

\definecolor{colorG}{rgb}{0, 0.5, 0.0} 
\definecolor{colorH}{rgb}{0.7, 0.7, 0.0} 
\definecolor{colorF}{rgb}{1.0, 0.5, 0.0} 
\definecolor{colorR}{rgb}{0.4, 0.5, 0.13} 
\definecolor{colorT}{rgb}{0.7, 0.5, 0.13} 
\definecolor{colorQ}{rgb}{0.6, 0.3, 0} 
\definecolor{rosepink}{rgb}{1.0, 0.4, 0.8}

\newtheorem{thm}{Theorem}[section]
\newtheorem{lemma}[thm]{Lemma}

\newtheorem{cor}[thm]{Corollary}
\newtheorem{prop}[thm]{Proposition}
\newtheorem{defn}[thm]{Definition}
\newtheorem{example}[thm]{Example}
\newtheorem{remark}[thm]{Remark}

\newtheorem{notation}[thm]{Notation}

\numberwithin{equation}{section}

\newcommand{\bb}{\boldsymbol{b}}

\newcommand{\AI}{A_\infty}
\newcommand{\CC}{\mathcal{C}}
\newcommand{\Hom}{{\rm Hom}}

\newcommand{\OL}[1]{\overline{#1}}

\newcommand{\CP}{\mathcal{P}}

\newcommand{\cF}{\mathcal{F}}

\newcommand{\Z}{\mathbb{Z}}
\newcommand{\R}{\mathbb{R}}
\newcommand{\C}{\mathbb{C}}

\newcommand{\cpp}{\mathcal{P}}
\newcommand{\bzz}{\mathbb{Z}}
\newcommand{\bff}{\mathbb{F}}

\newcommand{\bpp}{\mathbb{P}}
\newcommand{\bll}{\mathbb{L}}
\newcommand{\cff}{\mathcal{F}}

\newcommand{\fm}{\mathfrak{m}}
\newcommand{\fp}{\mathfrak{p}}
\newcommand{\fxx}{\mathfrak{X}}
\newcommand{\bcc}{\mathbb{C}}
\newcommand{\Tri}{\operatorname{Tri}}
\newcommand{\Trilf}{\operatorname{Tri}^{\operatorname{lf}}}
\newcommand{\CM}{\operatorname{CM}}
\newcommand{\CMlf}{\operatorname{CM}^{\operatorname{lf}}}
\newcommand{\dep}{\operatorname{depth}}
\newcommand{\ext}{\operatorname{Ext}}

\newcommand{\ho}{\operatorname{Hom}}

\newcommand{\cok}{\operatorname{coker}}

\newcommand{\id}{\operatorname{id}}
\newcommand{\len}{\operatorname{Length}}
\newcommand{\red}{\operatorname{Red}}

\newcommand{\MF}{\mathrm{MF}}
\newcommand{\WF}{\mathcal{WF}}

\newcommand{\bL}{\mathbb{L}}

\begingroup\lccode`~=`& \lowercase{\endgroup
\newcommand{\smat}[1]{%
  \let~=&
  \begin{smallmatrix}#1\end{smallmatrix}
}}

\begingroup\lccode`~=`& \lowercase{\endgroup
\newcommand{\spmat}[1]{%
  \left(
  \let~=&
  \begin{smallmatrix}#1\end{smallmatrix}
  \right)
}}

\newcommand{\muprime}{\mu'}   
\newcommand{\mupprime}{\mu''}  
\newcommand{\muppprime}{\mu'''} 
\newcommand{\operatorm}{m} 
\newcommand{\LocalF}{\mathcal{F}^\mathbb{L}} 
\newcommand{\LocalPhi}{\varPhi^\mathbb{L}} 
\newcommand{\LocalPsi}{\varPsi^\mathbb{L}} 

\makeatletter
\def\mathunderbar#1{\underline{\sbox\tw@{$#1$}\dp\tw@\z@\box\tw@}}
\makeatother
\usepackage[normalem]{ulem}

\begin{document}

\title[HMS of indecomposable Cohen-Macaulay modules for degenerate cusps]{Homological mirror symmetry of
indecomposable Cohen-Macaulay modules \\ for some degenerate cusp singularities}
\author[Cho]{Cheol-Hyun Cho}
\address{Department of Mathematical Sciences, Research Institute in Mathematics\\ Seoul National University\\ Gwanak-gu\\Seoul \\ South Korea}
\email{chocheol@snu.ac.kr}
\author[Jeong]{Wonbo Jeong}
\address{Department of Mathematical Sciences, Research Institute in Mathematics\\ Seoul National University\\ Gwanak-gu\\Seoul \\ South Korea}
\email{wonbo.jeong@gmail.com}
\author[Kim]{Kyoungmo Kim}
\address{Department of Mathematical Sciences\\ Seoul National University\\ Gwanak-gu\\Seoul \\ South Korea}
\email{kyoungmo@snu.ac.kr}
\author[Rho]{Kyungmin Rho}
\address{Department of Mathematical Sciences\\ Seoul National University\\ Gwanak-gu\\Seoul \\ South Korea}
\email{leftwing@snu.ac.kr}

\begin{abstract}
Burban-Drozd showed that the degenerate cusp singularities have tame Cohen-Macaulay representation type, and classified all indecomposable Cohen-Macaulay modules over them. One of their main example is the non-isolated singularity $W=xyz$. On the other hand,  Abouzaid-Auroux-Efimov-Katzarkov-Orlov showed that $W=xyz$ is mirror to a pair of pants. In this paper, we investigate homological mirror symmetry of these indecomposable Cohen-Macaulay modules for $xyz$.

Namely, we show that closed geodesics (with a flat $\C$-bundle) of a hyperbolic pair of pants  have a
one-to-one correspondence with indecomposable Cohen-Macaulay modules for $xyz$ with multiplicity one that are locally free on the punctured spectrum.
In particular, this correspondence is established first by a geometric $\AI$-functor from the Fukaya category of the pair of pants to the matrix factorization category of $xyz$, and next by the correspondence between Cohen-Macaulay modules and matrix factorizations due to Eisenbud.
For the latter, we compute explicit Macaulayfications of modules from Burban-Drozd's classification and find a canonical form of the
corresponding matrix factorizations. In the sequel, we will show that indecomposable modules with higher multiplicity correspond to twisted complexes of closed geodesics.

We also find mirror images of  rank $1$ indecomposable Cohen-Macaulay modules (of band type) over the singularity $W = x^{3} + y^{2} - xyz$ as closed loops
in the orbifold sphere $\mathbb{P}^1_{3,2,\infty}$.

%
%

\end{abstract}

\maketitle
\tableofcontents

\section{Introduction}

Given a polynomial $f \in \C[[x_1,\ldots,x_n]]=:S$, maximal Cohen-Macaulay modules
of $A=S/(f)$ are related to matrix factorizations of the polynomial $f$ by the theorem of Eisenbud \cite{E80}.
A matrix factorization of $f$ is given by a pair of  matrices $(\varphi,\psi) \in \operatorname{Mat}_{n \times n}(S)^{\times 2}$ such that  $ \varphi \cdot \psi = \psi \cdot \varphi = f \cdot 1_{n \times n}$ for some $n \in \mathbb{N}$.  

The Classification problem of maximal Cohen-Macaulay modules were systematically studied by
Auslander \cite{Au78}, \cite{Au86}, \cite{Au87},  Buchweitz \cite{Buc87}, Esnault \cite{Esn85} and many others.
Kn\"orrer \cite{KN87} and  Buchweitz-Greuel-Schreyer \cite{BGS87} have shown that
there are only finitely many isomorphism classes of indecomposable Cohen Macaulay $A$-modules
if and only if $A$ is an isolated simple hypersurface singularity. In this case $A$ is said to have finite CM-representation type. 

When there are infinitely many indecomposable Cohen-Macaulay $A$-modules, it is said to have tame representation
type  if indecomposable representations of any fixed rank form a finite set of at most 1-parameter families. It is called wild otherwise.

Recently, Burban and Drozd \cite{BD17} carried out a systematic study for an important class of non-isolated Gorenstein singularities called degenerate cusps, and show that they are of tame CM-representation type. 
The following are main examples of such singularities 
$$\C[[x,y,z]]/(xyz),\;\;\; \C[[x,y,z]]/(x^3+y^2-xyz).$$

Burban and Drozd developed a method to translate the classification problem of Cohen-Macaulay modules into a new class of linear algebra problem called decorated bunch of chains. In particular,  Burban and Drozd found a complete list of Cohen Macaulay modules that are locally free on the punctured spectrum. The list is given by a combinatorial data called ``band datum" $(w_{BD},\lambda,\mu)$ where $\lambda \in \bcc^*$ is called an eigenvalue, $\mu \in \mathbb{N}$ is called a multiplicity, and $w_{BD} \in \Z^{6\tau}$ for $\tau \in \mathbb{N}$  is a collection of positive integers 
\begin{equation}\label{eq:bw}
 ((a_1, b_1, c_1, d_1, e_1, f_2), (a_2, b_2,c_2,d_2,e_2, f_3), \ldots, (a_\tau, b_\tau, c_\tau, d_\tau, e_\tau, f_1))
 \end{equation}
with  $\min(f_i, a_i)=\min(b_i, c_i)=\min(d_i, e_i)=1$ for each $i$. We remark that the stable category $\underline{\operatorname{CM}}^{\operatorname{lf}}(A)$ of such locally free modules forms a Hom finite category and in fact a Calabi-Yau 1-category.  Also, Burban-Drozd showed that any indecomposable Cohen-Macaulay module falls into either
band or string type.

The main theme of this paper is to investigate homological mirror symmetry of these indecomposable Cohen-Macaulay modules.
Mirror symmetry is a conjectural relation between symplectic geometry of a space or a Landau-Ginzburg model and
complex geometry of another model. Here, Landau-Ginzburg model is nothing but a singularity $W:Y \to \C$.

We are interested in a version of homological mirror symmetry which is a (derived or $\AI$)-equivalence of categories between a symplectic manifold $X$ and a Landau-Ginzburg model given by a polynomial  $W:\C^n \to \C$. 
For $X$, we consider a Fukaya $\AI$-category, whose objects are Lagrangian submanifolds,  morphisms are Lagrangian Floer complexes
and $\AI$-operation is defined using the count of $J$-holomorphic discs. For $W:\C^n \to \C$, we consider the derived category of matrix factorization, or equivalently that of maximal Cohen-Macaulay modules. Orlov showed that these are equivalent to  the derived category of singularity, which is defined as a quotient of derived category of coherent sheaves on the singular fiber by the full subcategory of perfect complexes. Then the homological mirror symmetry claims that the split-closed derived Fukaya category of $X$ (wrapped version if $X$ is non-compact) is equivalent to the derived category of matrix factorizations of $W$.

For the case of $W= xyz$, its mirror is given by a pair of pants $\mathcal{P}$ (a Riemann surface obtained from $\mathbb{P}^1$ by removing three points, say $0,1,\infty$). Homological mirror symmetry has been proved by Abouzaid-Auroux-Efimov-Katzarkov-Orlov \cite{AAEKO} by comparing split generators of these two categories. Namely, two factorizations $x \cdot (yz),  y \cdot (xz)$ are known to split generate the $B$-model derived category, and two non-compact curves connecting two punctures (say $0,1$ and $1,\infty$) are known to split generate the wrapped Fukaya category of $\mathcal{P}$. Derived equivalence was proved by investigating possible $\AI$-structures that can be put on these two split generators.

For the case of $W=x^3+y^2-xyz$, its mirror is given by an orbifold $\mathbb{P}^1_{3,2,\infty}$, which is $\mathbb{P}^1$ with $\Z_3, \Z_2$ orbifold points and a puncture. Various versions of homological mirror symmetry between them has been proved (Lekili-Perutz \cite{LP12},
Lekili-Ueda \cite{LU18} for the compact Fukaya category and  Cho-Choa-Jung \cite{CHL} for the wrapped case).

Based on these mirror equivalences, we may ask the following natural questions.
What are mirrors of indecomposable Cohen-Macaulay modules of $W$? 
Are the mirrors given by geometric Lagrangians irreducible?
What are geometric meanings of the band datum and of band and string type?
Why does indecomposable representations of tame type appear as 1-parameter families?

Although derived HMS equivalence has been established, these questions cannot be answered directly
as there does not exist a direct explicit functor between the category of Cohen-Macaulay modules and the mirror Fukaya category. 

In this paper, we will answer the above questions for the singularities $W=xyz$ and $x^3+y^2-xyz$.

To explain our results, let us first recall one of the main classification theorem of Burban-Drozd.
\begin{thm}[\cite{BD17} Theorem 9.2]
Let $A = k[[x,y,z]]/(xyz)$, $(w_{BD},\lambda,\mu)$ be a band datum, and 
$\tilde{M}(w,\lambda,\mu)$ be the associated torsion-free $A$-module (see Definition \ref{defn:module}).
\begin{enumerate}
\item The module $M(w,\lambda,\mu):= \tilde{M}(w,\lambda,\mu)^\dagger$ obtained from the Macaulayfication
is an indecomposable maximal Cohen-Macaulay module over $A$, and locally free of rank $\mu\tau$  on punctured spectrum.
\item Any indecomposable object of $\operatorname{CM}^{\operatorname{lf}}(A)$ is isomorphic to some module  $M(w,\lambda,\mu)$
\item $M(w_1,\lambda_1,\mu_1) = M(w_2,\lambda_2,\mu_2)$ if and only if $\mu_1=\mu_2, \lambda_1 = \lambda_2$ and 
$w_1$ is a cyclic shift of $w_2$.
\end{enumerate}
\end{thm}
\begin{remark}
Here, we use a modified band word $w= \left(l_1,m_1,n_1,l_2,m_2,n_2,\dots,l_\tau,m_\tau,n_\tau\right)$ given by $3\tau$ integers (defined in Definition \ref{defn:mbd}), instead of $w_{BD}$ given by $6\tau$ integers with some conditions.
\end{remark}

The $A$-module $\tilde{M}(w,\lambda,\mu)$ is explicitly defined from a band datum,
but  this may not be maximal Cohen-Macaulay.  Then a process called Macaulayfication extends $\tilde{M}(w,\lambda,\mu)$ to the actual maximal Cohen-Macaulay module $M(w,\lambda,\mu) \in  \mathrm{CM}(A)$.

Our first main result is a systematic computation of the Macaulayfication of $A$-modules $\tilde{M}(w,\lambda,\mu)$.
  This process turns out to be complicated, as one needs to find  all Macaulayfying elements for each band data.
For this purpose, we introduce the notion of a generator diagram, which is quite useful to carry out this process.
This will be explained in Section \ref{sec:higherrank} which is rather technical and lengthy.

Once we extend $\tilde{M}(w,\lambda,\mu)$ to the corresponding maximal Cohen-Macaulay modules $M(w,\lambda,\mu)$, their associated matrix factorizations of $W$ can be obtained by finding their resolutions.
We show that for each band datum, they have the following canonical form for the case of $\mu=1$.
\begin{thm}\label{thm:1}
For the Cohen-Macaulay module $M(w,\lambda,1)$,  the corresponding matrix factorization of $xyz$ is given
as follows (Theorem \ref{thm:MFFromModule}).
\begin{equation}\label{eq:can}
\varphi\left(w',\lambda,1\right) :=
\begin{psmallmatrix}
z & -y^{m_1'-1} & 0 & 0 & \displaystyle\cdots & 0 & -\lambda^{-1}x^{-l_1'} \\[2mm]
-y^{-m_1'} & x & -z^{n_1'-1} & 0 & \displaystyle\cdots & 0 & 0 \\[2mm]
0 & -z^{-n_1'} & y & -x^{l_2'-1} & \displaystyle\cdots & 0 & 0 \\[0mm]
0 & 0 & -x^{-l_2'} & z & \ddots & \vdots & \vdots \\[0mm]
\vdots & \vdots & \vdots & \ddots & \ddots & -y^{m_\tau'-1} & 0 \\[2mm]
0 & 0 & 0 & \displaystyle\cdots & -y^{-m_\tau'} & x & -z^{n_\tau'-1} \\[2mm]
-\lambda x^{l_1'-1} & 0 & 0 & \displaystyle\cdots & 0 & -z^{-n_\tau'} & y
\end{psmallmatrix}_{3\tau\times3\tau}
\end{equation}
where $x^a$, $y^a$ or $z^a$ is regarded as zero if $a<0$.
Here, we write only one factor $\varphi$ and its counterpart $\psi$ can be computed using adjoint matrix (Corollary
\ref{cor:Psi}).

Here, exponents of entries can be put into a word  $w'= \left(l_1',m_1',n_1',\ldots,l_\tau',m_\tau',n_\tau'\right)$,  which is defined from $w$ by the following rule (Definition \ref{def:ConversionFromBandtoLoop}):  Given the modified band word $w$, we  first define the {\em sign word}  $\delta(w)$ and the correction word $\varepsilon(w)$, both in $\mathbb{Z}^{3\tau}$. 
For the sign word $\delta = (\delta_1,\ldots, \delta_{3\tau})$, we set 
$$
\delta_j :=
\left\{
\begin{array}{cl}

\setlength\arraycolsep{0pt}
\begin{array}{l}
0 \\ \\ \\
\end{array} 

&

\setlength\arraycolsep{0pt}
\begin{array}{l}
\text{if} \\ \\ \\
\end{array} 

\quad \left\{
                         \begin{array}{l}
                         w_j<0, \text{ or} \\
                         w_j=0 \text{ and at least one of the first non-zero entries adjacent to the} \\
                         \hspace{11mm}\text{string of $0$s containing $w_j$ (exists and) is negative,}
                         \end{array}
                         \right. \\[6mm]
1 & \text{otherwise}.
\end{array}
\right.
$$
We define the correction word $\varepsilon = (\varepsilon_1,\ldots,\varepsilon_{3\tau})$ by setting   $\varepsilon_j := -1 + \delta_{j-1} + \delta_{j} + \delta_{j+1}$ for all $j \in \Z_{3\tau}$.

Finally, we define  $w':= w + \varepsilon(w)$.
\end{thm}
\begin{remark}
For the case of higher multiplicity ($\mu >1$), we will show in the sequel that matrix factorizations still has 
a canonical form whose diagonal blocks are of the above form, and off-diagonal blocks illustrate the iterated cone structures.
\end{remark}

Surprisingly, the canonical form for the band data $w$ is best written in terms of another combinatorial data $w'$, which we will call {\em mirror loop data} as they turn out to describe Lagrangian loops in the mirror pair of pants.   What is happening here is that even though the band data provide generators of $A$-module, Macaulayfication adds new generators to this module, and this is reflected on the entries of corresponding matrix factorizations.
These exponents on matrix factorizations will be interpreted as winding numbers of a loop in the pair of pants under homological mirror symmetry.

Now, we turn to mirror side of symplectic geometry of the pair of pants $\mathcal{P}$. We consider its Fukaya $\AI$-category, whose objects are oriented  immersed curves  in $\mathcal{P}$ equipped with flat complex line bundles. First, closed connected oriented immersed curves can be described by their free homotopy class $[S^1,\CP]$ up to Lagrangian isotopy, and hence it is natural to describe them using 
the fundamental group of $\mathcal{P}$:
$$
\pi_1\left(\mathcal{P}\right) = \left<\alpha,\beta,\gamma\left|\alpha\beta\gamma=1\right.\right>,
$$
where $\alpha$, $\beta$ and $\gamma$ are the based loops in $\mathcal{P}$ shown in Figure \ref{fig:F1}. 
The set of conjugacy classes in $\pi_1(\mathcal{P})$ equals $\left[S^1,\mathcal{P}\right]$. 
\begin{defn}\label{defn:lpwd}
Given a word (which will be called a \emph{loop word} of \emph{length} $3\tau$)
$$
w' = \left(l_1',m_1',n_1',l_2',m_2',n_2',\dots,l_\tau',m_\tau',n_\tau'\right)\in \mathbb{Z}^{3\tau}
$$
with  $\tau\in\mathbb{Z}_{\ge1}$, $l_i'$, $m_i'$, $n_i'\in\mathbb{Z}$ for $i\in\left\{1,\dots,\tau\right\}$,
we define the associated free homotopy class in $\left[S^1,\mathcal{P}\right]$ as
$$
L\left[w'\right]:=\alpha^{l_1'}\beta^{m_1'}\gamma^{n_1'}\alpha^{l_2'}\beta^{m_2'}\gamma^{n_2'}\cdots\alpha^{l_\tau'}\beta^{m_\tau'}\gamma^{n_\tau'}.
$$
We say that two loop words $w'$ and $w_*'$ are \emph{equivalent} if $L\left[w'\right] = L\left[w_*'\right]$
(Definition \ref{def:LoopWord}).
\end{defn}
Here the prime symbol is used to help to distinguish the loop word and the band word. A flat complex line bundle can be
described by its total holonomy $\lambda' \in \C^*$ up to gauge equivalences.
Additional choice of $\muprime \in \mathbb{N}$ define a loop datum $(w',\lambda',\muprime)$.
\begin{figure}[h]
\begin{subfigure}{0.43\textwidth}
\includegraphics[width=50mm]{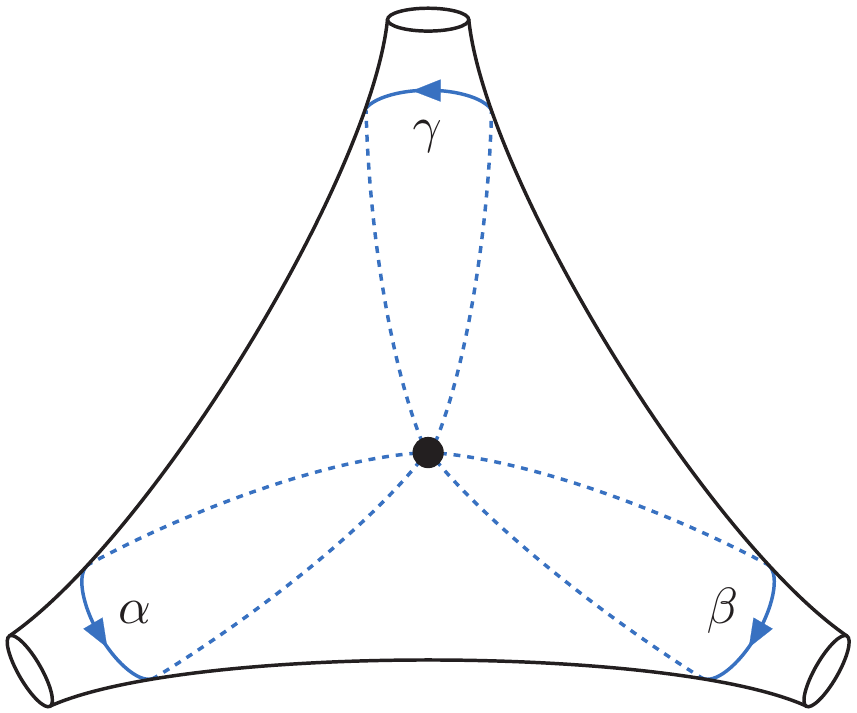}
\centering
\caption{Generators of $\pi_1(\mathcal{P})$}
\label{fig:gen}
\end{subfigure}
\begin{subfigure}{0.43\textwidth}
\includegraphics[width=50mm]{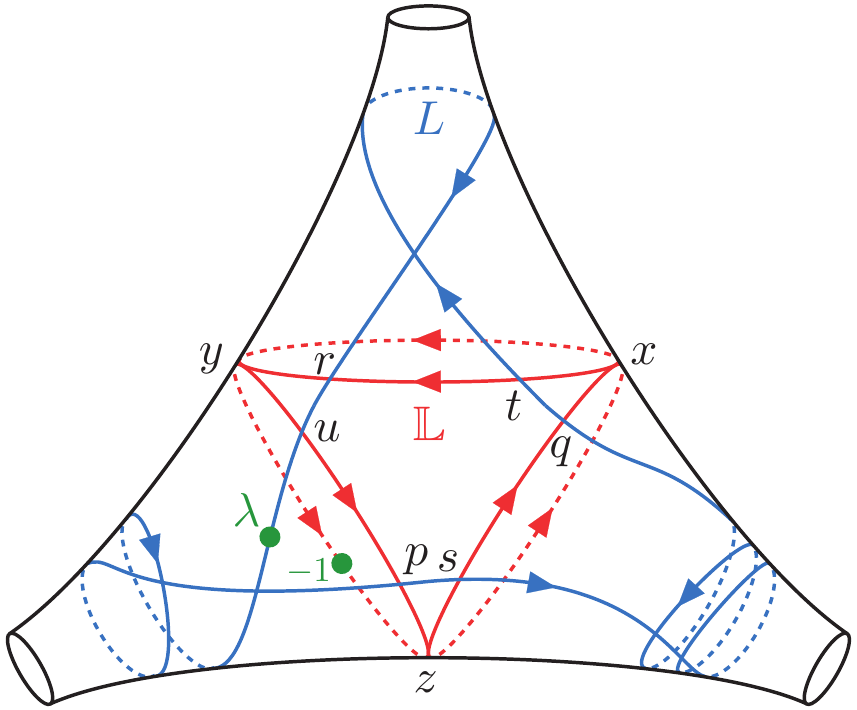}
\centering
\caption{Immersed Lagrangian for the loop data $( (3,-2,2), \lambda', 1 )$}
\label{fig:loopex}
\end{subfigure}
\centering
\caption{Fundamental group and loop data}
\label{fig:F1}
\end{figure}

We will assign an object of the Fukaya category for each loop datum $(w',\lambda',1)$
 by an immersed Lagrangian loop $L\left(w',\lambda'\right)$ (of homotopy class $L\left[w'\right]$) that winds around each puncture as described by the loop word $w'$, with holonomy $\lambda' \in \C^*$ for the case of $\muprime=1$.
If $\muprime \neq 1$, 
we will consider $\muprime$-times iterated cones of $L\left(w', \lambda'\right)$ with itself.
See Figure \ref{fig:F1} (b) for the case of $L\left((3,-2,2),\lambda'\right)$.

There are in fact several issues that we need to address, which are responsible for Appendix \ref{sec:normal}
and \ref{sec:T=1}.

First, different loop words may correspond to the same free homotopy class, due to the relation $\alpha\beta\gamma=1$ in 
the fundamental group. For the mirror correspondences, we would like to choose unique representatives in
each homotopy class. We will introduce the notion of a normal loop word (see Definition \ref{defn:normal1} for the case of $\alpha^{l}\beta^{m}\gamma^n$ and Definition \ref{defn:normal2} in general), and show that
there is a unique normal representative in each  class of $[S^1,\CP]$(except the loops that wind around single punctures) (Proposition \ref{prop:normalnormal}). Normality will play another 
important role in homological mirror symmetry in that matrix factorizations corresponding to normal loops will be always of the canonical form (Theorem \ref{thm:MFFromLag}). All detailed proofs are done in Appendix
\ref{sec:normal}.

Second, closed immersed Lagrangian loops that we consider in this paper are not exact Lagrangians
(Lagrangian arcs that connect punctures are still exact).
Hence, a priori,  we need to work with the Novikov field $\Lambda$ as  a coefficient of the Fukaya category $\mathcal{WF}(\CP)$.
Here the Novikov field is
$$\Lambda = \left\{ \left.\sum_{i=0}^\infty a_i T^{\lambda_i} \right| a_i \in \C, \lim_{i \to \infty} \lambda_i = \infty \right\}.$$ 
But we would like to use $\C$ as a coefficient, or equivalently we want to evaluate $T=1$, but such a process is not
well-defined for infinite sums in $\Lambda$. 
In this paper, we will choose our immersed Lagrangians in $\mathcal{P}$ to be regular curves (Definition \ref{def:regular}) so that we can
evaluate $T=1$ for the resulting matrix factorizations (Theorem \ref{thm:T=1Homotopy}). We remark that evaluation $T=1$ fails if multiples of $L$ and $\bL$ are in the same
homotopy class, and bound a cylinder, but this can be avoided by making a small new intersection between $L$ and $\bL$ in such
a case (Lemma \ref{lem:Regular}).

Third, Lagrangian isotopy (which is not a Hamiltonian isotopy) of curves does not give equivalences in Fukaya category in general,
especially because our closed Lagrangians are not exact. Thus if we move Lagrangian curves in its homotopy class, they are not in general equivalent.  But we show that  after the evaluation of $T=1$, the resulting matrix factorizations under Lagrangian isotopy are homotopic to each other (Theorem \ref{thm:T=1Homotopy}). Therefore, we can use Lagrangian isotopy of curves for mirror symmetry purposes.
Detailed proofs regarding the second and third issues are done in Appendix \ref{sec:T=1}.

Let us finally discuss homological mirror symmetry between $xyz$ and a pair of pants $\CP$.
We claim that Cohen-Macaulay modules for the band data $(w,\lambda,\mu)$ are
mirror to the Lagrangian object $L(w',\lambda',\muprime)$ for the corresponding loop data, where
we should set $\lambda' := \left(-1\right)^{w_1+w_4+\cdots+w_{3\tau-2}+\tau}\lambda \in \C^*$.
It is difficult to prove this using the derived equivalence of \cite{AAEKO} since we would need to express every object
in terms of split generators. 

Instead, we will use the explicit geometric functor, called localized mirror functor 
$$\mathcal{F}^\bL : \mathcal{WF}(\mathcal{P}) \to \mathcal{MF}(xyz)$$ defined 
by the first author with Hong and Lau \cite{CHL}.  
We consider the Seidel Lagrangian $\bL$, which is an immersed Lagrangian that is first introduced by Seidel to
investigate mirror symmetry of Riemann surfaces and their quotients (see  \cite{Se}, \cite{Ef}).
The Lagrangian $\bL$ (see Figure \ref{fig:F1}b) has three self-intersections labeled by $X,Y,Z$, which bounds
a triangle $XYZ$. It is shown in \cite{CHL} that the linear combination $b = xX +yY + zZ$ solves the weak Maurer-Cartan equation in the 
sense of Fukaya-Oh-Ohta-Ono (\cite{FOOO}) with a disc potential function $W = xyz$. 
(For this we need to equip $\bL$ with $\C$-bundle with $(-1)$-holonomy, or a non-trivial spin structure)

Then, localized mirror functor $\mathcal{F}^\bL$ is canonically defined in the sense of curved Yoneda embedding.
In particular, given any immersed Lagrangian $L$ in $\mathcal{P}$, its mirror matrix factorization can be computed
from $\AI$-operations of the Fukaya category, which in our case of $\mathcal{P}$, boils down to 
the count of immersed polygons bounded by $L$ and $\bL$.
Here, the target category $ \mathcal{MF}(xyz)$  is an $\AI$-version of the homotopy category $\underline{MF}(xyz)$,
and the functor is an $\AI$-functor.(see \cite{CHL} for more details) Also the functor induces an derived equivalence between two categories.

Now, let us state the precise homological mirror symmetry between $xyz$ and $\mathcal{P}$.
What we will establish is the following commutative diagram. 
\begin{equation}\label{dia:main}
\begin{tikzcd}[ampersand replacement=\&, sep=20pt] 
  \&
  \underline{\operatorname{MF}}(xyz)
    \arrow[dl, swap, "\mathrel{\rotatebox{48}{$\simeq$}}", shift right,
           "\begin{matrix}
              \operatorname{coker} \\
              \text{(Eisenbud, 1980)}
            \end{matrix}\hspace{4mm}"]
  \&
  \\[15mm]
  \underline{\operatorname{CM}}(A)
    \arrow[ur, swap, shift right,
           "\text{resolve}"]
  \&
  \&
  WF\left(\mathcal{P}\right)
    \arrow[ul, swap, "\mathrel{\rotatebox{-48}{$\simeq$}}",
           "\quad
            \begin{matrix}
              \text{localized mirror functor}
                       \end{matrix}"]
  \\[5mm]
  \text{\color{black} band data}
    \arrow[rr, leftrightarrow, color=black,
           "\text{conversion formula}"]
    \arrow[u, swap, color=black, swap,
           "\begin{matrix}
              \text{classification Thm} \\
              \text{(}\color{black}\text{Burban-Drozd)}
            \end{matrix}"]
  \&
  \&
  \text{\color{black} loop data}
    \arrow[u, swap, color=black,
           "\begin{matrix}
              \text{classify objects} \\
            \end{matrix}"]
\end{tikzcd}
\end{equation}

From the band data of Burban-Drozd (which parametrizes indecomposable Cohen-Macaulay modules), we obtain the corresponding matrix factorizations in Theorem \ref{thm:1}, which gives the composition of two upward arrows on the left hand side
(we perform Macaulayfication and find its resolution).

We first define the corresponding combinatorial data in geometry.
\begin{defn}
A loop data $(w',\lambda',\mu')$ consists of a normal loop word $w'$ (defined in Definition \ref{defn:normal1}, \ref{defn:normal2})
which is not periodic, a holonomy $\lambda' \in \C^*$ and the multiplicity $\mu' \in \mathbb{N}$.
Here $w'$ is periodic if $w' = \left(\tilde{w}'\right)^{\#N}$ for some normal loop word $\tilde{w}'$ and $N \geq 2$.
\end{defn}
If a word is periodic, it corresponds to decomposable objects both in algebra and geometry and hence periodic words are omitted from
both the band and loop data.

We show that rather technical normality condition guarantees that we have a unique representative in each (essential) homotopy class.
\begin{thm}
An essential loop in the free homotopy class $[S^1,\CP]$ has a unique normal representative (up to overall shifting of words).
Here a loop is essential if it is not of type $\alpha^l,\beta^m,\gamma^n$ that wind around each puncture
for some $l,m,n \in \Z$. Conversely, if a loop word is normal, then it is essential (Proposition \ref{prop:normalnormal}).
\end{thm}

On the other hand, essential non-periodic loop classes in $[S^1,\CP]$ can be represented by  closed geodesics.
\begin{cor}
There is a one-to-one correspondence between closed geodesics in hyperbolic pair of pants $\CP$ with three cusps and normal loop words
which are not periodic.
\end{cor}
\begin{proof}
It is well-known that for the complete hyperbolic manifold $\CP$, there is a one to one correspondence between hyperbolic (non-periodic) elements of $[S^1,\CP]$ and  closed geodesics in $\CP$.
Note that the classes $\alpha^l,\beta^m,\gamma^n$ are parabolic elements (winding
around cusps) and are not realized by closed geodesics. Also there are no elliptic elements as $\CP$ has no orbifold point.
Thus,  essential elements  are exactly hyperbolic elements in $[S^1,\CP]$. 
Here, a closed geodesic is defined as a non-periodic one and thus we have the desired correspondence.
\end{proof}

We are ready to state the mirror symmetry correspondence between band and loop data.
A band datum $\left(w,\lambda,\mu\right)$ is said to \emph{degenerate} if   $w=(0,0,0)$ and $\lambda=1$.
A loop datum $\left(w',\lambda',\mu' \right)$ is said to \emph{degenerate} if   $w' = (2,2,2)$ and $\lambda'=-1$.
They are called \emph{non-degenerate} otherwise.

\begin{thm}\label{thm:c}
There is a bijection between non-degenerate band data and  non-degenerate loop data.

More precisely, given a band data $(w,\lambda,\mu)$, the corresponding loop data $(w',\lambda',\mu')$ is defined as
$$
\begin{cases}
w' &=w+\epsilon(w)  \\
\lambda' &= \left(-1\right)^{w_1+w_4+\cdots+w_{3\tau-2}+\tau}\lambda \in \C^* \\
\mu' &= \mu
\end{cases}
$$
where $\epsilon(w)$ is defined in Theorem \ref{thm:1}, and  $w'$ is always a normal loop word.
The inverse conversion formula  is defined in Definition \ref{def:ConversionFromLooptoBand} and proved
in Proposition \ref{prop:ConversionFormulaInverseToEachOther}.
\end{thm}
This gives the correspondence in the bottom of the diagram \eqref{dia:main}.

Now, we are ready to state our main theorem on homological mirror symmetry.
We will restrict ourselves to the case of $\mu=1$ in this paper, and the cases of
$\mu >1$ will be proved in the sequel.

For a loop data $(w',\lambda',1)$, we will define a closed immersed Lagrangian $L(w',\lambda',1)$
in the homotopy class $L\left[w'\right]$ with $\C$-bundle of holonomy $\lambda' \in \C^*$.

\begin{thm}\label{thm:intromain}
The homological mirror symmetry diagram \eqref{dia:main} commutes (Theorem \ref{thm:MainThoerem}).

Namely, given a non-degenerate band data $(w,\lambda,1)$ and the corresponding non-degenerate loop data $(w', \lambda', 1)$ (from
Theorem \ref{thm:c}), we have
\begin{enumerate}
\item Maximal Cohen-Macaulay module $M\left(w,\lambda,1\right)$ corresponds to the matrix factorization $\varphi\left(w',\lambda,1\right)$ under Eisenbud's theorem, i.e.,
$$
M\left(w,\lambda,1\right) \cong \operatorname{coker}\underline{\varphi}\left(w',\lambda,1\right) \quad\text{in}\
\ \underline{\operatorname{CM}}(A),
$$
\item Immersed Lagrangian $L\left(w',\lambda',1\right)$ corresponds to the matrix factorization $\varphi\left(w',\lambda,1\right)$ under the localized
mirror functor, i.e.,
$$
\LocalF\left(L\left(w',\lambda',1\right)\right)\ \cong \varphi\left(w',\lambda,1\right)\quad
\text{in}\ \ \underline{\operatorname{MF}}(xyz).
$$
\end{enumerate}
\end{thm}
In particular, the image of the localized mirror functor $\mathcal{F}^\bL\big( L(w',\lambda',1)\big)$
  is always of the canonical form except for the Lagrangian defined from the loop word $(2, 2, \cdots, 2)$, which gives a matrix factorization homotopic to canonical form.
Here, we are using the definition that the loop word for a loop data is normal. For non-normal loop, the corresponding matrix factorization may not be of a canonical form.
   
\begin{cor}
There is  a one-to-one correspondence between indecomposable   Cohen-Macaulay modules
over $\C[[x,y,z]]/(xyz)$ of multiplicity one that are locally free on the punctured spectrum and  closed geodesics of hyperbolic pair of pants with three cusps under homological mirror symmetry.
\end{cor}
\begin{remark}
We conjecture that such a correspondence holds for other degenerate cusps as well if its mirror is
given by a hyperbolic orbi-surface. We check the corresponding statement for the singularity $x^3+y^2-xyz$
for rank one (hence multiplicity one) indecomposable modules (Lemma \ref{lem:32InfiniteGeodesic}).
\end{remark}

So far, we have considered the case of multiplicity one ($\mu=1$) (the notion of multiplicity was defined by Burban-Drozd).
For the band data with $\mu >1$, the corresponding Lagrangian object is not entirely geometric, and 
  only exists as a twisted complex in wrapped Fukaya category of $\mathcal{P}$.
In particular, this shows that the mirror of indecomposable Cohen-Macaulay modules for $xyz$ are not necessarily 
realized by geometric Lagrangians.  We postpone the discussion of the cases $\mu >1$ to the sequel.

There are another class of  indecomposable Cohen-Macaulay modules, namely the ones that are not locally free on the punctured spectrum.
Burban-Drozd  showed that such modules correspond to string data. String data should correspond under mirror symmetry to
non-compact Lagrangian immersions (connecting punctures).
We omit detailed discussion of them as it can be handled in a similar way.
The name of band and string data in Burban-Drozd came from
the properties of the corresponding combinatorial data, and under mirror symmetry they correspond to immersed circle and line respectively.
Therefore, the mirror Lagrangian also looks like a band ($S^1$) or a string ($\mathbb{R}$).

The homological mirror symmetry correspondence that we established in this paper can be used to study Cohen-Macaulay modules as well.
We will show how to transform some of the properties of Lagrangians to that of the modules in the sequel.

Lastly, we discuss another singularity $W = x^3+ y^2 -xyz$ among degenerate cusps.
Let us first explain the relation between $W$ and the mirror orbifold sphere $\mathbb{P}^1_{3,2,\infty}$ following \cite{CCJ}.
For the Fermat polynomial $F_{p,q} = x^{p} + y^{q}$ (assume that $p$ and $q$ are relatively prime), we denote by $M_{F_{p,q}}$ its Milnor fiber
(given by $x^p+y^q=1$). The maximal diagonal symmetry group $G_{F_{p,q}}= \Z_p\times \Z_q$ acts on the Milnor fiber  $M_{F_{p,q}}$ 
and the quotient space $M_{F_{p,q}}/G_{F_{p,q}}$ is an orbifold $\mathbb{P}^1_{p,q,\infty}$. 

 \begin{figure}[h]
\begin{subfigure}[t]{0.43\textwidth}
\raisebox{3ex}{\includegraphics[scale=0.55]{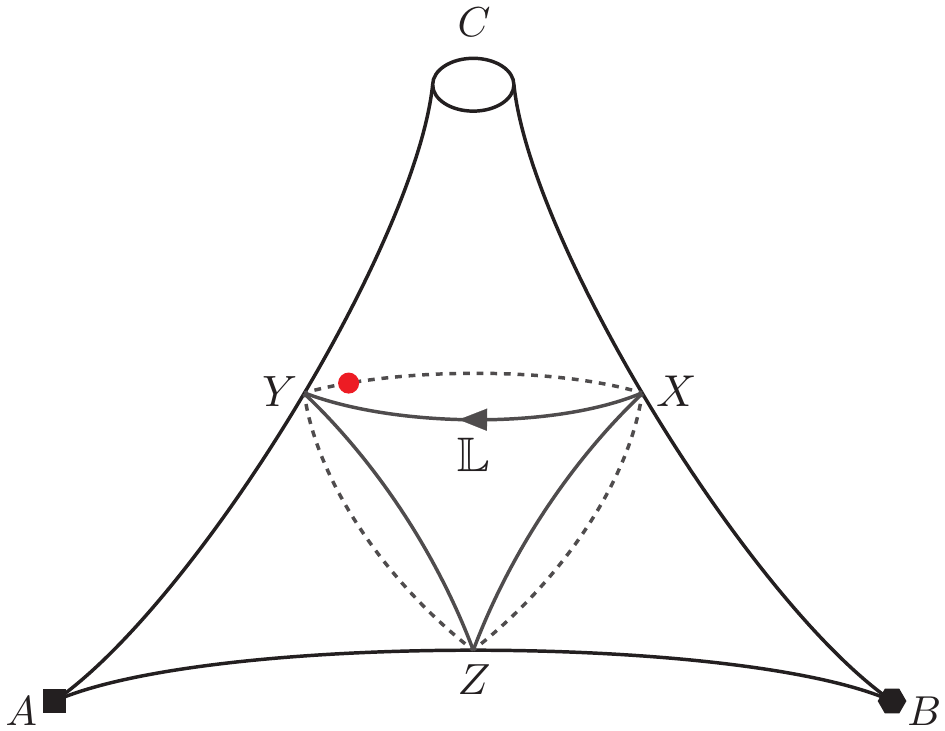}}
\centering
\caption{$M_{F_{3,2}}/G_{F_{3,2}}$ and the Seidel Lagrangian $\bL$}
\label{fig:Fquot}
\end{subfigure}
\begin{subfigure}[t]{0.43\textwidth}
\raisebox{4ex}{\includegraphics[scale=0.55]{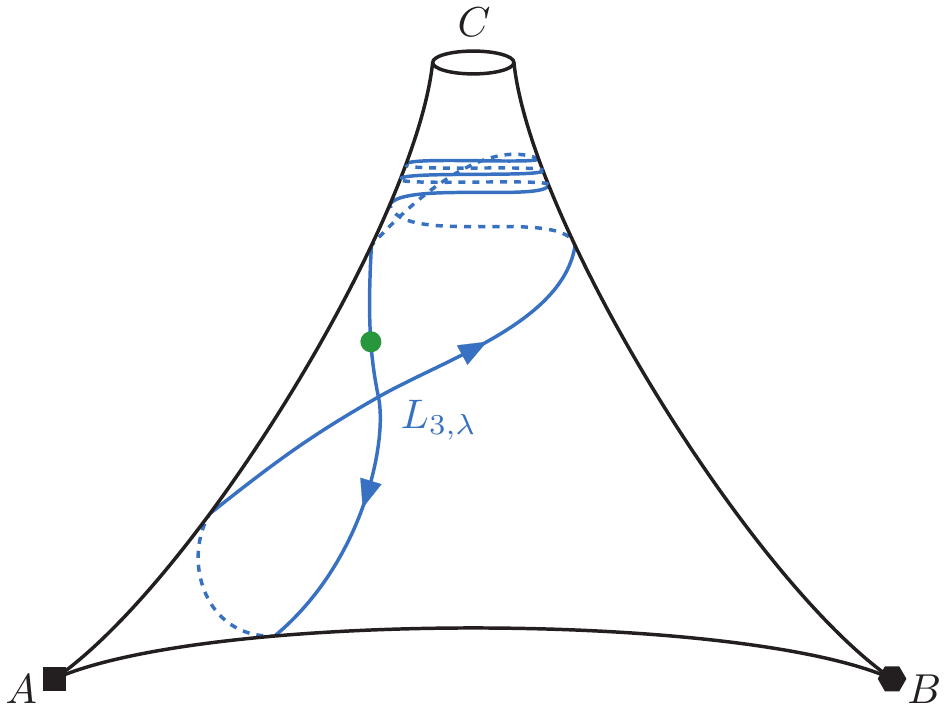}}
\centering
\caption{$L_{3,\lambda}$ in $M_{F_{3,2}}/G_{F_{3,2}}$}
\end{subfigure}
\centering
\caption{}
\label{fig:f32}
\end{figure}
For the Seidel Lagrangian $\bL$ in $\mathbb{P}^1_{p,q,\infty}$, the disc potential function $W_\bL$ (counting polygons with $X,Y,Z$-corners passing through a generic point) is nothing but $x^p+y^q-xyz$ (if we set $T=1$).
Localized mirror functor $$\mathcal{F}^\bL:\WF(M_{F_{p,q}}/G_{F_{p,q}}) \to \mathcal{MF}(x^p+y^q-xyz),$$
is shown to be a quasi-isomorphism. We refer readers to \cite{CCJ} for the chain and loop type case, and related Berglund-H\"ubsch homological mirror symmetry
 
 We first recall the following classification result of Burban and Drozd.

\begin{prop}[\cite{BD17} Proposition 5.3]\label{BD23}
For the ring $A = \C [[x,y,z]] / (x^3+y^2-xyz)$, the classification of maximal Cohen-Macaulay A-modules of rank one is the following.
\begin{enumerate}
\item There exist exactly one maximal Cohen-Macaulay $A$-module  of rank one which is not locally free on the punctured spectrum, 
given by the conductor ideal $I = \langle x, y \rangle \subset  A$.

\item A rank one maximal Cohen-Macaulay $A$-module which is locally free on the punctured spectrum, is either regular or has the following form:
\begin{itemize}
\item $I_{m,\lambda} = \langle x^{m+1}, yx^{m-1} + \lambda (xz - y)^{m} \rangle \subset A$
\item $J_{m,\lambda} = \langle x^{m+1}, y^{m} + \lambda x^{m-1} (xz - y) \rangle \subset A$
\end{itemize}
where $\lambda \in \C^{*}$ for $m \in \Z_{\geq 2}$  or  $\lambda \in \C^{*} \setminus \{ 1 \}$ for $m = 1$.
\end{enumerate}
\end{prop}
Burban-Drozd considered the normalization $R = \C[[u,v]]$ of $A$, where $u = \frac{y}{x}, v=\frac{xz-y}{x}$
and found exact relations between $A$-modules and $R$-modules for the classification.
 
Let us discuss homological mirror symmetry for these modules. The module corresponding to the conductor ideal $I$ is given by a non-compact Lagrangian  (see Figure \ref{fig:A2K}).
We also find immersed Lagrangians that correspond to the remaining rank one maximal Cohen-Macaulay modules, which are locally free on the punctured spectrum.
\begin{thm} \label{thm:J} 
Let $L_{m,\lambda}$ be an immersed curve in $\mathbb{P}^1_{3,2,\infty}$ that wind the puncture $m$ times and wind the $\Z/3$-orbifold point once as in Figure \ref{fig:f32}(b), equipped complex line bundle with holonomy $\lambda$.
Then, under the localized mirror functor $\cF^{\bL}$, we have the mirror symmetry correspondences
$$\begin{cases}
L_{m,\lambda} & \mapsto J_{m,-\lambda}, \\
\overline{L}_{m,\lambda} &\mapsto I_{m,-\frac{1}{\lambda}}
\end{cases}$$
where $\overline{L}_{m,\lambda}$ is the orientation reversal of $L_{m,\lambda}$ (Theorem \ref{thm:32InfiniteCorrespondence}).
\end{thm}

Here is the brief outline of the paper.
In Section 2, we give a brief introduction to algebraic preliminaries for geometry oriented readers and recall Burban-Drozd's work on
the classification of Cohen-Macaulay modules of degenerate cusps.
In Section 3, we give a brief introduction to Fukaya category of a surface and localized mirror functor.
In Section 4, we introduce the generator diagram for Macaulayfications of modules, and use it to 
find explicit rank one Cohen-Macaulay modules as well as their corresponding matrix factorizations 
from Burban-Drozd's combinatorial data. In this process, we find rank one conversion formula from the
band data to the loop data.
In Section 5, we define immersed Lagrangians for the rank one cases and introduce the notion of normal loop word
to find the unique representative in each homotopy class, and also for the computation of mirror functor.
For the normal rank one immersed Lagrangians,  we find the corresponding matrix factorizations under the localized mirror functor.
In Section 6, we compare the results of previous two sections to establish homological mirror symmetry for rank one modules.
In Section 7, we define general loop data, and modified version of band data of Burban-Drozd for better mirror comparison.
In Section 8, we establish general conversion formula, and an explicit inverse conversion formula, and state our main theorem.
In Section 9, we use generator diagram to perform Macaulayfication of modules from all band data of multiplicity one. As a result,
we find the canonical form of the corresponding matrix factorizations.
In Section 10, we define general immersed Lagrangians in pair of pants, and compute the mirror matrix factorizations  under the localized mirror functor.
In Section 11, we discuss the singularity $x^3+y^2-xyz$ and find Lagrangians corresponding to rank one maximal
Cohen-Macaulay modules that are classified by Burban-Drozd.
In Appendix A, we prove that normal condition on loop word indeed gives a unique representative in each essential homotopy class.
In Appendix B, we prove that localized mirror functor (after the evaluation $T=1$) takes isotopic Lagrangians to homotopic matrix factorizations.

\subsection{Acknowledgement}
We would like to thank Igor Burban and Yuriy Drozd for their interests and encouragements. We would like to thank Youngjin Bae
for the helpful discussion on Reidemeister moves.

%



\section{Preliminaries on Algebra}
In this section, we give a gentle introduction to the necessary algebraic notions for geometry minded readers.
We also recall the work of Burban and Drozd from \cite{BD17} on the  classification of maximal Cohen-Macaulay modules over degenerate cusps. They introduced the notions of {\em  bunch of decorated chains}, their category of representations $\text{rep}(\fxx_A)$
and another equivalent category  $\Tri(A)$.  They showed that locally free maximal Cohen-Macaulay modules for $xyz$ are
classified by a band data, which in turn produces an element of $\text{rep}(\fxx_A)$ or $\Tri(A)$.

\subsection{Algebraic preliminaries}

%
%

Recall that for a local ring $(A, \fm)$, the inverse system $(A/\fm^n)$ gives the {\em completion} $\hat{A} = \varprojlim(A/\fm^n)$ together with the canonical map $\iota_A : A\rightarrow \hat{A}$ induced from the canonical projection $A\rightarrow A/\fm^n$ for each $n\in\bzz_{>0}$. 
A local ring $(A, \fm)$ is said to be {\em complete} if the canonical map $\iota_A$ is an isomorphism. We are interested in a complete reduced Noetherian local ring with the residue field $\bcc$. For example, the power series ring $\bcc[[x_1, \cdots, x_n]]$ and its quotient ring $\bcc[[x_1, \cdots, x_n]]/I$, where $I$ is a radical ideal, are complete reduced local ring with the maximal ideal $(x_1, \cdots, x_n)$. Note that each of them has the residue field $\bcc$. Also we recall some notion of dimension.

\begin{defn}
        The {\em  Krull dimension} $\dim_{Kr}(A)$ of a local ring $(A, \fm)$ is defined as the maximal length of chains of prime ideals.
\end{defn}

For example, the Krull dimension of $\bcc[[x]]$ is $1$ since it has only one non-zero prime ideal $(x)$. More generally, the Krull dimension of the power series ring $\bcc[[x_1, \cdots, x_n]]$ is $n$. The following example is our main object.


\begin{example}
        The Krull dimension of $A = \bcc[[x, y, z]]/(xyz)$ is $2$. Indeed, suppose that $\fp_0\subsetneq \fp_1 \subsetneq \cdots \subsetneq \fp_r$ be a maximal chain of prime ideals of $A$. Then, since $xyz=0$ in $A$, one of $x, y$, and $z$ should be in $\fp_0$. Without loss of generality, we may assume $x\in \fp_0$. Then the chain can be viewed as a chain of prime ideals of $A/(x)\cong \bcc[[y, z]]$. Hence $\dim_{Kr}(A) = \dim_{Kr}(\bcc[[y, z]]) = 2$.
\end{example}

There is also a notion of algebraic dimension,  {\em  depth}. 
\begin{defn}
        Let $(A, \fm)$ be a local ring. A sequence of elements $(f_1, \cdots, f_r)$ is said to be {\em  regular} if the following are satisfied.
        \begin{itemize}
                \item $f_1$ is neither a unit nor a zero divisor in $A$.
                \item For each $i \geq 2$, $f_i$ is neither a unit nor a zero divisor in $A/\left(f_1, \cdots, f_{i-1}\right)$.
        \end{itemize}
        The {\em  depth} of $A$ is the maximal length of regular sequences of $A$ and is denoted by $\dep(A)$.
\end{defn}

There is a ring whose Krull dimension and depth are not equal. If the Krull dimension and depth of a ring are the same, the ring is called a Cohen-Macaulay ring.
\begin{defn}
        A local ring $(A, \fm)$ is said to be {\em  Cohen-Macaulay} if $\dim_{Kr}(A)=\dep(A)$.
\end{defn}

\begin{example}
        For example, a sequence $(x+y+z, y-z)$ in $\bcc[[x, y, z]]/(xyz)$ is regular. Hence $\dep(A)\geq 2$. It is known that $\dep(A)\leq\dim_{Kr}(A)$. Hence $\dep(A) = 2$. Thus the ring $\bcc[[x, y, z]]/(xyz)$ is Cohen-Macaulay.
\end{example} 

\begin{remark}
        It is known that a local ring $(A, \fm)$ with a residue field $\bcc$ is Cohen-Macaulay if and only if $\ext^i_A(\bcc, A)=0$ for $i<\dep(A)$.
\end{remark}

\begin{defn}
        Let $A$ be a local ring with a residue field $\bcc$. A finitely generated $A$-module $M$ is said to be {\em  maximal Cohen-Macaulay} if $\ext^i_A(\bcc, M)=0$ for $i<\dep(A)$.
\end{defn}

For a ring $A$, denote by $\text{Mod}(A)$ the category of $A$-modules. Denote by $\CM(A)$ the full subcategory of $\text{Mod}(A)$ whose objects are maximal Cohen-Macaulay modules.

We are interested in the case that $A$ is a non-isolated singularity. 
\begin{defn}
        Let $(A, \fm)$ be a local ring with a residue field $\bcc$ and $\fp$ be a prime ideal of $A$ so that there is an associated local ring $A_\fp$ with a maximal ideal $\fp A_\fp$. The prime ideal $\fp$ is said to be {\em  regular} if $\dim_\bcc \fm_\fp/\fm_\fp^2 = \dim_{Kr} A_\fp$. Otherwise, it is said to be {\em  singular}.
        
        Suppose that in a local ring $(A, \fm)$, the maximal ideal $\fm$ is singular. The ring $(A, \fm)$ is called an {\em  isolated singularity} if every non-maximal prime ideal is regular. Otherwise it is called a {\em  non-isolated singularity}.
\end{defn}

\begin{example}
        The maximal ideal $\fm = (x, y, z)$ of the ring $A = \bcc[[x, y, z]]/(xyz)$ is singular because $$\dim_\bcc \fm/\fm^2 = \dim_\bcc ((x, y, z)/(xyz))/((x^2, y^2, z^2, xy, yz, xz)/(xyz)) = \dim_\bcc (x, y, z)/(x^2, y^2, z^2, xy, yz, xz) = 3,$$ while $\dim_{Kr} A_\fm = \dim_{Kr} A = 2$.
        
        The ideal $\fp = (x, y)$ is non-maximal prime ideal, but is singular. Indeed, $$\dim_\bcc \fp/\fp^2 = \dim_\bcc (x, y)/(x^2, xy, y^2) = 2,$$ while $\dim_{Kr} A_\fp = 1$. Hence the ring $\bcc[[x, y, z]]/(xyz)$ is non-isolated singularity.
\end{example}

\subsection{The category $\Tri(A)$}

Let $A$ be a ring. Then the set of all non zero-divisors $S$ becomes a multiplicative set. The localization of $A$ with respect to $S$, which is denoted by $Q(A)$, is called the {\em  total ring} of $A$.
\begin{defn}\label{def:AlgebraSetting}
        Let $A$ be a ring and $Q(A)$ be the total ring of $A$. The {\em  normalization} of $A$ is the integral closure of $A$ in $Q(A)$. Let $R$ be the normalization of $A$. The {\em  conductor ideal} of $R$ is defined as $$\text{ann}(R/A) = \{r\in R : rR\subseteq A\}.$$
\end{defn}

From now on, let $(A, \fm)$ be a reduced complete Cohen-Macaulay ring of Krull dimension two which is non-isolated singularity. Also, we denote by $\overline{A}, \overline{R}$ the quotient ring $A/I$ and $R/I$, where $I$ is the conductor ideal of the normalization $R$ of $A$. 

\begin{example}
        The ring $A=\bcc[[x, y, z]]/\left(xyz\right)$ is a reduced complete Cohen-Macaulay ring of Krull dimension two. Let $R$ be a product ring $\bcc[[x_1, y_1]]\times \bcc[[y_2, z_2]] \times \bcc[[z_3, x_3]]$ and consider an embedding $\iota : A\rightarrow R$ which is given by $\iota(x) = x_1+x_3, \iota(y) = y_1+y_2, \iota(z) = z_2+z_3$. Then $R$ is integral over $\iota(A)$. For example, $x_1$ is a root of a monic polynomial $t^2-(x_1+x_3)t = 0$, where $x_1+x_3 = \iota(x)\in \iota(A)$. Since each component of $R$ is integrally closed, $R$ is a normalization of $A$.
        
        The conductor ideal $I$ is, as an $A$-module, $\left(xy, yz, xz\right)$. As an $R$-module, it is given by $\left(x_1y_1, y_2z_2, x_3z_3\right)$.      Therefore, $\overline{A} = A/I = \bcc[[x]]\times\bcc[[y]]\times\bcc[[z]]$ and $Q(\overline{A}) = \bcc((x))\times \bcc((y)) \times \bcc((z))$. Similarly, we have
        $$Q(\overline{R}) = \bcc((x_1))\times\bcc((x_3))\times\bcc((y_1))\times \bcc((y_2)\times \bcc((z_2))\times \bcc((z_3)).$$
\end{example}

We recall  the category $\Tri(A)$ introduced in \cite{BD17} which is  equivalent to $\CM(A)$.
\begin{defn}
        An object of $\mathrm{Tri}(A)$ is a triple $(\tilde{M}, V, \theta)$ consisting of the following data.
        \begin{itemize}
                \item A maximal Cohen-Macaulay $R$-module $\tilde{M}$.
                \item A finitely generated $Q(\overline{A})$-module $V$.
                \item A surjective $Q(\overline{R})$-module homomorphism $$\theta : Q(\overline{R})\otimes_{Q(\overline{A})} V \rightarrow Q(\overline{R})\otimes_R\tilde{M}$$ such that the following induced $Q(\overline{A})$-module homomorphism is injective: $$V\rightarrow Q(\overline{R})\otimes_{Q(\overline{A})} V \xrightarrow{\theta} Q(\overline{R})\otimes_R\tilde{M}$$
        \end{itemize}
        A morphism between two triples $(\tilde{M}, V, \theta)$ and $(\tilde{M}', V', \theta')$ is a pair of morphisms $(\varphi : \tilde{M}\rightarrow\tilde{M}', \psi : V\rightarrow V')$ with a suitable commutative diagram.
\end{defn}

\begin{example}\label{example:xyzTri}
        
        The category $\Tri(A)$ for $A=\bcc[[x, y, z]]/(xyz)$ is computed as follows. It is known that any maximal Cohen-Macaulay $R$-module $\tilde{M}$ is of the form $\bcc[[x_1, y_1]]^a\times \bcc[[y_2, z_2]]^b \times \bcc[[z_3, x_3]]^c$ for some nonnegative integers $a, b, c$ and a finitely generated $Q(\overline{A})$-module $V$ is of the form $\bcc((x))^d\times \bcc((y))^e\times\bcc((z))^f$ for nonnegative integers $d, e, f$. Thus the morphism $\theta : Q(\overline{R})\otimes_{Q(\overline{A})} V \rightarrow Q(\overline{R})\otimes_R\tilde{M}$ is written as 
        \begin{align*}
                \theta : &\bcc((x_1))^d\times \bcc((x_3))^d\times\bcc((y_1))^e\times \bcc((y_2))^e\times \bcc((z_2))^f\times \bcc((z_3))^f \\
                \rightarrow &\bcc((x_1))^a\times\bcc((x_3))^c\times\bcc((y_1))^a\times\bcc((y_2))^b\times \bcc((z_2))^b\times\bcc((z_3))^c.
        \end{align*}
        This morphism is the direct sum of six morphisms \begin{align*}
                &\theta^x_1 : \bcc((x_1))^d\rightarrow \bcc((x_1))^a, \quad \theta^x_3 : \bcc((x_3))^d \rightarrow \bcc((x_3))^c \\
                &\theta^y_1 : \bcc((y_1))^e\rightarrow \bcc((y_1))^a, \quad \theta^y_2 : \bcc((y_2))^e \rightarrow \bcc((y_2))^b \\
                &\theta^z_2 : \bcc((z_2))^f \rightarrow \bcc((z_2))^b, \quad \theta^z_3 : \bcc((z_3))^f \rightarrow \bcc((z_3))^c.
        \end{align*}
        Hence the morphism $\theta$ can be identified with a collection of six matrices $\Theta^x_1, \Theta^x_3, \Theta^y_1, \Theta^y_2, \Theta^z_2,$ and $\Theta^z_3$ over a field $\bcc((t))$      satisfying the following conditions : 
        \begin{itemize}
                \item each matrices $\Theta^x_1, \Theta^x_3, \Theta^y_1, \Theta^y_2, \Theta^z_2$ and $\Theta^z_3$ have full row-rank.
                \item induced matrices $\begin{pmatrix} \Theta^x_1 \\ \Theta^x_3 \end{pmatrix}, \begin{pmatrix} \Theta^y_1 \\ \Theta^y_3 \end{pmatrix}$ and $ \begin{pmatrix} \Theta^z_2 \\ \Theta^z_3 \end{pmatrix}$ have full column-rank.
        \end{itemize}
        These data may be illustrated as in the diagram below.
        \begin{center}
                \begin{tikzcd}[arrow style=tikz,>=stealth,row sep=4em]
                        & \circ x \arrow[rr, "\theta^x_3"] \arrow[ld, "\theta^x_1"'] & & \bullet \text{3} &  \\
                        \bullet \text{1} & & & & \circ z \arrow[lu, "\theta^z_3"'] \arrow[ld, "\theta^z_2"] \\
                        & \circ y \arrow[lu, "\theta^y_1"] \arrow[rr, "\theta^y_2"'] & & \bullet \text{2} & 
                \end{tikzcd}
        \end{center}
\end{example}

\subsection{Categorical equivalence between $\CM\boldsymbol{(A)}$ and $\Tri\boldsymbol{(A)}$}\label{sec:maci}
In this subsection, we recall a construction of a functor from the category $\CM(A)$ to $\Tri(A)$ introduced in \cite{BD17}, which gives a categorical equivalence.

The first step is the {\em  Macaulayfication}. This is a functor $(-)^\dagger$ from $A$-mod to $\CM(A)$. The following definition is borrowed from \cite{HB93}.

\begin{prop}\cite{HB93} For a Noetherian local Cohen-Macaulay ring $A$ the following hold.
        \begin{enumerate}
                \item There is a unique, up to isomorphism, $A$-module $K_A$ called the {\em  canonical module} such that $$\ext^i_A(\bcc, K_A) = \begin{cases} 0 \text{ if } i\neq \dim_{Kr}(A)\\ \bcc \text{ if } i=\dim_{Kr}(A)\end{cases}.$$
                \item For any finitely generated module $M$, the module $M^\dagger := \mathrm{Hom}_A(\mathrm{Hom}_A(M, K_A), K_A)$ is a maximal Cohen-Macaulay module over $A$ and  $M^\dagger$ is called the Macaulayfication of $M$. This gives a functor from $A$-mod to $\CM(A)$
        \end{enumerate}
\end{prop}
The functor $(-)^\dagger$  is left adjoint to the forgetful functor from $\CM(A)$ to $A$-mod. More precisely, for any finitely generated $A$-module $M$ and any maximal Cohen-Macaulay $A$-module $N$, we have a natural isomorphism $$\ho_{A-\mathrm{mod}}(M, N)\cong\ho_{\CM(A)}(M^\dagger, N).$$

Under the situation of \ref{def:AlgebraSetting}, we have two natural functors
\begin{itemize}
        \item $\CM(A)\rightarrow \CM(R)$, sending a maximal Cohen-Macaulay module $M$ over $A$ to the Macaulayfication $R\boxtimes_A M = (R\otimes_A M)^\dagger$.
        \item $\CM(A)\rightarrow Q(\overline{A})-\mathrm{mod}$, sending $M$ to the tensor product $Q(\overline{A})\otimes_A M$.
\end{itemize}
Then there is a natural $Q(\overline{R})$-module homomorphism $$\theta_M : Q(\overline{R})\otimes_{Q(\overline{A})}(Q(\overline{A})\otimes_A M)=Q(\overline{R})\otimes_R(R\otimes_A M)\rightarrow Q(\overline{R})\otimes_R (R\otimes_A M)^\dagger.$$ We get a triple $(R\boxtimes_A M, Q(\overline{A})\otimes_A M, \theta_M)$. The next result shows that this is an object of $\Tri(A)$.

\begin{lemma}\cite[Lemma 3.2]{BD17}
        The morphism $\theta_M$ is surjective. Moreover, the following canonical morphism of $Q(\overline{A})$-modules is injective:   
        $$\tilde{\theta}_M : Q(\overline{A})\otimes_A M \rightarrow Q(\overline{R})\otimes_A M\xrightarrow{\theta_M} Q(\overline{R})\otimes_R (R\boxtimes_A M)$$
\end{lemma}

Since the construction is natural to the module $M$, it defines a functor $\bff : \CM(A)\rightarrow \Tri(A)$. The main result in Section 3 of \cite{BD17} is that this functor is an equivalence of categories.

\begin{thm}\cite[Theorem 3.5]{BD17}\label{BurbanDrozdMainTheorem}
        The functor $\bff : \mathrm{CM}(A)\rightarrow \mathrm{Tri}(A)$, which sends a maximal Cohen-Macaulay module $M$ to a triple $(R\boxtimes_A M, Q(\overline{A})\otimes_A M, \theta_M)$, is an equivalence of categories.
\end{thm}

Now recall two full subcategories $\CMlf(A)$ of $\CM(A)$ and $\Trilf(A)$ of $\Tri(A)$.
\begin{defn}
        A maximal Cohen-Macaulay $A$-module $M$ is said to be {\em locally free on the punctured spectrum}  if the localization $M_\fp$ is a free $A_\fp$-module for any non-maximal prime ideal $\fp$. We denote by $\mathrm{CM}^{\operatorname{lf}}(A)$ the full subcategory of $\mathrm{CM}(A)$ consisting of such $A$-modules. 
        \end{defn}
We abbreviate it as locally free from now on.  Similarly, an object $(\tilde{M}, V, \theta)$ in $\Tri(A)$ is locally free if the morphism $\theta$ is an isomorphism.   It is shown in Theorem 3.9 \cite{BD17} that the functor $\bff$ induces an equivalence between full subcategories of locally free objects between $\CMlf(A)$ and $\Trilf(A)$.

\subsection{Burban \& Drozd's classification} We have mentioned that for a ring $A = \bcc[[x, y, z]]/(xyz)$, an obejct of $\Tri(A)$ is determined by the $6$ matrices $(\Theta^x_1, \Theta^x_3, \Theta^y_1, \Theta^y_2, \Theta^z_2, \Theta^z_3)$ in Example \ref{example:xyzTri}. By Theorem \ref{BurbanDrozdMainTheorem}, maximal Cohen-Macaulay $A$-modules are equivalent to this collection of matrices. In particular, locally free maximal Cohen-Macaulay $A$-modules are determined by six nonsingular square matrices. By choosing appropriate basis for each $\bcc((t))$-vector spaces, one can transform these matrices into a canonical form, which is given as follows.
 \begin{defn}\label{defn:bd}
        A {\em  band data} consists of 
        \begin{itemize}
                \item positive integers $\tau$ and $\mu$,
                \item a nonzero complex number $\lambda\in\bcc^*$,
                \item a collection of positive integers $\omega = ((a_1, b_1, c_1, d_1, e_1, f_2), (a_2, b_2, c_2, d_2, e_2, f_3), \cdots, (a_\tau, b_\tau, c_\tau, d_\tau, e_\tau, f_1))$ such that $\text{min}(f_i, a_i) = \text{min}(b_i, c_i) = \text{min}(d_i, e_i) = 1$ for all $i$.
        \end{itemize}
        
        Given band data $(\tau, \lambda, \mu, \omega)$, define the corresponding canonical form as follows.
        \begin{itemize}
                \item Let $I_r$ be the $r\times r$ identity matrix and $J_r(\lambda)$ be the Jordan block of size $r\times r$ with the eigenvalue $\lambda$.
                \item Consider the following diagonal matrices $$A_i:=t^{a_i}I_\mu, B_i:=t^{b_i}I_\mu, C_i:=t^{c_i}I_\mu, D_i:=t^{d_i}I_\mu, E_i:=t^{e_i}I_\mu, F_i:=t^{f_i}I_\mu, H=t^{f_1}J_\mu(\lambda).$$
                \item Then define the following matrices
                \begin{align*}
                        &\Theta^x_1 = \text{diag}(A_1, A_2, \cdots, A_\tau), \quad \Theta^y_1 = \text{diag}(B_1, B_2, \cdots, B_\tau), \quad \Theta^y_2 = \text{diag}(C_1, C_2, \cdots, C_\tau), \\
                        &\Theta^z_2 = \text{diag}(D_1, D_2, \cdots, D_\tau), \quad \Theta^z_3 = \text{diag}(E_1, E_2, \cdots, E_\tau),
                \end{align*} and a block matrix
                \begin{math}
                        \Theta^x_3:=
                        \begin{pmatrix}
                                0 & F_2 & 0 &\cdots & 0 \\
                                0 & 0 & F_3 & \cdots & 0 \\
                                \vdots & \vdots & \vdots & \ddots & \vdots \\
                                0 & 0 & 0 & \cdots & F_\tau \\
                                H & 0 & 0 & \cdots & 0
                        \end{pmatrix}
                \end{math}
        \end{itemize}
        
\end{defn}
The maximal Cohen-Macaulay $A$-module corresponding to a band data is described explicitly in Definition \ref{defn:module}.
%
%
%

We are interested in indecomposable maximal Cohen-Macaulay modules. They correspond to a special class of band data, called non-periodic band data.
\begin{defn}
        Let $(\tau, \lambda, \mu, \omega)$ be a band data. It is said to be {\em  periodic} if there is some $1<r<\tau$ such that $(a_i, b_i, c_i, d_i, e_i, f_{i+1}) = (a_{i+r}, b_{i+r}, c_{i+r}, d_{i+r}, e_{i+r}, f_{i+r+1})$ for each $1\leq i\leq \tau$, where indices are considered as an element of $\bzz_\tau$.
\end{defn}

\begin{thm}\cite[Theorem 8.2]{BD17}\label{thm:BD}
        Let $A$ be a ring $\bcc[[x, y, z]]/(xyz)$. There is a one to one correspondence between the set of isomorphism classes of indecomposable objects of $\mathrm{Tri}^{\operatorname{lf}}(A)$ and the set of non-periodic band data.
\end{thm}

\begin{cor}
        Locally free maximal Cohen-Macaulay modules over $\bcc[[x, y, z]]/(xyz)$ are classified by band data.
\end{cor}

\subsection{Matrix Factorization} 
The Eisenbud's theorem \cite{E80} says that for a hypersurface singularity, maximal Cohen-Macaulay modules are equivalent to matrix factorizations. In this subsection we recall basic definitions and properties of matrix factorizations. We refer readers to the book of Yoshino \cite{Yo} for more details.

Let $S$ be a regular local ring and $W$ be a nonzero-divisor in $S$.

\begin{defn}\label{defn:mf}
        A {\em  matrix factorization} of $W$ is a $\bzz_2$-graded free $S$-module $X=X^0\oplus X^1$ with an odd degree map $d : X\rightarrow X$ such that $d^2 = W\id_{X}$. A {\em morphism} from $(X, d_X)$ and $(Y, d_Y)$ is an $S$-module homomorphism from $X$ to $Y$. A morphism $f$ is decomposed into an {\em odd part} $f^0$ and an {\em even part} $f^1$ so that the set of morphisms $\operatorname{Hom}^{DG}_{\operatorname{MF}}((X, d_X), (Y, d_Y))$ has a natrual $\bzz_2$-graded $S$-module structure. A {\em differential} $d$ on the hom module is given by $$d(f)=d_Y\circ f + (-1)^{|f|} f \circ d_X,$$ where $f$ is homogeneous of degree $|f|$.
\end{defn}

From the definition, we have three categories for matrix factorization as follows. 
\begin{itemize}
        \item The DG-category $\mathrm{MF}^{DG}(S, W)$, whose morphism set is given as $\operatorname{Hom}^{DG}_{\mathrm{MF}}((X, d_X), (Y, d_Y))$.
        \item The ordinary category $\mathrm{MF}(S, W)$, whose morphism set is given as $Z^0(\operatorname{Hom}^{DG}_{\mathrm{MF}}((X, d_X), (Y, d_Y)))$.
        \item The homotopy category $\underline{\mathrm{MF}}(S, W)$, whose morphism set is given as $H^0(\operatorname{Hom}^{DG}_{\mathrm{MF}}((X, d_X), (Y, d_Y)))$.
\end{itemize}

\begin{remark}
        Note that any matrix factorization $(X, d)$ is isomorphic to $(S^n\oplus S^n = S^n_{even}\oplus S^n_{odd} , \varphi : S^n_{even}\rightarrow S^n_{odd}, \psi : S^n_{odd}\rightarrow S^n_{even})$ for some $n$. We will denote this matrix factorization by $(\varphi, \psi)$ and $n$ is called the {\em rank} of it . Also note that for two matrix factorizations $(\varphi, \psi), (\varphi', \psi')$, a morphism is given by a pair of matrices $(\alpha, \beta)$ satisfying $\beta\circ\varphi = \varphi'\circ\alpha$ and $\alpha\circ\psi = \psi'\circ\beta$.
\end{remark}

%
%
%

\begin{center}
        $\begin{matrix}
                \begin{tikzcd}
                        S^{n} \arrow[r, "\varphi"] \arrow[d, "\alpha"] & S^{n} \arrow[r, "\psi"] \arrow[d, "\beta"] & S^{n} \arrow[d, "\alpha"] \\
                        S^{n'} \arrow[r, "\varphi'"] & S^{n'} \arrow[r, "\psi'"]& S^{n'} 
                \end{tikzcd}
                \\[10mm]
                {\scriptstyle \text{A morphism of matrix factorizations}}
        \end{matrix}
        \qquad
        \begin{matrix}
                \begin{tikzcd}
                        & S^{n} \arrow[r, "\varphi"] \arrow[dl, "\xi"'] \arrow[d, "\alpha-\alpha'"] & S^{n} \arrow[r, "\psi"] \arrow[dl, "\eta"] \arrow[d, "\beta-\beta'"] & S^{n}  \arrow[dl, "\xi"] \\
                        S^{n'} \arrow[r, "\psi'"] & S^{n'} \arrow[r, "\varphi'"] & S^{n'} &
                \end{tikzcd}
                \\[10mm]
                {\scriptstyle \text{A homotopy between morphisms}}
        \end{matrix}$
\end{center}

Let us  consider a quotient ring $A = S/\left(W\right)$. Given a matrix factorization $\left(\varphi, \psi\right)$, we have 
an induced $2$-periodic acyclic chain complex of $A$-modules (see section 7.2.2 in \cite{Yo}).
$$
\cdots \rightarrow A^n \xrightarrow{\underline{\varphi}} A^n \xrightarrow{\underline{\psi}}  A^n \xrightarrow{\underline{\varphi}} A^n \xrightarrow{\underline{\psi}}       A^n \xrightarrow{\underline{\varphi}} A^n \rightarrow \cdots
$$
The cokernel $\cok\underline{\varphi}$ is a Cohen-Macaulay $A$-module, and it defines a functor 
$$\cok : \mathrm{MF}(W) \rightarrow \CM(A)$$
$$\hspace{14.5mm}\left(\varphi,\psi\right) \mapsto \cok\underline{\varphi}$$
Conversely,   Theorem 6.1 of \cite{E80} states that for any $A$-module $M$, its minimal free resolution is eventually periodic with periodicity $2$. Moreover, the theorem also tells that the minimal free resolution is $2$-periodic itself if and only if the module $M$ is maximal Cohen-Macaulay. In this case, the resulting resolution gives rise to a matrix factorization. Thus the functor $\cok$ is essentially surjective.

The stable category of Cohen-Macaulay modules is similarly defined as follows.
\begin{defn}
        For two Cohen-Macaulay $A$-modules $M$ and $N$, consider the set $I\left(M,N\right)$
        of morphisms $f : M\rightarrow N$ that passes through a projective $A$-module $P$
        (namely, there are morphisms $ g : M \rightarrow P$, $h : P \rightarrow N$ satisfying $f=  h \circ g$). Then $I\left(M,N\right)$ is an ideal of $\operatorname{Hom}_{\operatorname{CM}(A)}\left(M,N\right)$,
which allows us to define
$$\mathrm{Hom}_{\underline{\CM}(A)}\left(M, N\right) := \mathrm{Hom}_{\mathrm{CM}(A)}\left(M, N\right)/I\left(M, N\right).$$
The stable category $\underline{\mathrm{CM}}(A)$ of $\operatorname{CM}(A)$ is the category whose objects are the same as $\mathrm{CM}(A)$ and the hom set between two objects $M$, $N$ is given by $\ho_{\underline{\mathrm{CM}}(A)}\left(M, N\right)$.
\end{defn}

Now the theorem of Eisenbud can be stated as follows.
\begin{thm}(Eisenbud's matrix factorization theorem \cite{E80})\label{thm:Eisenbud}
        The induced functor
        $$\cok : \underline{\MF}(W) \rightarrow \underline{\mathrm{CM}}(A)$$
        is an equivalence of categories.
\end{thm}
Instead of $\underline{\MF}(W)$, we will work with the following ($\AI$-analogue of) dg-category of matrix factorizations, to which $\AI$-functor from the Fukaya category is defined. One can find the definitions of the $\AI$-category and functor in the next section.
\begin{defn}
The $\AI$-category $\mathcal{MF}(W)$ has the same set of objects as $\mathrm{MF}(W)$, and its $\Z/2$-graded hom space is defined as
$$\Hom_{\mathcal{MF}(W)}\left(\left(\varphi',\psi'\right),\left(\varphi,\psi\right)\right) := \Hom_{\mathrm{MF}(W)}^{\Z/2}\left(\left(\varphi,\psi\right),\left(\varphi',\psi'\right)\right)$$
with $\AI$-operations $\operatorm_k$ $\left(k=1,2,\dots\right)$:
$$\operatorm_1\left(\left(\alpha,\beta\right)\right) = Df = d \circ f - (-1)^{|f|} f \circ d,\quad \operatorm_2(f_1, f_2): = (-1)^{|f_1|} f_1 \circ f_2$$
and $\operatorm_k = 0$ for $k \neq 1,2$. Here, a morphism defined in Definition \ref{defn:mf} is of even degree and odd degree morphism is defined in a similar way but with a condition $\alpha\varphi_1 = \psi_2\beta$ and $\beta\psi_1 = \varphi_2\alpha$.
\end{defn}

\section{Preliminaries on Geometry}
Let us recall our geometric setup of Fukaya category and localized mirror functor.
We refer readers to Fukaya-Oh-Ohta-Ono \cite{FOOO}, Seidel \cite{Se}, Akaho-Joyce \cite{AJ} for general definitions and properties of Fukaya category, and \cite{CHL} for localized mirror functor formalism.

In this paper, a symplectic manifold $(M,\omega)$ is given by a punctured Riemann surface $\Sigma$ with an area form $\omega$ on it.
In particular, many operations on the Fukaya category can be explained combinatorially as counts of suitable (immersed) polygons
(instead of counting solutions of $J$-holomorphic curve equation).

An object of Fukaya category of $\Sigma$ will be given by oriented immersed curves $\iota : L\rightarrow M$, 
which automatically satisfies the Lagrangian condition ($\iota^*(\omega)=0, \dim(L)=\frac{1}{2}\dim(M)$). We will call $L$ an {\em  immersed Lagrangian}. Our Lagrangian is always oriented, and hence we will omit it from now on.
(In fact, we will only consider  regular immersed curves; see Definition \ref{def:regular}). 
We allow non-compact Lagrangians that start and end at punctures (morphisms between them will be defined as in  wrapped Fukaya category). Our Fukaya category is defined over the Novikov field $\Lambda$ where
$$\Lambda := \Big\{ \sum_{i=0}^\infty a_i T^{\lambda_i} \Big| a_i \in \C, \lim_{i \to \infty} \lambda_i = \infty \Big\}.$$
This was introduced to handle infinite sums whose energy (exponent of T) of summands approach infinity.
Note that $M$ is an exact symplectic manifold $\omega = d\theta$, and if we consider exact Lagrangians $\iota: L \to M$ only (with $\iota^* \theta = d f_L$ for some function $f_L: L \to \R$), we can work with $\C$-coefficients.
With exact Lagrangians, a $J$-holomorphic curve with prescribed inputs and an output has a fixed energy, and
hence its count is finite from the Gromov-Compactness theorem.

 But compact immersed Lagrangians that we are interested in are {\em not} exact, hence we need to work with $\Lambda$ a priori.
The energy filtration of $\Lambda$ is used to run Maurer-Cartan formalism as well as the localized mirror functor.

To compare with the matrix factorizations over $\C$, we will make an evaluation $T=1$ later.
In general, an evaluation $T=1$ for an element of $\Lambda$ does not make sense.
But for regular immersed curves, even though they are not exact, we will be able to make the evaluation $T=1$
(see Appendix \ref{sec:T=1}).

Let us recall the definition (and convention) of an $\AI$-category over the field $\Lambda$.

\begin{defn}\label{defn:AinftyCategory}
An $\AI$-category $\CC$ over $\Lambda$ consists of a collection of objects $Ob(\CC)$, a (graded) $\Lambda$-module $\Hom(A_{1},A_{2})$ for $A_{1},A_{2} \in Ob(\CC)$, and a set of $\AI$-operations $\left\{ \operatorm_{k} \right\}_{k \geq 1}$ where
$$\operatorm_{k} : \Hom(A_{1},A_{2}) \otimes \cdots \otimes \Hom(A_{k},A_{k+1}) \to \Hom(A_{1},A_{k+1}).$$
They satisfy $\AI$-relations
$$\sum_{p,q} (-1)^{\dagger} \operatorm_{n-q+1} \left (f_{1},\ldots,f_{p}, \operatorm_{q} \left( f_{p+1}, \ldots, f_{p+q} \right),f_{p+q+1}, \ldots, f_{n} \right)=0$$
for any fixed $k \geq 1$ and possible $p,q \geq 1$. Here, $\dagger = |a_{1}| + \cdots + |a_{p}| - p$ is related to the grading of inputs.

An $\AI$-functor $\mathcal{F}$ between two $\AI$-categories $\mathcal{A}$, $\mathcal{B}$ consists of maps $\left\{\mathcal{F}_{k} \right\}_{k \geq 0}$ where
$$\mathcal{F}_{0} : Ob(\mathcal{A}) \to Ob(\mathcal{B})$$
and for $k \geq 1$,
$$\mathcal{F}_{k} :  \Hom_{\mathcal{A}}(A_{1},A_{2}) \otimes \cdots \otimes \Hom_{\mathcal{A}}(A_{k},A_{k+1}) \to \Hom_{\mathcal{B}}(\mathcal{F}_{0}(A_{1}),\mathcal{F}_{0}(A_{k+1})).$$
They satisfy similar $\AI$-relations
\begin{align*}
&\sum_{t, i_{1} + \cdots + i_{t}=n} \operatorm_{t} \left (\mathcal{F}_{i_{1}} \left( a_{1}, \ldots, a_{i_{1}} \right), \ldots, \mathcal{F}_{i_{t}} \left( a_{i_{t-1} + 1}, \ldots, a_{n} \right) \right)\\
&=\sum_{p,q} (-1)^{\dagger} \mathcal{F}_{n-q+1} \left (a_{1}, \ldots, a_{p}, \operatorm_{q} \left( a_{p+1}, \ldots, a_{p+q} \right), \ldots, a_{n} \right).
\end{align*}
\end{defn}

We will only consider a countable family of regular immersed Lagrangians on $\Sigma$ as objects of Fukaya category.
Without loss of generality, we may assume that these curves intersect transversely away from self intersection points and there are no triple (or higher) intersections.

Since $M$ is non-compact and has cylindrical ends toward the punctures, we choose a Hamiltonian function $H$ on $M$ which is quadratic at infinity (to define wrapped Fukaya category). Let $\phi_H$ be a time one map of its Hamiltonian vector field $X_H$.
Often we will simply write $L$ instead of writing the immersion $\iota:L \to M$ for convenience.
\begin{defn}
For two immersed oriented Lagrangians $L_1,L_2$ that intersect transversely, 
$CF^*(L_1,L_2)$ is a $\Z/2$-graded vector space over $\Lambda$ generated by $\phi_H(L_1)\cap L_2$. 
Here an intersection $p \in CF^*(L_1,L_2)$ is odd
if the orientation of $T_pL_1 \oplus T_pL_2$ agrees with  that of $T_pM$ and it is even otherwise.
The self Hom space $CF^*(L, L)$ is a $\Z/2$-graded vector space over $\Lambda$ generated by $\phi_H(L)\cap L$. 
\end{defn}
It is well-known how to define a differential and $\AI$-operations in general.
We will add another object, compact immersed Lagrangian in the pair of pants, called Seidel Lagrangian $\bL$.
To run the Maurer-Cartan theory, we can work with the following Morse complex version of self-Hom space 
instead of Hamiltonian perturbation. We refer readers to  Seidel \cite{Se} for more details, and in particular the sign convention.

Let $\iota : L \rightarrow M$ be a regular {\em compact} Lagrangian immersion.
We fix a Morse function on $L$ so that its critical points are away from immersed points, and intersection points with other Lagrangians, and denote by $C^{*}_{Morse}(L)$ the resulting Morse complex of $L$
over $\Lambda$. Let us explain the association of two immersed generators for each self intersection point $p \in M$ of $\iota$.
For two branches $\tilde{L}_1,\tilde{L}_2$ of $\iota(L)$ in the neighborhood of $p$,
we can associate two local Floer generators $p \in CF(\tilde{L}_1,\tilde{L}_2)$, $\overline{p} \in CF(\tilde{L}_2, \tilde{L}_1)$.
We may choose branches so that $p$ is odd, and $\overline{p}$ is even. 
For the Seidel Lagrangian $\bL$, we define 
$$CF^*(\bL,\bL): =  C^{*}_{Morse}(\bL) \bigoplus_{p} (  \Lambda \langle p \rangle \oplus \Lambda \langle \overline{p} \rangle)$$
where the sum is over the self intersection points of $\bL$.

If $L_0, \ldots, L_k$ are immersed Lagrangians and for $w_i \in CF^*(L_{i-1}, L_{i})$ for $i=1,\cdots k$ given by
transverse intersections or immersed generators, 
an $\AI$-operation $m_k(w_1,\ldots, w_k)$ is defined by counting immersed rigid (holomorphic) polygons with convex corners at $w_1,\cdots, w_k, w_0$ with $w_0 \in CF^{|w_0|}(L_k,L_0)$, 
contributing a term $\pm T^{\omega(u)} \cdot \overline{w}_0 \in CF^{1-|w_0|}(L_0,L_k)$. We take the sum  over all such $u$ and $w_0$.  
If a polygon has a non-convex corner, it is not difficult to show that such a polygon comes with at least one parameter family of  holomorphic polygons and thus it is not rigid.



Let $\mathcal{P}$ be a pair of pants which is a three-punctured sphere $\bpp^1\setminus\{a, b, c\}$ and $\bL$ be the Seidel Lagrangian in $\mathcal{P}$ as in Figure \ref{fig:loopex}. Since $\bL$ has $3$ self-intersections, a Floer complex of $\bL$ has $8$ generators
$$CF^{*}(\bL,\bL) = \langle e, X, Y, Z, \OL{X}, \OL{Y}, \OL{Z}, p \rangle$$
where $e$ and $p$ are a minimum and a maximum of the chosen Morse function on $\bL$.

\begin{thm}\cite{CHL}
Assume that the areas of two triangles bounded by $\bL$ in $\mathcal{P}$ are the same.
A linear combination $b = xX + yY + zZ \in CF^{*}(\bL,\bL) \otimes_\Lambda \Lambda \langle x,y,z \rangle$  is
a weak bounding cochain. Namely, we have
\begin{equation} \label{eqn:wu}
\sum_{i=0}^{\infty} \operatorm_{k}(b,\ldots,b) = xyz \cdot e.
\end{equation}
\end{thm}
The only nontrivial $\AI$-operation is $m_k(X,Y,Z)$ which counts the front triangle bounded by $\bL$
passing through $e$, and we have the mirror potential $W=xyz$ from the pair of pants $\mathcal{P}$.

For weakly unobstructed Lagrangian $\bL$ in symplectic manifold $M$, localized mirror functor formalism \cite{CHL} gives an $\AI$-functor from Fukaya category of $M$ to the matrix factorization category of $W^\bL$ 
(here we follow the convention in \cite{CHLnc}).

\begin{thm}\cite{CHL}\label{thm:lmf}
Let $W^{\mathbb{L}}$ be the disc potential of $\mathbb{L}$. The localized mirror functor $\LocalF : \mathcal{WF}(X) \to \mathcal{MF}(W^{\bL})$ is defined as follows.
\begin{itemize}
\item For a given Lagrangian $L$,  mirror object $\LocalF(L) $
is  given by the following matrix factorization $M_L$ 
$$\big(CF(L,\mathbb{L}), -\operatorm_{1}^{0,b}\big), \textrm{where} \;\;\operatorm_1^{0,b}(x) =\sum_{l=0}^{\infty} \operatorm_k(x,\underbrace{b, \ldots, b}_{l}).$$
\item Higher component of the $\AI$-functor 
$$\LocalF_k : CF(L_{1},L_{2}) \otimes \cdots \otimes CF(L_{k},L_{k+1}) \to \mathcal{MF}(M_{L_{1}},M_{L_{k+1}})$$
is given by
$$\LocalF_k\left(f_{1}, \ldots,f_{k}\right) := \operatorm_{k+1}^{0,\dots,0,b}\left(f_1,\dots,f_k,\bullet\right)
= \sum_{l=0}^{\infty} \operatorm_{k+1+l}(f_{1},\ldots,f_{k},\bullet,\underbrace{b, \ldots, b}_{l}).$$
Here the input $\bullet$ is an element of $M_{L_{k+1}} = CF(L_{k+1},\mathbb{L})$.
\end{itemize}
Then, $\LocalF$ defines an $A_{\infty}$-functor, which is cohomologically injective on $\bL$.
\end{thm}
\begin{figure}[h]
\includegraphics[scale=0.6]{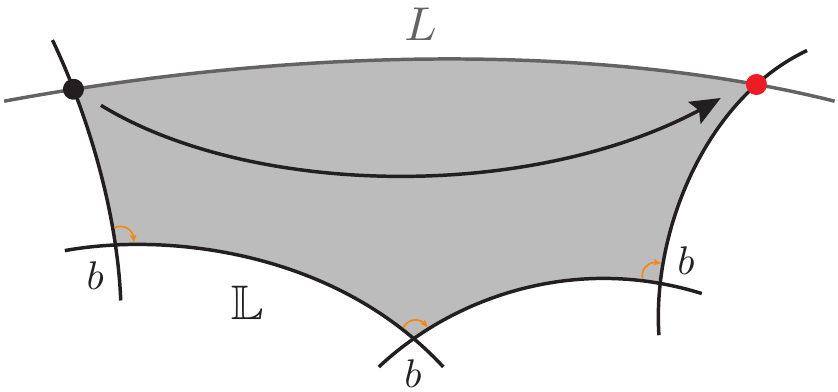}
\centering
\caption{Reading off entries of matrix factorization from decorated strips}
\label{fig:strip}
\end{figure}
The above theorem prescribes how to find the mirror matrix factorization from geometry.
Given a curve $L$, take the $\Z/2$-graded vector space generated by the intersection $L \cap \mathbb{L}$.
Count decorated strips bounded by $L$ and $\mathbb{L}$ as in Figure \ref{fig:strip}. Here decorated means that
we allow strips to have arbitrary many corners (of $X,Y,Z$), and the resulting count records the labels of them (by $x,y,z$).
By index reasons, each strip should map an even (resp. odd) intersection to an odd (resp. even) intersection.
If we record all these data in two matrices, they become the matrix factorization of the disc potential function $W_\bL$
corresponding to $L$.

In our case, a disc potential of Seidel Lagrangian $\bL$ is $W^{\bL} = xyz$, and three simplest non-compact Lagrangians
connecting different punctures are mapped to the three factorizations $x \cdot yz, y \cdot xz, z \cdot zy$.
From this, $\LocalF$ becomes an quasi-equivalence
recovering the results of \cite{AAEKO} for $\mathcal{P}$.
\begin{thm}
$\WF(\mathcal{P})$ is derived equivalent to $\MF(xyz)$.
\end{thm}




\section{Matrix Factorizations Arising from Modules: Rank $1$ Case}\label{sec:mod1}
Burban-Drozd classified maximal Cohen-Macaulay $A$-modules for $A=\left.\mathbb{C}[[x,y,z]]\right/(xyz)$ by the band data
in Theorem \ref{thm:BD}. In fact, a given band data produces an $A$-module $\tilde{M}$ which may {\em not} be maximal Cohen-Macaulay,
but it is known that any Noetherian $A$-module $\tilde{M}$ can be extended to a maximal Cohen-Macaulay module $\tilde{M}^\dagger$ that is called Macaulayfication of $\tilde{M}$. Once we obtain the maximal Cohen-Macaulay module, we can apply the Eisenbud's theorem to obtain the corresponding matrix factorizations.

Therefore, we need to carry out Macaulayfication of $A$-modules corresponding to band data, which turns out to
be quite subtle process. In this section,  we   introduce a combinatorial method to carry out  Macaulayfication for all $A$-modules in the list of Burban-Drozd. Namely we will introduce what we call a generator diagram and explain how to perform Macaulayfication using such a diagram.
In this section, we give a gentle introduction to this method by explaining the rank one ($\tau\mu=1$) cases in the list.
Higher rank cases are considerably more complicated, and will be discussed in Section \ref{sec:higherrank}.

%
%
%

\subsection{Macaulayfication and Macaulayfying Elements}

Recall from Section \ref{sec:maci} that Macaulayfication of $\tilde{M}$ is defined as  $\tilde{M}^\dagger := \ho_A(\ho_A(\tilde{M}, K_A), K_A)$ for the canonical module $K_A$.
It is more convenient to find its Macaulayfication using  \emph{Macaulayfying elements}.
%

\begin{defn}
For an $A$-submodule $\tilde{M}$ of a free $A$-module $A^r$, if there is an element $F\in A^r\setminus\tilde{M}$ such that $xF$, $yF$, $zF\in \tilde{M}$, we call it a \emph{Macaulayfying element} of $\tilde{M}$ in $A^r$.
\end{defn}

\begin{prop}\label{prop:Macaulayfying}
For an $A$-submodule $\tilde{M}$ of a free $A$-module $A^r$, the following hold:
\begin{enumerate}
\item $\tilde{M}$ is maximal Cohen-Macaulay if and only if there is no Macaulayfying element of $\tilde{M}$ in $A^r$. We have $\tilde{M}^\dagger \cong \tilde{M}$ in this case.
\item $\tilde{M}^\dagger \cong \left<\tilde{M},F\right>_A^\dagger$ holds for any Macaulayfying element $F$ of $\tilde{M}$ in $A^r$.
\end{enumerate}
\end{prop}

\begin{proof}
(1) Recall from Corollary 2.23 of \cite{BD08} that $\tilde{M}$ is maximal Cohen-Macaulay if and only
if $H_{\left\{\mathrm{m}\right\}}^i \left(\tilde{M}\right)=0$ for $i=0,1$, where $\mathrm{m} = \left(x,y,z\right)$
is the maximal ideal of $A$. By the long exact sequence
$$
\begin{tikzcd}[column sep = 10pt] 
  0 \arrow[r] & H_{\left\{\mathrm{m}\right\}}^0 \left(\tilde{M}\right) \arrow[r] & H_{\left\{\mathrm{m}\right\}}^0 \left(A^r \right) \arrow[r] & H_{\left\{\mathrm{m}\right\}}^0 \left(A^r \left/ \tilde{M} \right.\right) \arrow[r] & H_{\left\{\mathrm{m}\right\}}^1 \left(\tilde{M}\right) \arrow[r] & H_{\left\{\mathrm{m}\right\}}^1 \left(A^r \right) \arrow[r] & \cdots
\end{tikzcd}
$$
and the fact that $A^r$ is maximal Cohen-Macaulay, it is also equivalent to
say that$$
H_{\left\{\mathrm{m}\right\}}^0 \left(A^r \left/ \tilde{M} \right.\right)
\cong \left\{ \left[F\right]\in A^r \left/\tilde{M}\right.
~\left| ~\mathrm{m}^t \left[F\right] = 0 ~\text{for some} ~t\in\mathbb{Z}_{\ge1}\right.\right\}
$$
is trivial, which proves the first part. See Theorem 2.18 of \cite{BD08} for the second statement.

(2) See Lemma 1.5 of \cite{BD17}.
\end{proof}

We will compute the Macaulayfication of an $A$-submodule $\tilde{M}$ of a free $A$-module $A^r$ by finding all Macaulayfying elements of $\tilde{M}$ in $A^r$.

\subsection{Generator Diagram and Macaulayfying Elements}

Here we illustrate a method to find Macaulayfying elements of an $A$-submodule of $A^1$ (which is also an ideal of $A$). For this, we arrange the monomials in $A$ in a lattice form as in Figure \ref{fig:LatticeDiagram}, called the \emph{lattice diagram} for $A$.
\begin{figure}[h]
\includegraphics[scale=0.6]{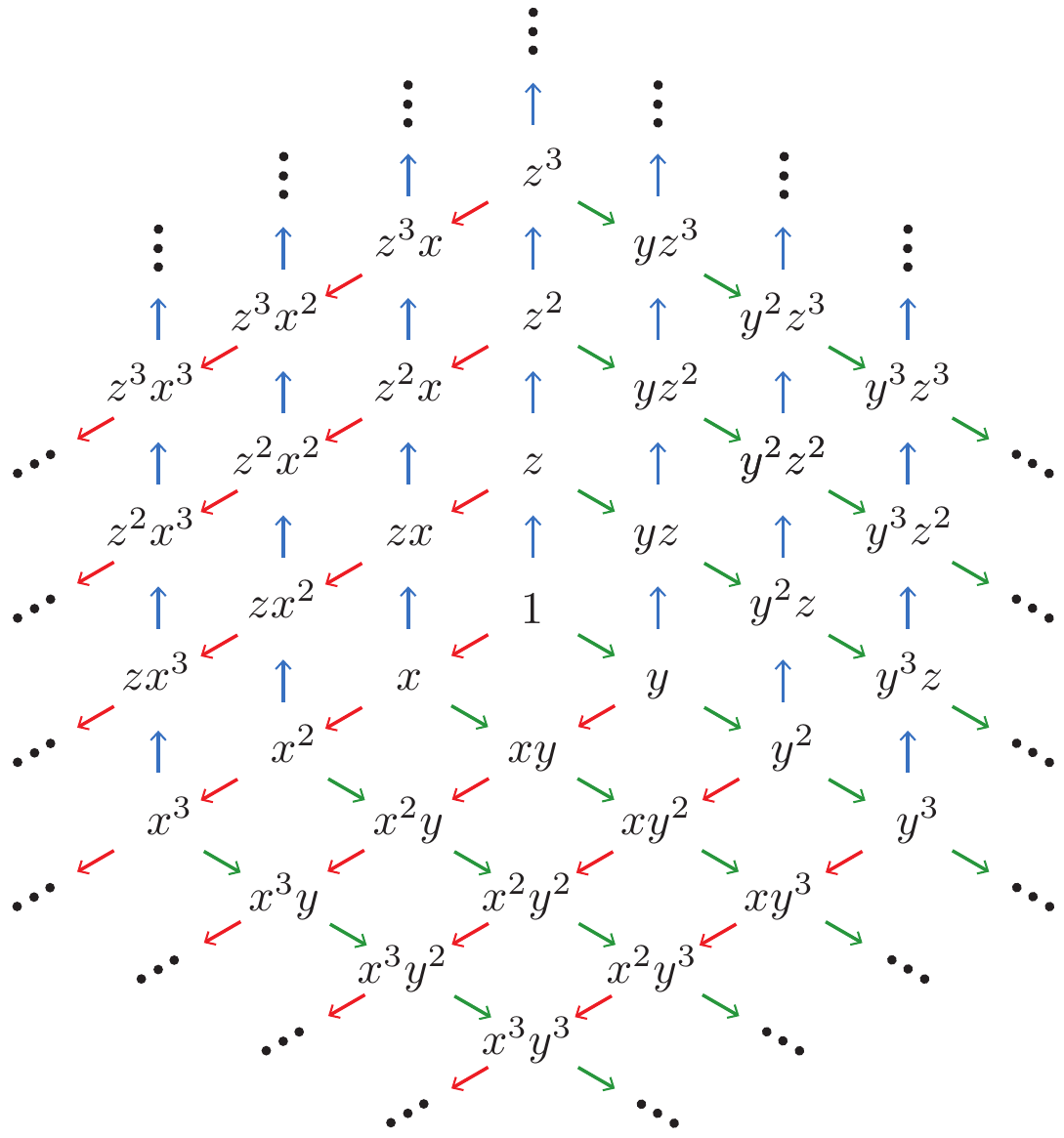}
\centering
\caption{Lattice diagram for $A=\left.\mathbb{C}[[x,y,z]]\right/(xyz)$}
\label{fig:LatticeDiagram}
\end{figure}

The colored arrows represent relations between them as elements of an $A$-module. Namely, the red, green and blue arrows indicate how each element changes when it is multiplied by $x$, $y$ and $z$, respectively. If there is no corresponding arrow, this means that the element becomes zero.

Now, let us explain how to find the Macaulayfication of the following $A$-submodule $\tilde{M}$ of $A^1$ generated by three elements
as an example.
$$
\tilde{M} = \left<zx^2+x^2y, xy^2+y^2z, yz^2+z^2x\right>_A.
$$
We can express $\tilde{M}$ on the lattice by denoting  its$A$-generators, which is described in Figure \ref{fig:DegenerateCaseGeneratorDiagram}.(a). We call it a \emph{generator diagram} of $\tilde{M}$.
\begin{figure}[h]
     \centering
     \begin{subfigure}[b]{0.45\textwidth}
         \includegraphics[scale=0.6]{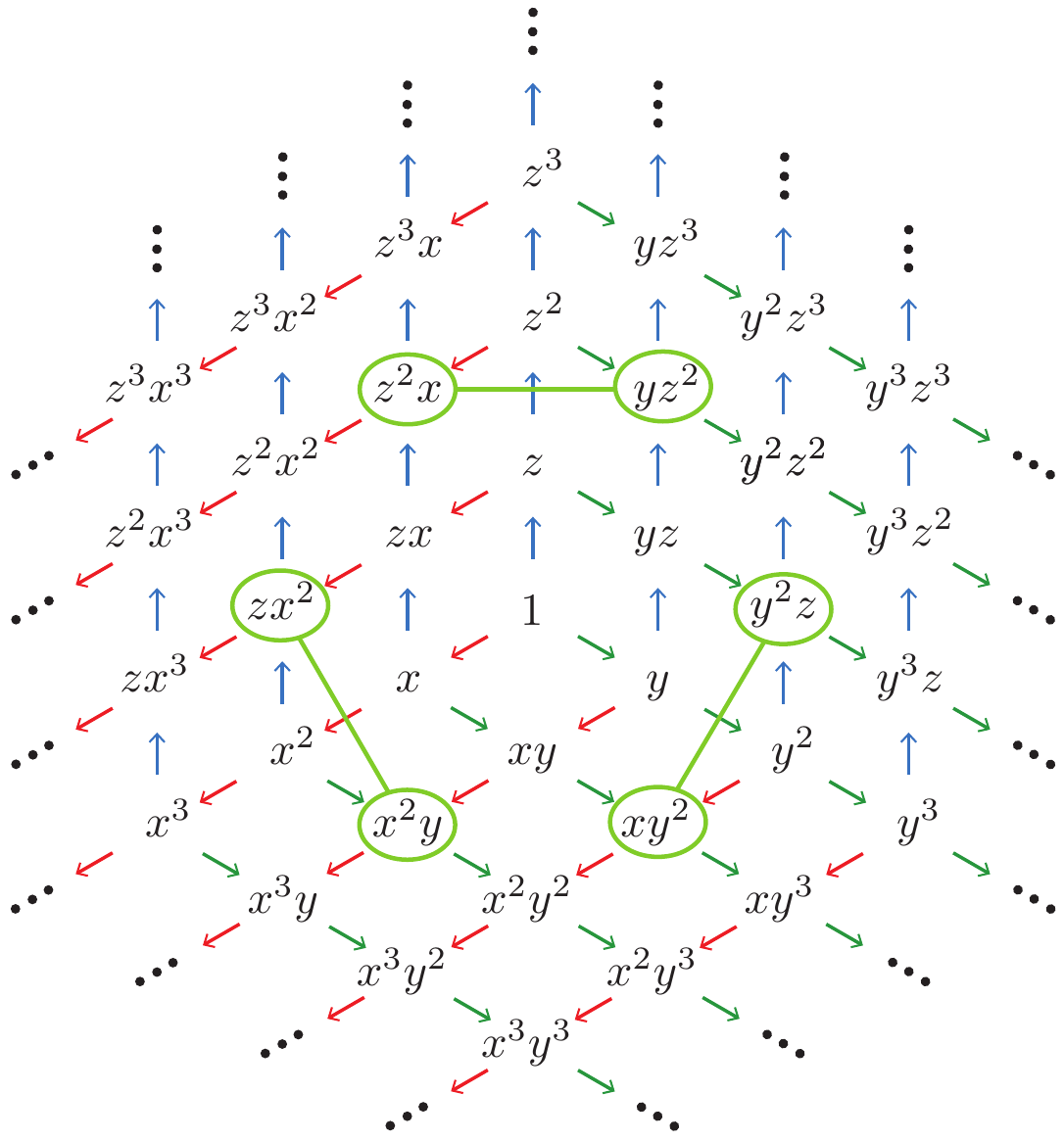}
         \caption{$\tilde{M}=\left<zx^2+x^2y, xy^2+y^2z, yz^2+z^2x\right>_A$}
         \label{fig:DegenerateCaseGeneratorDiagram1}
     \end{subfigure}
     \qquad
     \begin{subfigure}[b]{0.45\textwidth}
         \includegraphics[scale=0.6]{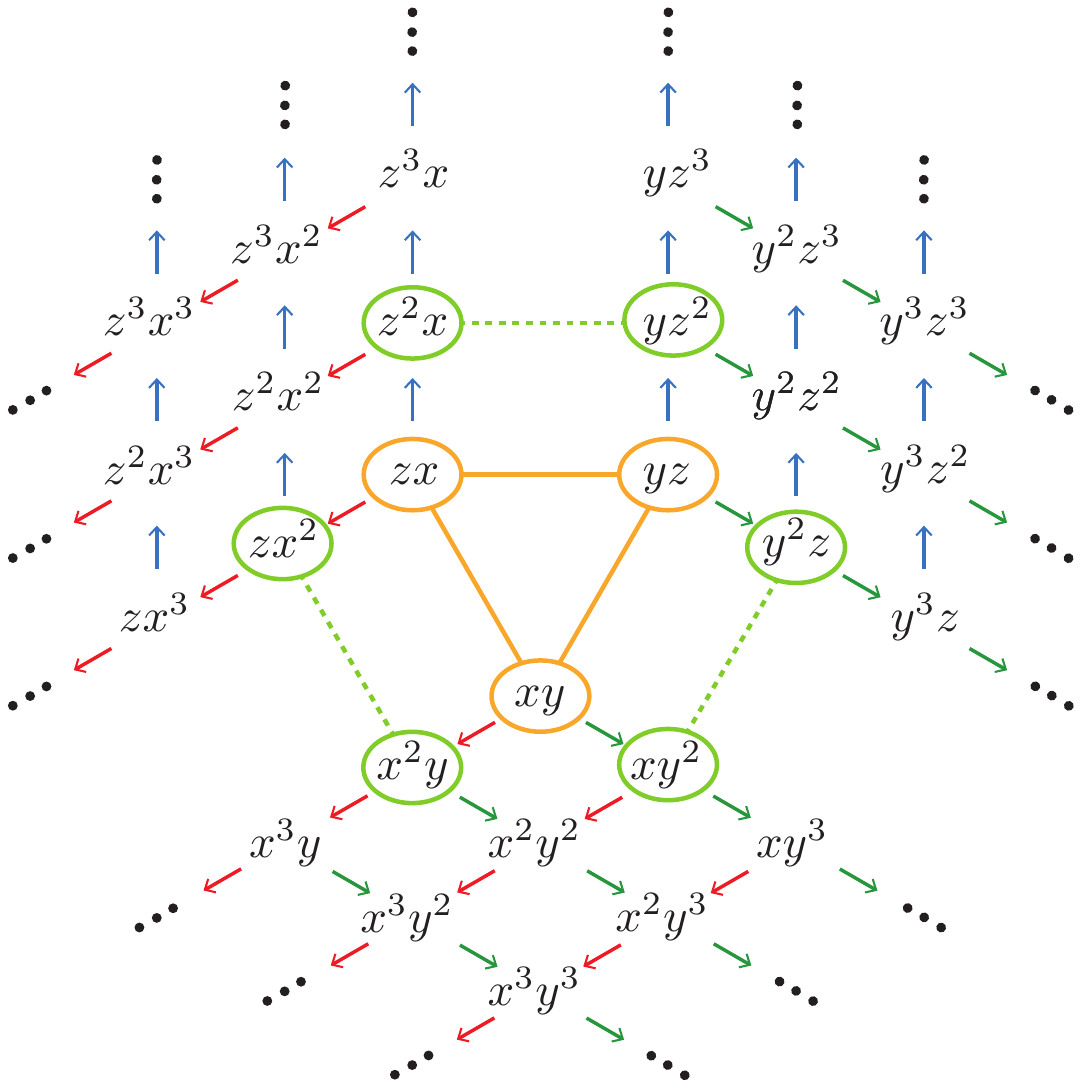}
         \caption{$\tilde{M}^\dagger=\left<xy+yz+zx\right>_A$}
         \label{fig:DegenerateCaseGeneratorDiagram2}
     \end{subfigure}
        \caption{Generator diagram in the degenerate case}
        \label{fig:DegenerateCaseGeneratorDiagram}
\end{figure}

In Figure \ref{fig:DegenerateCaseGeneratorDiagram}, each $A$-generator is marked with two (light green) circles with an edge between them.
Note that from these three generators, it is not hard to describe  the module $\tilde{M}$ as a $\mathbb{C}$-vector space:
Just multiply $x, y$ or $z$ to the generators repeatedly, which corresponds to moving the two circles (connected with an edge) along the red, green and blue arrows. A pair of circles may become one or disappear in this process, if the corresponding element after multiplication lies in the ideal $(xyz)$. Thus, we may  draw the generators only to describe $\tilde{M}$.

Let us show that $F:=xy+yz+zx\in A$ is an Macaulayfying element of $\tilde{M}$.
The element $F$ is specified in the right side of the picture by the three yellow circles and edges.
Clearly $F$ is not in $\tilde{M}$.
We can easily read from the diagram that $x \cdot F, y \cdot F, z \cdot F$ becomes three generators
$ zx^2+x^2y, xy^2+y^2z,  yz^2+z^2x$ of $\tilde{M}$ (by moving these circles and edges in three directions).
 This means that 
 $F$ is a Macaulayfying element of $\tilde{M}$ in $A^1$. 
 
 We add $F$ to $\tilde{M}$ to obtain $\left<\tilde{M},F\right>_A = \left<F\right>_A$ as
 original generators are all generated by $F$. It is easy to check that there are no more  Macaulayfying elements of $\left<F\right>_A$. 
 Therefore, Proposition \ref{prop:Macaulayfying} implies that $\tilde{M}^\dagger \cong \left<\tilde{M},F\right>_A^\dagger = \left<F\right>_A^\dagger \cong \left<F\right>_A$. Furthermore, it turns out that the map $A \rightarrow\left<F\right>_A$, $a\mapsto aF$ is an $A$-module isomorphism, which forces $\tilde{M}^\dagger \cong A$ as $A$-modules.

\subsection{Matrix Factorizations}\label{subsec:MFFromRank1Module}
Let us describe the Macaulayfication and the corresponding matrix factorization of the following $A$-modules:
\begin{defn}\label{defn:modrank1}
Rank $1$ (modified) band data is given by $l, m, n \in \mathbb{Z}$ and  $\lambda\in\mathbb{C}^*$, denoted as
$$\left(\left(l,m,n\right),\lambda,1\right)$$
We define the corresponding $A$-module as
$$
\tilde{M}\left(\left(l,m,n\right),\lambda,1\right) = \left<
x^2y^2, y^2z^2, z^2x^2, \lambda zx^{l^++2}+x^{l^-+2}y, xy^{m^++2}+y^{m^-+2}z, yz^{n^++2}+z^{n^-+2}x\right>_A
$$
where $a^+, a^-$ (satisfying $a = a^+ - a^-$) for $a\in\mathbb{Z}$ are defined as
$$
a^+:=\max\left\{0,a\right\} \quad \text{and} \quad a^- := \max\left\{0,-a\right\}
$$
We denote its Macaulayfication by 
$$ M\left(\left(l,m,n\right),\lambda,1\right) = \tilde{M}\left(\left(l,m,n\right),\lambda,1\right)^\dagger$$
\end{defn}
\begin{remark}
A general definition of modified band data, and its corresponding module will be given in Section \ref{defn:mbd}.
It is called modified as we replace sextuple $(a_1,b_1,c_1,d_1,e_1,f_1)$ of the band data in Definition \ref{defn:bd}  by the triple $(l,m,n)$.
\end{remark}

Now we compute the explicit Macaulayfication $M\left(\left(l,m,n\right),\lambda,1\right)$ using the generator diagram.
It turns out that we need to split into multiple cases according to the signs of $l$, $m$ and $n$.
We discuss each case separately and also find the corresponding matrix factorization as well.
\\
\\
\noindent\textbf{Case 1}:  $l=m=n=0$ and $\lambda=1$.
This is the "\emph{degenerate}" case. The other cases will be called non-degenerate.
This is the module that appeared in the previous subsection.
$$
M\left(\left(0,0,0\right),1,1\right) = \left<zx^2+x^2y, xy^2+y^2z, yz^2+z^2x\right>_A^\dagger = \left<xy+yz+zx\right>_A \cong A
$$
  It gives a trivial matrix factorization $(xyz)\cdot1 = 1\cdot (xyz)=xyz$ where we have $\operatorname{coker}\mathunderbar{\left(xyz\right)} = \operatorname{coker}\mathunderbar{\left(0\right)}\cong A$ as $A$-modules.
\\

\begin{figure}[h]
     \centering
     \begin{subfigure}[t]{0.45\textwidth}
         \includegraphics[scale=0.4]{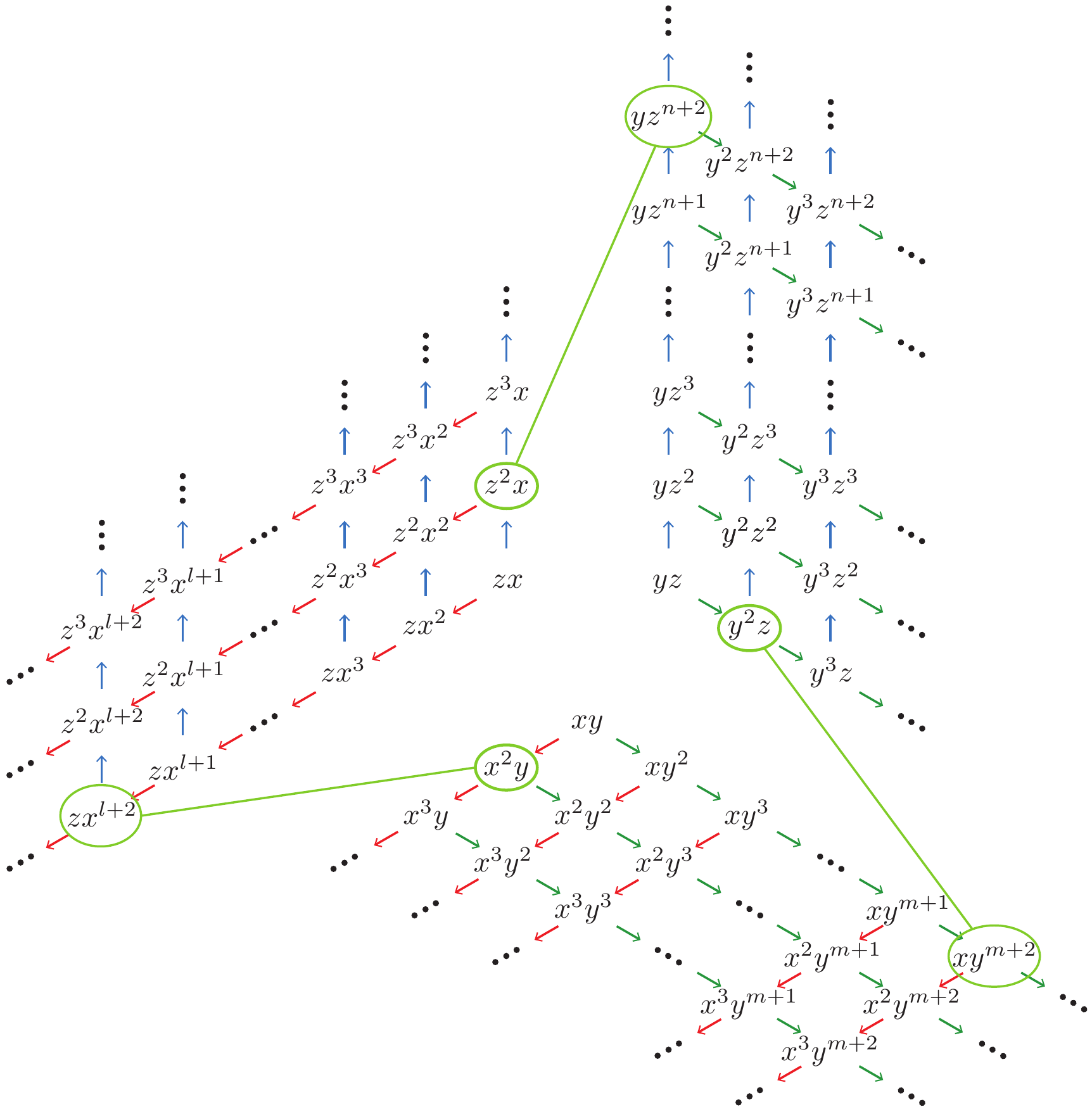}
         \caption{$l\ge0$, $m\ge0$, $n\ge0$}
         \label{fig:Case1GeneratorDiagram}
     \end{subfigure}
     \begin{subfigure}[t]{0.45\textwidth}
     \raisebox{4ex}{
         \includegraphics[scale=0.4]{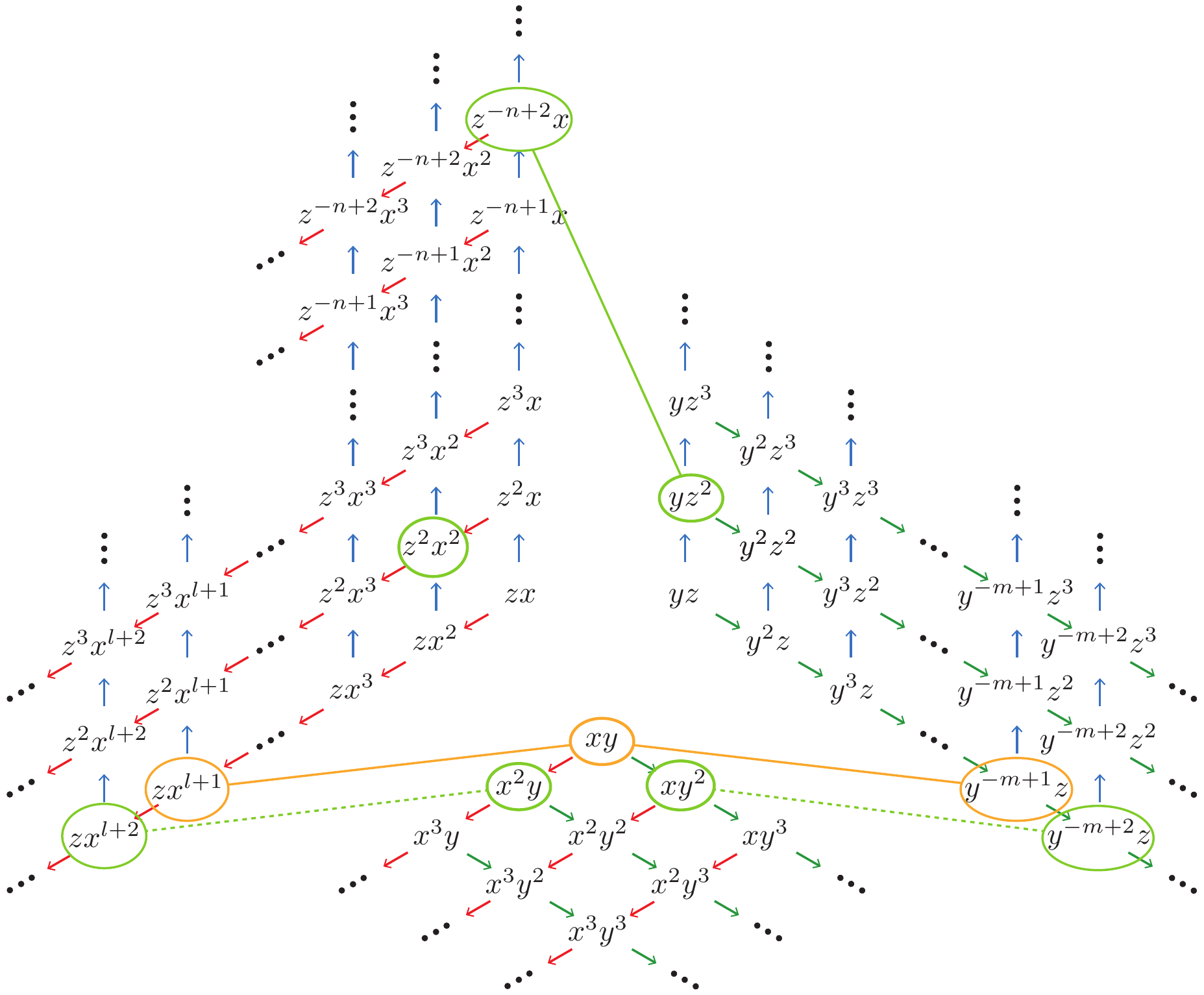}}
         \caption{$l>0$, $m<0$, $n\le 0$}
         \label{fig:Case2GeneratorDiagram}
     \end{subfigure}\\
     \begin{subfigure}[t]{0.5\textwidth}
     \hspace{1cm}
         \includegraphics[scale=0.4]{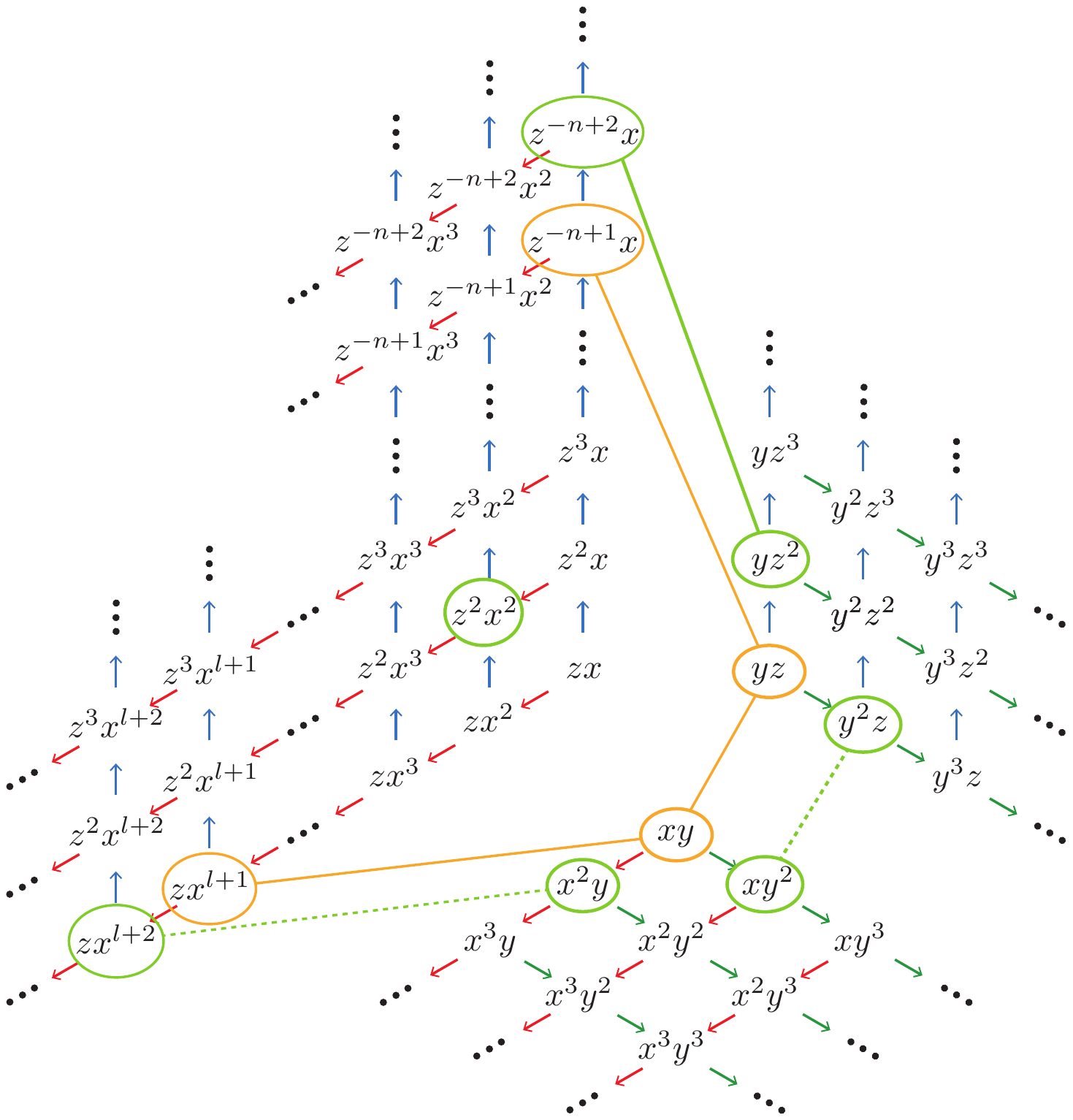}
         \caption{$l>0$, $m=0$, $n<0$}
         \label{fig:Case3GeneratorDiagram}
         \centering
     \end{subfigure}
        \caption{Generator diagram in some cases}
        \label{fig:CaseGeneratorDiagram}
\end{figure}

\noindent\textbf{Case 2}: $l,m,n\ge0$ and nondegenerate.

We have $l^+ = l$, $m^+ = m$, $n^+ = n$ and $l^- = m^- = n^- = 0$ and hence
$$
\tilde{M}\left(\left(l,m,n\right),\lambda,1\right) = \left<\lambda zx^{l+2}+x^2y, xy^{m+2}+y^2z, yz^{n+2}+z^2x\right>_A.
$$
The corresponding generator diagram is shown in Figure \ref{fig:Case1GeneratorDiagram}. Note that there are no Macaulayfying elements of $\tilde{M}\left(\left(l,m,n\right),\lambda,1\right)$ in $A^1$ and therefore it is already maximal Cohen-Macaulay. That is,
$$
M\left(\left(l,m,n\right),\lambda,1\right) = \left<\lambda zx^{l+2}+x^2y, xy^{m+2}+y^2z, yz^{n+2}+z^2x\right>_A.
$$

We find a matrix factorization arising from this module. Denoting $G_1 := \lambda zx^{l+2}+x^2y$, $G_2 := xy^{m+2}+y^2z$ and $G_3 := yz^{n+2}+z^2x$, then with the help of the generator diagram, we can find following $3$ relations (over $A$) among them:

$$
  \left\{
  \setlength\arraycolsep{1pt}
  \begin{array}{rrcrrcrrcc}
       z     & G_1 &   &         &     & - & \lambda x^{l+1} & G_3 & = & 0 \\
    -y^{m+1} & G_1 & + &       x & G_2 &   &                 &     & = & 0 \\
             &     & - & z^{n+1} & G_2 & + &               y & G_3 & = & 0
  \end{array}
  \right.
$$
From this we get matrix factors (over $\mathbb{C}[[x,y,z]]$)
$$
\varphi = 
\begin{pmatrix}
z & - y^{m+1} & 0 \\
0 & x & - z^{n+1} \\
- \lambda x^{l+1} & 0 & y
\end{pmatrix}
\quad\text{and}\quad
\psi = \left(1-\lambda x^l y^m z^n\right)^{-1}
\begin{pmatrix}
xy & y^{m+2} & y^{m+1} z^{n+1} \\
\lambda z^{n+1} x^{l+1} & yz & z^{n+2} \\
\lambda x^{l+2} & \lambda x^{l+1} y^{m+1} & zx
\end{pmatrix}
$$
of $xyz$ which satisfy $\varphi \psi = \psi \varphi = xyz I_3$. Here we computed $\psi$ using
the matrix adjoint $\operatorname{adj}\varphi$ of $\varphi$ which satisfies the relation $\varphi \cdot
\operatorname{adj}\varphi = \operatorname{adj}\varphi \cdot \varphi = \left(\det\varphi\right)I_3$.

Then we have
$$
\operatorname{coker}\mathunderbar{\varphi} \cong M\left(\left(l,m,n\right),\lambda,1\right)
$$
as $A$-modules. Indeed we will prove this rigorously in Theorem \ref{thm:MFFromModule} in a more general situation.
\\
\\
\noindent\textbf{Case 3}: $l>0$, $m<0$ and $n \le 0$.

We have
$$
\tilde{M}\left(\left(l,m,n\right),\lambda,1\right) = \left<z^2 x^2, \lambda zx^{l+2}+x^2y, xy^2+y^{-m+2}z, yz^{2}+z^{-n+2}x\right>_A.
$$
and the generator diagram given in Figure \ref{fig:Case2GeneratorDiagram}. It is enough to add one Macaulayfying element $\lambda zx^{l+1} + xy + y^{-m+1} z$ of $\tilde{M}\left(\left(l,m,n\right),\lambda,1\right)$ in $A^1$, which is marked in the picture by yellow circles and edges among them: 
$$
M\left(\left(l,m,n\right),\lambda,1\right) = \left<\lambda zx^{l+1} + xy + y^{-m+1} z, yz^2 + z^{-n+2}x,
z^2 x^2\right>_A
$$
where the order of the generators have been adjusted to obtain some desired form of matrix factorization. Note also that if $l=1$ or $n=0$, the third generator is redundant. But we do not exclude it to consistently get a $3\times3$ matrix in any cases below.

We can find relations among the above generators as before which gives the matrix factors
$$
\varphi = 
\begin{pmatrix}
z & 0 & 0 \\
-y^{-m} & x & 0 \\
-\lambda x^{l-1} & -z^{-n} & y
\end{pmatrix}
\quad\text{and}\quad
\psi =
\begin{pmatrix}
xy & 0 & 0 \\
y^{-m+1} & yz & 0 \\
y^{-m}z^{-n}+\lambda x^l & z^{-n+1} & zx
\end{pmatrix}.
$$
of $xyz$ satisfying $\operatorname{coker}\mathunderbar{\varphi} \cong M\left(\left(l,m,n\right),\lambda,1\right)$
as $A$-modules.
Also, when $l>0$, $m<0$ and $n>0$, one can compute its Macaulayfication in a similar way.
\\
\\
\noindent\textbf{Case 4}: $l>0$, $m=0$ and $n<0$.

The module and the generator diagram are given respectively by
$$
\tilde{M}\left(\left(l,0,n\right),\lambda,1\right) = \left<z^2 x^2, \lambda zx^{l+2}+x^2y, xy^2+y^2z, yz^2+z^{-n+2}x\right>_A.
$$
and Figure \ref{fig:Case3GeneratorDiagram}. Here we have a Macaulayfying element $\lambda zx^{l+1} + xy + yz + z^{-n+1}x$ of $\tilde{M}\left(\left(l,0,n\right),\lambda,1\right)$ in $A^1$ and no more, which implies
$$
M\left(\left(l,0,n\right),\lambda,1\right) = \left<\lambda zx^{l+1} + xy + yz + z^{-n+1}x, yz^2 + z^{-n+2}x, z^2 x^2\right>_A.
$$
This gives the matrix factors of $xyz$ as
$$
\varphi = 
\begin{pmatrix}
z & 0 & 0 \\
-1 & x & 0 \\
-\lambda x^{l-1} & -z^{-n} & y
\end{pmatrix}
\quad\text{and}\quad
\psi =
\begin{pmatrix}
xy & 0 & 0 \\
y & yz & 0 \\
z^{-n} + \lambda x^l & z^{-n+1} & zx
\end{pmatrix}.
$$
satisfying $\operatorname{coker}\mathunderbar{\varphi} \cong M\left(\left(l,0,n\right),\lambda,1\right)$
as $A$-modules.
\\
\\
\noindent\textbf{The other cases}: 
\begin{itemize}
\item The case $l,m,n\le 0$ and nondegenerate can be treated in the same way as in Case 2.
\item When the word $(l,m,n)$ contains two nonzero elements with different signs, it is further subdivided. In the case $(0,-,+)$ and $(-,+,0)$, there appears $+ \rightarrow 0 \rightarrow -$ in cyclic order. These cases are equivalent to Case 4 (for $(+,0,-)$) if we change the order of $x$, $y$ and $z$ cyclically (and the auxiliary parameter $\lambda$ should be handled appropriately), and we may rotate the corresponding generator diagram accordingly.
\item 
In the remaining cases, the word has a subword of the form $+\rightarrow -$ in cyclic order. In those cases, we can cycle the order of $x$, $y$ and $z$ and make some modifications to proceed as in Case 3.
\end{itemize}

\subsection{Correction number and a uniform expression of matrix factorizations}
Let us define the notion of `\emph{correction number}' which enables us to write matrix factorizations in the above multiple cases into a single formula.  This correction number turns out to be the data needed for mirror symmetry correspondence as well.

\begin{defn}
Define the \emph{correction number} $\varepsilon$ of the band word $(l,m,n)$ by
$$
\varepsilon := \varepsilon(l,m,n):=
\left\{
\setlength\arraycolsep{0pt}
\begin{array}{ccl}
2 & \quad\text{if} & \quad l,m,n\ge0, \\[2mm]
1 & \quad\text{if} & \quad \left\{
                           \setlength\arraycolsep{1pt}
                           \begin{array}{cccc}
                           l>0, & m>0, & n<0  & \text{or} \\
                           l>0, & m<0, & n>0  & \text{or} \\
                           l<0, & m>0, & n>0, &
                           \end{array}
                           \right. \\[6mm]
0 & \quad\text{if} & \quad \left\{
                           \setlength\arraycolsep{1pt}
                           \begin{array}{ccccc}
                           l>0,   & m\le0, & n\le0, & (m,n)\ne(0,0)  & \text{or} \\
                           l\le0, & m>0,   & n\le0, & (n,l)\ne(0,0)  & \text{or} \\
                           l\le0, & m\le0, & n>0,   & (l,m)\ne(0,0), &
                           \end{array}
                           \right. \\[6mm]
-1 & \quad\text{if} & \quad l,m,n\le0,\ (l,m,n)\ne(0,0,0)
\end{array}
\right.
$$
\end{defn}

\begin{prop}\label{prop:Rank1MFfromModule}
We add the correction number to define a new triple $(l', m', n')$ from $(l,m,n)$;
$$ \big(l', m', n'\big) = \big(l,m,n\big) + \big(\varepsilon(l,m,n),\varepsilon(l,m,n),\varepsilon(l,m,n)\big).$$
Then, for each non-degenerate modified band datum $\left((l,m,n),\lambda,1\right)$, 
the Cohen-Macaulay module
$$M\left((l,m,n),\lambda,1\right)$$
is equivalent to the following matrix factorization $(\varphi,\psi)$ over $\mathbb{C}[[x,y,z]]$ under Eisenbud's theorem.
$$
\varphi := 
\begin{pmatrix}
z & - y^{m'-1} & -\lambda^{-1}x^{-l'} \\
-y^{-m'} & x & - z^{n'-1} \\
- \lambda x^{l'-1} & -z^{-n'} & y
\end{pmatrix}
\quad\text{and}
$$
$$
\psi :=
u^{-1}
\begin{pmatrix}
xy & y^{m'} + \lambda^{-1}z^{-n'}x^{-l'} & \lambda^{-1}x^{-l'+1} + y^{m'-1}z^{n'-1} \\
y^{-m'+1} + \lambda z^{n'-1}x^{l'-1} & yz & z^{n'} + \lambda^{-1} x^{-l'}y^{-m'} \\
\lambda x^{l'} + y^{-m'}z^{-n'} & z^{-n'+1}+\lambda x^{l'-1}y^{m'-1} & zx
\end{pmatrix}
$$
where
$$
u:=1 - \lambda x^{l'-2}y^{m'-2}z^{n'-2} - \lambda^{-1} x^{-l'-1}y^{-m'-1}z^{-n'-1}.
$$
Here, we are using a non-standard notation that {\em $x^a$, $y^a$ or $z^a$ is considered as zero if $a<0$}.

Namely, we have  $$\varphi \cdot \psi = \psi  \cdot \varphi = xyz  \cdot I_3,  \;\;\;\;  \operatorname{coker}\mathunderbar{\varphi} \cong M\left(\left(l,m,n\right),\lambda,1\right)  \;\;\;\; \textrm{as} \; A\textrm{-modules}
$$

\end{prop}
\begin{defn}
The matrix $\varphi$  is called the
\emph{canonical form} of matrix factors arising from the band datum $\left(\left(l,m,n\right),\lambda,1\right)$
and denote it by $\varphi\left(\left(l',m',n'\right),\lambda,1\right)$.
\end{defn}
These definitions and propositions will be generalized into the higher rank case in Section \ref{sec:higherrank}
and the rigorous proof will be given there. 

If one of the entry in $\varphi$ or $\psi$ is in $\mathbb{C}^*$, then such a matrix factorization can be further reduced to the size $2 \times 2$
which describes the actual indecomposable matrix factorizations. (In the case 3 above, some of module generators were redundant).
In this way, we recover the following theorem of Burban and Drozd, and we leave as an exercise.
\begin{thm}\cite[Proposition 9.6]{BD17}
Let $M$ be a  rank one object of $\mathrm{CM}^{\operatorname{lf}}(A)$ which is not isomorphic to $A$ itself. Then its minimal number of generators is equal to two or three.
        \begin{enumerate}
                \item Assume $M$ is generated by two elements. Then there exist a bijection $\{u, v, w\}\rightarrow \{x, y, z\}$, a pair of integers $p, q \geq 1$ and $\lambda\in\bcc^*$ such that $M=\cok(A^2\xrightarrow{\theta_i((p, q), \lambda))}A^2)$ for some $1\leq i \leq 3$, where
                $$\theta_1((p, q), \lambda) = \begin{pmatrix} u & 0 \\ v^p+\lambda w^q & vw\end{pmatrix}, \quad \theta_2((p, q), \lambda) = \begin{pmatrix} \lambda u +v^pw^q & w^{q+1} \\ v^{p+1} & vw \end{pmatrix}$$ and $\theta_3((p, q), \lambda) = \theta_1((p, q), \lambda)^{tr}$.
                \item Assume $M$ is generated by three elements. Then there exist a bijection $\{u, v, w\}\rightarrow \{x, y, z\}$, a triple of integers $m, n, l\geq 1$ and $\lambda\in \bcc^*$ such that $M = \cok(A^3\xrightarrow{\theta_i((m, n, l), \lambda))} A^3)$ for some $4\leq i \leq 7$, where
                $$\theta_4((m, n, l), \lambda) = \begin{pmatrix} u & w^l & 0 \\ 0 & v & u^m \\ \lambda v^n & 0 & w \end{pmatrix}, \quad \theta_5((m, n, l), \lambda) = \begin{pmatrix} u & w^l & \lambda v^n \\ 0 & v & u^m \\ 0 & 0 & w \end{pmatrix},$$ $\theta_6((m, n, l), \lambda) = \theta_4((m, n, l), \lambda)^{tr}$ and $\theta_7((m, n, l), \lambda) = \theta_5((m, n, l), \lambda)^{tr}$.
        \end{enumerate}
\end{thm}

\section{Immersed Lagrangians of Rank $1$ in pair of pants}\label{sec:MFFromRank1Lagrangian}
In this section, we describe immersed Lagrangians of rank $1$ in $\mathcal{P}$, and compute the
corresponding mirror matrix factorization under the localized mirror functor (from the Seidel Lagrangian $\bL$ in $\mathcal{P}$). 
Such immersed Lagrangians will be indexed by what we call the loop data $((l',m',n'), \lambda')$ where $(l',m',n') \in \Z^3, \lambda' \in \C^*$ called a holonomy.
Explicit correspondence with the Cohen-Macaulay modules will be given in the next section.

%
%

%
%
%
%
%
%

%
\begin{figure}[h]
        \centering
        \begin{subfigure}[t]{0.45\textwidth}
                \includegraphics[scale=0.35]{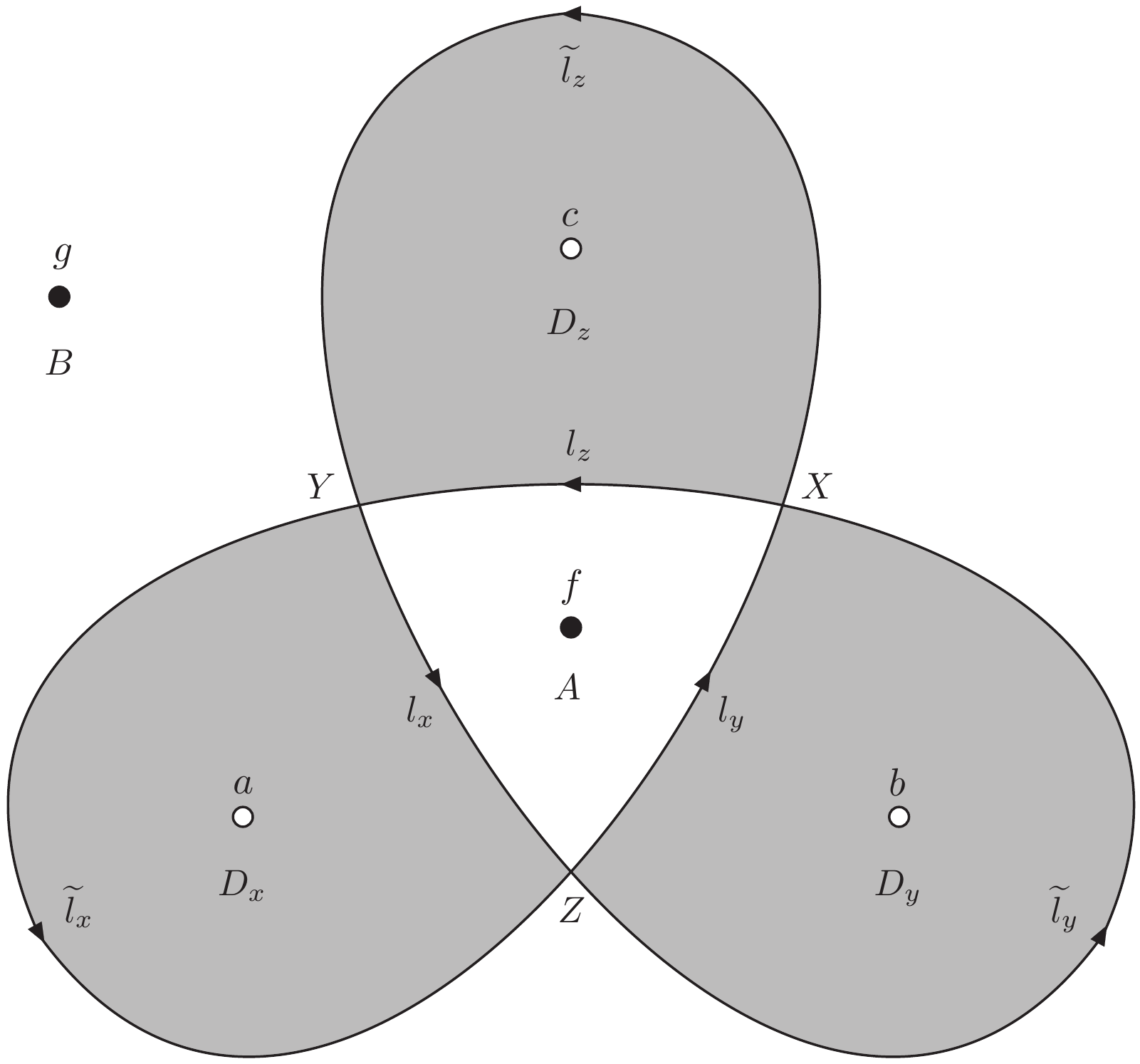}
                \centering
                \caption{Seidel Lagrangian centered at $f$}
                       \end{subfigure}
        \begin{subfigure}[t]{0.45\textwidth}
                \includegraphics[scale=0.35]{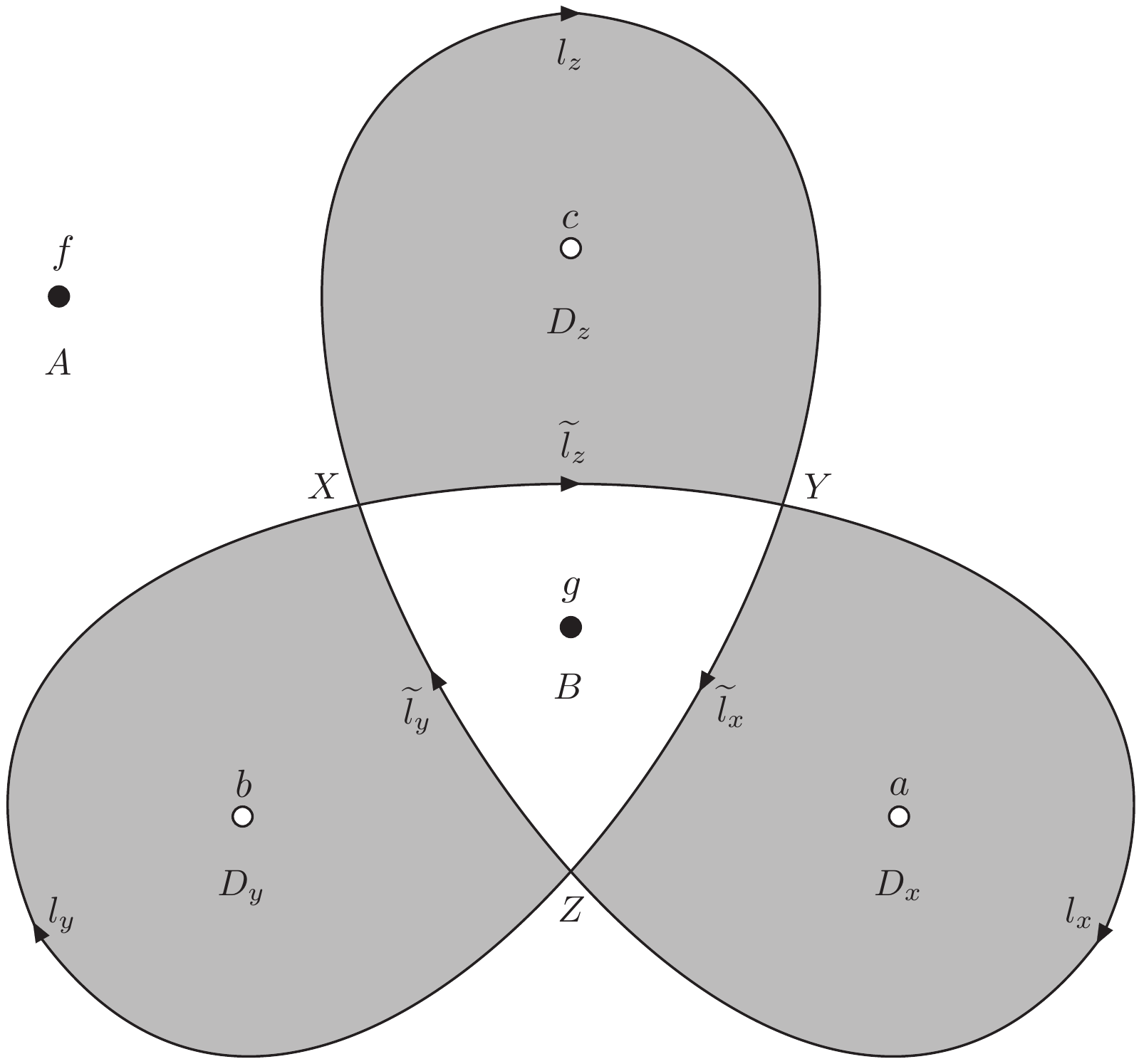}
                \centering
                \caption{Seidel Lagrangian centered at $g$}
                      \end{subfigure}
                      
        \begin{subfigure}[t]{0.45\textwidth}
                \includegraphics[scale=0.35]{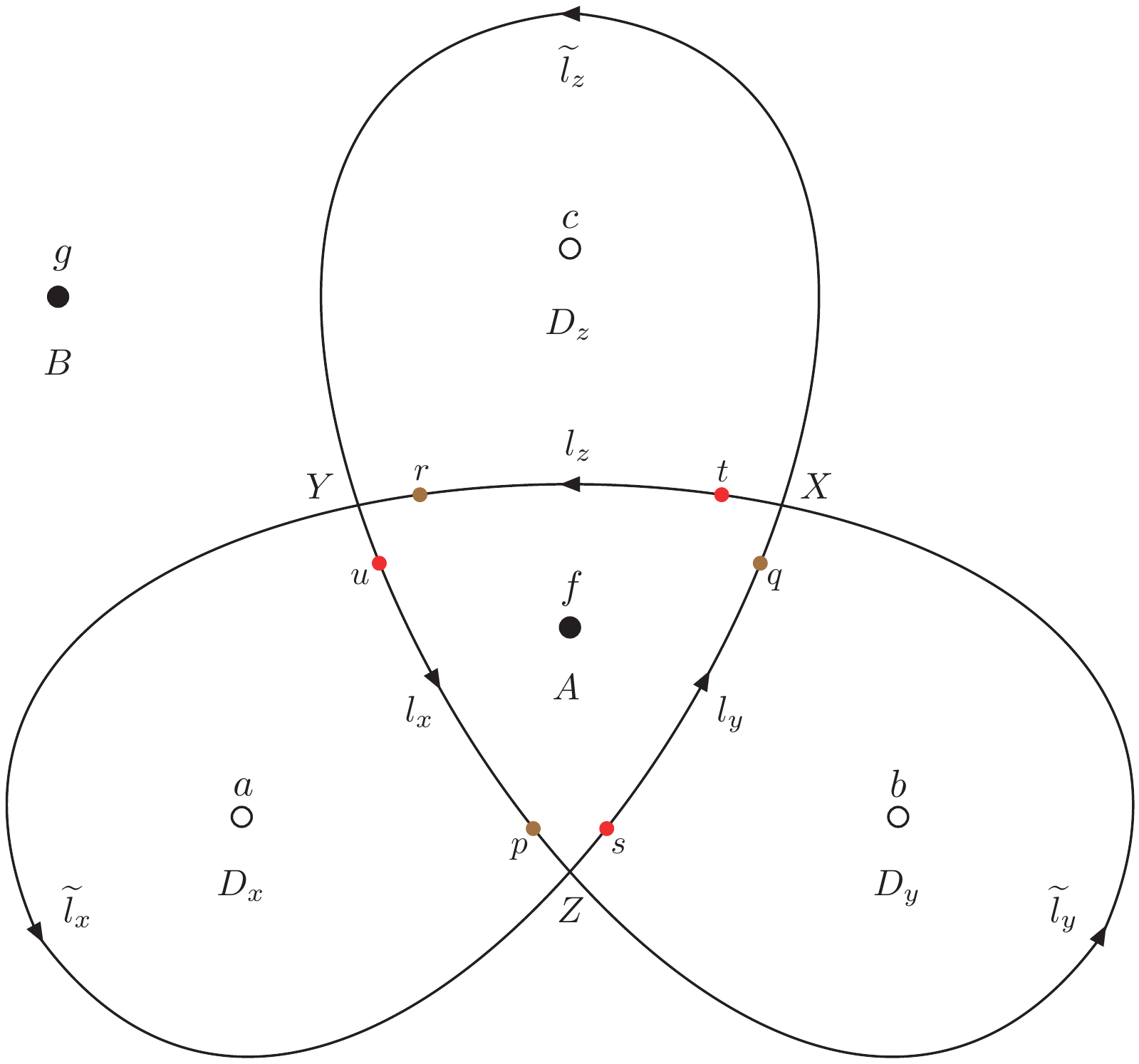}
                \centering
                \caption{Six points on $\bll$.}
                \label{fig:SixPoints}
        \end{subfigure}
        \begin{subfigure}[t]{0.45\textwidth}
                \includegraphics[scale=0.35]{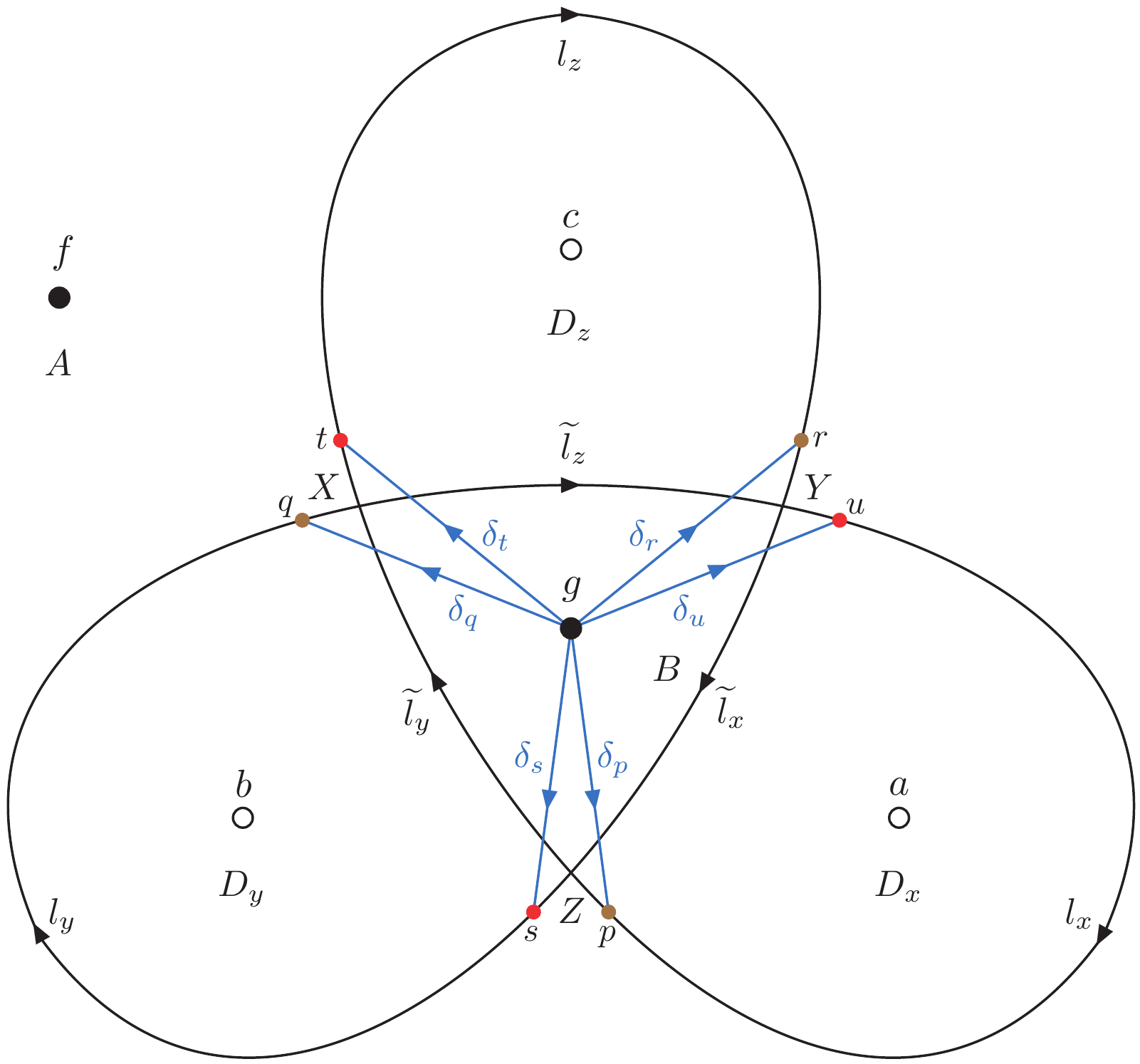}
                \centering
                \caption{Paths $\delta_\square$ with $\square\in\{p, q, r, s, t, u\}$.}
                \label{fig:PathFromBase}
        \end{subfigure}
                       \caption{}
        \label{fig:SeidelLagFigures}
\end{figure}

\subsection{Immersed Lagrangians}
The pair of pants, and the Seidel Lagrangian are given in Figure \ref{fig:F1}, but we will
work with different drawings (given by Figure \ref{fig:SeidelLagFigures}).
Let us first explain them.  The pair of pants $\mathcal{P}$ can be identified with the $\mathbb{P}^1$ with three distinct punctures $a, b, c$. We will consider the fundamental group with base point  $g \in \C$. 
For convenience, we take another identification of $\mathcal{P}$ with the plane $\mathbb{R}^2$ with 3 punctures, by sending $g$ to the infinity and thus having $f$ at the center. See the left side of Figure \ref{fig:SeidelLagFigures}, where the Seidel Lagrangian is also drawn and self-intersection points $X,Y,Z$ are marked.
Alternatively, we can send $f$ to the infinity, and have $g$ at the center to obtain 
the right side of Figure \ref{fig:SeidelLagFigures}.

Note that $\cpp\setminus\bll$ has five connected components, each of which is labeled as $D_x, D_y, D_z$, $A$, $B$ in Figure \ref{fig:SeidelLagFigures}. For instance, $D_x$ is the only component of which $x$ is not contained in the closure. Also, $A$ and $B$ are the component containing $f$ and $g$, respectively.
%


We label the edges of the Seidel Lagrangian as follows.
Namely,  $\bll\setminus \{X, Y, Z\}$ has six components. We denote them by \begin{align*}
        &l_x = \overline{D_x}\cap\overline{A}, \quad l_y = \overline{D_y}\cap\overline{A}, \quad l_z = \overline{D_z}\cap\overline{A},\\ &\tilde{l_x}=\overline{D_x}\cap\overline{B}, \quad \tilde{l_y}=\overline{D_y}\cap\overline{B}, \quad \tilde{l_x}=\overline{D_z}\cap\overline{B}.
\end{align*}
These are illustrated in Figure \ref{fig:SeidelLagFigures}. 
We may view them as paths on $\cpp$, for instance, $l_x$ is a path from $Y$ to $Z$
(along the orientation of $\bL$).
%
Now, we will describe a family of immersed Lagrangians as in Figure \ref{fig:loopex} that will correspond to Burban-Drozd's indecomposable
maximal Cohen-Macaulay modules. 
Here, we only describe the rank one part, and the general cases will be handled in Section \ref{sec:MfFromLag}.

First, put two points on the line $l_x$ along the orientation and call them $u$ and $p$ in order. Also put four points, $s, q$ on $l_y$ and $t, r$ on $l_z$. See Figure \ref{fig:SixPoints}. Then, for each  point $\square\in\{p, q, r, s, t, u\}$, define a path $\delta_\square$ from the base point $g$ to the  as in Figure \ref{fig:PathFromBase} and denote by  $\overline{\delta_\square}$ the reversing of it.
 

\begin{figure}[H]
        \centering
       \begin{subfigure}[t]{0.45\textwidth}
                \includegraphics[scale=0.35]{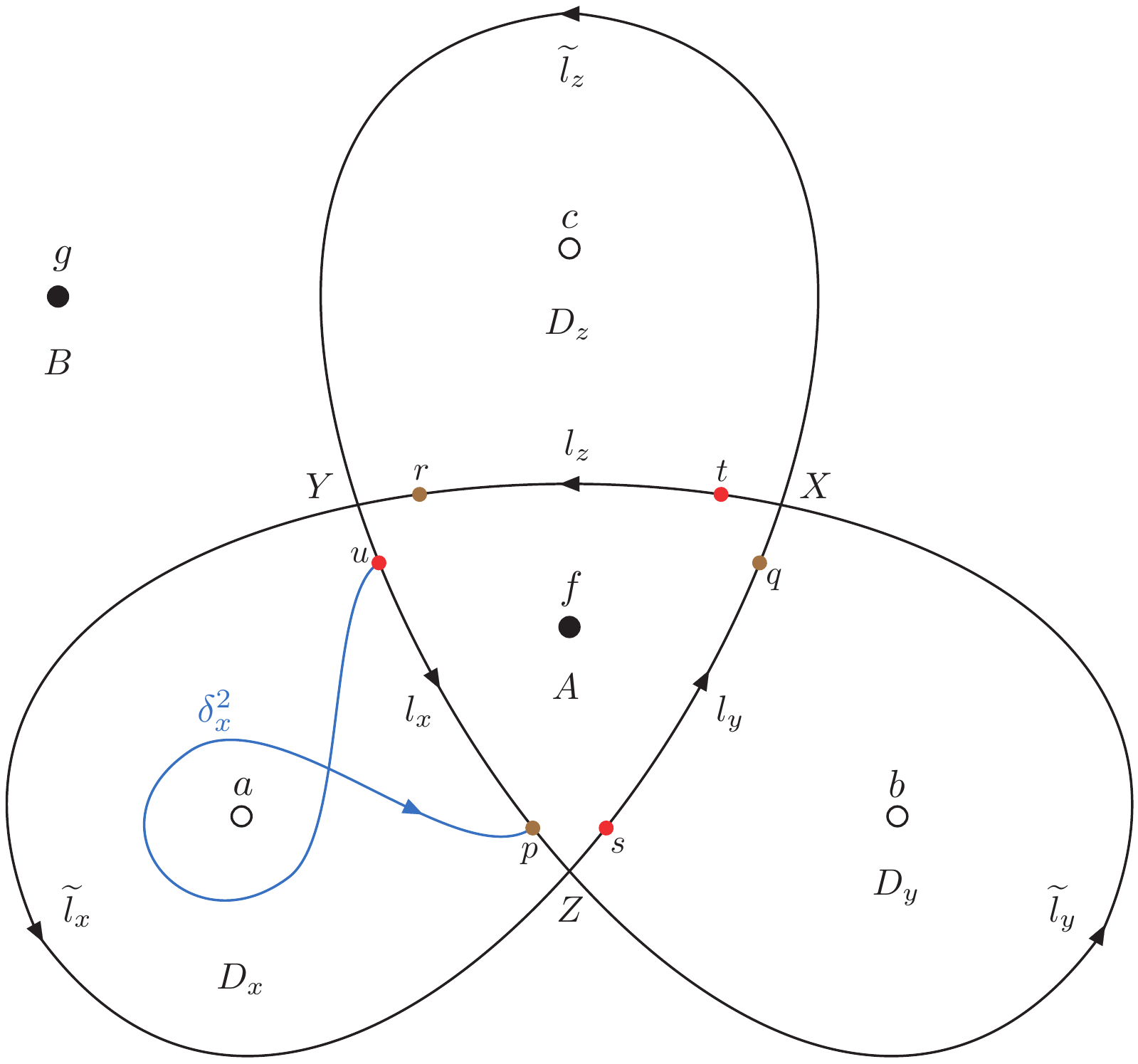}
                \centering
                \caption{The path $\delta^2_x$.}
                \label{fig:DeltaPosTwoA}
        \end{subfigure}
            \begin{subfigure}[t]{0.45\textwidth}
                \includegraphics[scale=0.35]{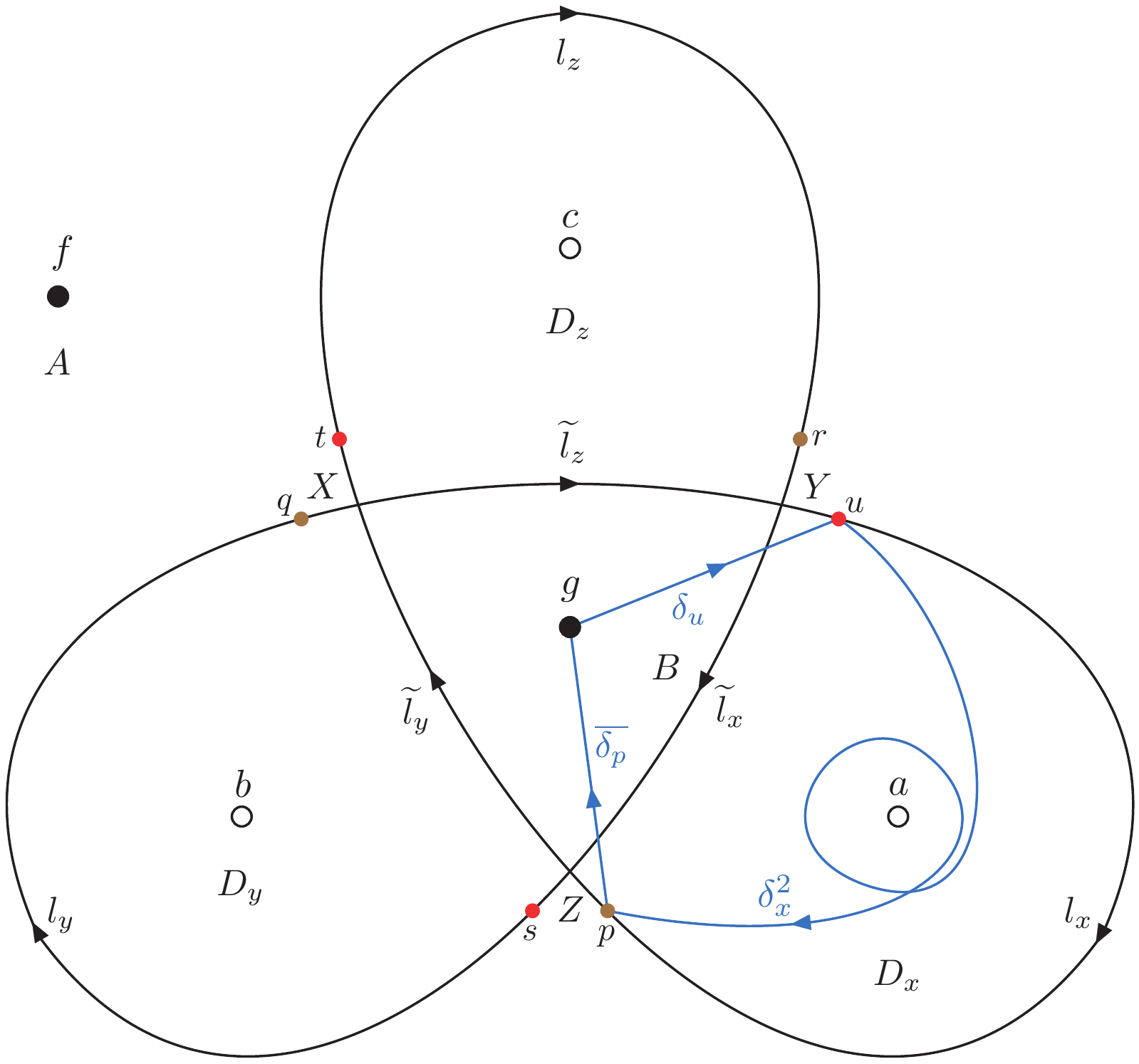}
                \centering
                \caption{The loop $\delta_u\cdot\delta^2_x\cdot\overline{\delta_p}$}
                \label{fig:DeltaPosTwoB}
        \end{subfigure}
     
    \begin{subfigure}[t]{0.45\textwidth}
                \includegraphics[scale=0.35]{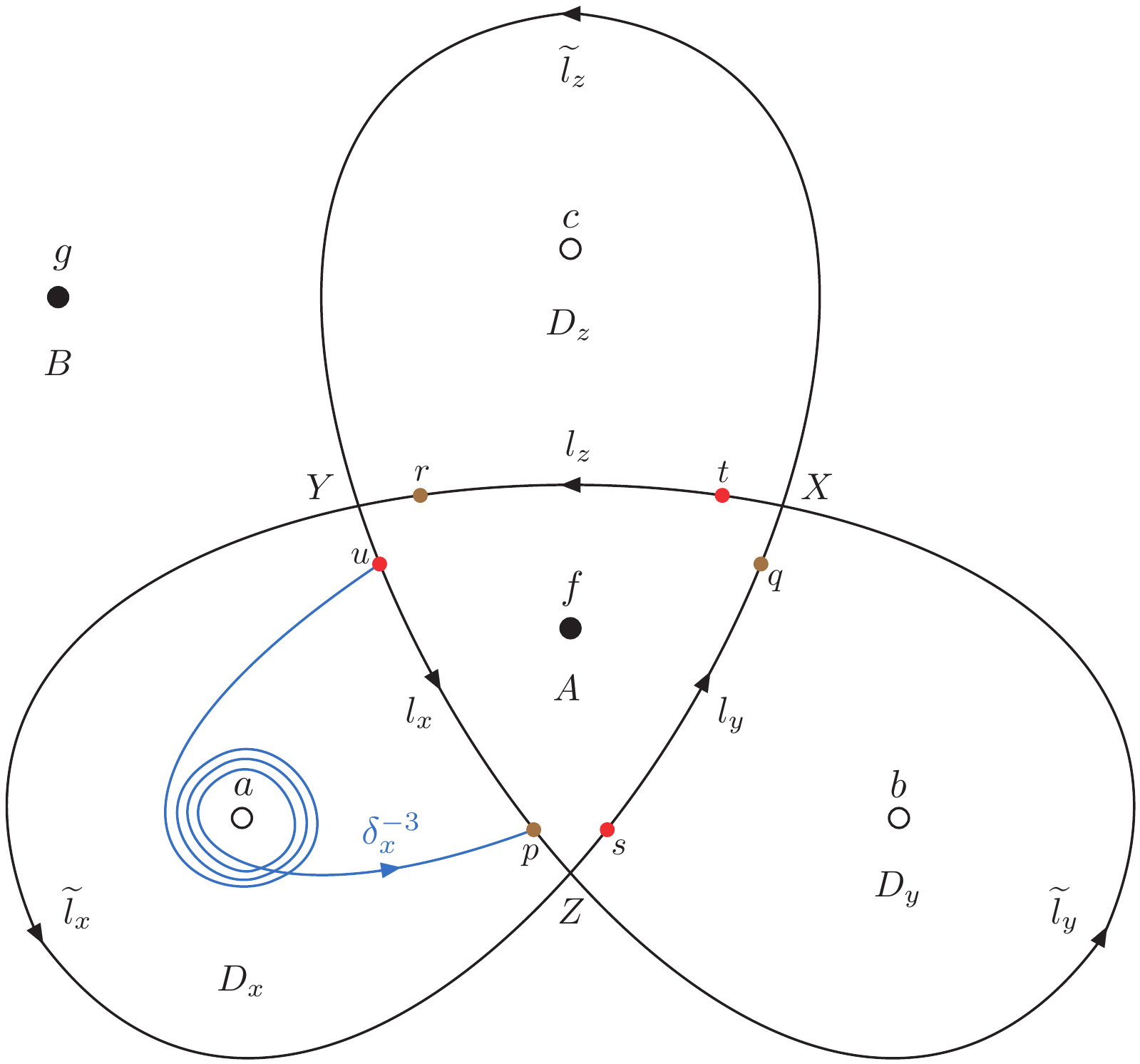}
                \centering
                \caption{The path $\delta^{-3}_x$.}
                \label{fig:DeltaNegThreeA}
        \end{subfigure}
              \begin{subfigure}[t]{0.45\textwidth}
                \includegraphics[scale=0.35]{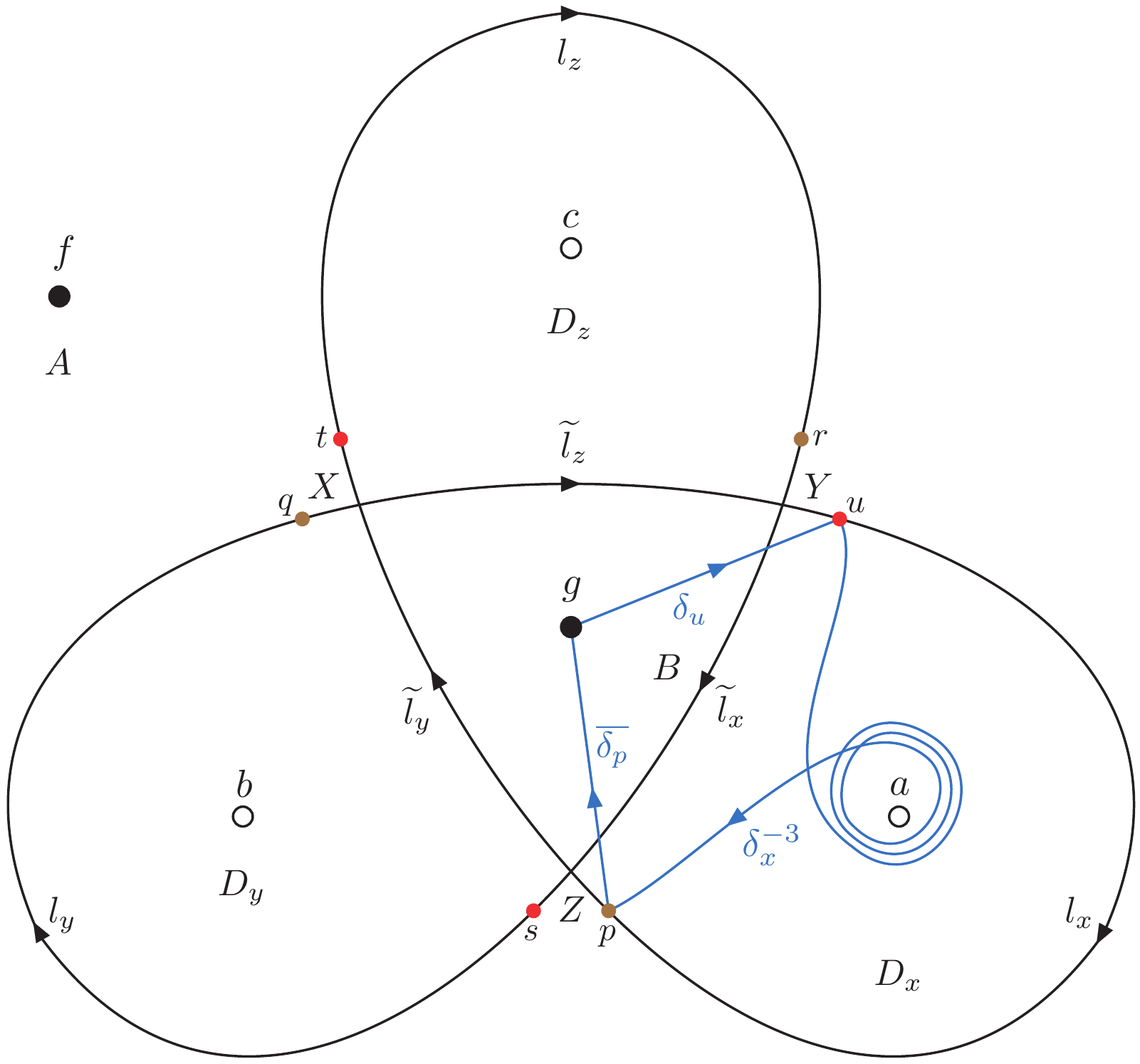}
                \centering
                \caption{The loop $\delta_u\cdot\delta^{-3}_x\cdot\overline{\delta_p}$}
                \label{fig:DeltaNegThreeB}
        \end{subfigure}   
          \caption{}
           \label{fig:SS2}
\end{figure}

\begin{defn}
We define the path $\delta_x^\nu$ for each $\nu \in \mathbb{Z}$ as the path from $u$ to $p$
winding around the puncture in the region $D_x$. See Figure \ref{fig:DeltaPosTwoA} for the path $\delta_x^2$ and
Figure \ref{fig:DeltaNegThreeA} for the path $\delta_x^{-3}$. 

To be more precise, we use the fundamental group $\pi_1(\cpp) = \left<\alpha, \beta, \gamma | \alpha\beta\gamma = 1\right>$.
We require that the homotopy class of the concatenated loop $[\delta_u\cdot\delta_x^{\nu}\cdot \overline{\delta_p}]$  is $\alpha^\nu$
(see  Figure \ref{fig:DeltaPosTwoB} and \ref{fig:DeltaNegThreeB}). Geometrically, $\nu$ may be understood as a signed intersection number between the loop and the shortest path from  the puncture $a$ to the center point $f$.

One can define paths $\delta_y^\nu$ from $s$ to $q$ and $\delta_z^\nu$ from $t$ to $r$ in a similar way. Next, we define a path  $\Delta_{\square_1\square_2}$  to be a line segment from $\square_1$ to $\square_2$ contained in $\overline{A}$, where $\square_1, \square_2\in\{p, q, r, s, t, u\}$.
\end{defn}

\subsection{Immersed Lagrangians $\boldsymbol{L((l', m', n'), \lambda')}$ and their mirror matrix factorizations}
We are ready to define our rank $1$ immersed Lagrangians.
\begin{defn}\label{def:RankOneLagrangian} 
        For three integers $l', m', n'$ and nonzero $\lambda'\in\bcc^*$, define an immersed Lagrangian submanifold $L((l', m', n'), \lambda')$ to be a smoothing of the loop $\delta_x^{l'}\cdot \Delta_{ps} \cdot \delta_y^{m'} \cdot \Delta_{qt} \cdot \delta_z^{n'} \cdot \Delta_{ru}$ whose holonomy $\lambda'$ is concentrated at a point in $\delta_x^{l'}$.
\end{defn}
See Figure \ref{fig:loopex} for the curve $L( (3,-2,2), \lambda' )$ and Figure \ref{fig:loopdata} for the general cases.
We will call the pair $((l', m', n'), \lambda')$ a \emph{loop data of rank $1$}. Note that the Lagrangian $L((l', m', n'), \lambda')$ has a free homotopy class $[\alpha^{l'}\beta^{m'}\gamma^{n'}]$.

\begin{figure}
\includegraphics[scale=0.5]{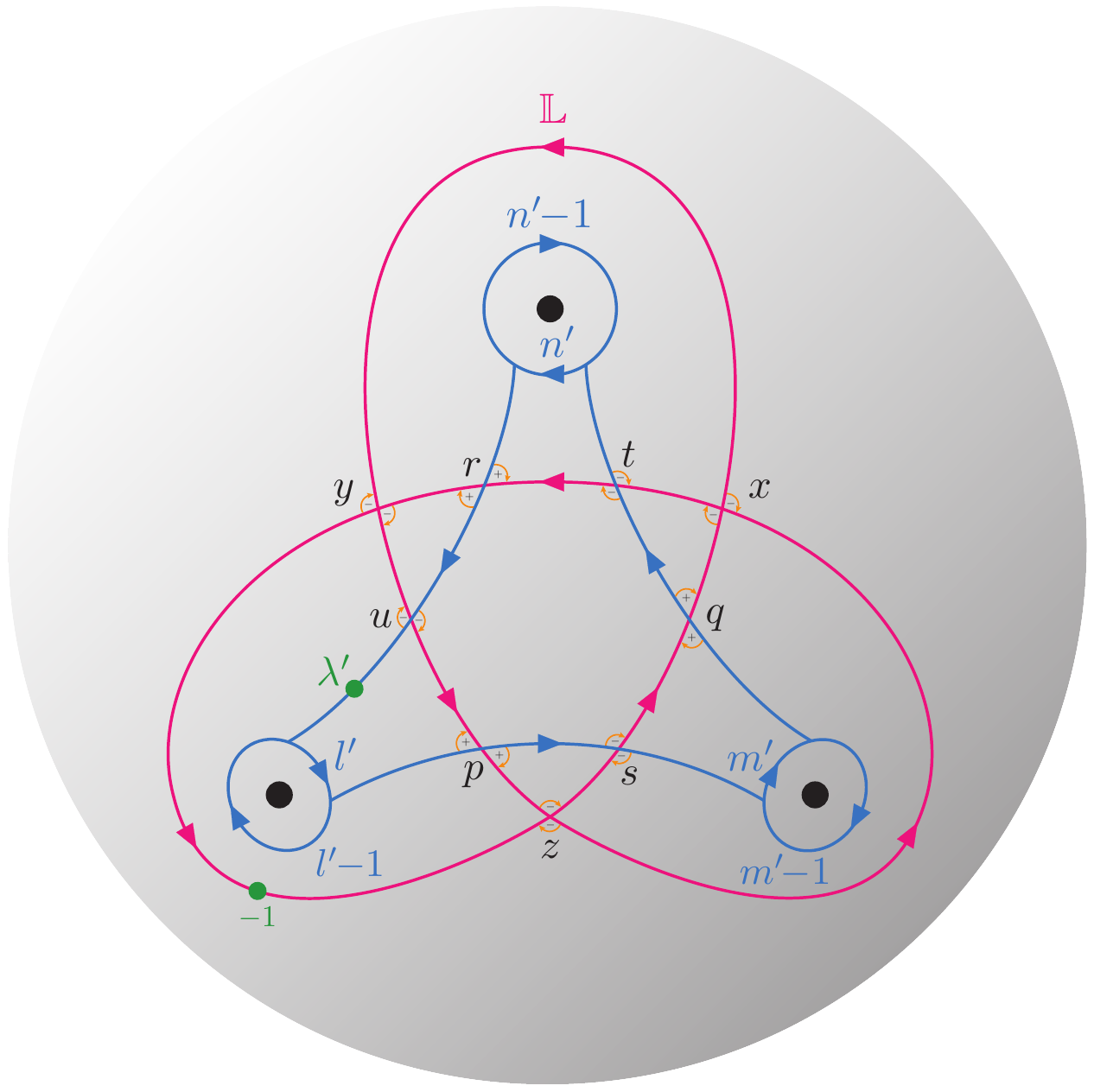}
\centering
\caption{$L((l', m', n'), \lambda')$ in the pair of pants}
\label{fig:loopdata}
\end{figure}

\begin{defn}\label{defn:normal1}
        The triple $(l', m', n')=(w_1', w_2', w_3')$ is said to be {\em normal} if the following are satisfied. Here the index $i$ is regarded as an element of $\bzz_3$.
        \begin{itemize}
                \item If $w_i'=1$, then $w_{i-1}'$, $w_{i+1}'\leq 0$.
                \item If $w_i'=0$ ,then $w_{i-1}'\leq -1, w_{i+1}'\geq 1$ or $w_{i-1}'\geq 1, w_{i+1}'\leq -1$ or $w_{i-1}'$, $w_{i+1}'\geq 1$.
                \item $(l', m', n')\neq (-1, -1, -1)$.
        \end{itemize}
\end{defn}

We need the normality for two purposes. The first is to choose a unique representative in each free homotopy class of maps from $S^1$ to $\mathcal{P}$. To prove this, we introduce some notions and an easy lemma.

\begin{defn}\label{def:LoopWordEquivalenceRankOne}
        Two triples $(l'_1, m'_1, n'_1)$ and $(l'_2, m'_2, n'_2)$ are said to be {\em equivalent} if one can obtained from the other by performing the following operations:
        \begin{itemize}
                \item if one of $l'_1, m'_1, n'_1$ is zero, then add $(1, 1, 1)$ to the triple.
                \item if one of $l'_1, m'_1, n'_1$ is one, then subtract $(1, 1, 1)$ from the triple.
                \item change $(-1, -1, -1)$ to $(2, 2, 2)$ and vice versa.
        \end{itemize}
        A triple $(l', m', n')$ is said to be {\em non-essential} if it is equivalent to a triple of the form $(n, 0, 0), (0, n, 0), or (0, 0, n)$ for some $n\in\bzz$. If a triple is not non-essential, we say it is {\em essential}.
\end{defn}
For more general definition of equivalence, see Definition \ref{def:LoopWord}. In there, we use the following property as a definition.

\begin{lemma}\label{lem:TripleEqui}
        If two triples $(l'_1, m'_1, n'_1)$ and $(l'_2, m'_2, n'_2)$ are equivalent, then the free homotopy classes $[\alpha^{l'_1}\beta^{m_1'}\gamma^{n'_1}]$ and $[\alpha^{l'_2}\beta^{m_2'}\gamma^{n'_2}]$ are the same.
\end{lemma}
\begin{proof}
        By the definition, it is enough to show that the free homotopy class does not change by each operation in definition \ref{def:LoopWordEquivalenceRankOne}. Indeed, the first two operations do not change the class because of the relation $\alpha\beta\gamma=1$. The third one also does not change the class as
$
\left[\alpha^{-1}\beta^{-1}\gamma^{-1}\right]
= \left[(\beta\gamma)(\gamma\alpha)(\alpha\beta)\right]
= \left[\alpha^2\beta^2\gamma^2\right].
$
\end{proof}

Now we prove the main lemma.

\begin{lemma}\label{lem:normal}
        For the loops $\alpha,\beta,\gamma$ in $\mathcal{P}$, consider the following (sub)set of equivalence classes of essential free loops in $[S^1,\mathcal{P}]$ up to free homotopy:
        $$\big\{ \alpha^{l'}\beta^{m'}\gamma^{n'}\big\} \big/ \sim $$
        Each equivalence class has the unique element which is normal.
\end{lemma}
\begin{proof}
        Let us first prove existence of a normal representative. Let $(l', m', n')$ be an essential non-normal triple. Then it is $(-1, -1, -1)$ or one of entries is $0$ or $1$. The triple $(-1, -1, -1)$ is by definition equivalent to a normal triple, $(2, 2, 2)$. If $l'=0$, the assumption implies  $m', n'<0$. Then, it is equivalent to  a normal triple $(1, m'+1, n'+1)$. Also, if $l'=1$, then it is equivalent to $(0, m'-1, n'-1)$, which should be normal. If not, the last argument implies normality of $(1, m', n')$. The case of $m'=0, 1$ or $n'=0, 1$ is proved by the same say.

        Now we prove uniqueness of a normal representative. Suppose that two distinct normal triples $(l'_1, m'_1, n'_1)$, $(l'_2, m'_2, n'_2)$ which give freely homotopic loops. Consider a group homomorphism
        \begin{align*}
                \left<\alpha, \beta, \gamma | \alpha\beta\gamma = 1\right>
                 & \rightarrow \left<\alpha, \beta | \alpha\beta\alpha^{-1}\beta^{-1}\right> \cong \bzz^2 \\
                \alpha^{l'}\beta^{m'}\gamma^{n'}&\mapsto(l'-n', m'-n').
        \end{align*}
        Since a group homomorphism preserves the conjugate action and a conjugacy classes of an element in an Abelian group is just a singleton, we have $l'_1-n'_1 = l'_2-n'_2$, $m'_1-n'_1 = l'_2-m'_2$. Thus $(l'_2, m'_2, n'_2) = (l'_1+r, m'_1+r, n'_1+r)$ for some $r\in \bzz$. For simplicity, let us write $(l', m', n')$ instead of $(l'_1, m'_1, n'_1)$ and assume $r> 0$.
        
        Note that a conjugacy class $w$ in the free group $\left<\alpha, \gamma\right>$ has a cyclically reduced representative $\alpha^{a_1}\gamma^{c_1}\cdots\alpha^{a_\nu}\gamma^{c_\nu}$ which is unique up to cyclic permutation. Define the length of $w$ as the integer $\nu$ and denote by $\len(w)$. Also define the length of $[\alpha^{l'}\beta^{m'}\gamma^{n'}]$ using the group isomorphism $\left<\alpha, \beta, \gamma | \alpha\beta\gamma\right> \cong \left<\alpha, \gamma\right>$. Then length is computed as follows.
        \begin{itemize}
                \item If $m'>0$, then $l', n'\neq 1$, and $\len(\alpha^{l'}\beta^{m'}\gamma^{n'}) = \len(\alpha^{l'-1}(\gamma^{-1}\alpha^{-1})^{m'-1}\gamma^{n'-1}) = m'$.
                \item If $m'=0$, then $\len(\alpha^{l'}\gamma^{n'})=1$.
                \item If $m'<0$, then $$\len(\alpha^{l'}\beta^{m'}\gamma^{n'}) = \len(\alpha^{l'}(\gamma\alpha)^{-m'}\gamma^{n'}) = \begin{cases} -m'+1 \text{ if } l'n'\neq 0 \\ -m' \text{ if } l'n'=0\end{cases}.$$
        \end{itemize}
        
        Two triples $\alpha^{l'}\beta^{m'}\gamma^{n'}$ and $\alpha^{l'+r}\beta^{m'+r}\gamma^{n'+r}$ should have the same length. Denote by $\ell_1$, $\ell_2$ the length of $\alpha^{l'}\beta^{m'}\gamma^{n'}$ and $\alpha^{l'+r}\beta^{m'+r}\gamma^{n'+r}$, respectively. If $m'>0$, then $\ell_1 = m' \neq m'+r = \ell_2$. Thus $m'$ should be nonpositive. Suppose that $m'=0$. Then the equality $\ell_2=\ell_1=1$ implies that $r=1$. However, $(l', 0, n')$ and $(l'+1, 0, n'+1)$ cannot be normal at the same time. Thus $m'$ cannot be zero. By the cyclic symmetry, we may assume $l', m', n' <0$. Then $\ell_1 = -m'+1\geq 2$. Since $m'+r\leq 0$ implies $\ell_2\leq -m'-r+1<\ell_1$, so we have $m'+r>0$ and $-m'+1=\ell_1=\ell_2=m'+r$. Simply comparing the number of exponent $(-1)$ in each reduced cyclic words of $\alpha^{l'}\beta^{m'}\gamma^{n'}$ and $\alpha^{l'+r}\beta^{m'+r}\gamma^{n'+r}$,
        $$\alpha^{l'}(\gamma\alpha)^{-m'}\gamma^{n'}\sim \alpha^{l'-2m'}(\gamma^{-1}\alpha^{-1})^{-m'}\gamma^{n'-2m'},$$ we have $m'=-1$. Considering cyclic symmetry again, we also have $l'=n'=-1$ also. This contradicts the given triple $(l', m', n')$ is normal. This proves the lemma.

\end{proof}

Secondly, normality will guarantee that we have the canonical form of the corresponding matrix factorization,
i.e., a non-normal triple may have a mirror matrix factorization that is not in the canonical form. (see Example \ref{ex:nonnormal}).

\begin{prop}\label{prop:LagToMFRankOne}
        If $(l', m', n')\neq (2, 2, 2)$, a {\em  normal} immersed Lagrangian $L((l', m', n'), \lambda')$ is mapped to
        the following matrix factorization under the localized mirror functor $\mathcal{F}^\bL$.
        We only give the $\varphi$-matrix factor.
        \begin{equation}\label{eq:33}
                \cff^\bll(L((l', m', n'), \lambda')) = \begin{pmatrix} z & -y^{m'-1} & (-1)^{l'+1}\lambda'^{-1}x^{-l'} \\
                        y^{-m'} & x & -z^{n'-1} \\ 
                        (-1)^{l'}\lambda' x^{l'-1} & z^{-n'} & y \end{pmatrix}.
        \end{equation}
        (Again, we are using a non-standard notation that {\em $x^a$, $y^a$ or $z^a$ is considered as zero if $a<0$} .)
\end{prop}

\begin{remark}\label{rem:DegenerateLoopLength1}
        Since the Lagrangian $L((2, 2, 2), \lambda')$ is nor regular in the sense \ref{def:regular}, we have to perturb it to compute the localized mirror functor with $\bcc$-coefficient as in the figure \ref{fig:Perturbed222}. The corresponding matrix factorization is given as follows.
        $$\begin{pmatrix}
                xz & 0 & 0 & 0 \\ z & -y & 0 & 0 \\ 0 & x & -z & 0 \\ \lambda'x & 0 & y & xy
        \end{pmatrix}, \quad \begin{pmatrix}
        y & 0 & 0 & 0 \\ z & -xz & 0 & 0 \\ x & -x^2 & -xy & 0 \\ -1-\lambda' & x & y & z
\end{pmatrix},$$ which turns out to be, when $\lambda'\neq -1$, homotopy equivalent to the following matrix factorization.
$$\begin{pmatrix}
        z & -y & 0 \\ 0 & x & -z \\ \lambda'x & 0 & y \end{pmatrix}, \quad \begin{pmatrix} \frac{xy}{1+\lambda'} & \frac{y^2}{1+\lambda'} & \frac{yz}{1+\lambda'} \\ \frac{-\lambda'xz}{1+\lambda'} & \frac{yz}{1+\lambda'} & \frac{z^2}{1+\lambda'} \\ \frac{-\lambda'x^2}{1+\lambda'} & \frac{-\lambda'xy}{1+\lambda'} & \frac{xz}{1+\lambda'}
\end{pmatrix},$$ where the $\phi$-component is canonical.
        
\end{remark}

\begin{figure}[h]
        \centering
                \includegraphics[scale=0.35]{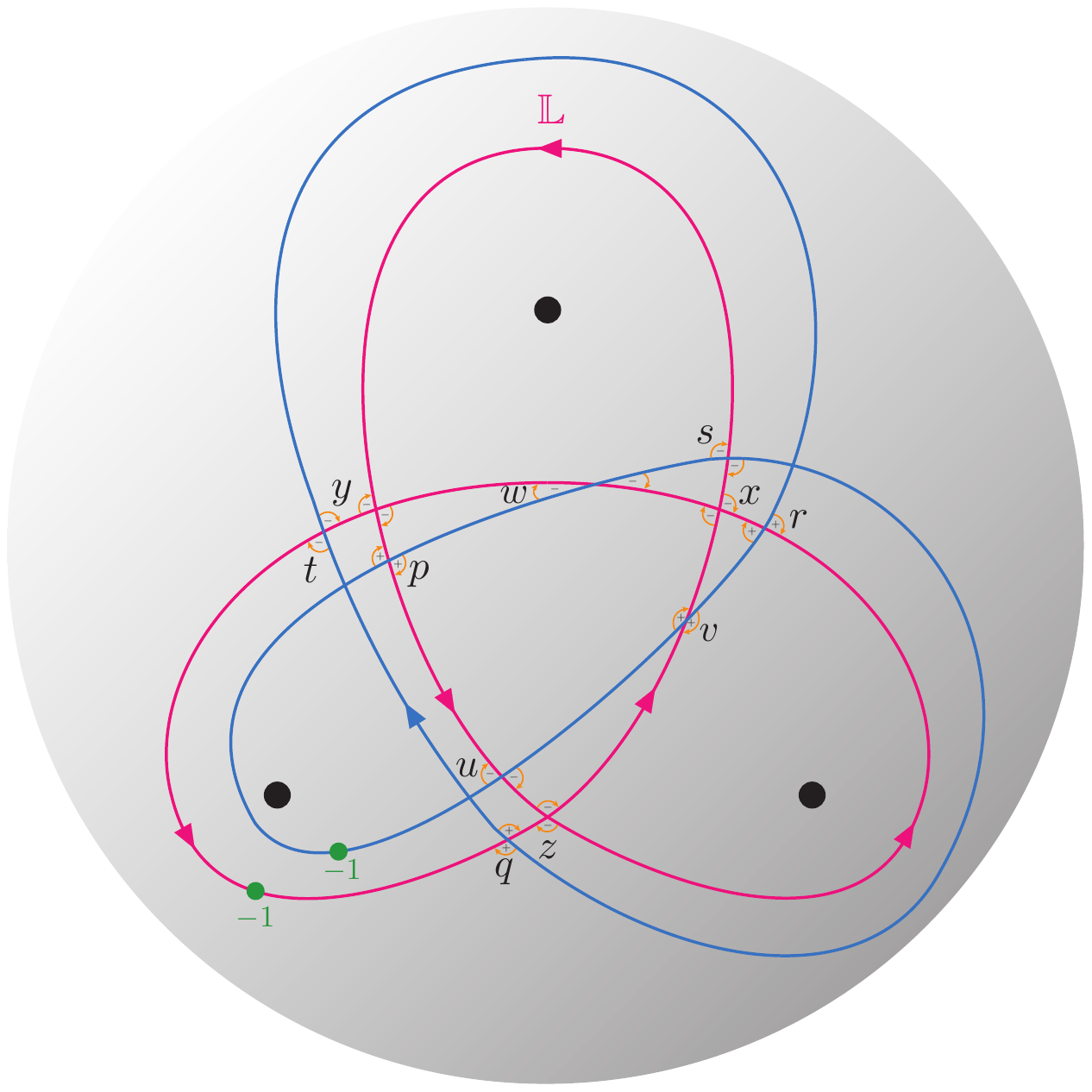}
                \caption{Perturbed Lagrangian $L(2, 2, 2)$.}
                \label{fig:Perturbed222}
\end{figure}
As we will handle the proof of much more general case in  Section \ref{sec:MfFromLag},
let us only illustrate the computation by working out one specific case of $L(2, 3, 2)$.

Consider the Lagrangian $L(2, 3, 2)$ which is drawn in Figure \ref{fig:L322}.
It is enough to find decorated strips (see Figure \ref{fig:strip}) bounded by $L(2,3,2)$ and $\mathbb{L}$ that corresponds to
the following matrix factor of $L(2,3,2)$: $$\begin{blockarray}{cccc}
        p & q & r \\
        \begin{block}{(ccc)c}
                z & -y^2 & 0 & s\\
                0 & x & -z & t\\
                \lambda x & 0 & y & u\\
        \end{block}
\end{blockarray}.$$

Here $\{p,q,r\}$ and $\{s, t, u\}$ are intersection points $L(2,3,2) \cap \mathbb{L}$, grouped according to their intersection signs.
Let us find the non-trivial strips from $p,q,r$ to $s,t,u$, which produce the above matrix factor.

Let us start with easy ones given by three small triangles. 
One of them is drawn in Figure \ref{fig:L322a}, namely the triangle $pZs$ corresponding  to the entry ``z'' of the matrix  for the map $p \to s$. Similar triangles  produces
the rest of diagonal entries $x, y$ of the above matrix.

Now, we may try to find another strip that contains $pZs$-triangle in Figure \ref{fig:L322a}.
If we do not turn at the $Z$-corner, and try to find a bigger (immersed) strip, we will meet the puncture as in Figure \ref{fig:L322ps}. This means that there are no more decorated bigons in that direction.

Thus, we can find all strips, starting from $p,q,r,s,t,u$ in every possible direction while keeping track of turns or not turns at each immersed points.
\begin{figure}[h]
       \centering
       \begin{subfigure}[b]{0.45\textwidth}
               \includegraphics[scale=0.35]{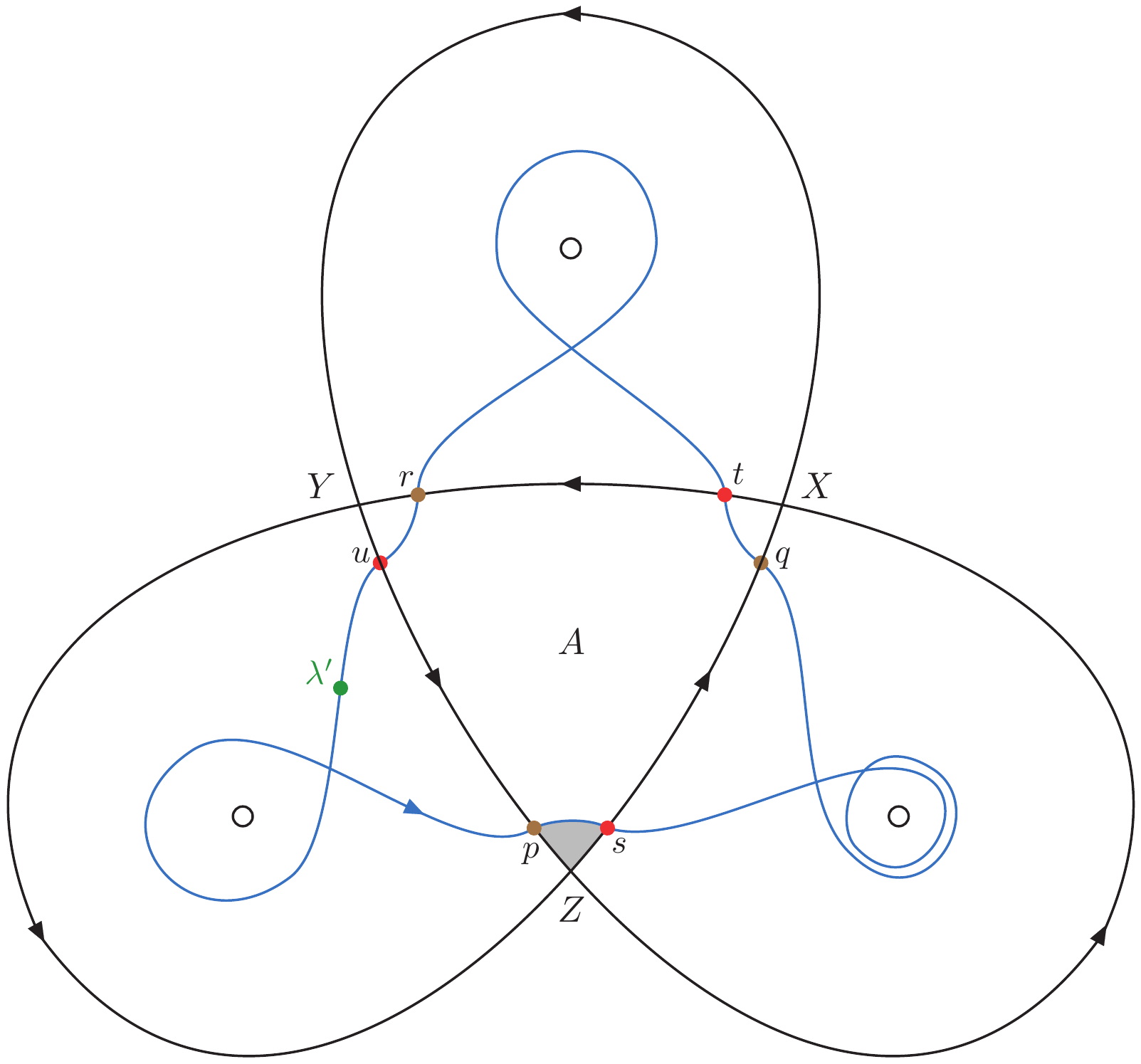}
               \centering
               \caption{}
               \label{fig:L322a}
       \end{subfigure}
       \qquad
       \begin{subfigure}[b]{0.45\textwidth}
               \includegraphics[scale=0.35]{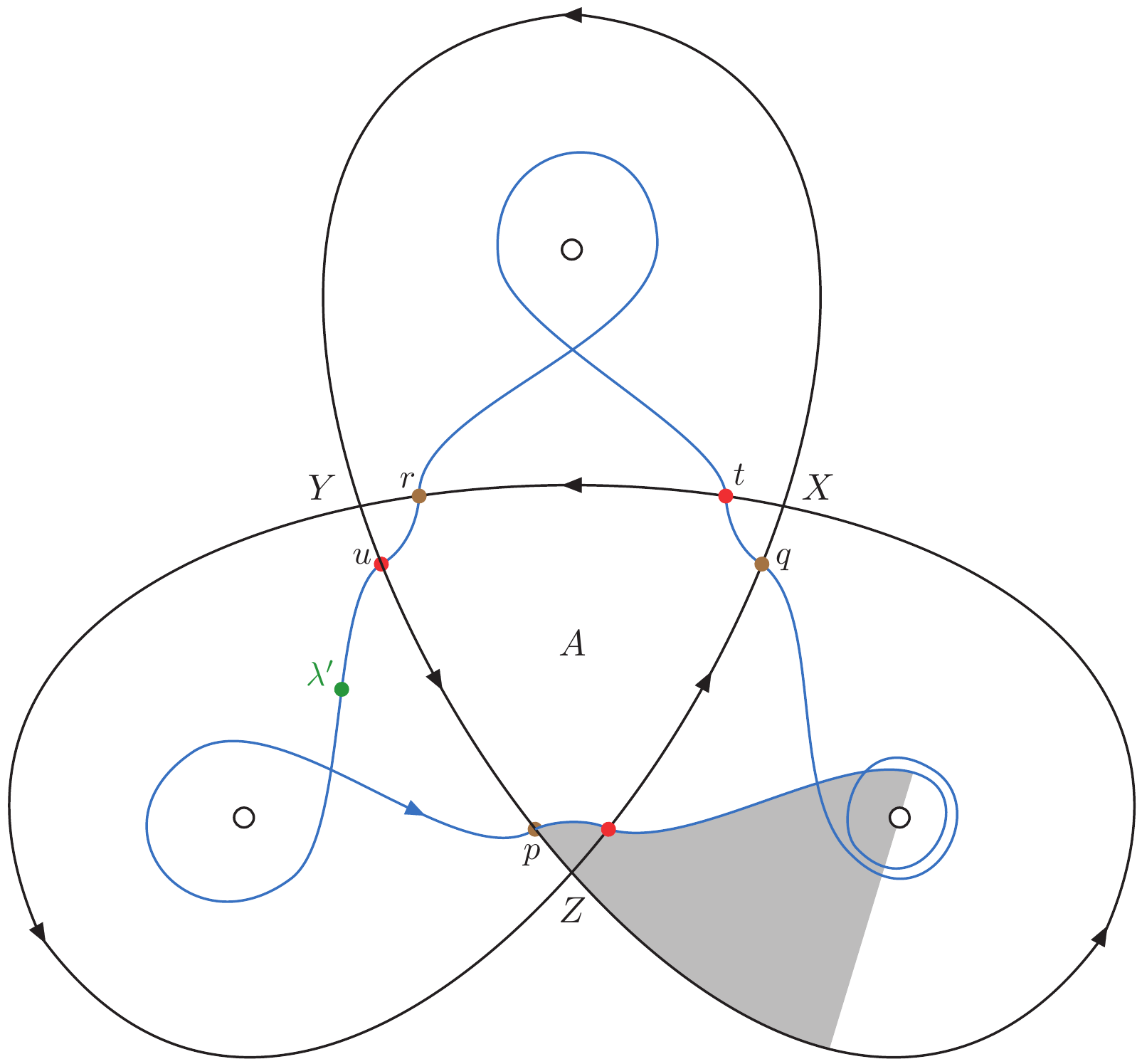}
               \centering
               \caption{}
               \label{fig:L322ps}
       \end{subfigure}
       \caption{}
       \label{fig:L322}
\end{figure}
Let us find other strips in this way. 
Figure \ref{fig:L322psin} describes a part of the decorated strip, which starts from $p$ and moving toward left hand side.
Note that since the infinity is a point $g$ that we can pass through, the strip looks unbounded, but actually it is a bounded strip.
Note that we turn at the $X$-corner, and the rest of the strip is drawn in \ref{fig:L322pu2}, ending at the point $u$.
Since the strip passes the holonomy point, we multiply  the weight $\lambda$ in our counting.
This gives the entry $\lambda x$ for the map $p \to u$. Since this strip follows the Seidel Lagrangian in the opposite direction, the negative sign contributes one more than the number of corners and special holonomy points on the path. This time, this number is $2$ and the strip gives the $\lambda x$ entry.

\begin{figure}[h]
       \centering
       \begin{subfigure}[b]{0.45\textwidth}
       \includegraphics[scale=0.35]{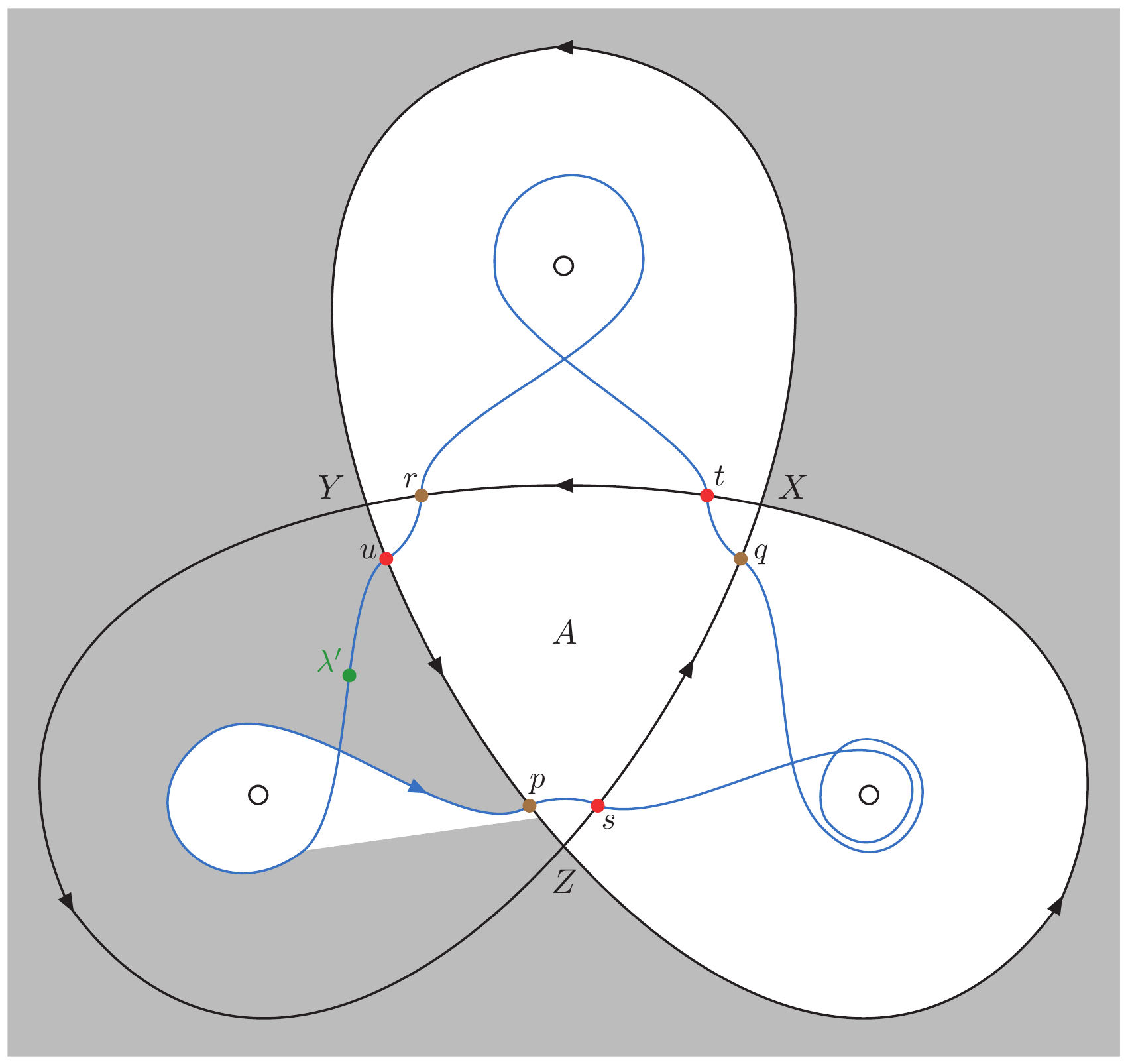}
       \centering
       \caption{}
       \label{fig:L322psin}
       \end{subfigure}
     \begin{subfigure}[b]{0.45\textwidth}
               \includegraphics[scale=0.35]{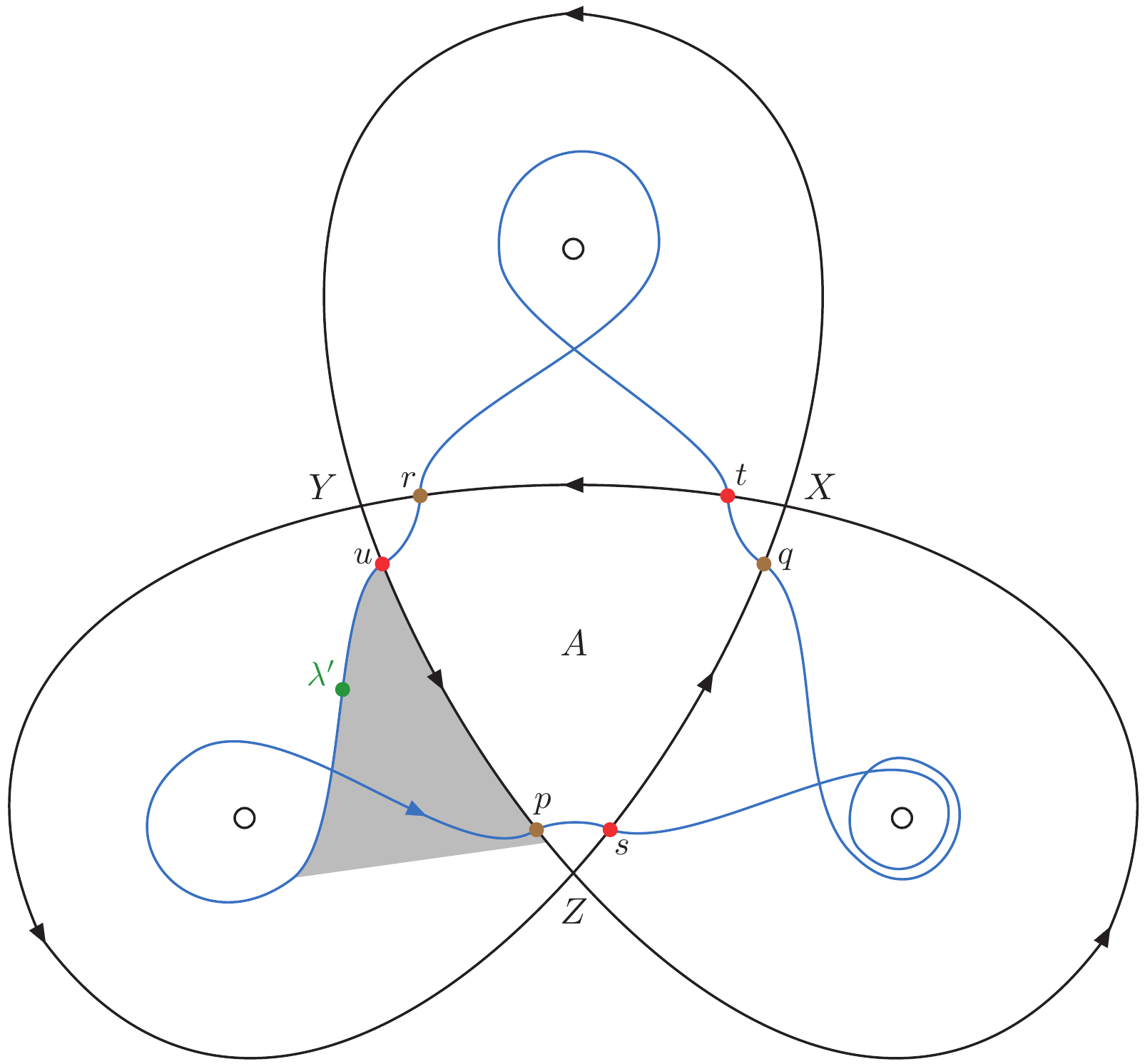}
               \centering
               \caption{}
               \label{fig:L322pu2}
       \end{subfigure}
       \caption{}
       \label{fig:L32222}    
\end{figure}

There is another strip starting from $r$ moving toward right hand side, and covering the whole non-compact region as before, and its boundary on $\bL$ turns at the
$Z$-corner, and the strip ends at the intersection $t$. This strip has one corner and meets the special holonomy point once, the corresponding entry from $r$ to $t$ is $-z$.

The final strip starts from $q$ moving downward, and covering the whole non-compact region twice while turning at the $Y$-corner twice, and ending at the intersection $s$.
This corresponds to $-y^2$ for the map $q \to s$.
One can check that these are all the decorated strips that we can find.

\subsection{ The cases of $2 \times 2$ mirror matrix factorizations}
We consider the cases when a mirror $3 \times 3$ matrix factorization admits a reduction to $2\times 2$.
From the matrix factor in \eqref{eq:33}, we need to consider only the cases when some of $l',m',n'$ are equal to zero or one. Geometrically, if $n'=0, 1$, then the free homotopy class $[\alpha^{l'}\beta^{m'}\gamma^{n'}]$ is reduced to  $[\alpha^{l'}\beta^{m'}], [\alpha^{l'-1}\beta^{m'-1}]$ respectively. Therefore, we may find immersed Lagrangians having these free homotopy classes and intersecting with $\bll$ at only four points. Also, from the symmetry of $x,y,z$, it is enough to consider the cases of $(l',m',1)$ or $(l', m', 0)$. Similar to \ref{def:RankOneLagrangian}, corresponding immersed Lagrangians are given as follows.



\begin{defn}
For two integers $l', m'$ and nonzero $\lambda'\in\bcc^*$, we define an immersed Lagrangian submanifold $L((l', m'), \lambda')$ to be a smoothing of the loop $\delta_x^{l'}\cdot \Delta_{ps} \cdot \delta_y^{m'} \cdot \Delta_{qu}$ whose holonomy $\lambda'$ is concentrated at a point in $\delta_x^{l'}$.
\end{defn}
\begin{remark}
This loop is homotopic to $L((l', m', 1), \lambda')$ (see Figure \ref{fig:SS3}), and $L((l'-1, m'-1, 0), \lambda')$ also because of Lemma \ref{lem:TripleEqui}.  More precisely,  the Lagrangian $L((l', m'), \lambda')$ does not pass through the region $D_z$,
whereas the Lagrangian $L((l', m', 1), \lambda')$ will pass through $D_z$. The latter curve will enter $D_z$ through segment $l_z$, but exits immediately forming an embedded bigon.
This results in a scalar factor  (the expression $z^{n'-1}$ in \eqref{eq:33}) in the corresponding matrix factorization, which can be reduced.
\end{remark}

\begin{figure}[h]
        \centering
        \begin{subfigure}[t]{0.45\textwidth}
                \includegraphics[scale=0.35]{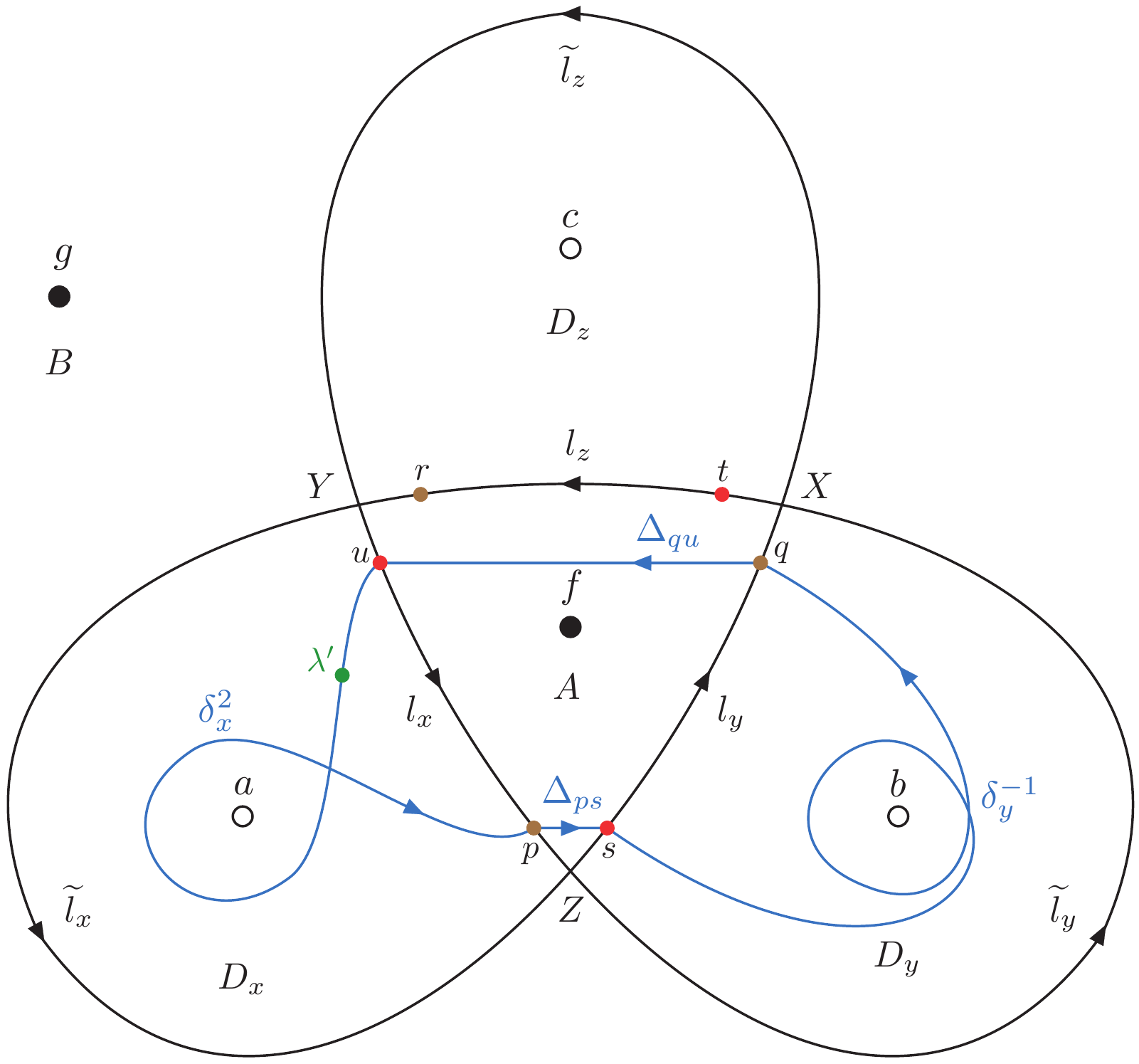}
                \centering
                \caption{The loop $\delta^2_x \cdot \Delta_{ps} \cdot \delta^{-1}_y \cdot \Delta_{qu}$}
                \label{fig:BeforeSmoothingL(2, -1)}
        \end{subfigure}
        \begin{subfigure}[t]{0.45\textwidth}
                \includegraphics[scale=0.35]{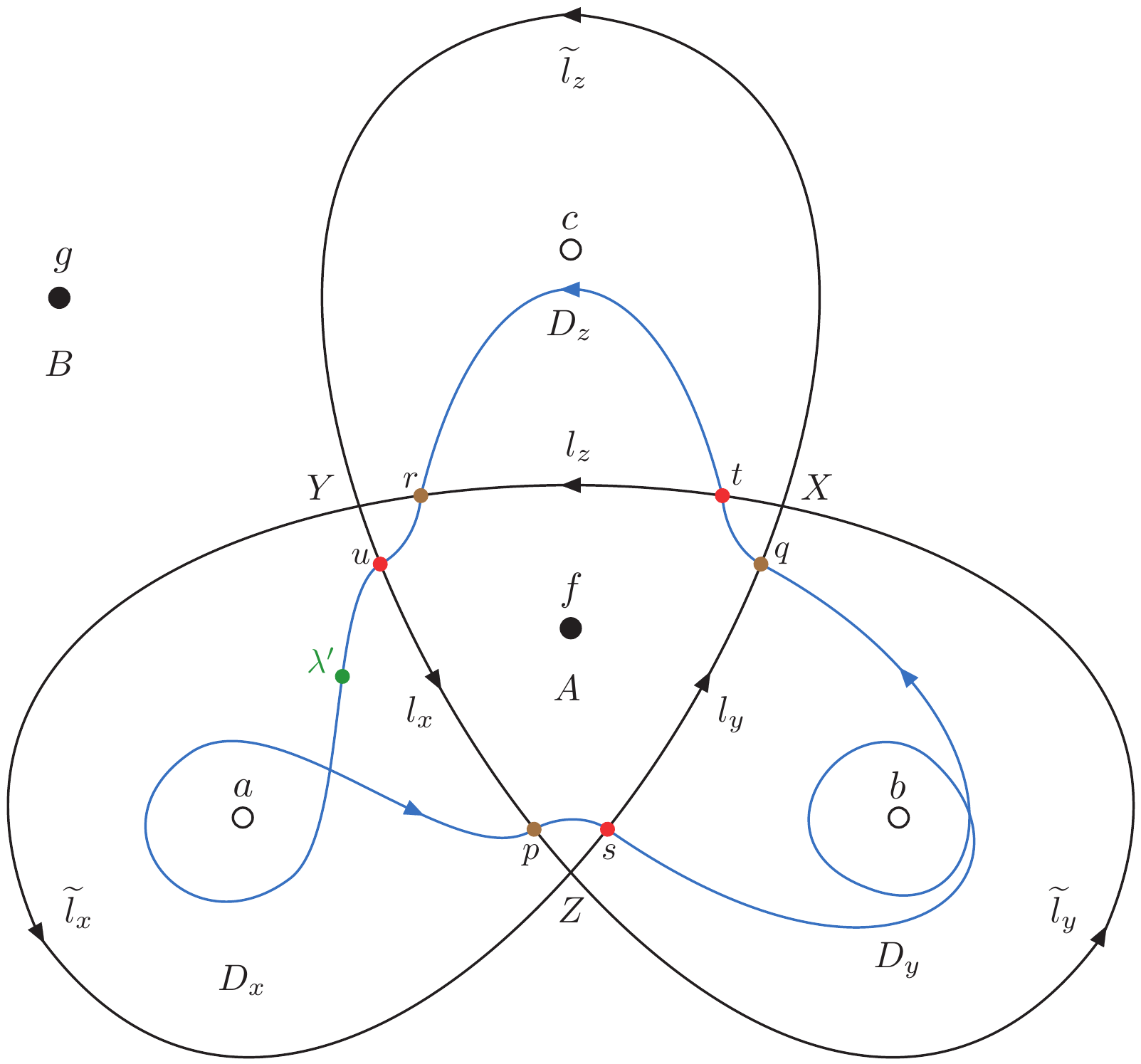}
                \centering
                \caption{The Lagrangian $L_1((2, -1), \lambda')$}
                \label{fig:AfterSmoothingL(2, -1)}
        \end{subfigure}
            \caption{}
           \label{fig:SS3}
\end{figure}

We leave the following proposition as an exercise. Readers can find a more general `\emph{matrix reduction}' process according to removing a bigon in lemma \ref{lem:HomotopyTypeV}.

\begin{prop}\label{prop:RemovingBigon}
Suppose that $(l',m',1)$ is normal. The immersed Lagrangian $L((l', m'), \lambda')$ is mapped to
the following matrix factorization ($\varphi$-matrix factor in fact) under the localized mirror functor $\mathcal{F}^\bL$.
 $$                 \cff^\bll(L((l', m'), \lambda')) = \begin{pmatrix} z+(-1)^{l'+1}\lambda'^{-1}x^{l'-2}y^{m'-2}+(-1)^{l'+1}\lambda' x^{-l'}y^{-m'} & y^{m'-1}+\lambda'^{-1}(-1)^{l'}x^{-l'+1}\\
                        y^{-m'+1}+\lambda'(-1)^{l'}x^{l'-1} & xy \end{pmatrix} $$
                                               (Again, we are using a non-standard notation that {\em $x^a$, $y^a$ or $z^a$ is considered as zero if $a<0$} .)
\end{prop}

\section{Rank one Mirror symmetry Correspondence between modules and Lagrangians}\label{sec:MSRank1}
In this section, we prove homological mirror symmetry between indecomposable maximal Cohen-Macaulay modules over $\mathbb{C}[[x,y,z]]/(xyz)$ 
and free homotopy classes of normal immersed Lagrangian curves in the pair of pants $\mathcal{P}$ given by rank one data.

Recall that for modules, rank one means that they are classified by (modified) band data of rank one $((l,m,n), \lambda)$.
In other word, their word length is 3 and multiplicity $\mu$ is one. 
The canonical forms of their associated matrix factorizations are given in Proposition \ref{prop:Rank1MFfromModule} and
due to the non-trivial  Macaulayfications of modules, canonical form used corrected band data $(l',m',n')$.

Rank one immersed Lagrangians have  their free homotopy classes as $\alpha^{l'}\beta^{m'}\gamma^{n'}$ (up to conjugation).
We have the unique normal Lagrangian loop in each homotopy class whose mirror matrix factorization under the localized mirror functor
is given by the canonical form in Section \ref{sec:MFFromRank1Lagrangian}.

We justify the use of the  same notation $(l',m',n')$ in algebra and geometry. 
First, we show that there is a bijective correspondence between (corrected) band words and normal loop words in the rank one case.
Next, we will verify that the corresponding matrix factorizations match and therefore, the coincidence is the
manifestation of mirror symmetry.

\begin{thm}\label{thm:c1}
The conversion $(l,m,n) \mapsto (l',m',n')$ given 
in Proposition \ref{prop:Rank1MFfromModule} gives a bijection between non-degenerate band words $(l,m,n)$ and
normal loop words $(l',m',n')$.  More precisely, the converted word  is always normal and there exists an inverse
from normal loop words to non-degenerate band words
$$
(l',m',n') \mapsto  (l',m',n') - (\varepsilon'(l',m',n'), \varepsilon'(l',m',n'), \varepsilon'(l',m',n'))
$$
where the correction number is
$$
\varepsilon'(l',m',n') := -1 + \#\left(\left\{l',m',n'\right\} \cap \mathbb{Z}_{\ge1}\right)
$$
\end{thm}
Before we give a proof of the theorem, let us prove rank one homological mirror symmetry first.
\begin{thm}\label{thm:r1}
Let  $((l,m,n),\lambda,1)$ be a non-degenerate (modified) band datum
and let $((l',m',n'),\lambda,1')$ be the converted word with the relation
$$ \lambda' =(-1)^{l+1}\lambda$$
Then, the  diagram \eqref{dia:main} commutes.
\end{thm}
\begin{proof}
We prove Theorem \ref{thm:r1}.
Given a rank one band data $\left((l,m,n),\lambda,1\right)$, Burban-Drozd defined the corresponding module 
$\tilde{M}\left(\left(l,m,n\right),\lambda,1\right)$. In Section \ref{sec:mod1}, we have computed its Macaulayfication $M\left(\left(l,m,n\right),\lambda,1\right)$
and found the following canonical matrix factor of $xyz$:
$$
\varphi((l',m',n'),\lambda,1) =
\begin{pmatrix}
z & - y^{m'-1} & -\lambda^{-1}x^{-l'} \\
-y^{-m'} & x & - z^{n'-1} \\
- \lambda x^{l'-1} & -z^{-n'} & y
\end{pmatrix}.
$$
This explains arrows between $\underline{\mathrm{CM}}(A)$ and $\mathrm{HMF}(xyz)$.

We showed that the normal immersed Lagrangian $L((l',m',n'),\lambda,1')$
is mapped to the following matrix factor under the localized mirror functor $\LocalF$
$$
\LocalF\left(L((l',m',n'),\lambda',1')\right) =
\begin{pmatrix}
z & - y^{m'-1} & (-1)^{l'+1}\lambda'^{-1}x^{-l'} \\
y^{-m'} & x & - z^{n'-1} \\
(-1)^{l'} \lambda'x^{l'-1} & z^{-n'} & y
\end{pmatrix}.
$$
Comparing the above two matrices, one can show that they are isomorphic in $\operatorname{MF}(xyz)$ if we set $\lambda'=(-1)^{l+1}\lambda$. Note that the exponent of $(-1)$ is not $(l'+1)$ but $(l+1)$. 
\end{proof}

Now, let us prove Theorem \ref{thm:c1}.

\begin{proof}
        Let $(l, m, n)$ be a nondegenerate triple and $(l', m', n')$ be the converted triple. We show that $(l', m', n')$ is normal and $\varepsilon'(l', m', n') = \varepsilon(l, m, n)$ by dividing the cases according to the value of $\varepsilon(l, m, n)$.
        \begin{itemize}
                \item $\varepsilon(l, m, n)=2$ : in this case, $(l', m', n') = (l+2, m+2, n+2)\in\bzz_{\geq 2}^3$, hence the converted triple is normal and $\varepsilon'(l', m', n') = 2$.
                \item $\varepsilon(l, m, n)=1$ : we may assume that $l, m>0$ and $n<0$. Then $(l', m', n') = (l+1, m+1, n+1) \in\bzz_{\geq 2}^2\times\bzz_{\leq 0}$,  the converted triple is normal and $\varepsilon'(l', m', n')=1$.
                \item $\varepsilon(l, m, n)=0$ : we may assume $l>0$, $m<0$, and $n\leq 0$. Then the converted triple $(l', m', n')=(l, m, n)$ is normal and $\varepsilon'(l', m', n') = 0$.
                \item $\varepsilon(l, m, n)=-1$ : assume that $l<0$, and $m, n\leq 0$. Then $(l', m', n') = (l-1, m-1, n-1)\in\bzz_{\leq -2}\times\bzz_{\leq -1}^2$ is normal and $\varepsilon'(l', m', n') = -1$.
        \end{itemize}
        Thus conversion from nondegenerate triples to normal triples is well defined and injective. Now consider the converse. Let $(l', m', n')$ be a normal triple and $(l'', m'', n'')$ be the converted triple. As before, we show $\varepsilon(l'', m'', n'') = \varepsilon'(l', m', n')$ by dividing the cases according to the value of $\varepsilon(l', m', n')$.
        \begin{itemize}
                \item $\varepsilon'(l', m', n')=2$ : in this case, $l', m', n'\geq 1$. Since $(l', m', n')$ is normal, neither of them is $1$. Thus, $(l'', m'', n'')=(l'-2, m'-2, n'-2)\in\bzz_{\geq 0}^3$ and $\varepsilon(l'', m'', n'')=2$.
                \item $\varepsilon'(l', m', n')=1$ : we may assume $l', m'\geq 1$ and $n'\leq 0$. By the same reason as above, $l', m'\neq 1$. Hence $(l'', m'', n'') = (l'-1, m'-1, n'-1)\in\bzz_{\geq 1}^2\times\bzz_{\leq-1}$ and $\varepsilon(l'', m'', n'') = 1$.
                \item $\varepsilon'(l', m', n')=0$ : we may assume $l'\geq 1$ and $m', n'\leq 0$. Note that not both $m'$ and $n'$ are zero by normality assumption. Then $\varepsilon(l'', m'', n'') = \varepsilon(l', m', n') = 0$.
                \item $\varepsilon(l', m', n')=-1$ : let $l', m', n'\leq 0$. Since $(l', m', n')$ is normal, they are all non-zero. Also, $(l', m', n')\neq (-1, -1, -1)$. Then $(l'', m'', n'') = (l'+1, m'+1, n'+1)\in\bzz_{\leq 0}^3\setminus\{(0, 0, 0)\}$ and $\varepsilon(l'', m'', n'')=-1$.
        \end{itemize}
        Thus conversion from normal triples to nondegenerate triples is also injective, which implies the conversion is bijective.
\end{proof}

We will also prove it in more general form for higher rank cases in Propositions \ref{prop:ConvertedToNormalForm} and \ref{prop:ConversionFormulaInverseToEachOther}.

\section{Loop Data and Modified Band Data}
To discuss higher length cases, we study some properties of loop data in this section.
We will also introduce a slightly modified version of the band data of Burban-Drozd \cite{BD17}
to compare it with the loop data.

\subsection{Loop data for immersed Lagrangians}
Recall that the fundamental group of $\mathcal{P}$ can be presented as
$\pi_1\left(\mathcal{P}\right) = \left<\alpha,\beta,\gamma\left|\alpha\beta\gamma=1\right.\right>$
with the based loops $\alpha$, $\beta$ and $\gamma$ in $\mathcal{P}$ shown in Figure \ref{fig:gen}.
Also recall that there is a one-to-one correspondence between the elements of $\left[S^1,\mathcal{P}\right]$ and the conjugacy classes in $\pi_1\left(\mathcal{P}\right)$. 

\begin{defn}\label{def:LoopWord}
A \emph{loop word} of \emph{length} $3\tau$ is
$$
w' = (l_1',m_1',n_1',l_2',m_2',n_2',\dots,l_\tau',m_\tau',n_\tau')\in \mathbb{Z}^{3\tau}
$$
with  $\tau\in\mathbb{Z}_{\ge1}$, $l_i'$, $m_i'$, $n_i'\in\mathbb{Z}$ for $i\in\left\{1,\dots,\tau\right\}$.
Its associated element in $[S^1,\CP]$ is denoted as
$$
L\left[w'\right]:=\alpha^{l_1'}\beta^{m_1'}\gamma^{n_1'}\alpha^{l_2'}\beta^{m_2'}\gamma^{n_2'}\cdots\alpha^{l_\tau'}\beta^{m_\tau'}\gamma^{n_\tau'}
$$
Two loop words $w'$ and $w_*'$ are \emph{equivalent} if  the associated loops are homotopic to each other.
A \emph{loop datum} $(w',\lambda',\mu')$ consists of the following:
\begin{itemize}
\item (loop word) $w' = (l'_1,m'_1,n'_1,l'_2,m'_2,n'_2,\dots,l'_\tau,m'_\tau,n'_\tau)\in \mathbb{Z}^{3\tau}$ for some $\tau \in \mathbb{Z}_{\ge1}$,
\item (holonomy)\ $\lambda' \in \mathbb{C}^*$,
\item (multiplicity) $\mu' \in \mathbb{Z}_{\ge1}$.
\end{itemize}
Two loop data $(w',\lambda',\mu')$ and $(w'_*,\lambda'_*,\mu'_*)$ are said to be \emph{equivalent} if $w'$ and $w'_*$ are equivalent as loop words, $\lambda'=\lambda'_*$ and $\mu'=\mu'_*$.
\end{defn}

For such a word $w'$, define its \emph{$1$-shift} to be
$$
w'^{(1)}=\left(l'_2,m'_2,n'_2,\dots,l'_\tau,m'_\tau,n'_\tau,l'_1,m'_1,n'_1\right)
$$
and \emph{$k$-shift} to be $w'^{(k)}$ where $w'^{(k)}$ is obtained from $w'$ by applying the $1$-shift $k$-times for some $k \in \Z$.

We sometimes denote the $j$-th value of $w'$ as $w'_j$ so that
$$
w' = (w'_1, w'_2, w'_3, w'_4, w'_5, w'_6, \dots, w'_{3\tau-2}, w'_{3\tau-1}, w'_{3\tau})
$$
where $w'_{3i-2} = l'_i$, $w'_{3i-1} = m'_i$ and $w'_{3i} = n'_i$ for $i\in\mathbb{Z}_\tau$. Here we regard the index $i$ of $l'_i$, $m'_i$ and $n'_i$ to be in $\mathbb{Z}_\tau$ (hence $3i\in\mathbb{Z}_{3\tau}$)
and the index $j$ of $w'_j$ to be in $\mathbb{Z}_{3\tau}$. Then any tuple
$
(w'_{k},w'_{k+1},\dots,w'_l)
$
for some distinct $k$, $l\in\mathbb{Z}_{3\tau}$ is called a \emph{subword} in $w$. For example, $(w'_{3\tau-1},w'_{3\tau},w'_1)$ is a subword.
The following lemma is easy to check.

\begin{lemma}\label{lem:equimove}
The following $5$ operations on a loop word $w'$ do not change the equivalence class of $w'$:
\begin{itemize}
\item (shifting) take $k$-shift on $w'$ for some $k\in\mathbb{Z}$,
\item (inserting 0s) insert the subword $(0,0,0)$ in any place of $w'$,
\item (removing 0s) remove a subword $(0,0,0)$ in $w'$ if it exists,
\item (adding $1$s around $0$) add $(1,1,1)$ to the subword $(w_{j-1}',0,w_{j+1}')$ in $w'$ where $w_j'=0$, and
\item (subtracting $1$s around $1$) subtract $(1,1,1)$ from the subword $(w_{j-1}',1,w_{j+1}')$ in $w'$ where $w_j'=1$.
\end{itemize}
\end{lemma}

 \begin{prop}\label{prop:equihomotopy}
Two loop words $w'$ and $w_*'$ are equivalent if and only if $w_*'$ can be obtained from $w'$ by performing the above operations finitely many times.
\end{prop}
\begin{proof}
We give the proof in Appendix \ref{subsec:equiv}. See Proposition \ref{prop:a1}. 
\end{proof}

Note that several equivalent loop words can represent the same element of  $\left[S^1,\mathcal{P}\right]$.
To find a unique representative in each homotopy class, we propose the following `\emph{normal form}'  of a loop word.
It turns out that this also play an important role in the conversion formula between loop and band datum.

\begin{defn}\label{defn:normal2}
A loop word $w'$ is said to be \emph{normal} if it satisfies the following $4$ conditions:
\begin{itemize}
\item any subword of the form $\left(a,1,b\right)$ in $w'$ satisfies $a, b\le0$,
\item any subword of the form $\left(a,0,b\right)$ in $w'$ satisfies $a\le-1$, $b\ge1$ or $a\ge1$, $b\le-1$ or $a, b\ge1$,
\item $w'$ has no subword of the form $(0,-1,-1,\dots,-1,0)$, and
\item $w'$ does not consist only of $-1$, that is, $w'\ne\left(-1,-1,\dots,-1\right)$.
\end{itemize}
\end{defn}

The normal form actually gives exactly one representative among equivalent loop words, with a few exceptions. To rule those out, we say a loop word $w'$ is \emph{essential} if $w'$ is not equivalent to a  loop word of the form $(l',0,0)$, $(0,m',0)$ or $(0,0,n')$ for some $l',m', n'\in\mathbb{Z}$.

\begin{prop}\label{prop:normalnormal}
Any essential loop word is equivalent to a unique normal loop word up to shifting. Conversely, any normal
loop word is essential.

\end{prop}
\begin{proof}
We prove the existence and the uniqueness in Appendix \ref{subsec:exist} and \ref{subsec:unique},
respectively. See Proposition \ref{prop:NormalExistence} and \ref{prop:NormalUniqueness}. For the last statement, see Corollary \ref{cor:NormalEssential}.
\end{proof}

A normal loop word $w'$ is called \emph{periodic} if $w'^{\left(\tilde{\tau}\right)}=w'$ for some positive divisor $\tilde{\tau}$ of $\tau$ other than itself. The smallest $\tilde{\tau}$ among those satisfying such relation is called the \emph{period} of $w'$. In this case, 
for $N:=\frac{\tau}{\tilde{\tau}}\in\mathbb{Z}_{\ge2}$, 
there exists another normal loop word $\tilde{w}'$ of length $3\tilde{\tau}$ such that $w'$ is $N$ repetitions
of $\tilde{w}'$. We denote it as
$$
w' = \left(\tilde{w}',\tilde{w}',\dots,\tilde{w}'\right) =: \left(\tilde{w}'\right)^{\#N}
$$
and call $w'$ as the \emph{$N$-concatenation} of $\tilde{w}'$.

\subsection{Modified band datum}\label{sec:mbd}
Let us denote the band datum of Burban-Drozd by $(w_{\text{\rm BD}},\lambda,\mu)$ (see \eqref{eq:bw}).
Let us define a modified version of band datum.
\begin{defn}\label{defn:mbd}
A \emph{modified band datum} $(w,\lambda,\mu)$ consists of the following:
\begin{itemize}
\item (band word) $w = \left(l_1,m_1,n_1,l_2,m_2,n_2,\dots,l_\tau,m_\tau,n_\tau\right)\in \mathbb{Z}^{3\tau}$ for some $\tau \in \mathbb{Z}_{\ge1}$,
\item (eigenvalue) $\lambda\in \mathbb{C}^*$,
\item (multiplicity) $\mu \in \mathbb{Z}_{\ge1}$.
\end{itemize}
The \emph{length} of the band word $w$ or the band datum $(w,\lambda,\mu)$ is defined as $3\tau$, and the \emph{rank} of the band datum $(w,\lambda,\mu)$ is defined as $\tau\mu$.
\end{defn}

It is easy to see that two notions of band datum are equivalent.
\begin{lemma}
Modified band word $w$ can be obtained from $w_{\text{\rm BD}}$ by recoding the differences $l_i:=f_i-a_i$,$m_i:=b_i-c_i$, $n_i:=d_i-e_i$ for $i\in\mathbb{Z}_\tau$. Conversely, we can recover $w_{\text{\rm BD}}$ from $w$ by setting $f_i:=\max\left\{0, l_i\right\}+1$, $a_i:=\max\left\{0, -l_i\right\}+1$, $b_i:=\max\left\{0, m_i\right\}+1$, $c_i:=\max\left\{0, -m_i\right\}+1$, $d_i:=\max\left\{0, n_i\right\}+1$ and $e_i:=\max\left\{0, -n_i\right\}+1$ for $i\in\mathbb{Z}_\tau$.
\end{lemma}

Since the word $w$ (resp. $w_{\text{\rm BD}}$) is in $\mathbb{Z}^{3\tau}$(resp. $\mathbb{Z}^{6\tau}$)
and it is easy to tell whether band datum is a modified version or not. Therefore, we will call the modified band datum $\left(w,\lambda,\mu\right)$, simply as a band datum.

The notions of \emph{shift}, \emph{subword}, \emph{periodicity}, \emph{period} and \emph{concatenation} of a band word can be defined similarly following those of a normal loop word. Then two band words $w$ and $w_*$  are considered \emph{equivalent} if they are the same as cyclic words, that is, $w_* = w^{(k)}$ for some $k\in\mathbb{Z}$. Note that the periodicity and the period of a band word are invariant under this equivalence relation.
Two band data $(w,\lambda,\mu)$ and $(w_*,\lambda_*,\mu_*)$ are said to be \emph{equivalent} if $w$ and $w_*$ are equivalent as band words, $\lambda=\lambda_*$ and $\mu=\mu_*$.

\begin{defn}
Given a band datum $(w,\lambda,\mu)$ as above, define $T(w,\lambda,\mu)$ to be the following quiver representation:

\adjustbox{scale=0.7,center}{
\begin{tikzcd}[arrow style=tikz,>=stealth,row sep=10em,column sep=6em]
&
\mathbb{C}((t))^{\tau \mu}
  \arrow[rr,"\spmat{
               0 & t^{l_2^+ +1}I_\mu & & 0 \\
               \vdots &  & \ddots &  \\
               0 & 0 &  & t^{l_\tau^+ +1}I_\mu \\
               t^{l_1^+ +1}J_\mu(\lambda) & 0 & \displaystyle\cdots & 0
               }"]
  \arrow[dl,swap,"\spmat{
               t^{l_1^- +1}I_\mu & & 0 \\
                & \ddots &  \\
               0 &  & t^{l_\tau^- +1}I_\mu
               }"]
&&
\mathbb{C}((t))^{\tau\mu} 
\\
\mathbb{C}((t))^{\tau\mu}
&&&&
\mathbb{C}((t))^{\tau\mu}
  \arrow[ul,swap,"\spmat{
               t^{n_1^- +1}I_\mu & & 0 \\
                & \ddots &  \\
               0 &  & t^{n_\tau^- +1}I_\mu
               }"]
  \arrow[dl,"\spmat{
               t^{n_1^+ +1}I_\mu & & 0 \\
                & \ddots &  \\
               0 &  & t^{n_\tau^+ +1}I_\mu
               }"]
\\
&
\mathbb{C}((t))^{\tau\mu}
  \arrow[ul,"\spmat{
               t^{m_1^+ +1}I_\mu & & 0 \\
                & \ddots &  \\
               0 &  & t^{m_\tau^+ +1}I_\mu
               }"]
  \arrow[rr,swap,"\spmat{
               t^{m_1^- +1}I_\mu & & 0 \\
                & \ddots &  \\
               0 &  & t^{m_\tau^- +1}I_\mu
               }"]
&&
\mathbb{C}((t))^{\tau\mu}
\end{tikzcd}
}
where $I_\mu$ is the $\mu\times \mu$ identity matrix and $J_\mu(\lambda)$ is the $\mu\times \mu$ Jordan block with the eigenvalue $\lambda$.
\end{defn}

This generalizes the quiver in Example \ref{example:xyzTri}. This $T(w,\lambda,\mu)$ represents an object in the category $\operatorname{Tri}(A)$,
which is equivalent to the category of maximal Cohen-Macaulay modules (Theorem \ref{BurbanDrozdMainTheorem}). 
We describe the module corresponding to $T(w,\lambda,\mu)$ from \cite{BD17}.
\begin{defn}\label{defn:module}
Given a band datum $(w,\lambda,\mu)$ as above, define $\tilde{M}(w,\lambda,\mu)$ to be the $A$-submodule of $A^{\tau\mu}$ generated by all columns of the following $6$ matrices of size $\tau\mu\times \tau\mu$:
\begin{center}
$x^2y^2I_{\tau\mu}$,\quad $y^2z^2I_{\tau\mu}$,\quad $z^2x^2I_{\tau\mu}$,\quad
$
\pi_x\left(w,\lambda,\mu\right):=
\begin{pmatrix}
x^{l_1^- +2}yI_\mu & zx^{l_2^+ +2}I_\mu & \cdots & 0 \\
0 & x^{l_2^- +2}yI_\mu & \ddots & \vdots \\
\vdots & \vdots & \ddots & zx^{l_\tau^+ +2}I_\mu \\
zx^{l_1^+ +2}J_\mu(\lambda) & 0 & \cdots & x^{l_\tau^- +2}yI_\mu
\end{pmatrix}_{\tau\mu \times \tau\mu},
$

$
\pi_y\left(w,\lambda,\mu\right):=
\begin{pmatrix}
\left(xy^{m_1^+ +2} + y^{m_1^- +2}z\right)I_\mu & 0 & \cdots & 0 \\
0 & \left(xy^{m_2^+ +2} + y^{m_2^- +2}z\right)I_\mu & \cdots & 0 \\
\vdots & \vdots & \ddots & \vdots \\
0 & 0 & \cdots & \left(xy^{m_\tau^+ +2} + y^{m_\tau^- +2}z\right)I_\mu
\end{pmatrix}_{\tau\mu \times \tau\mu},
$

$
\pi_z\left(w,\lambda,\mu\right):=
\begin{pmatrix}
\left(yz^{n_1^+ +2}+z^{n_1^- +2}x\right)I_\mu & 0 & \cdots & 0 \\
0 & \left(yz^{n_2^+ +2} + z^{n_2^- +2}x\right)I_\mu & \cdots & 0 \\
\vdots & \vdots & \ddots & \vdots \\
0 & 0 & \cdots & \left(yz^{n_\tau^+ +2} + z^{n_\tau^- +2}x\right)I_\mu
\end{pmatrix}_{\tau\mu \times \tau\mu}.
$
\end{center}
Then we define $M(w,\lambda,\mu)$ to be the Macaulayfication $\tilde{M}(w,\lambda,\mu)^\dagger$ of $\tilde{M}(w,\lambda,\mu)$.
\end{defn}
This generalizes the rank one case given in Definition \ref{defn:modrank1}.
%
%
%
%
%
Now we can state one of the main results in \cite{BD17} which establishes the relationship between band data and modules.

\begin{thm}
\begin{enumerate}
\item For a non-periodic band datum $(w,\lambda,\mu)$, the module $M(w,\lambda,\mu)$ is an indecomposable maximal Cohen-Macaulay module over $A$. Furthermore, it is locally free of rank $\tau\mu$ on the punctured spectrum where the length of $w$ is $3\tau$.
\item Any indecomposable maximal Cohen-Macaulay module over $A$ which is locally free on the punctured spectrum is isomorphic to $M(w,\lambda,\mu)$ for some non-periodic band datum $(w,\lambda,\mu)$.
\item For two band data $(w,\lambda,\mu)$ and $(w_*,\lambda_*,\mu_*)$, the corresponding modules $M(w,\lambda,\mu)$ and $M(w_*,\lambda_*,\mu_*)$ are isomorphic if and only if $(w,\lambda,\mu)$ and $(w_*,\lambda_*,\mu_*)$ are equivalent.
\end{enumerate}
\end{thm}

\section{Conversion Formula between Band Data and Loop Data}
We have defined the conversion formula from rank one band data $(l,m,n)$ to normal loop data $(l',m',n')$ in Proposition \ref{prop:Rank1MFfromModule},
and showed that it defines a bijection in Theorem \ref{thm:c1}.
In this section, we discuss the conversion formula for higher rank cases that appeared in Theorem \ref{defn:lpwd}.
\begin{defn}[Conversion from band data to loop data]\label{def:ConversionFromBandtoLoop}
Pick a band datum $(w,\lambda,\mu)$, with
$$
w = \left(w_1,w_2,w_3,w_4,w_5,w_6,\dots,w_{3\tau-2},w_{3\tau-1},w_{3\tau}\right)\in\mathbb{Z}^{3\tau}.
$$
We define the \emph{sign word} $\delta = \delta(w)\in\left\{0,1\right\}^{3\tau}$, the \emph{correction word} $\varepsilon = \varepsilon(w)\in\left\{-1,0,1,2\right\}^{3\tau}$ of $w$ and the loop word $w'\in\mathbb{Z}^{3\tau}$ converted from $w$ below, where we regard the index $j$ of $w_j$, $\delta_j$, $\varepsilon_j$ and $w_j'$ to be in $\mathbb{Z}_{3\tau}$. First, each entry of the sign word $\delta=\delta(w)$ is defined as 
$$
\delta_j :=
\left\{
\begin{array}{cl}

\setlength\arraycolsep{0pt}
\begin{array}{l}
0 \\ \\ \\
\end{array} 

&

\setlength\arraycolsep{0pt}
\begin{array}{l}
\text{if} \\ \\ \\
\end{array} 

\quad \left\{
                         \begin{array}{l}
                         w_j<0, \text{ or} \\
                         w_j=0 \text{ and at least one of the first non-zero entries adjacent to the} \\
                         \hspace{11mm}\text{string of }$0$\text{s containing $w_j$ (exists and) is negative,}
                         \end{array}
                         \right. \\[6mm]
1 & \text{otherwise}.
\end{array}
\right.
$$
Next, each entry of the correction word $\varepsilon=\varepsilon(w)$ is defined as
$$
\varepsilon_j := -1 + \delta_{j-1} + \delta_j + \delta_{j+1}.
$$
Then each entry of the converted loop word $w'$ is defined as
$$
w_j' = w_j + \varepsilon_j
$$
and the \emph{conversion} from the band datum to the loop datum is given by
$$
(w,\lambda,\mu)\mapsto (w' = w + \varepsilon(w), \lambda' = (-1)^{l_1 + \cdots + l_\tau + \tau} \lambda , \muprime = \mu)
$$
where $l_i = w_{3i-2}$ for $i\in\mathbb{Z}_\tau$.
\end{defn}

\begin{lemma}\label{lem:SignWord}
For a band word $w\in\mathbb{Z}^{3\tau}$ and its sign word $\delta=\delta(w)\in\left\{0,1\right\}^{3\tau}$, assume $w_k=\cdots=w_l=0$ for some $k$, $l\in\mathbb{Z}_{3\tau}$. Then
\begin{enumerate}
\item $\delta_{k-1}=\delta_{l+1}=1$ implies $\delta_k=\cdots=\delta_l=1$, and
\item either $\delta_{k-1}=0$ or $\delta_{l+1}=0$ implies $\delta_k=\cdots=\delta_l=0$.
\end{enumerate}
\end{lemma}
\begin{proof}
It is obvious from the definition of the sign word. 
\end{proof}

\begin{prop}\label{prop:ConvertedToNormalForm}
The loop word $w'$ converted from a band word $w$ is always normal.
\end{prop}
\begin{proof}
Let $w$ be any band word, $\varepsilon = \varepsilon(w)$ the correction word of $w$, and $w'=w+\varepsilon(w)$ the converted loop word. We will show that $w'$ satisfies all of $4$ conditions to be normal in order.

\begin{description}[font=\normalfont\itshape\textbullet\space,labelindent=3mm,leftmargin=5mm]
\setlength\itemsep{1em}

\item[Condition 1] Assume that $w_j'=1$ for some $j\in\mathbb{Z}_{3\tau}$. As $\varepsilon_j$ takes its value in one of $-1$, $0$, $1$ and $2$, the possible combination of $w_j$ and $\varepsilon_j$ are $\left(w_j,\varepsilon_j\right) = (2,-1)$, $(1,0)$, $(0,1)$ and $(-1,2)$. But the first one is impossible as $w_j=2$ means $\delta_j=1$ so that $\varepsilon_j\ge0$. The last one is also ruled out as $w_j=-1$ yields $\delta_j=0$ so that $\varepsilon_j\le1$. In the third case,
in order for $\varepsilon_j=1$ to be hold, only one of $\delta_{j-1}$, $\delta_j$ and $\delta_{j+1}$ is $0$. But this cannot hold under $w_j=0$ according to Lemma \ref{lem:SignWord}.

Thus only the second combination remains. In this case, in order to hold $w_j=1$ and $\varepsilon_j=0$, we must have $\delta_j=1$ and $\delta_{j-1}=\delta_{j+1}=0$. Then we have $w_{j-1}\le0$. If $\delta_{j-2}=0$, we have $\varepsilon_{j-1}=0$ and hence $w_{j-1}'\le0$. Otherwise, if $\delta_{j-2}=1$, we have $\varepsilon_{j-1}=1$ and $w_{j-1}\le-1$ by Lemma \ref{lem:SignWord}.(1) which gives $w_{j-1}'\le0$. Therefore, $w_{j-1}'\le0$ holds in any case and similarly we conclude that $w_{j+1}'\le0$ also holds. This establishes the first normality condition of $w'$.

\item[Condition 2] Assume that $w_j'=0$ for some $j\in\mathbb{Z}_{3\tau}$. The possible combination of $w_j$ and $\varepsilon_j$ are
$$\left(w_j,\varepsilon_j\right)=(1,-1), (0,0), (-1,-1) \mbox{ and } (-2,-2).$$
 As before, one can easily exclude the first and the last cases.

In the second case, because $\varepsilon_j=0$, only one of $\delta_{j-1}$, $\delta_j$ and $\delta_{j+1}$ is $1$. As $w_j=0$, it follows from Lemma \ref{lem:SignWord}.(2) that one of $\delta_{j-1}$ and $\delta_{j+1}$ is $1$.  Hence, we can assume without loss of generality that $\delta_{j-1}=1$ and $\delta_j=\delta_{j+1}=0$. Then we must have $w_{j-1}\ge1$ by Lemma \ref{lem:SignWord}.(2) and $\varepsilon_{j-1}\ge0$ and hence $w_{j-1}'\ge1$. Also, we have $w_{j+1}\le0$ and $\varepsilon_{j+1}\le0$. If $w_{j+1}\le-1$, we get $w_{j+1}'\le-1$. Otherwise, if $w_{j+1}=0$, Lemma \ref{lem:SignWord}.(1) gives $\delta_{j+2}=0$ so that $\varepsilon_{j+1}=-1$ and $w_{j+1}'\le-1$ follow. Consequently, we have $w_{j-1}'\ge1$ and $w_{j+1}'\le-1$. If $\delta_{j+1}=1$, by symmetry, we get $w_{j-1}'\le-1$ and $w_{j+1}'\ge1$, establishing the second normality condition of $w'$.

In the third case, $\varepsilon_j=1$ and $\delta_j=0$ yield $\delta_{j-1}=\delta_{j+1}=1$. Then Lemma \ref{lem:SignWord}.(2) gives $w_{j-1}$, $w_{j+1}\ge1$. Since $\varepsilon_{j-1}$, $\varepsilon_{j+1}\ge0$, we have $w_{j-1}'$, $w_{j+1}'\ge1$, establishing again the second normality condition of $w'$.

\item[Condition 3] Assume there are integers $k$, $l\in\mathbb{Z}_{3\tau}$ such that $l\ne k+1$ and $w_k'=0$, $w_{k+1}'=\cdots=w_{l-1}'=-1$, $w_l'=0$. By discussion in the previous case, we have $w_k=w_l=0$ and $\delta_{k-1}=\delta_{l+1}=1$, $\delta_k=\delta_l=0$. It can be easily checked that $w_{k+1}'=\cdots=w_{l-1}'=-1$ implies $\delta_{k+1}=\cdots=\delta_{l-1}=0$, yielding $\varepsilon_{k+1}=\cdots=\varepsilon_{l-1}=-1$ and that $w_{k+1}=\cdots=w_{l-1}=0$. Putting these together would contradict Lemma \ref{lem:SignWord}(1). Therefore, we conclude that there is no subword of the form $(0,-1,-1,\dots,-1,0)$ in $w'$, establishing the third normality condition of $w'$.

\item[Condition 4] Assume $w_j'=-1$ for all $j\in\mathbb{Z}_{3\tau}$. It can be easily checked that $\delta_j=0$ for all $j$, yielding $\varepsilon_j=-1$ and $w_j=0$. But then by definition we have $\delta_j=1$ for any such $j$, which is a contradiction. Therefore, $w'$ does not consist only of $-1$, establishing the last normality condition of $w'$.
\end{description}
\end{proof}

Next we define the inverse of the above conversion formula.

\begin{defn}[Conversion from normal loop data to band data]\label{def:ConversionFromLooptoBand}
Pick a normal loop datum $(w',\lambda',\muprime)$, with
$$
w' = (w_1',w_2',w_3',w_4',w_5',w_6',\dots,w_{3\tau-2}',w_{3\tau-1}',w_{3\tau}')\in\mathbb{Z}^{3\tau}.
$$
We define the \emph{sign word} $\delta' = \delta'(w')\in\left\{0,1\right\}^{3\tau}$, the \emph{correction word} $\varepsilon' = \varepsilon'(w')\in\left\{-1,0,1,2\right\}^{3\tau}$ of $w'$ and the band word $w\in\mathbb{Z}^{3\tau}$ converted from $w'$ below, where we regard the index $j$ of $w_j'$, $s_j'$, $\varepsilon_j'$ and $w_j$ to be in $\mathbb{Z}_{3\tau}$. First, each entry of the sign word $\delta'=\delta'(w')$ is defined as 
$$
\delta_j' :=
\left\{
\begin{array}{ll}
0 & \text{if} \quad w_j'\le0 \\[1mm]
1 & \text{if} \quad w_j'>0.
\end{array}
\right.
$$
Next, each entry of the correction word $\varepsilon'=\varepsilon'(w')$ is defined as
$$
\varepsilon_j' := -1 + \delta_{j-1}' + \delta_j' + \delta_{j+1}'.
$$
Then each entry of the converted loop word $w$ is defined as
$$
w_j = w_j' - \varepsilon_j'
$$
and the \emph{conversion} from the normal loop datum to the band datum is given by
$$
(w',\lambda',\muprime) \mapsto (w = w' - \varepsilon'(w'), \lambda= (-1)^{l_1 + \cdots + l_\tau + \tau}\lambda', \mu = \muprime )
$$
where $l_i = w_{3i-2}$ for $i\in\mathbb{Z}_\tau$.
\end{defn}

\begin{example}\label{ex:ConversionExample}
Consider a band datum $(w,\lambda,\mu)$ whose band word $w$ is given as below.

\newcolumntype{x}[1]{>{\centering\hspace{0pt}}p{#1}} 

\newcommand{\corcmidrule}[1][0.6pt]{\\[\dimexpr-\normalbaselineskip-\belowrulesep-\aboverulesep-#1\relax]} 

\begin{center}
\setlength{\tabcolsep}{0.9pt}
\begin{tabular}{cccx{4.5mm}cx{4.5mm}cx{4.5mm}cx{4.5mm}cx{4.5mm}cx{4.5mm}cx{4.5mm}cx{4.5mm}cx{4.5mm}cx{4.5mm}cx{4.5mm}cx{4.5mm}cx{4.5mm}cx{4.5mm}cx{4.5mm}c}
$w$ & $=$ & $($ & $6$ & $,$ & $0$ & $,$ & $2$ & $,$ & $-1$ & $,$ & $0$ & $,$ & $-3$ & $,$ & $0$ & $,$ & $0$ & $,$ & $5$ & $,$ & $0$ & $,$ & $-2$ & $,$ & $1$ & $,$ & $-1$ & $,$ & $3$ & $,$ & $4$ & $)$ \\
\arrayrulecolor{blue}\cmidrule[0.6pt]{4-8}
\corcmidrule\arrayrulecolor{red}\cmidrule[0.6pt]{10-18}
\corcmidrule\arrayrulecolor{blue}\cmidrule[0.6pt]{20-20}
\corcmidrule\arrayrulecolor{red}\cmidrule[0.6pt]{22-24}
\corcmidrule\arrayrulecolor{blue}\cmidrule[0.6pt]{26-26}
\corcmidrule\arrayrulecolor{red}\cmidrule[0.6pt]{28-28}
\corcmidrule\arrayrulecolor{blue}\cmidrule[0.6pt]{30-32} \\[-3mm]

$\delta(w)=\delta'\left(w'\right)$ & $=$ & $($ & $1$ & $,$ & $1$ & $,$ & $1$ & $,$ & $0$ & $,$ & $0$ & $,$ & $0$ & $,$ & $0$ & $,$ & $0$ & $,$ & $1$ & $,$ & $0$ & $,$ & $0$ & $,$ & $1$ & $,$ & $0$ & $,$ & $1$ & $,$ & $1$ & $)$ \\[3.8mm]

$\varepsilon(w)=\varepsilon'\left(w'\right)$ & $=$ & $($ & $2$ & $,$ & $2$ & $,$ & $1$ & $,$ & $0$ & $,$ & $-1$ & $,$ & $-1$ & $,$ & $-1$ & $,$ & $0$ & $,$ & $0$ & $,$ & $0$ & $,$ & $0$ & $,$ & $0$ & $,$ & $1$ & $,$ & $1$ & $,$ & $2$ & $)$ \\[3.8mm]

$w'$ & $=$ & $($ & $8$ & $,$ & $2$ & $,$ & $3$ & $,$ & $-1$ & $,$ & $-1$ & $,$ & $-4$ & $,$ & $-1$ & $,$ & $0$ & $,$ & $5$ & $,$ & $0$ & $,$ & $-2$ & $,$ & $1$ & $,$ & $0$ & $,$ & $4$ & $,$ & $6$ & $)$ \\
\arrayrulecolor{blue}\cmidrule[0.6pt]{4-8}
\corcmidrule\arrayrulecolor{red}\cmidrule[0.6pt]{10-18}
\corcmidrule\arrayrulecolor{blue}\cmidrule[0.6pt]{20-20}
\corcmidrule\arrayrulecolor{red}\cmidrule[0.6pt]{22-24}
\corcmidrule\arrayrulecolor{blue}\cmidrule[0.6pt]{26-26}
\corcmidrule\arrayrulecolor{red}\cmidrule[0.6pt]{28-28}
\corcmidrule\arrayrulecolor{blue}\cmidrule[0.6pt]{30-32}

\end{tabular}
\end{center}

We computed the sign word $\delta(w)$, the correction word $\varepsilon(w)$ of $w$ and the loop word $w' = w + \varepsilon(w)$ converted from $w$. Note that $w'$ is presented in the normal form. Then we computed the sign word $\delta'\left(w'\right)$, the correction word $\varepsilon'\left(w'\right)$ of $w'$.

We underlined the spots of $w$ in blue or red, respectively, according to whether the value of $\delta$ on them is $1$ or $0$. Likewise, the spots of $w'$ are underlined according to the value of $\delta'$. Observe that both $w$ and $w'$ have the same underline pattern, implying $\delta'\left(w'\right) = \delta(w)$ and hence $\varepsilon'\left(w'\right) = \varepsilon(w)$. Therefore, the band word $w'-\varepsilon'\left(w'\right)$ converted from $w'$ is the same as the original band word $w$. 

The parameter $\lambda$ and  the holonomy parameter $\lambda'$ in this case are related by
$\lambda' = (-1)^{6-1+0+0-1+5}\lambda = -\lambda$.

\end{example}

\begin{prop}\label{prop:ConversionFormulaInverseToEachOther}
The conversion from the band datum to the loop datum and the conversion from the normal loop datum to the band datum are the inverses of each other.
\end{prop}
\begin{proof}
Let $\left(w,\lambda,\mu\right)$ be a band datum and
$$
\left(w' = w+\varepsilon(w),\lambda' = (-1)^{l_1+\cdots+l_\tau+\tau}\lambda,\muprime\right)
$$
the converted loop datum, where $l_i=w_{3i-2}$ for $i\in\mathbb{Z}_\tau$. Then let
$$
\left(w'' = w'-\varepsilon'\left(w'\right),\lambda'' = (-1)^{l_1''+\cdots+l_\tau''+\tau}\lambda',\mupprime =\mu' \right)
$$
be the band datum converted from $(w',\lambda',\muprime)$, where $l_i'' = w_{3i-2}''$ for $i\in\mathbb{Z}_\tau$. In order to show $(w'',\lambda'',\mupprime) = (w,\lambda,\mu)$, we notice that it is enough to show $w''=w$, which is also equivalent to $\varepsilon(w) = \varepsilon'(w')$. By the construction of $w$ and $w'$, therefore, we only need to show that $\delta(w)=\delta'(w')$. Denoting $\delta=\delta(w)$ and $\delta'=\delta'(w')$, we can prove $\delta_j=\delta_j'$ for each $j\in\mathbb{Z}_{3\tau}$ as follows.

\begin{description}[font=\normalfont\itshape\textbullet\space,labelindent=3mm,leftmargin=5mm]
\setlength\itemsep{1em}

\item[Case 1] $w_j\le-1$

In this case, we have $\delta_j=0$ and hence $\varepsilon_j\le1$, implying $w_j'\le0$ so that $\delta_j'=0$ follows.

\item[Case 2] $w_j=0$

If $\delta_j=0$, by Lemma \ref{lem:SignWord}.(1), at least one of $\delta_{j-1}$ and $\delta_{j+1}$ must be $0$, implying $\varepsilon_j\le0$ so that $w_j'\le0$ and hence $\delta_j'=0$. Otherwise, if $\delta_j=1$, by Lemma \ref{lem:SignWord}.(2), both $\delta_{j-1}$ and $\delta_{j+1}$ must be $1$, implying $\varepsilon_j=2$ so that $w_j'=2$ and hence $\delta_j'=1$.

\item[Case 3] $w_j\ge1$

We have $\delta_j=1$ and hence $\varepsilon_j\ge0$, implying $w_j'\ge1$ so that $\delta_j'=1$ follows.
\end{description}

Therefore, we proved that if we convert a given band datum to a loop datum and then convert it back to a band datum, it returns to itself.

Conversely, let $(w',\lambda',\muprime)$ be a normal loop datum and
$$
\left(w'' = w'-\varepsilon'\left(w\right),\lambda'' = (-1)^{l_1''+\cdots+l_\tau''+\tau}\lambda',\mupprime\right)
$$
the converted band datum, where $l_i''=w_{3i-2}''$ for $i\in\mathbb{Z}_\tau$. Then let
$$
\left(w''' = w''-\varepsilon''\left(w''\right),\lambda''' = (-1)^{l_1''+\cdots+l_\tau''+\tau}\lambda'',\muppprime = \mupprime\right)
$$
be the loop datum converted from $(w'',\lambda'',\mupprime)$. In order to show $(w''',\lambda''',\muppprime) = (w',\lambda',\muprime)$, we only need to show $w'''=w'$, which again follows from $\delta'(w')=\delta''(w'')$. As above, denoting $\delta'=\delta'(w')$ and $\delta''=\delta''(w'')$, we can prove $\delta_j'=\delta_j''$ for each $j\in\mathbb{Z}_{3\tau}$ by dividing the case.

\begin{description}[font=\normalfont\itshape\textbullet\space,labelindent=3mm,leftmargin=5mm]
\setlength\itemsep{1em}
  
\item[Case 1] $w_j'\le-2$

We have $\delta_j'=0$. As $\varepsilon_j'\ge-1$, we also have $w_j''\le-1$ and hence $\delta_j''=0$.

\item[Case 2] $w_j'=-1$

We have $\delta_j'=0$. Since $w'$ is normal, $w'$ does not consist only of $-1$.

If the first non-$(-1)$ element to the left of $w_j'$ is less than or equal to $-2$, namely, if there is an integer $k\in\mathbb{Z}_{3\tau}\setminus\left\{j\right\}$ such that $w_k'\le-2$, $w_{k+1}'=\cdots=w_j'=-1$, we have $w_k''\le-1$, $w_{k+1}'', \dots, w_j''\le0$, implying $\delta_j''=0$ by Lemma \ref{lem:SignWord}.(2).

If the first non-$(-1)$ element to the left of $w_j'$ is greater than or equal to $1$, namely, if there is $k\in\mathbb{Z}_{3\tau}\setminus\left\{j\right\}$ such that $w_k'\ge1$, $w_{k+1}'=\cdots=w_j'=-1$, we have $\varepsilon_{k+1}'\ge0$ and hence $w_{k+1}''\le-1$, $w_{k+2}'', \dots, w_j''\le0$, implying $\delta_j''=0$ by Lemma \ref{lem:SignWord}.(2).

The above discussion also applies to the elements to the right of $w_j'$. It remains a case where both the first non-$(-1)$ element to the left and the right of $w_j'$ are $0$, namely, there are $k$, $l\in\mathbb{Z}_{3\tau}\setminus\left\{j\right\}$ such that $k<j<l$ and $w_k'=0$, $w_{k+1}'=\cdots=w_{l-1}'=-1$, $w_l'=0$. But this
violates the third condition for $w'$ to be normal. Thus we conclude that $\delta_j''=0$ whenever $w_j'=-1$ and $w'$ is normal.

\item[Case 3] $w_j'=0$

We have $\delta_j'=0$. Because $w'$ is normal, one of $w_{j-1}'\le-1$, $w_{j+1}'\ge1$ or $w_{j-1}'\ge1$, $w_{j+1}'\le-1$ or $w_{j-1}'$, $w_{j+1}'\ge1$ must hold. In the first case, we have $\varepsilon_j'=0$, $w_j''=0$ and $\delta_{j-1}''=0$ bt arguments in the case 1 and 2. Hence, Lemma \ref{lem:SignWord}.(2) gives $\delta_j''=0$. The second case can also be handled in the same way. In the third case, we have $\varepsilon_j'=1$, $w_j''=-1$ and hence $\delta_j''=0$.

\item[Case 4] $w_j'=1$

We have $\delta_j'=1$. As $w'$ is normal, both $w_{j-1}'$ and $w_{j+1}'$ are less than or equal to $0$. Therefore, we have $\varepsilon_j'=0$, $w_j''=1$ and hence $\delta_j''=1$.

\item[Case 5] $w_j'=2$

We have $\delta_j'=1$.  If $w'$ consists only of $2$, then $w''$ consists only of $0$, whence $\delta_j''=1$. Now assume that this is not the case.

If the first non-$2$ element to the left of $w_j'$ is less than or equal to $1$, namely, if there is $k\in\mathbb{Z}_{3\tau}\setminus\left\{j\right\}$ such that $w_k'\le1$, $w_{k+1}'=\cdots=w_j'=2$, we automatically get $w_k'\le0$ by the first condition that $w'$ is normal. Thus we have $\varepsilon_{k+1}'=1$ and hence $w_{k+1}''=1$, $w_{k+2}'', \dots, w_j''\ge0$. And if the first non-$2$ element to the left of $w_j'$ is greater than or equal to $3$, namely, if there is $k\in\mathbb{Z}_{3\tau}\setminus\left\{j\right\}$ such that $w_k'\ge3, w_{k+1}'=\cdots=w_j'=2$, we have $w_k''\ge1$, $w_{k+1}'', \dots, w_j''\ge0$. Therefore, in any cases, either $w_j''\ge1$ or there is $k\in\mathbb{Z}_{3\tau}\setminus\left\{j\right\}$ such that $w_k''\ge1$, $w_{k+1}'', \dots, w_j''\ge0$. This result also applies to the elements to the right of $w_j'$ and therefore we conclude that either $w_j''\ge1$, or there are $k$, $l\in\mathbb{Z}_{3\tau}\setminus\left\{j\right\}$\ with $k<j<l$ such that $w_k''\ge1$, $w_{k+1}'', \dots, w_{l-1}''\ge0$, $w_l''\ge1$. In either case, we have $\delta_j''=1$.

\item[Case 6] $w_j'\ge3$

We have $\delta_j'=1$. As $\varepsilon_j'\le2$, we also have $w_j''\ge1$ and hence $\delta_j''=1$.

\end{description}

We proved that if we convert a given normal loop datum to a band datum and then convert it back to a loop datum, it returns to itself.
\end{proof}

\begin{remark}
We could have chosen different forms in the conversion formula from band words to loop words. To be more specific, we can take any sign word $\delta^*:=\delta^*(w)\in\left\{0,1\right\}^{3\tau}$ in Definition \ref{def:ConversionFromBandtoLoop}
satisfying the following:
\begin{itemize}
\item
$\delta_j^* = 1$ \quad if \ \ $w_j >0$,
\item
$\delta_j^* = 0$ \quad if \ \ $w_j <0$,
\item
if $w_j =0$ and $\delta_j^* < \delta_{j+1}^*$, then the first non-zero entry to the left of $w_{j+1}$ (exists and) is negative and
\\
\hspace*{26mm} the first non-zero entry to the right of $w_{j}$ (exists and) is positive,
\item
if $w_j =0$ and $\delta_j^* > \delta_{j+1}^*$, then the first non-zero entry to the left of $w_{j+1}$ (exists and) is positive and
\\
\hspace*{26mm} the first non-zero entry to the right of $w_{j}$ (exists and) is negative.
\end{itemize}
Then we can prove that loop words obtained from the same band word should be equivalent to each other,
no matter which sign word $\delta^*(w)$ is used. In this case, all converted loop words satisfy the following `\emph{quasi-normal}'
conditions:
\begin{itemize}
\item any subword of the form $\left(a,0,b\right)$ in $w'$ satisfies $a\le-1$, $b\ge1$ or $a\ge1$, $b\le-1$ or $a, b\ge1$,
\item any subword of the form $\left(a,1,b\right)$ in $w'$ satisfies $a\ge2$, $b\le0$ or $a\le0$, $b\ge2$ or $a, b\le0$,
\item $w'$ has no subword of the form $(0,-1,-1,\dots,-1,0)$, and
\item $w'$ has no subword of the form $(1,2,2,\dots,2,1)$.
\end{itemize}
The conversion formula from `quasi-normal' loop words to band words remains the same as in Definition \ref{def:ConversionFromLooptoBand}. If one convert a given band word to a loop word using any sign word and then convert it back to a band word, it returns to itself. Conversely, if one convert a given `quasi-normal' loop word to a band
word and then convert it back to a loop word using some sign word, it is equivalent to the original one.

\end{remark}

Finally, we can state our main theorem.

\begin{thm}\label{thm:MainThoerem}
The following diagram commutes.

\begin{equation*}
\begin{tikzcd}[ampersand replacement=\&, sep=20pt] 
  \&
  \varphi\left(w',\lambda,1\right)\in
  \hspace{-15mm}
  \&
  \underline{\operatorname{MF}}(xyz)
    \arrow[dl, swap, "\mathrel{\rotatebox{48}{$\simeq$}}", shift right,
           "\begin{matrix}
              \vspace{-9mm}
              \text{(Eisenbud, 1980)}
              \\
              \operatorname{coker}
            \end{matrix}\hspace{4mm}"]
  \& \&
  \\[15mm]
  M\left(w,\lambda,1\right) \in
  \hspace{-12mm}
  \&
  \underline{\operatorname{CM}}(A)
    \arrow[ur, swap, shift right,
           "\text{resolve}"]
  \&
  \&
  WF\left(\mathcal{P}\right)
    \arrow[ul, swap, "\mathrel{\rotatebox{-48}{$\simeq$}}",
           "\quad
            \begin{matrix}
              \text{localized mirror functor}
                       \end{matrix}"]
  \&
  \hspace{-11mm}
  \ni L\left(w',\lambda',1\right)
  \\[5mm]
  \left(w,\lambda,1\right)
  \hspace{-12mm}
  \&
  \text{\color{black} band data}
    \arrow[rr, leftrightarrow, color=black,
           "\text{conversion formula}"]
    \arrow[u, swap, color=black, swap,
           "\begin{matrix}
              \text{classification Thm} \\
              \text{(}\color{black}\text{Burban-Drozd)}
            \end{matrix}"]
  \&
  \&
  \text{\color{black} loop data}
    \arrow[u, swap, color=black,
           "\begin{matrix}
              \text{classify objects} \\
            \end{matrix}"]
  \&
  \hspace{-12mm}
  \left(w',\lambda',1\right)
\end{tikzcd}
\end{equation*}

Namely, given a non-degenerate band datum $\left(w,\lambda,1\right)$ and the corresponding loop datum $\left(w',\lambda',1\right)$,

(1) $M\left(w,\lambda,1\right)$ corresponds to $\varphi\left(w',\lambda,1\right)$ under Eisenbud's theorem, i.e.,
$$
M\left(w,\lambda,1\right) \cong \operatorname{coker}\underline{\varphi}\left(w',\lambda,1\right) \quad\text{in}\
\ \underline{\operatorname{CM}}(A),
$$

(2) $L\left(w',\lambda',1\right)$ corresponds to $\varphi\left(w',\lambda,1\right)$ under the localized
mirror functor, i.e.,
$$
\LocalF\left(L\left(w',\lambda',1\right)\right)\ \cong \varphi\left(w',\lambda,1\right)\quad
\text{in}\ \ \underline{\operatorname{MF}}(xyz).
$$

\end{thm}

\begin{proof}
We will prove (1) and (2) in Section \ref{sec:higherrank} and \ref{sec:MfFromLag}, respectively.
See Theorem \ref{thm:MFFromModule} and \ref{thm:MFFromLag}.
\end{proof}

\begin{remark}
The above theorem is between the band data and the loop data.
On the other hand, there are indecomposable maximal Cohen-Macaulay modules that are not locally free on the punctured spectrum, and they correspond to the string data (Burban-Drozd).  On the mirror side,  these correspond to non-compact Lagrangians which start and end at punctures (i.e. Lagrangian immersions of $\mathbb{R}$). Most of the proof in this paper would carry over to these cases without much difficulty, and hence we do not discuss them.
\end{remark}

\begin{remark}
In the sequel, we will generalize the above theorem in the following sense.

(1) The above theorem is for the multiplicity one ($\mu=1$). We will prove the
corresponding statements for the higher multiplicity, namely
for $(w,\lambda, \mu)$ with  $\mu \geq 2$ as well. The mirror Lagrangian will be given by
twisted complexes (of length $\mu$)  of a single Lagrangian  $L(w',\lambda,1)$.

(2)  The above theorem holds even for periodic words.
(Recall that the definition of the band and loop data excludes periodic ones). They correspond to some
direct sums of indecomposable objects corresponding to non-periodic words in each category.

(3)  The theorem holds for degenerate band/loop data.  For degenerate cases, the correspondences become somewhat subtle, and we will explain how to handle them in the sequel.
\end{remark}

\section{Matrix Factorizations Arising from Modules: Higher Rank Cases}\label{sec:higherrank}
In this section, we work out the higher rank analogue of Section \ref{sec:mod1} to prove Theorem \ref{thm:1}.

Let us briefly recall our setting.
Let $S:=\mathbb{C}[[x,y,z]]$ be the formal power series ring of three variables and $A:=\mathbb{C}[[x,y,z]]/(xyz)$ its quotient ring. Given a band datum $\left(w,\lambda,1\right)$ of length $3\tau$ and multiplicity $1$, let $\delta:=\delta\left(w\right)$ and $\varepsilon:=\varepsilon\left(w\right)$ be the sign word and the correction word of $w$, respectively and $(w',\lambda',1)$ the loop datum obtained by the conversion formula, which are described in Definition \ref{def:ConversionFromBandtoLoop}. Then the module $M\left(w,\lambda,1\right)$ corresponding to $\left(w,\lambda,1\right)$ and the Lagrangian $L(w',\lambda',1)$ corresponding to $(w',\lambda',1)$ give matrix factors of $xyz$ over $S$. To compare them, we proposed the \emph{canonical form} $\varphi(w',\lambda,1)$ of matrices arising from $(w,\lambda,1)$ as

\adjustbox{scale=1,center}{$
\varphi(w',\lambda,1) :=
\begin{psmallmatrix}
z & -y^{m_1'-1} & 0 & 0 & \displaystyle\cdots & 0 & -\lambda^{-1}x^{-l_1'} \\[2mm]
-y^{-m_1'} & x & -z^{n_1'-1} & 0 & \displaystyle\cdots & 0 & 0 \\[2mm]
0 & -z^{-n_1'} & y & -x^{l_2'-1} & \displaystyle\cdots & 0 & 0 \\[0mm]
0 & 0 & -x^{-l_2'} & z & \ddots & \vdots & \vdots \\[0mm]
\vdots & \vdots & \vdots & \ddots & \ddots & -y^{m_\tau'-1} & 0 \\[2mm]
0 & 0 & 0 & \displaystyle\cdots & -y^{-m_\tau'} & x & -z^{n_\tau'-1} \\[2mm]
-\lambda x^{l_1'-1} & 0 & 0 & \displaystyle\cdots & 0 & -z^{-n_\tau'} & y
\end{psmallmatrix}_{3\tau\times3\tau}.
$}

As always, we denote $l_i' := w_{3i-2}'$, $m_i' := w_{3i-1}'$, $n_i' := w_{3i}'$ and \uline{regard $x^a$, $y^a$ or $z^a$ to be zero if $a<0$}. The entries of the matrix $\varphi(w',\lambda,1)$ are considered to be in $S$ and then it yields an obvious $S$-module homomorphism $\varphi(w',\lambda,1):S^{3\tau}\rightarrow S^{3\tau}$ between free $S$-modules.

Then we denote by $\mathunderbar{\varphi}(w',\lambda,1)$ the matrix $\varphi(w',\lambda,1)$ modulo $xyz$. That is, $\mathunderbar{\varphi}(w',\lambda,1)$ is the same form as $\varphi(w',\lambda,1)$ but entries are considered to be elements of $A$. This yields an $A$-module homomorphism $\mathunderbar{\varphi}(w',\lambda,1):A^{3\tau}\rightarrow A^{3\tau}$ between free $A$-modules.

Note that each $A$-module $M$ also has a natural $S$-module structure, and we will denote by $M_S$ the
corresponding $S$-module. Namely, the underlying set and abelian group structure of $M_S$ is the same as $M$ whereas the scalar multiplication is defined by $f\cdot u := \left[f\right]\cdot u$, where $f\in S$, $\left[f\right]\in A = S/(xyz).$ and $u$ in the left hand side is an element of $M_S$ and $u$ in the right hand side is the same element of $M$. Readers should note that $xyz \cdot u =0$ for any $u \in M_S$ even if it is an $S$-module operation. Note also that a subset of $M$ or $M_S$ generates the whole set with the $A$-module structure if and only if it does with the $S$-module structure.

Our goal in this section is to show that the matrix factor of $xyz$ over $S$ arising from the module $M(w,\lambda,1)$ fits into the above canonical form $\varphi(w',\lambda,1)$.

\begin{thm}\label{thm:MFFromModule}
For a nondegenerate band datum $(w,\lambda,1)$ of length $3\tau$ and multiplicity $1$, we can construct the following free resolution of $M_S\left(w,\lambda,1\right)$ as an $S$-module:
\begin{equation}\label{eqn:ExactSequence1}
\begin{tikzcd}[column sep = 20pt] 
  0 \arrow[r] & S^{3\tau} \arrow[rr, "\varphi{(w',\lambda,1)}"] & & S^{3\tau} \arrow[r, "\pi"] & M_S\left(w,\lambda,1\right) \arrow[r] & 0.
\end{tikzcd}
\end{equation}
The map $\pi$ will be constructed during the proof. In particular,
$$
M\left(w,\lambda,1\right) \cong \operatorname{coker}\mathunderbar{\varphi}(w',\lambda,1)
$$
holds as $A$-modules.
\end{thm}
This together with the conversion formula proves Theorem \ref{thm:1}.

\subsection{Notations and Properties}

For  rings $S:=\mathbb{C}[[x,y,z]]$ and $A = \mathbb{C}[[x,y,z]]/(xyz)$,
let
$$
A^\tau = \mathbb{C}[[x_1,y_1,z_1]]/(x_1y_1z_1) \times \cdots \times \mathbb{C}[[x_\tau,y_\tau,z_\tau]]/(x_\tau y_\tau z_\tau)
$$
be the free $A$-module of rank $\tau$. We introduce an alternative notation for elements of $S$, $A$ and
$A^\tau$. This is to represent the variables such as $x$, $y$ and $z$ at once.

\begin{notation}

\begin{enumerate}
\item
We denote by
$$
\chi_{3i-2}:=x,\quad \chi_{3i-1}:=y\quad\text{and}\quad \chi_{3i}:=z
$$
for $i \in\mathbb{Z}_\tau$ the elements in the ring $S$ or $A$, depending on the context. Then
we \uline{consider $\chi_j^a$ to be zero if $a<0$}, for any $j\in\mathbb{Z}_{3\tau}$.

\item
We abbreviate some monomials in $A^\tau$ as
$$
^l X_{3i-2}^m:=x_i^l y_i^m
,\quad
^m X_{3i-1}^n:= y_i^m z_i^n
\quad\text{and}\quad
^n X_{3i}^l:= z_i^n x_i^l
$$
for $i\in\mathbb{Z}_\tau$ and $l$, $m$, $n\in \mathbb{Z}_{\ge1}$.

\item
To deal with the eigenvalue $\lambda$, we introduce the symbols
$$
\Lambda_{j}^+ := \begin{cases}\lambda & \quad\text{if }j=1\text{ and }\delta_{1} = 1, \\ 1 & \quad\text{otherwise}\end{cases}
\quad\text{and}\quad
\Lambda_{j}^- := \begin{cases}\lambda^{-1} & \quad\text{if }j=1\text{ and }\delta_{1}=0, \\ 1 & \quad\text{otherwise}\end{cases}
$$
for $j \in \mathbb{Z}_{3\tau}$, depending on the band datum $\left(w,\lambda,1\right)$.

\end{enumerate}

\end{notation}

For later use, we summarize some useful properties of new notations in the following proposition.

\begin{prop}\label{prop:ChiRelations}

\begin{enumerate}
\item
As elements of $A^\tau$,
$$
^a X_j^b = \chi_j^a \chi_{j+1}^b \mathbf{e}_i
$$
for $i\in\mathbb{Z}_\tau$, $j\in\mathbb{Z}_{3\tau}$ with $3i-2\le j \le 3i$ and $a$, $b\in\mathbb{Z}_{\ge1}$, where $\mathbf{e}_i$ is the column vector in $A^\tau$ whose $i$-th entry is $1$ and the rest are $0$.
Here the multiplication of $\chi_j^a\chi_{j+1}^b\in (S\text{ or }A)$ to $\mathbf{e}_i\in \left(\left(A^\tau\right)_S\text{
or }A^\tau\right)$ can be either an $S$-module or an $A$-module operation.

\item
We have the arithmetic rules
\begin{equation}\label{eqn:ScalarMultiplicationArithmetic}
\chi_{j+3i}\left(^a X_{j}^b\right) = {^{a+1} X_{j}^b},\quad
\chi_{j+3i+1} \left(^a X_{j}^b\right) = {^a X_{j}^{b+1}},\quad\text{and}\quad
\chi_{j+3i+2} \left(^a X_{j}^b\right) = 0
\end{equation}
for any $i\in\mathbb{Z}_\tau$, $j\in\mathbb{Z}_{3\tau}$ and $a$, $b\in\mathbb{Z}_{\ge1}$, regardless whether these are $S$-module or $A$-module operations.

\end{enumerate}
\end{prop}

Using new notations, we can rewrite the relevant modules and matrices.

\begin{prop}
\begin{enumerate}
\item
$\tilde{M}\left(w,\lambda,1\right)$ is the $A$-submodule of $A^{\tau}$ generated by $6\tau$ elements
$$
H_{j}:= {^2 X_{j}^2} \quad \text{and} \quad G_{j}:= \Lambda_{j}^+ \left(^1 X_{j-1}^{w_{j}^+ +2}\right) + \Lambda_{j}^- \left(^{w_{j}^- +2} X_{j}^1\right) \qquad \left(j\in\mathbb{Z}_{3\tau}\right).
$$
\item
The canonical matrix $\varphi(w',\lambda,1)$ can be written as

$$
\varphi(w',\lambda,1) =
\begin{psmallmatrix}
\chi_{3\tau} & -\Lambda_{2}^+ \chi_{2}^{w_{2}'-1} & 0 & 0 & \displaystyle\cdots & 0 & -\Lambda_{1}^- \chi_{1}^{-w_{1}'} \\[2mm]
-\Lambda_{2}^- \chi_{2}^{-w_{2}'} & \chi_{1} & -\Lambda_{3}^+ \chi_{3}^{w_{3}'-1} & 0 & \displaystyle\cdots & 0 & 0 \\[2mm]
0 & -\Lambda_{3}^- \chi_{3}^{-w_{3}'} & \chi_{2} & -\Lambda_{4}^+ \chi_{4}^{w_{4}'-1} & \displaystyle\cdots & 0 & 0 \\[0mm]
0 & 0 & -\Lambda_{4}^- \chi_{4}^{-w_{4}'} & \chi_{3} & \ddots & \vdots & \vdots \\[1mm]
\vdots & \vdots & \vdots & \ddots & \ddots & -\Lambda_{3\tau-1}^+ \chi_{3\tau-1}^{w_{3\tau-1}'-1} & 0 \\[2mm]
0 & 0 & 0 & \displaystyle\cdots & -\Lambda_{3\tau-1}^- \chi_{3\tau-1}^{-w_{3\tau-1}'} & \chi_{3\tau-2} & -\Lambda_{3\tau}^+ \chi_{3\tau}^{w_{3\tau}'-1} \\[2mm]
-\Lambda_{1}^+ \chi_{1}^{w_{1}'-1} & 0 & 0 & \displaystyle\cdots & 0 & -\Lambda_{3\tau}^- \chi_{3\tau}^{-w_{3\tau}'} & \chi_{3\tau-1}
\end{psmallmatrix}_{3\tau\times3\tau}
$$

\noindent where entries are elements of $S$. The matrix $\mathunderbar{\varphi}(w',\lambda,1)$ can be written in the same form but entries are considered to be in $A$.

\end{enumerate}
\end{prop}

\begin{proof}
The second statement can be checked easily. For the first one, 
recall that $\tilde{M}\left(w,\lambda,1\right)$ is the $A$-submodule of $A^\tau$ generated by all columns of the $6$ matrices
$$
x^2 y^2 I_\tau,\quad y^2 z^2 I_\tau,\quad z^2 x^2 I_\tau,\quad
\pi_x\left(w,\lambda,1\right)
,\quad
\pi_y\left(w,\lambda,1\right)
\quad\text{and}\quad
\pi_z\left(w,\lambda,1\right)
$$
in $A^{\tau\times\tau}$, where the last three are given by

\adjustbox{scale=1,center}{$
\begin{psmallmatrix}
x^{l_1^- +2}y & zx^{l_2^+ +2} & \displaystyle\cdots & 0 \\[0mm]
0 & x^{l_2^- +2}y & \ddots & \vdots \\[1mm]
\vdots & \vdots & \ddots & zx^{l_\tau^+ +2} \\[2mm]
\lambda zx^{l_1^+ +2} & 0 & \displaystyle\cdots & x^{l_\tau^- +2}y
\end{psmallmatrix}_{\tau\times\tau}
,\quad
\begin{psmallmatrix}
xy^{m_1^+ +2} + y^{m_1^- +2}z & \displaystyle\cdots & 0 \\[1mm]
\vdots & \ddots & \vdots \\[2mm]
0 & \displaystyle\cdots & xy^{m_\tau^+ +2} + y^{m_\tau^- +2}z
\end{psmallmatrix}_{\tau\times\tau}
,\quad\begin{psmallmatrix}
yz^{n_1^+ +2}+z^{n_1^- +2}x & \displaystyle\cdots & 0 \\[1mm]
\vdots & \ddots & \vdots \\[2mm]
0 & \displaystyle\cdots & yz^{n_\tau^+ +2} + z^{n_\tau^- +2}x
\end{psmallmatrix}_{\tau\times\tau}
$}
in order. Now the $i$-th column of each matrix as an element of $A^\tau$ is
$$
{^2 X_{3i-2}^2}
,\quad
{^2 X_{3i-1}^2}
,\quad
{^2 X_{3i}^2},
$$
$$
\Lambda_{3i-2}^+ \left(\Lambda_{3i-2}^-\right)^{-1} \left(^1 X_{3i-3}^{w_{3i-2}^+ +2}\right) + {^{w_{3i-2}^- +2} X_{3i-2}^1}
,\quad
{^1 X_{3i-2}^{w_{3i-1}^+ +2}} + {^{w_{3i-1}^- +2} X_{3i-1}^1}
\quad\text{and}\quad
{^1 X_{3i-1}^{w_{3i}^+ +2}} + {^{w_{3i}^- +2} X_{3i}^1},
$$
respectively, in order.
\end{proof}

Recall in linear algebra that the adjoint matrix $\operatorname{adj}B$ of a square matrix $B$ is the transpose of the cofactor matrix of $B$ whose $(a,b)$-entry is $(-1)^{a+b}$ times the $(a,b)$-minor of $B$. A useful property of the adjoint matrix is that it satisfies
\begin{equation}\label{eqn:AdjointEquation}
B\cdot\operatorname{adj}B = \operatorname{adj}B \cdot B = \left(\det B\right)I
\end{equation}
where $I$ is the identity matrix of the same size as $B$. This is valid whenever the matrix $B$ has entries in a commutative ring.
\begin{prop}\label{prop:DetAndAdjOfPhiS}
Let $\left(w,\lambda,1\right)$ be a band datum of length $3\tau$ and multiplicity $1$.

\begin{enumerate}
\item The determinant of $\varphi:=\varphi\left(w',\lambda,1\right)$ is
$$
\det\varphi\left(w',\lambda,1\right) = x^\tau y^\tau z^\tau u
$$
where
$$
u := u\left(w',\lambda,1\right)
:=
1
- \displaystyle\prod_{j=1}^{3\tau}\Lambda_{j}^+\chi_{j}^{w_{j}'-2}
- \displaystyle\prod_{j=1}^{3\tau} \Lambda_{j}^- \chi_{j}^{-w_{j}'-1}
$$
is a unit in $S$ if and only if $\left(w,\lambda,1\right)$ is nondegenerate. In particular, $u = 1$ unless $w_j\ge0$ for all $j$ or $w_j\le0$ for all $j$.

\item The $(a,b)$-entry of the adjoint matrix of $\varphi\left(w',\lambda,1\right)$ is

$$
\left(\operatorname{adj}\varphi\left(w',\lambda,1\right)\right)_{ab}=
\begin{cases}
\displaystyle\prod_{j=a}^{a-2} \chi_j & \text{if}\quad a=b, \\[7mm]
\left(\displaystyle\prod_{j=b}^{a-2} \chi_j\right)
\left(\displaystyle\prod_{j=a+1}^{b} \Lambda_{j}^+ \chi_j^{w_{j}'-1}\right)
+
\left(\displaystyle\prod_{j=a}^{b-2} \chi_j\right)
\left(\displaystyle\prod_{j=b+1}^{a} \Lambda_{j}^- \chi_j^{-w_{j}'}\right) & \text{if} \quad a \ne b
\end{cases}
$$
where we regard the product $\displaystyle\prod_{j=a}^b h_j$ as $1$ if $b=a-1$ and $h_a \cdots h_{3\tau}h_1 \cdots h_b$ if $b\le a-2$.
\end{enumerate}
\end{prop}

\begin{proof}
It is a straightforward calculation using the fact $\chi_j^{w_{j}'-1}\chi_j^{-w_{j}'}=0$ for all $j$.
\end{proof}

In particular, Proposition \ref{prop:DetAndAdjOfPhiS}.(1) together with  the equation \ref{eqn:AdjointEquation} and the fact that $S$ is an integral domain demonstrates that the map $\varphi:S^{3\tau}\rightarrow S^{3\tau}$ is injective. This establishes the exactness of the sequence \ref{eqn:ExactSequence1} at the leftmost arrow.

\begin{cor}\label{cor:Psi}
For a nondegenerate band datum $\left(w,\lambda,1\right)$ of length $3\tau$ and multiplicity $1$, define a matrix $\tilde{\psi} := \tilde{\psi}\left(w',\lambda,1\right)\in S^{3\tau\times 3\tau}$ by
$$
\left(\tilde{\psi}\left(w',\lambda,1\right)\right)_{ab}
:=
\begin{cases}
\chi_{a}\chi_{a+1}
& \text{if} \quad a=b, \\[3mm]
\displaystyle\prod_{j=a+1}^{b} \Lambda_{j}^+ \chi_j^{w_{j}'-2+\mathbb{1}_{j\in\left\{a+1,b\right\}}}
+
\displaystyle\prod_{j=b+1}^{a} \Lambda_{j}^- \chi_j^{-w_{j}'-1+\mathbb{1}_{j\in\left\{b+1,a\right\}}}
& \text{if} \quad a \ne b
\end{cases}
$$
where we regard the product symbol as in Proposition \ref{prop:DetAndAdjOfPhiS}.(2) and $\mathbb{1}_{j\in\left\{a,b\right\}}$ is $1$ if $j\in\left\{a,b\right\}$ and $0$ otherwise. Then the matrix $\psi := \psi\left(w',\lambda,1\right) \in S^{3\tau \times 3\tau}$ defined by
$$
\psi\left(w',\lambda,1\right)
:=
\left(u\left(w',\lambda,1\right)\right)^{-1}\tilde{\psi}\left(w',\lambda,1\right)
$$
satisfies
$$
\varphi \psi = \psi \varphi = xyz I_{3\tau}.
$$
That is, both $\varphi$ and $\psi$ are matrix factors of $xyz$ over $S$.
\end{cor}

\begin{proof}
One can check that
$$
\operatorname{adj}\varphi = x^{\tau-1} y^{\tau-1} z^{\tau-1} \tilde{\psi}.
$$
Substituting this and $\det\varphi = x^\tau y^\tau z^\tau u$ into the equation \ref{eqn:AdjointEquation} gives
$$
x^{\tau-1} y^{\tau-1} z^{\tau-1} \varphi \tilde{\psi} = x^{\tau-1} y^{\tau-1} z^{\tau-1} \tilde{\psi} \varphi = x^\tau y^\tau z^\tau u I_{3\tau}.
$$
We may cancel out the common terms $x^{\tau-1} y^{\tau-1} z^{\tau-1}$ in both sides since $S$ is an integral domain, which yields the desired result.
\end{proof}

\subsection{Generators, Relations and Macaulayfying Elements}\label{subsubsec:GeneratorsAndMacaulayfyingElements}

Recall that $\tilde{M} := \tilde{M}\left(w,\lambda,\mu\right)$ is generated by elements $G_1, \dots, G_{3\tau}$ together with $H_1, \dots, H_{3\tau}$ in $A^\tau$. Here we investigate the relations among them by plotting them on the lattice diagram for $A^\tau$. Figure \ref{fig:HigherRankGeneratorDiagram} illustrates  a part of generator diagram for $\tilde{M}$ representing $G_j$, $H_j$ and $G_{j+1}$ in each case depending on the value of $\delta_j$ and $\delta_{j+1}$.

\begin{figure}[h]
     \centering
     \begin{subfigure}[t]{0.47\textwidth}
         \includegraphics[scale=0.5]{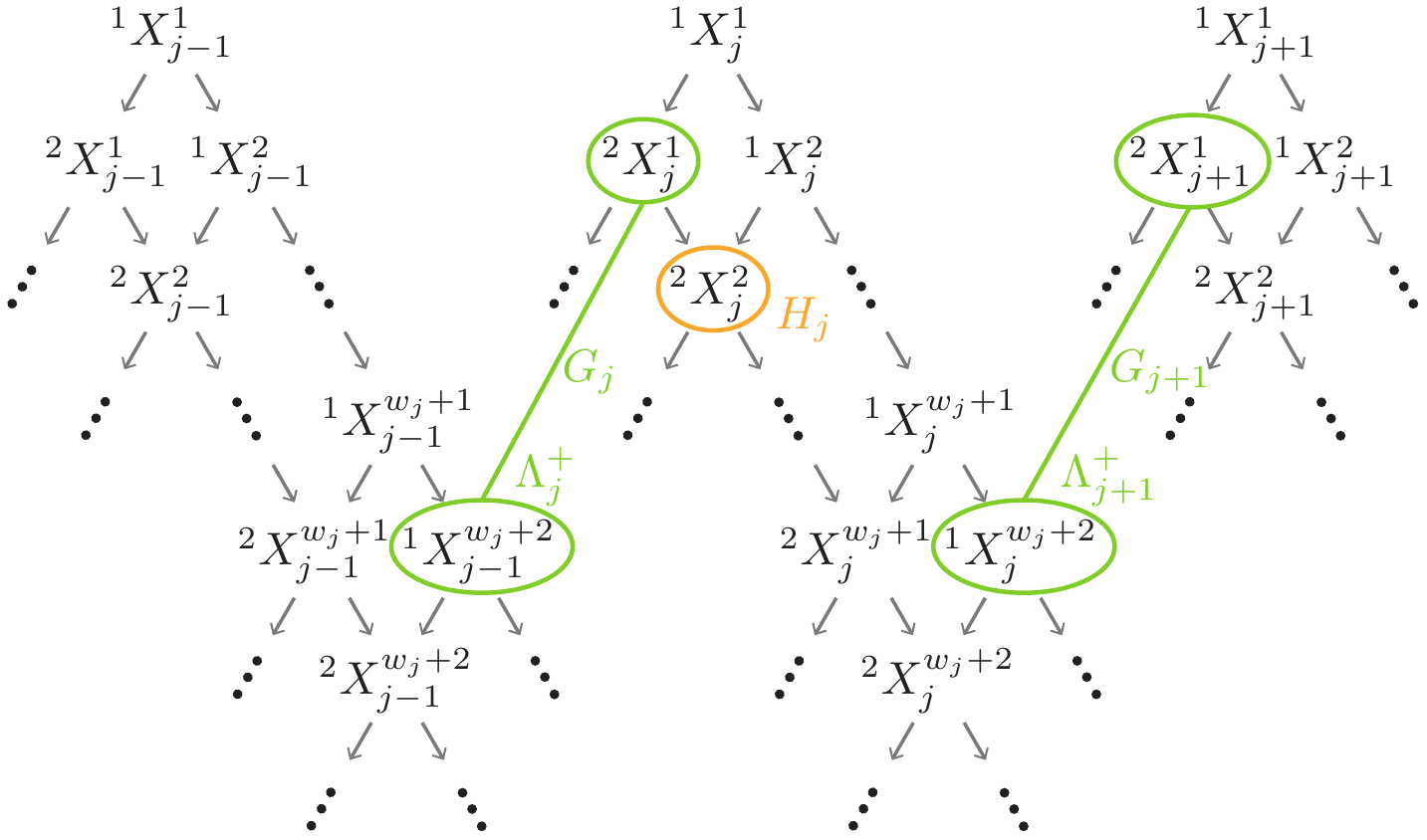}
         \caption{$\delta_j=\delta_{j+1}=1$}
         \label{fig:HigerRankGeneratorDiagramCase1}
         \centering
     \end{subfigure}
     \begin{subfigure}[t]{0.47\textwidth}
         \includegraphics[scale=0.5]{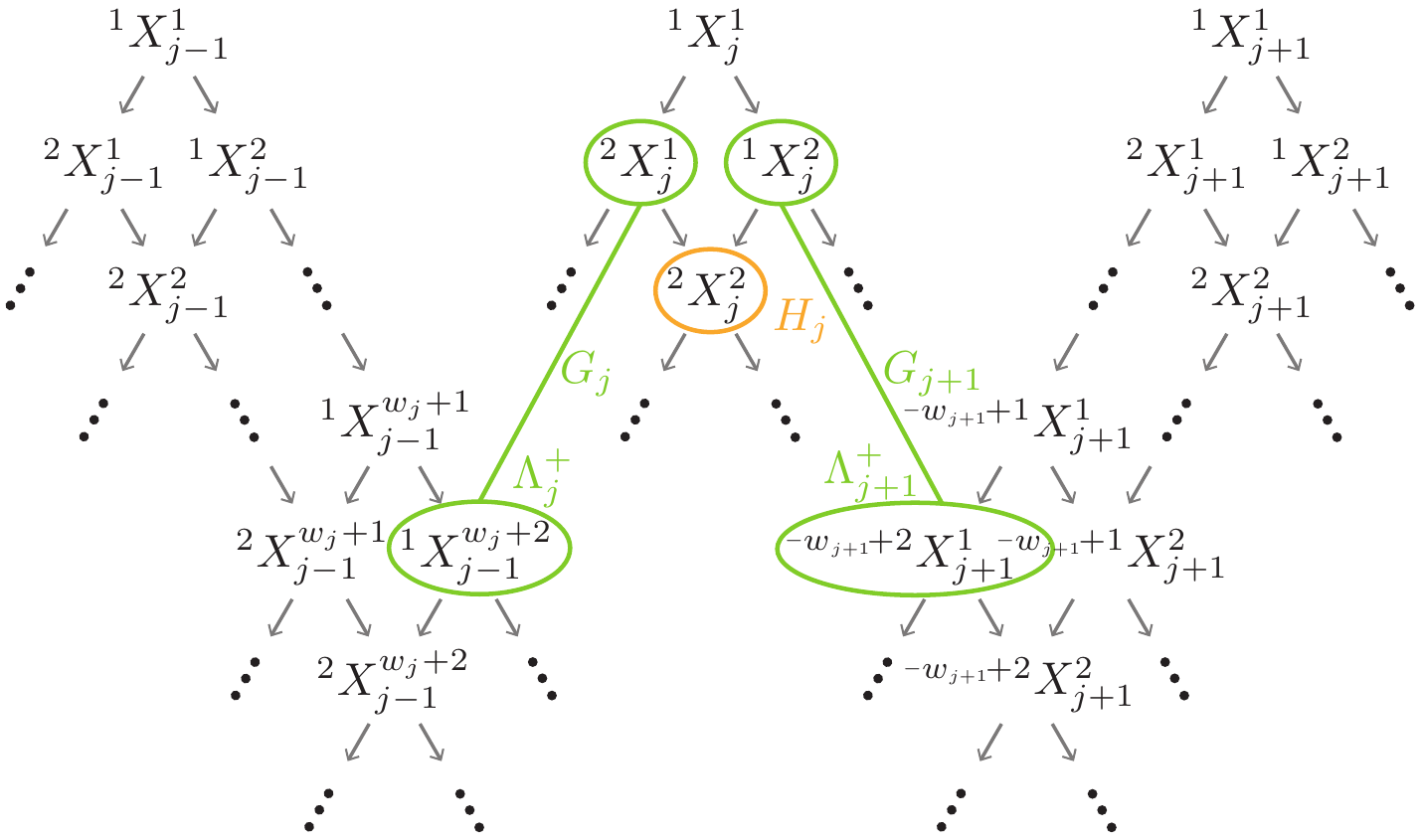}
         \caption{$\delta_j=1$ and $\delta_{j+1}=0$}
         \label{fig:HigherRankGeneratorDiagramCase2}
         \centering
     \end{subfigure}
     \\[5mm]
     \begin{subfigure}[t]{0.47\textwidth}
         \includegraphics[scale=0.5]{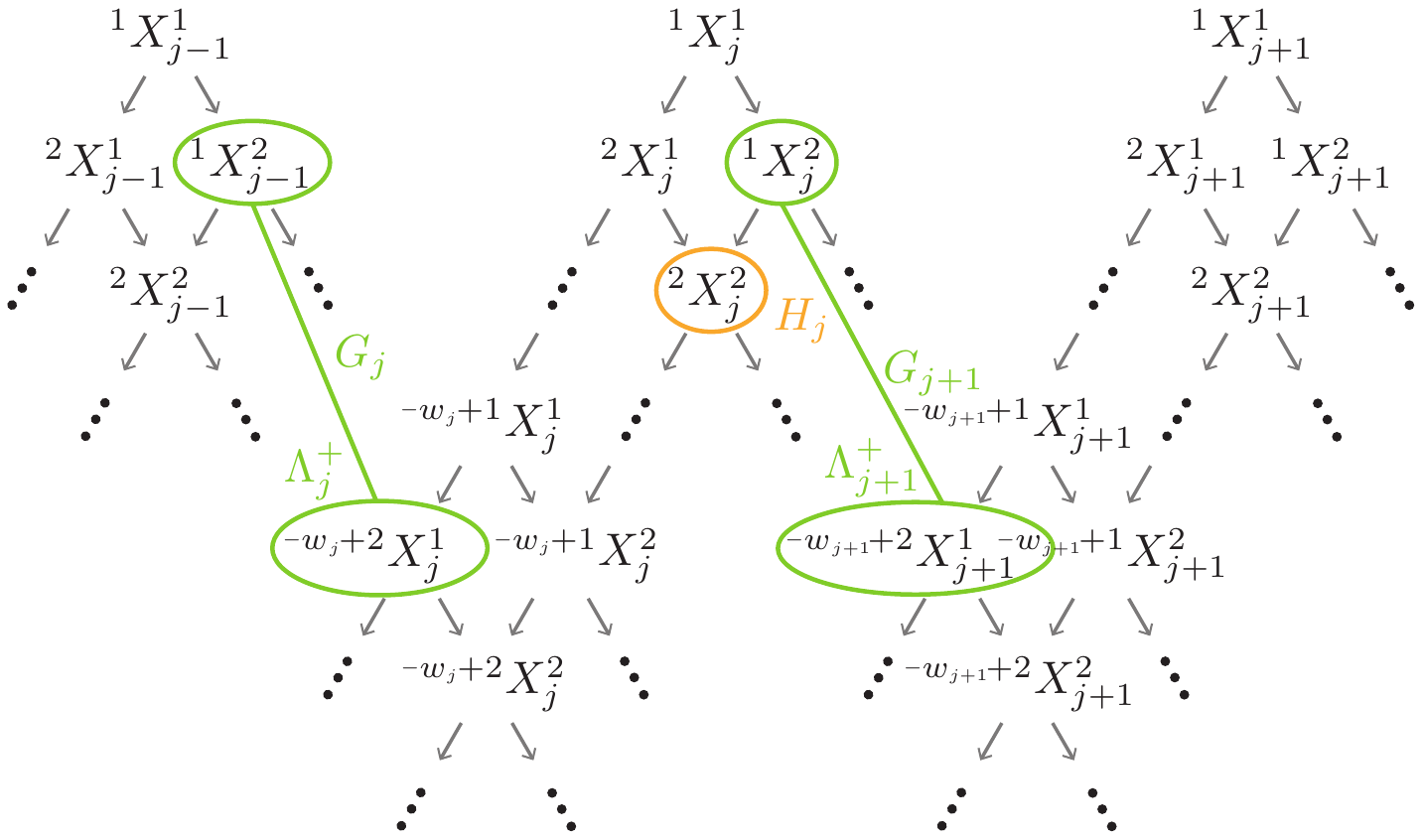}
         \caption{$\delta_j=\delta_{j+1}=0$}
         \label{fig:HigerRankGeneratorDiagramCase3}
         \centering
     \end{subfigure}
     \begin{subfigure}[t]{0.47\textwidth}
         \includegraphics[scale=0.5]{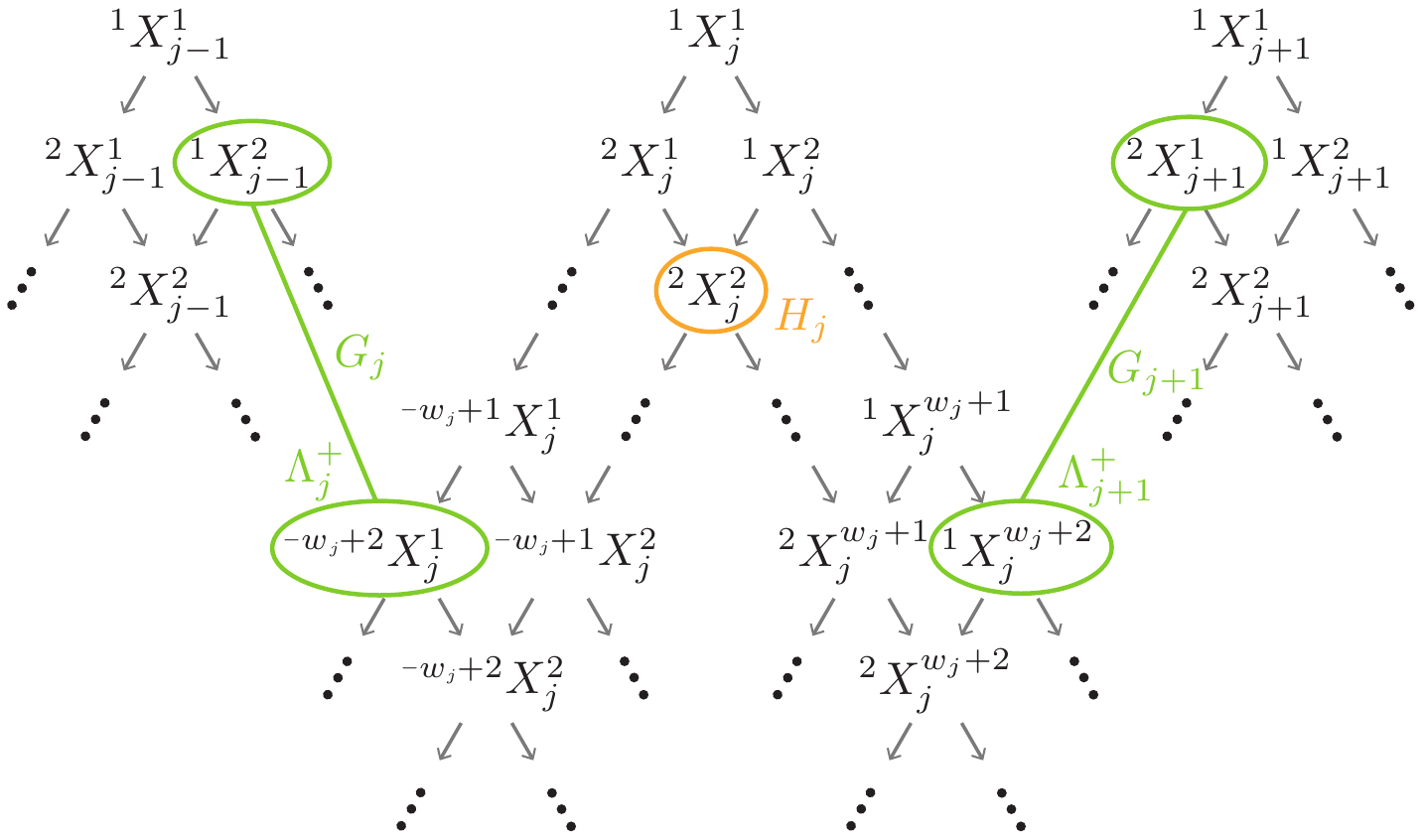}
         \caption{$\delta_j=0$ and $\delta_{j+1}=1$}
         \label{fig:HigherRankGeneratorDiagramCase4}
         \centering
     \end{subfigure}
        \caption{A part of generator diagram for $\tilde{M}$ containing $G_j$, $H_j$ and $G_{j+1}$ in each case}
        \label{fig:HigherRankGeneratorDiagram}
\end{figure}

In each case, we can observe the relations among $G_j$'s and determine whether $H_j$ is spanned by them or not, as follows:
\begin{itemize}
\item
$\delta_j = \delta_{j+1} = 1$:
\hspace{6.5mm}
$
\Lambda_{j+1}^+ \chi_{j+1}^{w_{j+1}+1} G_j
=
\chi_j G_{j+1}
$
,\hspace{14mm}
$
H_j
=
\chi_{j+1} G_j
$

\item
$\delta_j = 1$, $\delta_{j+1} = 0$:
\hspace{15mm}
$
\chi_{j+1} G_j
=
\chi_j G_{j+1}
$
,\hspace{14mm}
$
H_j
=
\chi_{j+1} G_j
=
\chi_j G_{j+1}
$
\item
$\delta_j = \delta_{j+1} = 0$:
\hspace{18mm}
$
\chi_{j+1} G_j
=
\Lambda_j^- \chi_j^{-w_j+1} G_{j+1}
$
,\quad
$
H_j = \chi_j G_{j+1}
$

\item
$\delta_j = 0$, $\delta_{j+1} = 1$:
\quad
$
\Lambda_{j+1}^+ \chi_{j+1}^{w_{j+1}+1} G_j
=
\Lambda_j^- \chi_j^{-w_j+1} G_{j+1}
$
,\quad
$H_j$ may not be spanned by $G_1, \dots, G_{3\tau}$
\end{itemize}
The relations among $G_j$'s can be summarized as
\begin{equation}\label{eqn:RelationsOnG}
\Lambda_{j+1}^+ \chi_{j+1}^{w_{j+1}^++1} G_j
=
\Lambda_j^- \chi_j^{w_j^-+1} G_{j+1}
\end{equation}
for $j\in\mathbb{Z}_{3\tau}$ in any cases, noting that $\Lambda_j^- = 1$, $w_j\ge0$ when $\delta_j = 1$ and $\Lambda_j^+ = 1$, $w_j\le0$ when $\delta_j = 0$.

On the other hand, $H_j$ is not spanned by $G_1, \dots, G_{3\tau}$ only if $\delta_j=0$ and $\delta_{j+1}=1$ hold simultaneously. The converse is not true in general, as $w_j = 0$ may still hold in this case. As a consequence, $\tilde{M}$ is generated by $G_1, \dots, G_{3\tau}$ together with those $H_j$'s. In such a case, we find in Figure \ref{fig:HigherRankGeneratorDiagramCase4} the relations
\begin{equation}\label{eqn:RelationsOnH}
-\chi_{j-2} G_{j}
+
\Lambda_{j}^- \chi_{j}^{-w_{j}} H_{j}
= 0,\quad
\chi_{j-1} H_{j}
= 0\quad\text{and}\quad
\Lambda_{j+1}^+ \chi_{j+1}^{w_{j+1}} H_{j}
-
\chi_{j} G_{j+1}
=0
\end{equation}
among $H_j$ and $G_j$'s. Note that these relations are `\emph{enough}' in the sense that any other
relations containing $H_j$ are $S$-linear combinations of these. Namely, if $sH_j$ can be represented as an $S$-linear combination of $G_j$'s for some $s\in S$, then $s$ should be an $S$-linear combination of $\chi_j^{-w_j}$, $\chi_{j-1}$ and $\chi_{j+1}^{w_{j+1}}$ and the whole equation should
be an $S$-linear combination of the above $3$ relations.

Next, we find Macaulayfying elements of $\tilde{M}$ in $A^\tau$. For any $\imath\in\mathbb{Z}_{3\tau}$ with $\delta_{\imath}=1$ and $\delta_{\imath+1}=0$, there are $\jmath\in\mathbb{Z}_{3\tau}\setminus\left\{\imath\right\}$ and $\kappa \in \left\{0,\dots,\jmath-\imath-1\right\}$ such that
\begin{equation}\label{eqn:InequalitiesOfWord}
w_\imath \ge1
,\quad
w_{\imath+1} = \cdots = w_{\imath+\kappa} = 0
,\quad
w_{\imath+\kappa+1}\le-1
,\quad
w_{\imath+\kappa+2}, \dots, w_{\jmath} \le 0
\quad\text{and}\quad
w_{\jmath+1} \ge 1.
\end{equation}
One should be careful about the indices. If $\kappa = 0$, then the second condition in \ref{eqn:InequalitiesOfWord} becomes an empty condition. If $\kappa = \jmath-\imath-1$, then the fourth one is empty.

In this case, we have the identities
\begin{equation}\label{eqn:IdentitiesOfSignWord}
\delta_{\imath} = 1
,\quad
\delta_{\imath+1} = \cdots = \delta_{\imath+\kappa+1} = \cdots = \delta_{\jmath} = 0
,\quad
\delta_{\jmath+1} = 1
\end{equation}
by definition of the sign word $\delta$. Figure \ref{fig:MacaulayfyingElementGeneratorDiagram} shows the relevant part of the generator diagram for $\tilde{M}$.
\begin{figure}[H]
\includegraphics[scale=0.55]{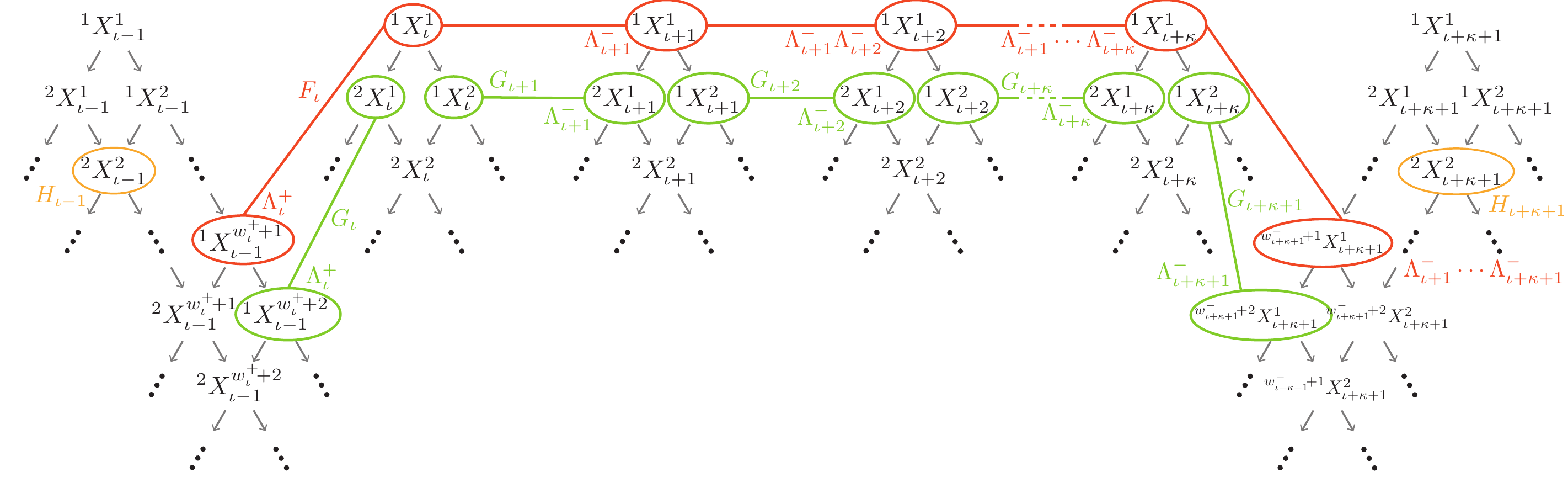}
\centering
\caption{A Macaulayfying element of $\tilde{M}$ in $A^{3\tau}$}
\label{fig:MacaulayfyingElementGeneratorDiagram}
\end{figure}

One may notice that the element $F_{\imath}\in A^{3\tau}\setminus\tilde{M}$ marked with orange color is a Macaulayfying element of $\tilde{M}$ in $A^{3\tau}$, which can be written as
\begin{align}\label{eqn:MacaulayfyingElement}
\begin{split}
F_{\imath}
&:=
\Lambda_\imath^+ \left(^1 X_{\imath-1}^{w_\imath+1}\right)
+
^1X_\imath^1
+
\Lambda_{\imath+1}^- \left(^1 X_{\imath+1}^1\right)
+
\Lambda_{\imath+1}^- \Lambda_{\imath+2}^- \left(^1 X_{\imath+2}^1\right)
\\
&\hspace{48mm}
+\cdots+
\left(\Lambda_{\imath+1}^- \cdots \Lambda_{\imath+\kappa}^-\right) \left(^1 X_{\imath+\kappa}^1\right) +
\left(\Lambda_{\imath+1}^- \cdots \Lambda_{\imath+\kappa+1}^-\right) \left(^{-w_{\imath+\kappa+1} +1} X_{\imath+\kappa+1}^1\right)
\\
&=
\Lambda_\imath^+ \left(^1 X_{\imath-1}^{w_\imath+1}\right)
+
\sum_{a=0}^{\kappa} \left(\prod_{b=1}^a \Lambda_{\imath+b}^- \right)\left(^1X_{\imath+a}^1 \right)
+
\left(\prod_{b=1}^{\kappa+1} \Lambda_{\imath+b}^- \right) \left(^{-w_{\imath+\kappa+1}+1}X_{\imath+\kappa+1}^1\right).
\end{split}
\end{align}
Indeed, we can express $\chi_{\imath+1}F_{\imath}$, $\chi_{\imath+2}F_{\imath}$ and $\chi_{\imath+3}F_{\imath}$ as $S$-linear combinations of
$$
G_{\imath-1}''
,\ \
G_{\imath}
,\ \
G_{\imath+1}
,\ \
G_{\imath+2}'
,\ \
\dots
,\ \
G_{\imath+\kappa+1}'
$$
where
\begin{equation}\label{eqn:DefinitionOfG'}
G_{\imath-1}''
:=
\begin{cases}
H_{\imath-1} & \text{if}\ \ \delta_{\imath-1} = 0
\\
G_{\imath-1} & \text{if}\ \ \delta_{\imath-1} = 1
\end{cases}
,\quad
G_{\imath}'
:=
G_{\imath}
,\quad\dots,\quad
G_{\jmath}'
:=
G_{\jmath}
\quad\text{and}\quad
G_{\jmath+1}'
:=
H_{\jmath},
\end{equation}
confirming that $xF_{\imath}$, $yF_{\imath}$ and $zF_{\imath}$ are elements of $\tilde{M}$. Note that $G_{\imath}' = G_{\imath}$ and $G_{\imath+1}' = G_{\imath+1}$ always hold. Moreover, each of $G_{\imath-1}''$ and $G_{\imath+a}'$ is one of $G_j$'s or one of $H_j$'s that is not generated by $G_j$'s, by the discussion in the above paragraph and identities $\delta_{\imath}=1$, $\delta_{\jmath}=0$
and $\delta_{\jmath+1}=1$. For later use, we write down the explicit formulas as follows:
\begin{equation}\label{eqn:RelationsOnF}
\left\{
\setlength\arraycolsep{1pt}
\begin{array}{rcccccccccccccccccc}
\chi_{\imath} F_\imath
&=& &\phantom{+}& 
G_{\imath}
& & & & &+&
\zeta_{\imath,-1,2}^{-1} G_{\imath+3}'
&+&
\cdots
&+&
\zeta_{\imath,-1,\kappa+1}^{-1} G_{\imath+\kappa+1}'
\\[4mm]
\chi_{\imath+1} F_\imath
&=& & & &\phantom{+}&
G_{\imath+1}
& & &+&
\zeta_{\imath,-1,2}^0 G_{\imath+3}'
&+&
\cdots
&+&
\zeta_{\imath,-1,\kappa+1}^0 G_{\imath+\kappa+1}'
\\[3mm]
\chi_{\imath+2} F_\imath
&=&
\Lambda_{\imath}^+ \chi_{\imath}^{w_{\imath}'-1} G_{\imath-1}''
& & & & & & +
\Lambda_{\imath+1}^- \chi_{\imath+1}^{-w_{\imath+1}'} G_{\imath+2}'
&+&
\zeta_{\imath,-1,2}^1 G_{\imath+3}'
&+&
\cdots
&+&
\zeta_{\imath,-1,\kappa+1}^1 G_{\imath+\kappa+1}'
\end{array}
\right.
\end{equation}
where
\begin{equation}\label{eqn:DefinitionOfZeta}
\zeta_{\imath,a,b}^c
:=
\begin{cases}
\left(\displaystyle\prod_{d=a+1}^{b}\Lambda_{\imath+d}^-\right) \chi_{\imath+b}^{-w_{\imath+b}'-1}
&
\text{if}\ \ c\equiv b\ \ \left(\operatorname{mod}\ 3\right)
\\[5mm]
\ 0
&
\text{otherwise}
\end{cases}
\end{equation}
for $b\in \left\{2, \dots, \kappa+1\right\}$, $a\in \left\{-1, \dots, b-1\right\}$ and $c\in\mathbb{Z}$. These can be observed in Figure \ref{fig:MacaulayfyingElementGeneratorDiagram} or computed straightforwardly from equation \ref{eqn:MacaulayfyingElement} applying the arithmetic rules \ref{eqn:ScalarMultiplicationArithmetic}, the (in)equalities \ref{eqn:InequalitiesOfWord}, the identities \ref{eqn:IdentitiesOfSignWord} and the conversion formula
$
w_j' = w_j + \delta_{j-1} + \delta_j + \delta_{j+1} - 1
$
in Definition \ref{def:ConversionFromBandtoLoop}.

We remark here on some basic properties of the symbols $\zeta_{\imath,a,b}^c$. First note that the monomial
$\chi_{\imath+b}^{-w_{\imath+b}'-1}$ is $1$ for $b \in \left\{2, \dots, \kappa\right\}$ and is nonzero for $b=\kappa+1$.
This follows from the observations $w_{\imath+2}' = \cdots = w_{\imath+\kappa}' = -1$ and $w_{\imath+\kappa+1}' \le -1$.
$\zeta_{\imath,a,b}^c$ is just an element of $\mathbb{C}$, namely $\displaystyle \prod_{d=a+1}^b \Lambda_{\imath+d}^-$, unless $b=\kappa+1$.
Based on this fact, we find that the symbols satisfy  algebraic relations
\begin{equation}\label{eqn:RelationsOnZeta}
\zeta_{\imath,a,b}^{b+d} \zeta_{\imath,b,c}^{c}
=
\zeta_{\imath,a,c}^{c+d}
\end{equation}
for any $c\in\left\{2,\dots,\kappa-2\right\}$, $b\in\left\{2,\dots,c-1\right\}$, $a\in\left\{-1,\dots,b-1\right\}$ and $d\in\mathbb{Z}$,
i.e. whenever all terms are defined.

Other than the form of $F_{\imath}$, there seem to be no further Macaulayfying elements. This will be revealed to be true in the proof of Theorem \ref{thm:MFFromModule}.

\subsection{Proof of the Theorem}

Let $M_{{0}}:=M_{{0}}\left(w,\lambda,1\right)$ be the $A$-submodule of $A^\tau$ generated by the $3\tau$ elements $G_{1}, \dots, G_{3\tau}$. Then $\tilde{M}$ is generated by elements of $M_{{0}}$ together with $H_1, \dots, H_{3\tau}$ in $A^\tau$. And then $M$ is generated by elements of $\tilde{M}$ together with the Macaulayfying elements of $\tilde{M}$ in $A^\tau$.

The overall strategy of the proof is as follows. We will complete it in four steps. In step 1, we will first find a (part of) free resolution of $\left(M_{{0}}\right)_S$ as an $S$-module. In step 3, we modify it to get a (part of) free resolution of $\tilde{M}_S$. In step 4, we fix it again to finally establish the desired free resolution \ref{eqn:ExactSequence1} of $M_S$. But for some technical reasons, we treat the special cases where $w_j\ge0$ for all $j$ or $w_j\le0$ for all $j$ separately in step 2.

In step 3 and step 4, we need some technical lemmas, which allow us to modify the free resolution of a module when we add or replace its generators.
Lemma \ref{lem:MatrixExpansionLemma} describes the `\emph{matrix expansion}' process according to adding
a generator and lemma \ref{lem:MatrixReductionLemma} gives the `\emph{matrix reduction}' process according
to replacing a redundant generator. There is also a geometric version of matrix reduction in lemma \ref{lem:HomotopyTypeV}
and remark \ref{rmk:MatrixReduction}.

\begin{lemma}\label{lem:MatrixExpansionLemma}
Let $\pi_{{0}} \in A^{a \times b}$ and $\varphi_{{0}} \in S^{b \times c}$ be matrices such that the sequence
$$
\begin{tikzcd}[column sep = 20pt] 
S^{c} \arrow[r, "\varphi_{{0}}"] & S^{b} \arrow[r, "\pi_{{0}}"] & \left(A_S\right)^{a} \arrow[r] &
0
\end{tikzcd}
$$
of $S$-modules is exact and $B \in A^{a\times 1}$  and $C^T \in S^{1\times d}$, $D \in S^{b\times d}$ be matrices such that the enlarged matrices
$$
\pi_{{1}}
:=
\left(
\begin{array}{c|c}
& \\
\qquad \pi_{{0}} \qquad \ & B \hspace*{-1mm} \\
& \\
\end{array}
\right)
\in A^{a \times \left(b+1\right)}
\quad\text{and}\quad
\varphi_{{1}}
:=
\left(
\begin{array}{c|c}
\\
\qquad \mbox{$\varphi_{{0}}$} \qquad \ &
\hspace{3mm} D \hspace*{1mm}
\\
\\ \hline
\\[-4mm]
\quad 0 \quad \ & \hspace{3mm} C^T \hspace*{1mm}
\end{array}
\right)
\in S^{\left(b+1\right)\times\left(c+d\right)}
$$
satisfy $\pi_{{1}}\varphi_{{1}} = 0\in A^{a\times\left(c+d\right)}$. Assume further that $\operatorname{im}C^T \subseteq S$ contains the conductor
$$
\operatorname{ann}_S \left(\operatorname{im}\pi_{{1}}\left/\operatorname{im}\pi_{{0}}\right.\right)
:=
\left\{s \in S \left| s \left(\operatorname{im}\pi_{{1}}\right) \subseteq \operatorname{im}\pi_{{0}}\right.\right\}
=
\left\{s \in S \left| s B \subseteq \operatorname{im}\pi_{{0}}\right.\right\}
$$
of $\operatorname{im} \pi_{{0}}$ in $\operatorname{im} \pi_{{1}}$. Then we have the following modified exact sequence of $S$-modules:
$$
\begin{tikzcd}[column sep = 20pt] 
S^{c+d} \arrow[r, "\varphi_{{1}}"] & S^{b+1} \arrow[r, "\pi_{{1}}"] & \left(A_S\right)^{a} \arrow[r]
& 0.
\end{tikzcd}
$$
\end{lemma}

\begin{proof}
Note that $\operatorname{im}\pi_{{1}} = \left(A_S\right)^a$ immediately follows from $\operatorname{im}\pi_{{0}}
= \left(A_S\right)^a$. The inclusion $\operatorname{im}\varphi_{{1}} \subseteq \ker\pi_{{1}}$ is also
immediate from $\pi_{{1}} \varphi_{{1}} = 0$. For the opposite inclusion, assume that $v:= \begin{pmatrix} v_1 \\ v_2 \end{pmatrix} \in S^{b+1}$ is an element of $\ker\pi_{{1}}$ for some $v_1\in S^b$ and $v_2 \in S$. Then $\pi_{{1}} v=0\in A^a$ yields $\pi_{{0}}v_1+Bv_2 = 0 \in A^a$. Thus, we have $v_2B \in \operatorname{im}\pi_{{0}}$ and the assumption implies $v_2 \in \operatorname{im}C^T$. Put $v_2=C^Tu_2$ for some $u_2 \in S^d$ and substitute it to the previous equation, which gives $\pi_{{0}}v_1 + BC^Tu_2 = 0 \in A^a$. Note that the assumption $\pi_{{1}}\varphi_{{1}} = 0$ implies the equation $\pi_{{0}}D + B C^T = 0 \in A^{a \times d}$. Combining these, we get $\pi_{{0}}v_1 - \pi_{{0}}Du_2 = 0 \in A^a$, or equivalently, $v_1-Du_2 \in \ker\pi_{{0}} = \operatorname{im}\varphi_{{0}}$. Thus we have $v_1 = \varphi_{{0}}u_1 + Du_2$ for some $u_1 \in S^c$. Now put $u:= \begin{pmatrix} u_1 \\ u_2 \end{pmatrix} \in S^{c+d}$ then one can check $v = \varphi_{{1}} u \in \operatorname{im}\varphi_{{1}}$, implying $\ker\pi_{{1}} \subseteq \operatorname{im}\varphi_{{1}}$.
\end{proof}

\begin{lemma}\label{lem:MatrixReductionLemma}
Let $\pi_{{1}}\in A^{a\times b}$, $B\in A^{a\times 1}$, $C\in S^{b\times c}$, $D\in S^{b\times1}$ and $E^T\in S^{1\times c}$
be matrices and $u\in S$ be a unit such that for matrices
$$
\pi_{{0}}
:=
\left(
\begin{array}{c|c}
& \\
\qquad \pi_{{1}} \qquad \ & B \hspace*{-1mm} \\
& \\
\end{array}
\right)
\in A^{a \times \left(b+1\right)}
\quad\text{and}\quad
\varphi_{{0}}
:=
\left(
\begin{array}{c|c}
\\
\qquad \mbox{$C$} \qquad \ &
D \hspace*{-1mm}
\\
\\ \hline
\\[-4mm]
\quad E^{T} \quad \ & u \hspace*{-1mm}
\end{array}
\right)
\in S^{\left(b+1\right)\times\left(c+1\right)}
$$
the sequence
$$
\begin{tikzcd}[column sep = 20pt] 
S^{c+1} \arrow[r, "\varphi_{{0}}"] & S^{b+1} \arrow[r, "\pi_{{0}}"] & \left(A_S\right)^{a} \arrow[r]
& 0
\end{tikzcd}
$$
of $S$-modules is exact. Setting
$
\varphi_{{1}} := C-Du^{-1}E^T \in S^{b\times c}
$, we have the following modified exact sequence of $S$-modules:
$$
\begin{tikzcd}[column sep = 20pt] 
S^{c} \arrow[r, "\varphi_{{1}}"] & S^{b} \arrow[r, "\pi_{{1}}"] & \left(A_S\right)^{a} \arrow[r] &
0.
\end{tikzcd}
$$
\end{lemma}

\begin{proof}
Consider the following diagram of $S$-modules:
$$
\newcommand{\scriptverteq}{\mathrel{\rotatebox{90}{$\scriptstyle=$}}}
\tikzset{
    labl/.style={anchor=south, rotate=90, inner sep=.5mm}
}
\begin{tikzcd}[arrow style=tikz,>=stealth,row sep=4em,column sep=7em,
    execute at end picture={
    \path (\tikzcdmatrixname-1-1) -- (\tikzcdmatrixname-2-1)
    coordinate[pos=0.5] (aux1)
    (\tikzcdmatrixname-1-2) -- (\tikzcdmatrixname-2-2)
    coordinate[pos=0.5] (aux2)
    (aux1) -- (aux2) node[midway,sloped]{$\hspace{-2mm} \begin{matrix} \scriptstyle\varphi_{{1}} \\[-1mm] \scriptverteq \end{matrix}$};
    }] 
S^{c}\oplus S
  \arrow[r, "\varphi_{{0}} = \spmat{ C & D \\[1mm] E^T & u }"]
  \arrow[d, "\cong" labl, swap, "\spmat{ I_c & 0 \\[1mm] u^{-1}E^T & 1 }\hspace{3mm}"]
&
S^{b}\oplus S
  \arrow[r, "\pi_{{0}} = \spmat{\pi_{{1}} & B}"]
  \arrow[d, "\cong" labl, "\spmat{ I_b & -Du^{-1} \\[1mm] 0 & u^{-1} }"]
&
\left(A_S\right)^{a}
  \arrow[r]
  \arrow[d, "\cong" labl, "I_a"]
&
0
  \arrow[d, equal]
\\
S^{c}\oplus S
  \arrow[r, "\spmat{ C-Du^{-1}E^T & 0 \\[1mm] 0 & 1 }"]
&
S^{b}\oplus S
  \arrow[r, "\spmat{\pi_{{1}} & 0}"]
&
\left(A_S\right)^{a}
  \arrow[r]
&
0
\end{tikzcd}
$$
One can immediately check by matrix calculations that both squares in the diagram commute. Also, note
that the vertical maps are all isomorphisms, hence the exactness of the top row yields the exactness
of the bottom row, which also implies the exactness of the desired sequence.
\end{proof}

\begin{proof}[Proof of Theorem \ref{thm:MFFromModule}]
We now prove the theorem, step by step as illustrated above.
\\
\\
\noindent\textbf{Step 1}: Find a (part of) free resolution of $\left(M_{{0}}\right)_S$.

Note that $M_{{0}} = \left<G_1,\dots,G_{3\tau}\right>_A \in A^\tau$ yields an $S$-module $\left(M_{{0}}\right)_S = \left<G_1,\dots,G_{3\tau}\right>_S \in \left(A_S\right)^\tau$, where $A_S$ is the same as $A$\ but considered as an $S$-module. That is, $M_{{0}}$ and $\left(M_{{0}}\right)_S$ are the same as underlying sets and the subset $\left\{G_1,\dots,G_{3\tau}\right\}$ $A$-generates $M_{{0}}$ and at the same time $S$-generates $\left(M_{{0}}\right)_S$.

Consider the $\tau\times 3\tau$ matrix $\pi_{{0}}:= \pi_{{0}}\left(w,\lambda,1\right)\in A^{\tau\times 3\tau}$ whose $j$-th column is $G_j\in A^\tau$ for $j\in\mathbb{Z}_{3\tau}$.
That is,
$$
\pi_{{0}}
:=
\left(
\arraycolsep=2pt\def\arraystretch{1}
\begin{array}{c|c|c|c|c|c}
\color{colorG} G_1 & \color{colorG} G_2 & \color{colorG} G_3 & \color{colorG} \cdots & \color{colorG} G_{3\tau-1} & \color{colorG} G_{3\tau}
\end{array}
\right)_{\tau\times 3\tau}.
$$
Then regard the matrix as an $S$-module homomorphism $\pi_{{0}}:S^{3\tau}\rightarrow \left(A_S\right)^\tau$, which is well-defined using the natural isomorphism $\operatorname{Hom}_S\left(S,A_S\right)\cong A_S$ where $A_S$ is the same as $A$ but considered as an $S$-module. As a result, the $S$-module $\left(M_{{0}}\right)_S$ is the image of $\pi_{{0}}:S^{3\tau}\rightarrow\left(A_S\right)^\tau$. Restricting the codomain, we also denote the map by $\pi_{{0}}:S^{3\tau}\rightarrow \left(M_{{0}}\right)_S$, which becomes automatically surjective.

We need to figure out the kernel of $\pi_{{0}}$ to further resolve $\left(M_{{0}}\right)_S$. It is equivalent to find relations among generators $G_{j}$ of $\left(M_{{0}}\right)_S$, or the columns of $\pi_{{0}}$. Based on the observations
\begin{equation}\label{eqn:RelationsOnG2}
\Lambda_{j}^+ \chi_{j}^{w_{j}^+ +1} G_{j-1} = \Lambda_{j-1}^- \chi_{j-1}^{w_{j-1}^- +1} G_{j}
\quad\text{and}\quad
\chi_{j-2}\chi_{j-1} G_{j}
=
0
\quad
\left(j \in \mathbb{Z}_{3\tau}\right)
\end{equation}
from equation \ref{eqn:RelationsOnG} in Subsection \ref{subsubsec:GeneratorsAndMacaulayfyingElements}, we define two matrices below. First, the matrix $\varphi_{{0}}:=\varphi_{{0}}\left(w,\lambda,1\right)\in S^{3\tau\times 3\tau}$ is defined by
\begin{equation}\label{eqn:DefinitionOfTtildePhiS}
\varphi_{{0}}:=
\setlength\arraycolsep{0pt}
\begin{matrix}
&
\color{colorR} \SMALL\text{$R_1$ \hspace{10mm} $R_2$ \hspace{9.5mm} $R_3$} \hspace{5.5mm} \cdots \hspace{8mm} \text{$R_{3\tau-1}$ \hspace{11.5mm} $R_{3\tau}$ \hspace{9mm}}
\\
\begin{smallmatrix}
\color{colorG} G^1
\\[3.7mm]
\color{colorG} G^2
\\[3.7mm]
\color{colorG} G^3
\\[1.5mm]
\color{colorG} \vdots
\\[5mm]
\color{colorG} G^{3\tau-1}
\\[4.5mm]
\color{colorG} G^{3\tau}
\\[1mm]
\end{smallmatrix}
&
\begin{psmallmatrix}
\Lambda_{3\tau}^- \chi_{3\tau}^{w_{3\tau}^- +1} & -\Lambda_{2}^+ \chi_{2}^{w_{2}^+ +1} & 0 & \displaystyle\cdots & 0 & 0 \\[1mm]
0 & \Lambda_{1}^- \chi_{1}^{w_{1}^- +1} & -\Lambda_{3}^+ \chi_{3}^{w_{3}^+ +1} & \displaystyle\cdots & 0 & 0 \\[0mm]
0 & 0 & \Lambda_{2}^- \chi_{2}^{w_{2}^- +1} & \ddots & \vdots & \vdots \\[1mm]
\vdots & \vdots & \vdots & \ddots & -\Lambda_{3\tau-1}^+\chi_{3\tau-1}^{w_{3\tau-1}^+ +1} & 0 \\[2mm]
0 & 0 & 0 & \displaystyle\cdots & \Lambda_{3\tau-2}^- \chi_{3\tau-2}^{w_{3\tau-2}^- +1} & -\Lambda_{3\tau}^+ \chi_{3\tau}^{w_{3\tau}^+ +1} \\[2mm]
-\Lambda_{1}^+ \chi_{1}^{w_{1}^+ +1} & 0 & 0 & \displaystyle\cdots & 0 & \Lambda_{3\tau-1}^- \chi_{3\tau-1}^{w_{3\tau-1}^- +1} \\[2mm]
\end{psmallmatrix}_{3\tau\times3\tau}
.
\end{matrix}
\end{equation}
We denote the $j$-th row and the $k$-th column of $\varphi_{{0}}$ as $G^j$ and $R_k$, respectively, for each $j$, $k\in\mathbb{Z}_{3\tau}$. This is
considering that column $R^k$ represents the $k$-th relation of $G_j$'s given in the left
part of \ref{eqn:RelationsOnG2}, when the entry in row $G^j$ is multiplied to $G_j$ for each $j$.
This proves $\pi_{{0}} \varphi_{{0}} = 0$.
 
Next, we define the matrix $\varphi_{{0\#}}\in S^{3\tau\times3\tau}$ by
$$
\begin{matrix}
\color{colorR} \SMALL\text{\hspace{4mm} $R_{1\#}$ \hspace{12mm} $R_{2\#}$ \hspace{9mm} $R_{3\#}$} \hspace{6mm} \cdots \hspace{10mm} \text{$R_{\left(3\tau-1\right)\#}$ \hspace{17mm} $R_{3\tau\#}$}
\\
\varphi_{{0\#}}
:=
\left(
\arraycolsep=2pt\def\arraystretch{1}
\begin{array}{c|c|c|c|c|c}
\chi_{3\tau-1}\chi_{3\tau} \mathbf{e}_{G^1} & \chi_{3\tau}\chi_1 \mathbf{e}_{G^2} & \chi_1\chi_2 \mathbf{e}_{G^3} & \cdots & \chi_{3\tau-3}\chi_{3\tau-2}
\mathbf{e}_{G^{3\tau-1}} & \chi_{3\tau-2}\chi_{3\tau-1} \mathbf{e}_{G^{3\tau}}
\end{array}
\right)_{3\tau \times 3\tau}
\end{matrix}
$$
where $\mathbf{e}_{G^j}$ is the column vector in $S^{3\tau}$ whose unique nonzero entry is $1$ and lies in the same
position corresponding to row $G^j$, namely the $j$-th row. We denote the $j$-th row and the $k$-th column
of $\varphi_{{0\#}}$
by $G^{j\#}$ and $R_{k\#}$, respectively, for each $j$, $k\in\mathbb{Z}_{3\tau}$. Then column $R_{k\#}$ corresponds to the $k$-th relation of $G_j$'s in the right part of \ref{eqn:RelationsOnG2}, when the entry in row $G^{j\#}$ is multiplied to $G_j$
for each $j$. This yields $\pi_{{0}} \varphi_{{0\#}} = 0$.

Now we claim that the following sequence of $S$-modules is exact:
\begin{equation}\label{eqn:ExactSequence2}
\begin{tikzcd}[column sep = 20pt] 
S^{6\tau} \arrow[rrr, "\left(\varphi_{{0}} \left|\varphi_{{0\#}}\right.\right)"] & & & S^{3\tau} \arrow[r, "\pi_{{0}}"] & \left(M_{{0}}\right)_S \arrow[r] & 0.
\end{tikzcd}
\end{equation}
Since we already know $\pi_{{0}}\left(\varphi_{{0}} \left|\varphi_{{0\#}}\right.\right)=0$, we only need to show that $\ker \pi_{{0}} \subseteq \operatorname{im}\left(\varphi_{{0}} \left|\varphi_{{0\#}}\right.\right)$. Let $a = \left(a_{1},\dots,a_{3\tau}\right)\in S^{3\tau}$ be an element of $\operatorname{ker}\pi_{{0}}$. Then, as the $j$-th column of $\pi_{{0}}$ is $G_{j}$ as an element of $A^\tau$, the equation $\pi_{{0}} a = 0$ is equivalent to the relation
$$
a_{1}G_{1} + \cdots + a_{3\tau} G_{3\tau} =0,
$$
which can be also written as
\begin{equation}\label{eqn:CokerEqn}
\sum_{j=1}^{3\tau}\left( \Lambda_{j+1}^+ \chi_{j+1}^{w_{j+1}^+ +1} a_{j+1}
+
\Lambda_{j}^- \chi_{j}^{w_{j}^- +1} a_{j} \right)
\left(^1 X_{j}^1 \right)
=0.
\end{equation}
This also implies that each summand vanishes.

Note that each $a_j \in S$ can be uniquely expressed as
\begin{equation}
a_j = a_{j,0} + a_{j,1}\chi_{j-1} + a_{j,2}\chi_{j} + a_{j,3}\chi_{j+1} + a_{j,4}\chi_{j-1}\chi_{j} +
a_{j,5}\chi_{j}\chi_{j+1} + a_{j,6}\chi_{j+1}\chi_{j-1} + a_{j,7}xyz
\end{equation}\label{eqn:ajExpression}
for some $a_{j,0}\in\mathbb{C}$, $a_{j,1}\in\mathbb{C}[[\chi_{j-1}]]$, $a_{j,2}\in\mathbb{C}[[\chi_{j}]]$,
 $a_{j,3}\in\mathbb{C}[[\chi_{j+1}]]$, $a_{j,4}\in\mathbb{C}[[\chi_{j-1},\chi_{j}]]$, $a_{j,5}\in\mathbb{C}[[\chi_{j},\chi_{j+1}]]$, $a_{j,6}\in\mathbb{C}[[\chi_{j+1},\chi_{j-1}]]$ and $a_{j,7}\in S$. Substituting this into the $j$-th summand in \ref{eqn:CokerEqn}, simplifying using proposition \ref{prop:ChiRelations}, and comparing coefficients
of each monomials yield the following:
$$
a_{j,0} = a_{j,2} = a_{j+1,0} = a_{j+1,2} = 0
\quad\text{and}
$$
$$
\Lambda_{j+1}^+\chi_{j+1}^{w_{j+1}^++1}\left(a_{j+1,1}\chi_j
+ a_{j+1,4}\chi_j\chi_{j+1}\right) + \Lambda_{j}^-\chi_{j}^{w_{j}^-+1}\left(a_{j,3}\chi_{j+1} + a_{j,5}\chi_{j}\chi_{j+1}\right)
= 0.
$$
From the latter one, we can write
$$
a_{j,3}\chi_{j+1} + a_{j,5}\chi_{j}\chi_{j+1} = -\Lambda_{j+1}^+\chi_{j+1}^{w_{j+1}^++1}b_{j+1}
\quad\text{and}\quad
a_{j+1,1}\chi_{j} + a_{j+1,4}\chi_{j}\chi_{j+1} = \Lambda_{j}^-\chi_{j}^{w_{j}^-+1}b_{j+1}
$$
for some $b_{j+1}\in S$.

Substituting into equation \ref{eqn:ajExpression} what we have earned so far, we get
$$
a_j = \Lambda_{j-1}^-\chi_{j-1}^{w_{j-1}^-+1}b_{j} - \Lambda_{j+1}^+\chi_{j+1}^{w_{j+1}^++1}b_{j+1} +
\left(a_{j,6} + a_{j,7}\chi_{j}\right)\chi_{j+1}\chi_{j-1}
$$
for any $j\in\left\{1,\dots,3\tau\right\}$, or equivalently,
$$
a = \sum_{j=1}^{3\tau}\left(b_jR_j + \left(a_{j,6}+a_{j,7}\chi_j\right)R_{j\#}\right).
$$
Thus $a$ is an $S$-linear combination of the columns $R_k$ and $R_{k\#}$, that is, $a \in \operatorname{im}\left(\varphi_{{0}} \left|\varphi_{{0\#}}\right.\right)$.
Consequently, the sequence \ref{eqn:ExactSequence2} is exact.
\\
\\
\noindent\textbf{Step 2}: Establish a free resolution of $M_S$ in special cases.

Here we first consider the case where $w_j\ge0$ for all $j$. Then we have $w_j^+ = w_j$, $w_j^- = 0$, $w_j' = w_j + 2$ and $\Lambda_j^- = 1$ for all $j$. Relations $H_j = \chi_{j+1} G_j$ show that $\tilde{M}$ is generated by only $3\tau$ elements $G_j$ $\left(j\in\mathbb{Z}_{3\tau}\right)$ in $A^\tau$ and hence $M_{{0}} = \tilde{M}$. We denote $\pi_{{0}}$ by $\tilde{\pi}$. One can also notice $\varphi_{{0}} = \varphi$.

Meanwhile, one can observe in Corollary \ref{cor:Psi} that $\psi_{ab}$ is a multiple of $\chi_{a+1} \chi_{b}$ for any $a$, $b\in\mathbb{Z}_{3\tau}$. Therefore, the $b$-th column of $\psi$ is a multiple of $\chi_{b}$ and can be written as $\chi_{b} v_b$ for some $v_b \in S^{3\tau}$. Using the equation
$\varphi \psi = xyz I_{3\tau}$, we now have
$$
\chi_{b} \varphi v_b = xyz \mathbf{e}_{G^b}.
$$
Since $S$ is an integral domain, we may cancel out common terms $\chi_{b}$ in both sides. The result
is
$$
\varphi v_b
=
\chi_{b-2} \chi_{b-1} \mathbf{e}_{G^b},
$$
where the right side is the $b$-th column of $\varphi_{{0\#}}$. This proves $\operatorname{im}\varphi_{{0+}} \subseteq\operatorname{im}\varphi$ and the exact sequence \ref{eqn:ExactSequence2} yields a new exact sequence
$$
\begin{tikzcd}[column sep = 20pt] 
  0 \arrow[r] & S^{3\tau} \arrow[r, "\varphi"] & S^{3\tau} \arrow[r, "\tilde{\pi}"] & \tilde{M}_S \arrow[r] & 0.
\end{tikzcd}
$$
From this we deduce $\tilde{M} \cong \operatorname{coker}\mathunderbar{\varphi}$\ as $A$-modules, which guarantees that $\tilde{M}$ is already a maximal Cohen-Macaulay $A$-module and hence $\tilde{M} = M$. Setting $\pi := \tilde{\pi}$, we get the desired free resolution \ref{eqn:ExactSequence1} in the theorem.

We can handle the case where $w_j\le0$ for all $j$ in a similar way but using some coordinate translation in the matrix. We leave it to the readers. For the other cases, we move on to the next step.
\\
\\
\noindent\textbf{Step 3}: Get a (part of) free resolution of $\tilde{M}_S$ in the other cases.

In the rest cases, some of $w_j$ are positive and some are negative. Therefore, some of $\delta_j$ are $1$ and some are $0$. Let's say the value of $\delta_j$ changes $2\xi$ times for some $\xi\in\mathbb{Z}_{\ge1}$ with
$2\xi\le3\tau$. Then there are $2\xi$ different integers $\imath_1$, $\jmath_1$,\dots ,$\imath_\xi$, $\jmath_\xi$ in cyclic order in $\mathbb{Z}_{3\tau}$ such that
$$
\delta_{\jmath_\xi+1} = \cdots = \delta_{\imath_1} = 1,\quad
\delta_{\imath_1+1} = \cdots = \delta_{\jmath_1} = 0,\quad
\delta_{\jmath_1+1} = \cdots = \delta_{\imath_2} = 1,\quad
\dots,\quad
\delta_{\imath_\xi+1} = \cdots = \delta_{\jmath_\xi} = 0.
$$

By the discussion in the first paragraph in Subsection \ref{subsubsec:GeneratorsAndMacaulayfyingElements}, $\tilde{M}$ is generated by elements of $M_{{0}}=\left<G_1,\dots,G_{3\tau}\right>_A$ together with $H_{\jmath_1}, \dots, H_{\jmath_\xi}$. We denote it by $M_{{1}}:=\tilde{M}$. To resolve it, we enlarge the matrix $\pi_{{0}} \in A^{\tau\times3\tau}$ to a new matrix $\pi_{{1}} := \pi_{{1}} \left(w,\lambda,1\right) \in A^{\tau\times\left(3\tau+\xi\right)}$ by
inserting the new column $H_{\jmath_\nu} \in A^\tau$ between the columns $G_{\jmath_\nu}$ and $G_{\jmath_\nu+1}$
of $\pi_{{0}}$
for each $\nu\in\left\{1, \dots, \xi\right\}$. As a result, we get
$$
\pi_{{1}}
:=
\left(
\arraycolsep=2pt\def\arraystretch{1}
\begin{array}{c|c|c|c|c|c|c|c|c|c|c|c|c|c|c}
\color{colorG} G_1 & \color{colorG} \cdots & \color{colorG} G_{\jmath_1} & \color{colorH} H_{\jmath_1} & \color{colorG} G_{\jmath_1+1} & \color{colorG} \cdots & \color{colorG} G_{\jmath_2} & \color{colorH} H_{\jmath_2} & \color{colorG} G_{\jmath_2+1} & \color{colorG} \cdots & \color{colorG} G_{\jmath_\xi} & \color{colorH} H_{\jmath_\xi} & \color{colorG} G_{\jmath_\xi+1} & \color{colorG} \cdots & \color{colorG} G_{3\tau}
\end{array}
\right)_{\tau\times \left(3\tau+\xi\right)}.
$$
Then it can be viewed as an $S$-module map $\pi_{{1}}:S^{3\tau+\xi}\rightarrow \left(A_S\right)^\tau$ whose image is $\left(M_{{1}}\right)_S = \tilde{M}_S$. Restricting the codomain, we get the surjective map $\pi_{{1}}:S^{3\tau+\xi}\rightarrow \left(M_{{1}}\right)_S$.

Next, denote by
$$
\varphi_{{0}}
\left[
G^{j_1}: G^{j_2}
;
R_{k_1}: R_{k_2}
\right]
$$
the submatrix of $\varphi_{{0}}$ taking rows from $G^{j_1}$ to $G^{j_2}$ and columns from $R_{k_1}$
to $R_{k_2}$. Now, for each $\nu\in\left\{1,\dots,\xi\right\}$, look at submatrices
$$
\left(\varphi_{{0}}\right)_\nu^-
:=
\varphi_{{0}}
\left[
G^{\imath_\nu+2}: G^{\jmath_\nu+1};R_{\imath_\nu+2}: R_{\jmath_\nu+1}
\right]
\quad\text{and}\quad
\left(\varphi_{{0}}\right)_{\nu+1}^+
:=
\varphi_{{0}} \left[G^{\jmath_\nu}: G^{\imath_{\nu+1}-1};R_{\jmath_\nu+1}: R_{\imath_{\nu+1}}\right]
$$
of $\varphi_{{0}}$, which are highlighted in Figure \ref{fig:Step3BeforeModifying}
respectively as a red box and a blue box. Check that the entries are the same but only expressions have changed from the original definition \ref{eqn:DefinitionOfTtildePhiS} of $\varphi_{{0}}$, using the inequalities $w_{\imath_\nu+1}, \dots, w_{\jmath_\nu}\le0$ and $w_{\jmath_\nu+1}, \dots, w_{\imath_{\nu+1}}\ge0$.

\begin{figure}[H]
\adjustbox{scale=0.7,center}{$
\begin{tikzpicture}
\setcounter{MaxMatrixCols}{100}
\matrix(MatrixPhi)[matrix of math nodes, nodes in empty cells] 
{
& \hspace{5mm} & \scalebox{1}{$\color{colorR} R_{\imath_\nu}$} & \scalebox{1}{$\color{colorR} R_{\imath_\nu + 1}$} & \scalebox{1}{$\color{colorR} R_{\imath_\nu + 2}$} &  &  & \scalebox{1}{$\color{colorR} R_{\jmath_\nu}$} & \scalebox{1}{$\color{gray} R_{\jmath_\nu + 1}$} & \scalebox{1}{$\color{colorR} R_{\jmath_\nu + 2}$} &  &  & \scalebox{1}{$\color{colorR} \hspace{-6mm} R_{\imath_{\nu+1}}$}
\\[7mm]
\color{colorG} G^{\imath_\nu - 1} & & -\Lambda_{\imath_\nu}^+ \chi_{\imath_\nu}^{w_{\imath_\nu}+1} & & & & & & & & & &
\\[5mm]
\color{colorG} G^{\imath_\nu} & & \chi_{\imath_\nu-1} & -\chi_{\imath_\nu+1} & & & & & & & & &
\\[5mm]
\color{colorG} G^{\imath_\nu + 1} & & & \chi_{\imath_\nu} & -\chi_{\imath_\nu+2} & & & & & & & &
\\[5mm]
\color{colorG} G^{\imath_\nu + 2} & & & & \Lambda_{\imath_\nu+1}^- \chi_{\imath_\nu+1}^{-w_{\imath_\nu+1} +1} & \hspace{-6mm} - \chi_{\imath_{\nu}+3} & \cdots & 0 & \color{gray} 0 & & & &
\\[5mm]
& & & & 0 & \hspace{-6mm} \Lambda_{\imath_{\nu}+2}^- \chi_{\imath_{\nu}+2}^{-w_{\imath_{\nu}+2} +1} & \ddots & \vdots & \color{gray} \vdots & & & &
\\[5mm]
& & & & \vdots & \hspace{-6mm} \vdots & \ddots & -\chi_{\jmath_\nu} & \color{gray} 0 & & & &
\\[5mm]
\color{colorG} G^{\jmath_\nu} & & & & 0 & \hspace{-6mm} 0 & \cdots & \Lambda_{\jmath_\nu-1}^- \chi_{\jmath_\nu-1}^{-w_{\jmath_\nu-1}+1} & \color{gray} - \Lambda_{\jmath_\nu+1}^+ \chi_{\jmath_\nu+1}^{w_{\jmath_\nu+1} +1} & 0 & \cdots & 0 & \hspace{-6mm} 0
\\[5mm]
\color{colorG} G^{\jmath_\nu + 1} & & & & 0 & \hspace{-6mm} 0 & \cdots & 0 & \color{gray} \quad\Lambda_{\jmath_\nu}^- \chi_{\jmath_\nu}^{-w_{\jmath_\nu} +1}\quad & - \Lambda_{\jmath_\nu+2}^+ \chi_{\jmath_\nu+2}^{w_{\jmath_\nu+2} +1} & \cdots & 0 & \hspace{-6mm} 0
\\[5mm]
& & & & & & & & \color{gray} 0 & \chi_{\jmath_\nu+1} & \ddots & \vdots & \hspace{-6mm} \vdots
\\[5mm]
& & & & & & & & \color{gray} \vdots & \vdots & \ddots & -\Lambda_{\imath_{\nu+1}-1}^+ \chi_{\imath_{\nu+1}-1}^{w_{\imath_{\nu+1}-1} +1} & \hspace{-6mm} 0
\\[5mm]
\color{colorG} G^{\imath_{\nu+1} - 1} & & & & & & & & \color{gray} 0 & 0 & \cdots & \chi_{\imath_{\nu+1}-2} & \hspace{-6mm} -\Lambda_{\imath_{\nu+1}}^+ \chi_{\imath_{\nu+1}}^{w_{\imath_{\nu+1}} +1}
\\
};

\draw[color=red] (MatrixPhi-5-5.north west) rectangle (MatrixPhi-9-9.south east);
\draw[color=blue] (MatrixPhi-8-9.north west) rectangle (MatrixPhi-12-13.south east);

\draw[blue] ($0.5*(MatrixPhi-2-3.north east)+0.5*(MatrixPhi-2-3.north east)$) -- ($0.5*(MatrixPhi-2-3.south east)+0.5*(MatrixPhi-2-3.south east)$);
\draw[blue] ($0.5*(MatrixPhi-2-3.south west)+0.5*(MatrixPhi-2-3.south west)$) -- ($0.5*(MatrixPhi-2-3.south east)+0.5*(MatrixPhi-2-3.south east)$);

\node[fit=(MatrixPhi-2-4)(MatrixPhi-2-4)]{$\color{blue} \left(\varphi_{{0}}\right)_{\nu}^+$};

\node[fit=(MatrixPhi-10-5)(MatrixPhi-10-7)]{$\color{red} \left(\varphi_{{0}}\right)_\nu^-$};
\node[fit=(MatrixPhi-7-11)(MatrixPhi-7-13)]{$\color{blue} \left(\varphi_{{0}}\right)_{\nu+1}^+$};

\node[fit=(MatrixPhi-1-6)(MatrixPhi-1-7)]{\scalebox{1}{$\color{colorR} \hspace{-6mm} \cdots$}};
\node[fit=(MatrixPhi-1-10)(MatrixPhi-1-13)]{\scalebox{1}{$\color{colorR} \cdots$}};
\node[fit=(MatrixPhi-6-1)(MatrixPhi-7-1)]{$\color{colorG} \vdots$};
\node[fit=(MatrixPhi-10-1)(MatrixPhi-11-1)]{$\color{colorG} \vdots$};

\tikzset{
    cheating dash/.code args={on #1 off #2}{
        \csname tikz@addoption\endcsname{%
            \pgfgetpath\currentpath%
            \pgfprocessround{\currentpath}{\currentpath}%
            \csname pgf@decorate@parsesoftpath\endcsname{\currentpath}{\currentpath}%
            \pgfmathparse{\csname pgf@decorate@totalpathlength\endcsname-#1}\let\rest=\pgfmathresult%
            \pgfmathparse{#1+#2}\let\onoff=\pgfmathresult%
            \pgfmathparse{max(floor(\rest/\onoff), 1)}\let\nfullonoff=\pgfmathresult%
            \pgfmathparse{max((\rest-\onoff*\nfullonoff)/\nfullonoff+#2, #2)}\let\offexpand=\pgfmathresult%
            \pgfsetdash{{#1}{\offexpand}}{0pt}}%
    }
}
\path(MatrixPhi-8-1.south) -- (MatrixPhi-9-1.north) coordinate[midway] (Y);
\draw[cheating dash=on 3pt off 3pt,gray] (Y -| MatrixPhi.west) -- (Y -| MatrixPhi.east);

\end{tikzpicture}
$}
\caption{Submatrix $\varphi_{{0}}
\left[
G^{\imath_\nu-1}: G^{\imath_{\nu+1}-1}
;
R_{\imath_\nu}: R_{\imath_{\nu+1}}
\right]$
of $\varphi_{{0}}$}
\label{fig:Step3BeforeModifying}
\end{figure}

Note that the diagonal entries of $\left(\varphi_{{0}}\right)_\nu^-$ and $\left(\varphi_{{0}}\right)_{\nu+1}^+$ can be read off respectively as
\begin{align}\label{eqn:DiagonalsOfPhi0}
\begin{split}
\operatorname{diag} \left(\varphi_{{0}}\right)_\nu^-
&=
\left(
\Lambda_{\imath_\nu+1}^- \chi_{\imath_\nu+1}^{-w_{\imath_\nu+1} - \delta_{\imath_\nu+1} +1},
\dots,
\Lambda_{\jmath_\nu}^- \chi_{\jmath_\nu}^{-w_{\jmath_\nu} - \delta_{\jmath_\nu} +1}
\right)
\in
S^{\jmath_\nu - \imath_\nu}
\quad\text{and}
\\[2mm]
\operatorname{diag} \left(\varphi_{{0}}\right)_{\nu+1}^+
&=
\left(
- \Lambda_{\jmath_\nu+1}^+ \chi_{\jmath_\nu+1}^{w_{\jmath_\nu+1} + \delta_{\jmath_\nu+1}},
\dots,
-\Lambda_{\imath_{\nu+1}}^+ \chi_{\imath_{\nu+1}}^{w_{\imath_{\nu+1}} + \delta_{\imath_{\nu+1}} +1}
\right)
\in
S^{\imath_{\nu+1} - \jmath_\nu}
\end{split}
\end{align}
by the identities $\delta_{\imath_\nu+1}=\cdots=\delta_{\jmath_\nu}=0$ and $\delta_{\jmath_\nu+1}=\cdots=\delta_{\imath_{\nu+1}}=1$, where $\operatorname{diag}B$ denotes the tuple consisting of the main diagonal entries of $B$ for any matrix $B$.

Now we modify the matrix $\varphi_{{0}} \in S^{3\tau \times 3\tau}$ to a new matrix $\varphi_{{1}} \in S^{\left(3\tau+\xi\right) \times \left(3\tau+2\xi\right)}$ by performing the following procedure for each $\nu\in\left\{1,\dots,\xi\right\}$:
\begin{itemize}
\item
Multiply the columns $\color{colorR} R_{\imath_\nu+1}$, $\color{colorR} \dots$, $\color{colorR} R_{\jmath_\nu}$ by $-1$,
\item
Insert a new row $\color{colorH}H^{\jmath_\nu}$ consisting of zeros between the rows $\color{colorG} G^{\jmath_\nu}$ and $\color{colorG} G^{\jmath_\nu+1}$, and then
\item
Replace column $\color{gray} R_{\jmath_\nu+1}$ with the new three columns
$\color{colorT} T_{\jmath_\nu}$,
$\color{colorT} T_{\jmath_\nu+1}$
and
$\color{colorT} T_{\jmath_\nu+2}$
\end{itemize}
as described in Figure \ref{fig:Step3AfterModifying}. We keep the same names for the original rows and
columns, although their size and entries may have changed. But submatrices $\left(\varphi_{{0}}\right)_{\nu}^-$ and $\left(\varphi_{{0}}\right)_{\nu+1}^+$ of $\varphi_{{0}}$ are readjusted to submatrices
$$
\left(\varphi_{{1}}\right)_\nu^-
:=
\scalebox{1}{$
\varphi_{{1}}
\left[
G^{\imath_\nu+2}: H^{\jmath_\nu};R_{\imath_\nu+2}: T_{\jmath_\nu}
\right]
$}
\quad\text{and}\quad
\left(\varphi_{{1}}\right)_{\nu+1}^+
:=
\scalebox{1}{$
\varphi_{{1}} \left[H^{\jmath_\nu}: G^{\imath_{\nu+1}-1};T_{\jmath_\nu+2}: R_{\imath_{\nu+1}}\right]
$}
$$
of $\varphi_{{1}}$, respectively. Check that the replaced column $R_{\jmath_\nu+1}$ is an $S$-linear combination of the new columns, namely
$$
-\Lambda_{\jmath_\nu+1}^+ \chi_{\jmath_\nu+1}^{w_{\jmath_\nu+1}} T_{\jmath_\nu}
+
\Lambda_{\jmath_\nu}^- \chi_{\jmath_\nu}^{-w_{\jmath_\nu}} T_{\jmath_\nu+2},
$$
meaning that the presence of the replaced column does not affect the image of $\varphi_{{1}}$. We also shifted indices in the upper diagonal entries of $\left(\varphi_{{1}}\right)_\nu^-$ by $3$,  which does not actually change the entries, to make them consecutive with those in the lower diagonal of $\left(\varphi_{{1}}\right)_{\nu+1}^+$.

\begin{figure}[h]
\adjustbox{scale=0.67,center}{$
\begin{tikzpicture}[column sep = -7.5pt]
\setcounter{MaxMatrixCols}{100}
\matrix (MatrixPhi)[matrix of math nodes, nodes in empty cells]
{
& \hspace{7mm} & \scalebox{1}{$\color{colorR} R_{\imath_\nu}$} & \scalebox{1}{$\color{colorR} R_{\imath_\nu + 1}$} & \color{colorR} R_{\imath_\nu+2} & \color{colorR} \hspace{-6mm} R_{\imath_\nu+3} & \color{colorR} \cdots & \color{colorR}
R_{\jmath_\nu} & \color{colorT} T_{\jmath_\nu} & \color{colorT} T_{\jmath_\nu+1} & \color{colorT} T_{\jmath_\nu+2} & \color{colorR} R_{\jmath_\nu+2} & \color{colorR} \cdots & \color{colorR} R_{\imath_{\nu+1}-1} & \color{colorR}
\hspace{-6mm} R_{\imath_{\nu+1}}
\\[5mm]
\color{colorG} G^{\imath_\nu - 1} & & -\Lambda_{\imath_\nu}^+ \chi_{\imath_\nu}^{w_{\imath_\nu}+1} & & & & & & & & & & & &
\\[5mm]
\color{colorG} G^{\imath_\nu} & & \chi_{\imath_\nu-1} & \chi_{\imath_\nu+1} & & & & & & & & & & &
\\[5mm]
\color{colorG} G^{\imath_\nu + 1} & & & -\chi_{\imath_\nu} & \chi_{\imath_\nu+2} & & & & & & & & &
&
\\[5mm]
\color{colorG} G^{\imath_\nu + 2} & & & & -\Lambda_{\imath_\nu+1}^- \chi_{\imath_\nu+1}^{-w_{\imath_\nu+1} +1} & \hspace{-6mm} \chi_{\imath_{\nu}} & \cdots & 0 & \color{colorT} 0 & & & & & &
\\[5mm]
 & & & & 0 & \hspace{-6mm} -\Lambda_{\imath_{\nu}+2}^- \chi_{\imath_{\nu}+2}^{-w_{\imath_{\nu}+2} +1} & \ddots & \vdots & \color{colorT} \vdots & & & & & &
\\[5mm]
 & & & & \vdots & \hspace{-6mm} \vdots & \ddots & \chi_{\jmath_\nu-3} & \color{colorT} 0 & & & & & &
 \\[5mm]
\color{colorG} G^{\jmath_\nu} & & & & 0 & \hspace{-6mm} 0 & \cdots & -\Lambda_{\jmath_\nu-1}^- \chi_{\jmath_\nu-1}^{-w_{\jmath_\nu-1} +1} & \color{colorT} \chi_{\jmath_\nu-2} & & & & & &
\\[5mm]
\color{colorH} H^{\jmath_\nu} & & & & \color{colorH} 0 & \color{colorH} \hspace{-6mm} 0 & \color{colorH} \cdots & \color{colorH} 0 & \color{colorT} -\Lambda_{\jmath_\nu}^- \chi_{\jmath_\nu}^{-w_{\jmath_\nu} } & \color{colorT} \qquad\chi_{\jmath_\nu-1}\qquad & \color{colorT} -\Lambda_{\jmath_\nu+1}^+ \chi_{\jmath_\nu+1}^{w_{\jmath_\nu+1}} & \color{colorH} 0 & \color{colorH} \cdots & \color{colorH} 0 & \color{colorH} \hspace{-6mm} 0
\\[5mm]
\color{colorG} G^{\jmath_\nu + 1} & & & & & & & & & & \color{colorT} \chi_{\jmath_\nu} & - \Lambda_{\jmath_\nu+2}^+ \chi_{\jmath_\nu+2}^{w_{\jmath_\nu+2} +1} & \cdots & 0 & \hspace{-6mm} 0
\\[5mm]
 & & & & & & & & & & \color{colorT} 0 & \chi_{\jmath_\nu+1} & \ddots & \vdots & \hspace{-6mm} \vdots
 \\[5mm]
 & & & & & & & & & & \color{colorT} \vdots & \vdots & \ddots & -\Lambda_{\imath_{\nu+1}-1}^+ \chi_{\imath_{\nu+1}-1}^{w_{\imath_{\nu+1}-1} +1} & \hspace{-6mm} 0
\\[5mm]
\color{colorG} G^{\imath_{\nu+1} - 1} & & & & & & & & & & \color{colorT} 0 & 0 & \cdots & \chi_{\imath_{\nu+1}-2} & \hspace{-6mm} -\Lambda_{\imath_{\nu+1}}^+ \chi_{\imath_{\nu+1}}^{w_{\imath_{\nu+1}} +1}
\\
};

\draw[color=red] (MatrixPhi-5-5.north west) rectangle (MatrixPhi-9-9.south east);
\draw[color=blue] (MatrixPhi-9-11.north west) rectangle (MatrixPhi-13-15.south east);

\draw[blue] ($0.5*(MatrixPhi-2-3.north east)+0.5*(MatrixPhi-2-3.north east)$) -- ($0.5*(MatrixPhi-2-3.south east)+0.5*(MatrixPhi-2-3.south east)$);
\draw[blue] ($0.5*(MatrixPhi-2-3.south west)+0.5*(MatrixPhi-2-3.south west)$) -- ($0.5*(MatrixPhi-2-3.south east)+0.5*(MatrixPhi-2-3.south east)$);

\node[fit=(MatrixPhi-2-4)(MatrixPhi-2-4)]{$\color{blue} \left(\varphi_{{1}}\right)_{\nu}^+$};

\node[fit=(MatrixPhi-10-6)(MatrixPhi-10-7)]{$\color{red} \left(\varphi_{{1}}\right)_{\nu}^-$};
\node[fit=(MatrixPhi-8-13)(MatrixPhi-8-14)]{$\color{blue} \left(\varphi_{{1}}\right)_{\nu+1}^+$};

\node[fit=(MatrixPhi-6-1)(MatrixPhi-7-1)]{$\color{colorG} \vdots$};
\node[fit=(MatrixPhi-11-1)(MatrixPhi-12-1)]{$\color{colorG} \vdots$};

\end{tikzpicture}
$}
\caption{Submatrix $\varphi_{{1}} \left[G^{\imath_\nu-1}: G^{\imath_{\nu+1}-1}; R_{\imath_\nu}: R_{\imath_{\nu+1}}\right]$ of $\varphi_{{1}}$}
\label{fig:Step3AfterModifying}
\end{figure}

The construction of $\varphi_{{1}}$ is based on the relations
$$
-\chi_{\jmath_\nu-2} G_{\jmath_\nu}
+
\Lambda_{\jmath_\nu}^- \chi_{\jmath_\nu}^{-w_{\jmath_\nu}} H_{\jmath_\nu}
= 0,\quad
\chi_{\jmath_\nu-1} H_{\jmath_\nu}
= 0\quad\text{and}\quad
\Lambda_{\jmath_\nu+1}^+ \chi_{\jmath_\nu+1}^{w_{\jmath_\nu+1}} H_{\jmath_\nu}
-
\chi_{\jmath_\nu} G_{\jmath_\nu+1}
=0
$$
on $G_{\jmath_\nu}$, $H_{\jmath_\nu}$ and $G_{\jmath_\nu+1}$ for $\nu\in\left\{1,\dots,\xi\right\}$, already observed in equation \ref{eqn:RelationsOnH}. This enables the matrix equation $\pi_{{1}} \varphi_{{1}} = 0$ to hold.

Comparing the exponents of the diagonal entries of $\left(\varphi_{{1}}\right)_\nu^-$ and $\left(\varphi_{{1}}\right)_{\nu+1}^+$
with those of $\left(\varphi_{{0}}\right)_\nu^-$ and $\left(\varphi_{{0}}\right)_{\nu+1}^+$, only the last one of $\left(\varphi_{{1}}\right)_\nu^-$ and the first
one of $\left(\varphi_{{1}}\right)_{\nu+1}^+$ have been changed. Now we can rewrite \ref{eqn:DiagonalsOfPhi0}
for those as
\begin{align*}
\operatorname{diag} \left(\varphi_{{1}}\right)_\nu^-
&=
\left(
-\Lambda_{\imath_\nu+1}^- \chi_{\imath_\nu+1}^{-w_{\imath_\nu+1} - \delta_{\imath_\nu+1} - \delta_{\imath_\nu+2} +1},
\dots,
- \Lambda_{\jmath_\nu}^- \chi_{\jmath_\nu}^{-w_{\jmath_\nu} - \delta_{\jmath_\nu} - \delta_{\jmath_\nu+1} +1}
\right)
\in
S^{\jmath_\nu - \imath_\nu}
\quad\text{and}
\\[2mm]
\operatorname{diag} \left(\varphi_{{1}}\right)_{\nu+1}^+
&=
\left(
- \Lambda_{\jmath_\nu+1}^+ \chi_{\jmath_\nu+1}^{w_{\jmath_\nu+1} + \delta_{\jmath_\nu} + \delta_{\jmath_\nu+1} - 1},
\dots,
-\Lambda_{\imath_{\nu+1}}^+ \chi_{\imath_{\nu+1}}^{w_{\imath_{\nu+1}} + \delta_{\imath_{\nu+1}-1} + \delta_{\imath_{\nu+1}} -1}
\right)
\in
S^{\imath_{\nu+1} - \jmath_\nu}
\end{align*}
using $\delta_{\imath_\nu+1}=\cdots=\delta_{\jmath_\nu}=0$ and $\delta_{\jmath_\nu+1}=\cdots=\delta_{\imath_{\nu+1}}=1$ again. Taking one step further, by the conversion formula
$$
w_j' = w_j + \delta_{j-1} + \delta_j + \delta_{j+1} -1
\quad
\left(j\in\mathbb{Z}_{3\tau}\right).
$$
in Definition \ref{def:ConversionFromBandtoLoop}, these can also be written as
\allowdisplaybreaks
\begin{align}\label{eqn:DiagonalsOfPhi1}
\begin{split}
\operatorname{diag} \left(\varphi_{{1}}\right)_\nu^-
&=
\scalebox{0.83}{$
\left(
-\Lambda_{\imath_\nu+1}^- \chi_{\imath_\nu+1}^{-w_{\imath_\nu+1} - \delta_{\imath_\nu} - \delta_{\imath_\nu+1} - \delta_{\imath_\nu+2} +2},
-\Lambda_{\imath_\nu+2}^- \chi_{\imath_\nu+2}^{-w_{\imath_\nu+2} - \delta_{\imath_\nu+1} - \delta_{\imath_\nu+2} - \delta_{\imath_\nu+3} +1},
\dots,
- \Lambda_{\jmath_\nu}^- \chi_{\jmath_\nu}^{-w_{\jmath_\nu} - \delta_{\jmath_\nu-1} - \delta_{\jmath_\nu} - \delta_{\jmath_\nu+1} +1}
\right)
$}
\\[2mm]
&=
\left(
-\Lambda_{\imath_\nu+1}^- \chi_{\imath_\nu+1}^{-w_{\imath_\nu+1}' +1},
-\Lambda_{\imath_\nu+2}^- \chi_{\imath_\nu+2}^{-w_{\imath_\nu+2}'},
\dots,
- \Lambda_{\jmath_\nu}^- \chi_{\jmath_\nu}^{-w_{\jmath_\nu}'}
\right)
\in
S^{\jmath_\nu - \imath_\nu}
\quad\text{and}
\\[2mm]
\operatorname{diag} \left(\varphi_{{1}}\right)_{\nu+1}^+
&=
\scalebox{0.83}{$
\left(
- \Lambda_{\jmath_{\nu}+1}^+ \chi_{\jmath_{\nu}+1}^{w_{\jmath_{\nu}+1} + \delta_{\jmath_{\nu}} + \delta_{\jmath_{\nu}+1} + \delta_{\jmath_{\nu}+2}- 2},
\dots,
-\Lambda_{\imath_{\nu+1}-1}^+ \chi_{\imath_{\nu+1}-1}^{w_{\imath_{\nu+1}-1} + \delta_{\imath_{\nu+1}-2} + \delta_{\imath_{\nu+1}-1} + \delta_{\imath_{\nu+1}} -2},
-\Lambda_{\imath_{\nu+1}}^+ \chi_{\imath_{\nu+1}}^{w_{\imath_{\nu+1}} + \delta_{\imath_{\nu+1}-1} + \delta_{\imath_{\nu+1}} + \delta_{\imath_{\nu+1}+1} -1}
\right)
$}
\\[2mm]
&=
\left(
- \Lambda_{\jmath_{\nu}+1}^+ \chi_{\jmath_{\nu}+1}^{w_{\jmath_{\nu}+1}' -1},
\dots,
-\Lambda_{\imath_{\nu+1}-1}^+ \chi_{\imath_{\nu+1}-1}^{w_{\imath_{\nu+1}-1}' -1},
-\Lambda_{\imath_{\nu+1}}^+ \chi_{\imath_{\nu+1}}^{w_{\imath_{\nu+1}}'}
\right)
\in
S^{\imath_{\nu} - \jmath_{\nu-1}}.
\end{split}
\end{align}

We also enlarge the matrix $\varphi_{{0\#}} \in S^{3\tau \times 3\tau}$ to a new matrix $\varphi_{{1\#}} \in S^{\left(3\tau+\xi\right)\times 3\tau}$ by
inserting a new row $H^{\jmath_\nu\#}$ consisting of zeros between the rows $G^{\left(\jmath_\nu+1\right)\#}$ and $G^{\left(\jmath_\nu+2\right)\#}$ of $\varphi_{{0\#}}$
for each $\nu\in\left\{1,\dots,\xi\right\}$. We use the same name $R_{k\#}$ for the $k$-th column. Then it is easy to see that $\pi_{{1}}\varphi_{{1\#}} = 0$. Moreover, we have the inclusion $\operatorname{im}\varphi_{{1\#}}\subseteq \operatorname{im}\varphi_{{1}}$.
Indeed, the equations
\newcommand{\tagarray}{
\mbox{}\refstepcounter{equation}%
$(\theequation)$%
}
\setlength\tabcolsep{1pt}
\begin{longtable}{ccccl}
$R_{\left(\imath_\nu+1\right)\#}$
&$=$&
$\chi_{\imath_\nu-1} \chi_{\imath_\nu} \mathbf{e}_{G^{\imath_\nu+1}}$
&$=$&
$\chi_{\imath_\nu} R_{\imath_\nu+2}
+
\Lambda_{\imath_\nu+1}^- \chi_{\imath_\nu+1}^{-w_{\imath_\nu+1}}
R_{\left(\imath_\nu+2\right)\#}$
\\
& $\vdots$ & & $\vdots$ &
\\
$R_{\left(\jmath_\nu-1\right)\#}$
&$=$&
$\chi_{\jmath_\nu-3} \chi_{\jmath_\nu-2} \mathbf{e}_{G^{\jmath_\nu-1}}$
&$=$&
$\chi_{\jmath_\nu-2} R_{\jmath_\nu}
+
\Lambda_{\jmath_\nu-1}^- \chi_{\jmath_\nu-1}^{-w_{\jmath_\nu-1}}
R_{\jmath_\nu\#}$
\\[3mm]
$R_{\jmath_\nu\#}$
&$=$&
$\chi_{\jmath_\nu-2} \chi_{\jmath_\nu-1} \mathbf{e}_{G^{\jmath_\nu}}$
&$=$&
$\chi_{\jmath_\nu-1} T_{\jmath_\nu}
+
\Lambda_{\jmath_\nu}^- \chi_{\jmath_\nu}^{-w_{\jmath_\nu}} T_{\jmath_\nu+1}$ \hspace{-123.7mm}\tagarray\label{eqn:GeneratingRSharp}
\\[3mm]
$R_{\left(\jmath_\nu+1\right)\#}$
&$=$&
$\chi_{\jmath_\nu-1} \chi_{\jmath_\nu} \mathbf{e}_{G^{\jmath_\nu+1}}$
&$=$&
$\Lambda_{\jmath_\nu+1}^+ \chi_{\jmath_\nu+1}^{w_{\jmath_\nu+1}} T_{\jmath_\nu+1}
+
\chi_{\jmath_\nu-1} T_{\jmath_\nu+2}$
\\[3mm]
$R_{\left(\jmath_\nu+2\right)\#}$
&$=$&
$\chi_{\jmath_\nu} \chi_{\jmath_\nu+1} \mathbf{e}_{G^{\jmath_\nu+2}}$
&$=$&
$\Lambda_{\jmath_\nu+2}^+ \chi_{\jmath_\nu+2}^{w_{\jmath_\nu+2}} R_{\left(\jmath_\nu+1\right)\#}
+
\chi_{\jmath_\nu} R_{\jmath_\nu+2}$
\\
& $\vdots$ & & $\vdots$ &
\\
$R_{\imath_{\nu+1}\#}$
&$=$&
$\chi_{\imath_{\nu+1}-2} \chi_{\imath_{\nu+1}-1} \mathbf{e}_{G^{\imath_{\nu+1}}}$
&$=$&
$\Lambda_{\imath_{\nu+1}}^+ \chi_{\imath_{\nu+1}}^{w_{\imath_{\nu+1}}} R_{\left(\imath_{\nu+1}-1\right)\#}
+
\chi_{\imath_{\nu+1}-2} R_{\imath_{\nu+1}}$
\end{longtable}
\noindent show that the columns $R_{\left(\imath_\nu+1\right)\#}$ through $R_{\imath_{\nu+1}\#}$ of $\varphi_{{1\#}}$
are spanned by the columns of $\varphi_{{1}}$. This holds for all $\nu\in\left\{1,\dots,\xi\right\}$. Thus, we have $\operatorname{im}\varphi_{{1\#}}\subseteq \operatorname{im}\varphi_{{1}}$.

Now we have the matrix equation $\pi_{{1}}\left(\varphi_{{1}}\left|\varphi_{{1\#}}\right.\right)=0$. We also checked in the below of equation \ref{eqn:RelationsOnH} in Subsection \ref{subsubsec:GeneratorsAndMacaulayfyingElements} that any element $s\in S$ satisfying $s H_{\jmath_\nu}\in \left(M_{{0}}\right)_S$ is $S$-generated by three elements
$$
- \Lambda_{\jmath_\nu}^- \chi_{\jmath_\nu}^{-w_{\jmath_\nu}},\quad
\chi_{\jmath_\nu-1}
\quad\text{and}\quad
- \Lambda_{\jmath_\nu+1}^+ \chi_{\jmath_\nu+1}^{w_{\jmath_\nu+1}}
$$
in $S$, which are the nonzero entries in the newly added row of $\varphi_{{1}}$.
Thus, we can apply Lemma \ref{lem:MatrixExpansionLemma} for each $\nu \in \left\{1,\dots,\xi\right\}$ to get a part of free resolution of $\left(M_{{1}}\right)_S = \tilde{M}_S$ as follows:
$$
\begin{tikzcd}[column sep = 20pt] 
S^{6\tau+2\xi} \arrow[rrr, "\left(\varphi_{{1}} \left|\varphi_{{1\#}}\right.\right)"] & & & S^{3\tau+\xi} \arrow[r, "\pi_{{1}}"] & \left(M_{{1}}\right)_S \arrow[r] & 0.
\end{tikzcd}
$$
Using $\operatorname{im}\varphi_{{1\#}}\subseteq \operatorname{im}\varphi_{{1}}$, we can drop $\varphi_{{1\#}}$ and have the following simplified one:
$$
\begin{tikzcd}[column sep = 20pt] 
S^{3\tau+2\xi} \arrow[r, "\varphi_{{1}}"] & S^{3\tau+\xi} \arrow[r, "\pi_{{1}}"] & \left(M_{{1}}\right)_S \arrow[r] & 0.
\end{tikzcd}
$$
\\
\noindent\textbf{Step 4}: Continue to establish a free resolution of $M_S$.

Recall that we have $2\xi$ different numbers $\imath_1$, $\jmath_1$,\dots ,$\imath_\xi$, $\jmath_\xi \in \mathbb{Z}_{3\tau}$ in cyclic order such that
$$
\delta_{\jmath_\xi+1} = \cdots = \delta_{\imath_1} = 1,\quad
\delta_{\imath_1+1} = \cdots = \delta_{\jmath_1} = 0,\quad
\delta_{\jmath_1+1} = \cdots = \delta_{\imath_2} = 1,\quad
\dots,\quad
\delta_{\imath_\xi+1} = \cdots = \delta_{\jmath_\xi} = 0.
$$
For each $\nu\in\left\{1,\dots,\xi\right\}$, we know that $w_{\imath_\nu}\ge1$ and all $w_{\imath_\nu+1}, \dots, w_{\jmath_\nu}$ must be non-positive
while at least one must be negative, by the definition of the sign word $\delta_j$. There is $\kappa_\nu \in \left\{0,\dots,\jmath_\nu-\imath_\nu-1\right\}$ for each $\nu$ such that
$$
w_{\imath_\nu} \ge1,
\quad
w_{\imath_\nu+1}=\cdots=w_{\imath_\nu+\kappa_\nu}=0
\quad\text{and}\quad
w_{\imath_\nu+\kappa_\nu+1}\le-1.
$$
Then we have a Macaulayfying element
$$
F_{\imath_\nu}
:=
\Lambda_{\imath_\nu}^+ \left(^1 X_{\imath_\nu-1}^{w_{\imath_\nu}+1}\right)
+
\sum_{a=0}^{\kappa_\nu} \left(\prod_{b=1}^a \Lambda_{\imath_\nu+b}^- \right)\left(^1X_{\imath_\nu+a}^1 \right)
+
\left(\prod_{b=1}^{\kappa_\nu+1} \Lambda_{\imath_\nu+b}^- \right) \left(^{-w_{\imath_\nu+\kappa_\nu+1}+1}X_{\imath_\nu+\kappa_\nu+1}^1\right)
$$
of $M_{{1}} = \tilde{M}$ in $A^\tau$ for each $\nu$, by the observation in Subsection \ref{subsubsec:GeneratorsAndMacaulayfyingElements}.

We know by Proposition \ref{prop:Macaulayfying}.(2) that $M = \left(M_{{1}}\right)^\dagger \cong \left<M_{{1}},F_{\imath_1},\dots,F_{\imath_\xi}\right>_A^\dagger$. Let $M_{{2}}:=\left<M_{{1}},F_{\imath_1},\dots,F_{\imath_\xi}\right>_A \subseteq A^\tau$, for a while. Later we will find out that $M_{{2}}$ is actually maximal Cohen-Macaulay and hence $M_{{2}}=M$. Therefore, the goal is to construct a free resolution of $\left(M_{{2}}\right)_S$.

We further enlarge the matrix $\pi_{{1}} \in A^{\tau\times\left(3\tau+\xi\right)}$ constructed in Step 3 to a new matrix $\pi_{{2}} \in A^{\tau\times(3\tau+2\xi)}$ by
inserting column $F_{\imath_\nu} \in A^\tau$ between the columns $G_{\imath_\nu-1}''$ and $G_{\imath_\nu}$ of $\pi_{{1}}$
for each $\nu\in\left\{1,\dots,\xi\right\}$. Here $G_{\imath_\nu-1}''$ is $H_{\imath_\nu-1}$ if $\delta_{\imath_\nu-1}=0$,
or equivalently $\jmath_{\nu-1}+1=\imath_\nu$, and is $G_{\imath_\nu-1}$ otherwise, as in equation \ref{eqn:DefinitionOfG'}. As a result, we get
$$
\pi_{{2}}
:=
\left(
\arraycolsep=2pt\def\arraystretch{1}
\begin{array}{c|c|c|c|c|c|c|c|c|c|c|c|c|c|c|c|c|c|c}
\color{colorG} G_1 & \color{colorG} \cdots & \color{colorG} G_{\imath_1-1} & \color{colorF} F_{\imath_1} & \color{colorG} G_{\imath_1} & \color{colorG} \cdots & \color{colorG} G_{\jmath_1} & \color{colorH} H_{\jmath_1} & \color{colorG} G_{\jmath_1+1} & \color{colorG} \cdots & \color{colorG} G_{\imath_\xi-1} & \color{colorF} F_{\imath_\xi} & \color{colorG} G_{\imath_\xi} & \color{colorG} \cdots & \color{colorG} G_{\jmath_\xi} & \color{colorH} H_{\jmath_\xi} & \color{colorG} G_{\jmath_\xi+1} & \color{colorG} \cdots & \color{colorG} G_{3\tau}
\end{array}
\right)_{\tau\times \left(3\tau+2\xi\right)}.
$$
Then it can be viewed as an $S$-module map $\pi_{{2}}:S^{3\tau+2\xi}\rightarrow \left(A_S\right)^\tau$ whose image is $\left(M_{{2}}\right)_S$. Restricting the codomain, we get the surjective map $\pi_{{2}}:S^{3\tau+2\xi}\rightarrow \left(M_{{2}}\right)_S$.

\begin{figure}[H]
\adjustbox{scale=0.67,center}{$
\begin{tikzpicture}[column sep = -7.5pt]
\setcounter{MaxMatrixCols}{100}
\matrix (MatrixPhi)[matrix of math nodes, nodes in empty cells]
{
& \hspace{7mm} & \color{colorT} T_{\jmath_{\nu-1}+2} & \hspace{-5mm} \color{colorR} R_{\jmath_{\nu-1}+2} & \color{colorR} \cdots & \color{colorR} R_{\imath_\nu-1} & \color{colorR} R_{\imath_\nu} & \quad\color{colorR} R_{\imath_\nu+1} & \color{colorR} R_{\imath_\nu+2} & \hspace{-5mm}\color{colorR} R_{\imath_\nu+3} & \color{colorR} \cdots & \color{colorR} R_{\jmath_\nu} & \color{colorT} T_{\jmath_\nu} & \color{colorT} T_{\jmath_\nu+1} & \color{colorT} T_{\jmath_\nu+2}
\\[5mm]
\color{colorH} H^{\jmath_{\nu-1}} & & -\Lambda_{\jmath_{\nu-1}+1}^+ \chi_{\jmath_{\nu-1}+1}^{w_{\jmath_{\nu-1}+1}'
-1} & \hspace{-5mm} 0 & \cdots & 0 & 0 & & & & & & & &
\\[5mm]
\color{colorG} G^{\jmath_{\nu-1} + 1} & & \chi_{\jmath_{\nu-1}} & \hspace{-5mm} - \Lambda_{\jmath_{\nu-1}+2}^+ \chi_{\jmath_{\nu-1}+2}^{w_{\jmath_{\nu-1}+2}' -1} & \cdots & 0 & 0 & & & & & & & &
\\[5mm]
 & & 0 & \hspace{-5mm} \chi_{\jmath_{\nu-1}+1} & \ddots & \vdots & \vdots & & & & & & & &
\\[5mm]
 & & \vdots & \hspace{-5mm} \vdots & \ddots & -\Lambda_{\imath_{\nu}-1}^+ \chi_{\imath_{\nu}-1}^{w_{\imath_{\nu}-1}' -1} & 0 & & & & & & & &
\\[5mm]
\color{colorG} G^{\imath_{\nu} - 1} & & 0 & \hspace{-5mm} 0 & \cdots & \chi_{\imath_{\nu}-2} & -\Lambda_{\imath_{\nu}}^+ \chi_{\imath_{\nu}}^{w_{\imath_{\nu}}'} & & & & & & & &
\\[5mm]
\color{colorG} G^{\imath_{\nu}} & & & & & & \chi_{{\imath_\nu}-1} & \quad\chi_{\imath_\nu+1} & & & & & & &
\\[5mm]
\color{colorG} G^{\imath_{\nu} + 1} & & & & & & & \quad-\chi_{\imath_\nu} & \chi_{\imath_\nu+2} & & & & & &
\\[5mm]
\color{colorG} G^{\imath_{\nu} + 2} & & & & & & & & -\Lambda_{\imath_\nu+1}^- \chi_{\imath_\nu+1}^{-w_{\imath_\nu+1}'+1} & \hspace{-5mm} \chi_{\imath_{\nu}} & \cdots & 0 & 0 & &
\\[5mm]
 & & & & & & & & 0 & \hspace{-5mm} -\Lambda_{\imath_{\nu}+2}^- \chi_{\imath_{\nu}+2}^{-w_{\imath_{\nu}+2}'} & \ddots & \vdots & \vdots & &
\\[5mm]
 & & & & & & & & \vdots & \hspace{-5mm} \vdots & \ddots & \chi_{\jmath_\nu-3} & 0 & &
\\[5mm]
\color{colorG} G^{\jmath_\nu} & & & & & & & & 0 & \hspace{-5mm} 0 & \cdots & -\Lambda_{\jmath_\nu-1}^- \chi_{\jmath_\nu-1}^{-w_{\jmath_\nu-1}'} & \chi_{\jmath_\nu-2} & &
\\[5mm]
\color{colorH} H^{\jmath_\nu} & & & & & & & & 0 & \hspace{-5mm} 0 & \cdots & 0 & -\Lambda_{\jmath_\nu}^- \chi_{\jmath_\nu}^{-w_{\jmath_\nu}'} & \qquad\chi_{\jmath_\nu-1} \qquad &  -\Lambda_{\jmath_\nu+1}^+ \chi_{\jmath_\nu+1}^{w_{\jmath_\nu+1}'}
\\
};

\draw[color=blue] (MatrixPhi-2-3.north west) rectangle (MatrixPhi-6-7.south east);
\draw[color=red] (MatrixPhi-9-9.north west) rectangle (MatrixPhi-13-13.south east);

\draw[blue] ($0.5*(MatrixPhi-13-15.north west)+0.5*(MatrixPhi-13-15.north west)$) -- ($0.5*(MatrixPhi-13-15.south west)+0.5*(MatrixPhi-13-15.south west)$);
\draw[blue] ($0.5*(MatrixPhi-13-15.north west)+0.5*(MatrixPhi-13-15.north west)$) -- ($0.5*(MatrixPhi-13-15.north east)+0.5*(MatrixPhi-13-15.north east)$);

\node[fit=(MatrixPhi-4-8)(MatrixPhi-4-8)]{$\color{blue} \left(\varphi_{{1}}\right)_{\nu}^+$};
\node[fit=(MatrixPhi-11-8)(MatrixPhi-11-8)]{\hspace{-9mm}$\color{red} \left(\varphi_{{1}}\right)_{\nu}^-$};

\node[fit=($0.4*(MatrixPhi-12-14.north west)+0.6*(MatrixPhi-13-15.north west)$)]{$\hspace{10mm} \color{blue} \left(\varphi_{{1}}\right)_{\nu+1}^+$};

\tikzset{
    cheating dash/.code args={on #1 off #2}{
        \csname tikz@addoption\endcsname{%
            \pgfgetpath\currentpath%
            \pgfprocessround{\currentpath}{\currentpath}%
            \csname pgf@decorate@parsesoftpath\endcsname{\currentpath}{\currentpath}%
            \pgfmathparse{\csname pgf@decorate@totalpathlength\endcsname-#1}\let\rest=\pgfmathresult%
            \pgfmathparse{#1+#2}\let\onoff=\pgfmathresult%
            \pgfmathparse{max(floor(\rest/\onoff), 1)}\let\nfullonoff=\pgfmathresult%
            \pgfmathparse{max((\rest-\onoff*\nfullonoff)/\nfullonoff+#2, #2)}\let\offexpand=\pgfmathresult%
            \pgfsetdash{{#1}{\offexpand}}{0pt}}%
    }
}
\path(MatrixPhi-6-1.south) -- (MatrixPhi-7-1.north) coordinate[midway] (Y);
\draw[cheating dash=on 3pt off 3pt,gray] (Y -| MatrixPhi.west) -- (Y -| MatrixPhi.east);
\path(MatrixPhi-5-6.east) -- (MatrixPhi-6-7.west) coordinate[midway] (X);
\draw[cheating dash=on 3pt off 3pt,gray] (X |- MatrixPhi.north) -- (X |- MatrixPhi.south);

\node[fit=(MatrixPhi-4-1)(MatrixPhi-5-1)]{$\color{colorG} \vdots$};
\node[fit=(MatrixPhi-10-1)(MatrixPhi-11-1)]{$\color{colorG} \vdots$};

\end{tikzpicture}
$}
\caption{Submatrix $\varphi_{{1}} \left[H^{\jmath_{\nu-1}}:H^{\jmath_\nu};T_{\jmath_{\nu-1}+2}:T_{\jmath_\nu+2}\right]$ of $\varphi_{{1}}$}
\label{fig:Step4BeforeModifying}
\end{figure}

Next, for each $\nu\in\left\{1,\dots,\xi\right\}$, look at the submatrix
$
\varphi_{{1}} \left[H^{\jmath_{\nu-1}}:H^{\jmath_\nu};T_{\jmath_{\nu-1}+2}:T_{\jmath_\nu}\right]
$
of $\varphi_{{1}}$ containing $\left(\varphi_{{1}}\right)_\nu^+$ and $\left(\varphi_{{1}}\right)_\nu^-$ as shown in Figure \ref{fig:Step4BeforeModifying}, whose diagonal entries have been calculated in \ref{eqn:DiagonalsOfPhi1}.

We modify the matrix $\varphi_{{1}} \in S^{\left(3\tau+\xi\right)\times\left(3\tau+2\xi\right)}$ to a new matrix $\varphi_{{2}} := \varphi_{{2}} \left(w,\lambda,1\right)\ \in S^{\left(3\tau+2\xi\right)\times\left(3\tau+5\xi\right)}$ in the following way. For each $\nu\in\left\{1,\dots,\xi\right\}$,
\begin{itemize}
\item
Insert a new row $\color{colorF} F^{\imath_\nu}$ consisting of zeros above the row $\color{colorG} G^{\imath_\nu}$, and then
\item
Insert the new three columns $\color{colorQ} Q_{\imath_\nu-1}$, $\color{colorQ} Q_{\imath_\nu}$ and $\color{colorQ}
Q_{\imath_\nu+1}$ to the left of the column $\color{colorR} R_{\imath_\nu}$ (or $\color{colorT} T_{\jmath_{\nu-1}+2}$
in the case $\jmath_{\nu-1}+1 = \imath_\nu$)
\end{itemize}
as described in Figure \ref{fig:Step4AfterAdding}, where the entries of the new columns that do not appear
in the figure are all zeros. In the figure, the rows indicate
$$
\left(G^{\imath_\nu-1}\right)''
:=
\begin{cases}
H^{\jmath_{\nu-1}} & \text{if}\ \ \jmath_{\nu-1}+1 = \imath_\nu
\\
G^{\imath_\nu-1} & \text{otherwise}
\end{cases} 
,\quad
\left(G^{\imath_\nu}\right)'
:=
G^{\imath_\nu}
,\quad\dots,\quad
\left(G^{\jmath_\nu}\right)'
:=
G^{\jmath_\nu}
\quad\text{and}\quad
\left(G^{\jmath_\nu+1}\right)'
:=
H^{\jmath_\nu}
$$
depending on the case, imitating the definition of $G_{\imath_\nu-1}''$ and $G_{\imath_\nu+a}'$ given in \ref{eqn:DefinitionOfG'}, and the columns indicate
$$
R_{\imath_\nu}''
:=
\begin{cases}
T_{\jmath_{\nu-1}+2} & \text{if}\ \ \jmath_{\nu-1}+1 = \imath_\nu
\\
R_{\imath_\nu} & \text{otherwise}
\end{cases}
,\quad
R_{\imath_\nu+1}'
:=
R_{\imath_\nu+1}
,\quad\dots,\quad
R_{\jmath_\nu}'
:=
R_{\jmath_\nu}
\quad\text{and}\quad
R_{\jmath_\nu+1}'
:=
T_{\jmath_\nu},
$$
similarly. Note that we always have $\left(G^{\imath_\nu}\right)' = G^{\imath_\nu}$, $\left(G^{\imath_\nu+1}\right)'
= G^{\imath_\nu+1}$ and $R_{\imath_\nu+1}' = R_{\imath_\nu+1}$. The symbols $\zeta_{\imath_\nu,a,b}^c$ are as defined in \ref{eqn:DefinitionOfZeta}.

\begin{figure}[h]
\adjustbox{scale=0.65,center}{$
\begin{tikzpicture}[ = -7.5pt]
\setcounter{MaxMatrixCols}{100}
\matrix(MatrixPhi)[matrix of math nodes, nodes in empty cells]
{
 & \hspace{3mm} & \color{colorQ} Q_{\imath_\nu-1} & \color{colorQ} Q_{\imath_\nu} & \color{colorQ} Q_{\imath_\nu+1} & \color{colorR} R_{\imath_\nu}'' & \quad \color{colorR} R_{\imath_\nu+1} & \color{colorR} R_{\imath_\nu+2}' & & & & & \color{colorR} R_{\imath_\nu+\kappa_\nu+1}' & \hspace{-5mm} \color{colorR} R_{\imath_\nu+\kappa_\nu+2}'
\\[5mm]
\color{colorG} \left(G^{\imath_{\nu} - 1}\right)'' & & \color{colorQ} 0 & \color{colorQ} 0 & \hspace{-10mm} \color{colorQ} -\Lambda_{\imath_\nu}^+ \chi_{\imath_\nu}^{w_{\imath_\nu}' -1} \hspace{-10mm} & -\Lambda_{\imath_{\nu}}^+ \chi_{\imath_{\nu}}^{w_{\imath_{\nu}}'} & & & & & & & &
\\[5mm]
\color{colorF} F^{\imath_\nu} & & \color{colorQ} \chi_{\imath_\nu} & \color{colorQ} \chi_{\imath_\nu+1} & \color{colorQ} \chi_{\imath_\nu+2} & \color{colorF} 0 & \quad \color{colorF} 0 & \color{colorF} 0 & \hspace{-3mm} \color{colorF} 0 & \color{colorF} 0 & \color{colorF} 0 & \color{colorF} \cdots & \color{colorF} 0 & \hspace{-5mm} \color{colorF} 0
\\[5mm]
\color{colorG} G^{\imath_{\nu}} & & \color{colorQ} -1 & \color{colorQ} 0 & \color{colorQ} 0 & \chi_{{\imath_\nu}-1} & \quad \chi_{\imath_\nu+1} & & & & & & &
\\[5mm]
\color{colorG} G^{\imath_{\nu}+1} & & \color{colorQ} 0 & \color{colorQ} -1 & \color{colorQ} 0 & & \quad -\chi_{\imath_\nu} & \chi_{\imath_\nu+2} & & & & \phantom{\color{red} -\zeta_{\imath_\nu,\kappa_\nu-2,\kappa_\nu+1}^{\kappa_\nu-2}}
& &
\\[5mm]
\color{colorG} \left(G^{\imath_{\nu} + 2}\right)' & & \color{colorQ} 0 & \color{colorQ} 0 & \color{colorQ} -\Lambda_{\imath_\nu+1}^- \chi_{\imath_\nu+1}^{-w_{\imath_\nu+1}'} &  & & -\Lambda_{\imath_\nu+1}^- \chi_{\imath_\nu+1}^{-w_{\imath_\nu+1}'+1} & \hspace{-3mm} \chi_{\imath_{\nu}} & 0 & 0 & \cdots & 0 & \hspace{-5mm} 0
\\[5mm]
\color{colorG} \left(G^{\imath_{\nu} +3}\right)' & & \color{colorQ} -\zeta_{\imath_\nu,-1,2}^{-1} & \color{colorQ} -\zeta_{\imath_\nu,-1,2}^{0} & \color{colorQ} -\zeta_{\imath_\nu,-1,2}^{1} & & & 0 & \hspace{-5mm} -\Lambda_{\imath_{\nu}+2}^- \chi_{\imath_{\nu}+2}^{-w_{\imath_\nu+2}'} & \chi_{\imath_\nu+1} & 0 & \cdots & 0 & \hspace{-5mm} 0
\\[5mm]
\color{colorG} \left(G^{\imath_{\nu} + 4}\right)' & & \color{colorQ} -\zeta_{\imath_\nu,-1,3}^{-1} & \color{colorQ} -\zeta_{\imath_\nu,-1,3}^{0} & \color{colorQ} -\zeta_{\imath_\nu,-1,3}^{1} & & & 0 &
\hspace{-3mm} 0 & \hspace{-2mm} -\Lambda_{\imath_{\nu}+3}^- \chi_{\imath_{\nu}+3}^{-w_{\imath_\nu+3}'} & \chi_{\imath_\nu+2} & \cdots & 0 & \hspace{-5mm} 0
\\[5mm]
& & \color{colorQ} -\zeta_{\imath_\nu,-1,4}^{-1} & \color{colorQ} -\zeta_{\imath_\nu,-1,4}^{0} & \color{colorQ} -\zeta_{\imath_\nu,-1,4}^{1} & & & 0 &
\hspace{-3mm} 0 & 0 & \hspace{-2mm} -\Lambda_{\imath_{\nu}+4}^- \chi_{\imath_{\nu}+4}^{-w_{\imath_\nu+4}'} & \ddots & \vdots & \hspace{-5mm} \vdots
\\[5mm]
& & \color{colorQ} \vdots & \color{colorQ} \vdots & \color{colorQ} \vdots & & & \vdots & \hspace{-3mm} \ddots & \ddots & \ddots & \ddots & \chi_{\imath_\nu+\kappa_\nu-2} & \hspace{-5mm} 0
\\[5mm]
\color{colorG} \left(G^{\imath_\nu+\kappa_\nu+1}\right)' & & \color{colorQ} -\zeta_{\imath_\nu,-1,\kappa_\nu}^{-1} & \color{colorQ} -\zeta_{\imath_\nu,-1,\kappa_\nu}^{0} & \color{colorQ} -\zeta_{\imath_\nu,-1,\kappa_\nu}^{1} & & & 0 & \hspace{-3mm} \cdots & 0 & 0 & 0 & -\Lambda_{\imath_\nu+\kappa_\nu}^- \chi_{\imath_\nu+\kappa_\nu}^{-w_{\imath_\nu+\kappa_\nu}'} & \hspace{-5mm} \chi_{\imath_\nu+\kappa_\nu-1}
\\[5mm]
\color{colorG} \left(G^{\imath_\nu+\kappa_\nu+2}\right)' & & \color{colorQ} -\zeta_{\imath_\nu,-1,\kappa_\nu+1}^{-1}
\hspace{1.5mm} & \color{colorQ} -\zeta_{\imath_\nu,-1,\kappa_\nu+1}^{0}\hspace{-1mm} & \color{colorQ} -\zeta_{\imath_\nu,-1,\kappa_\nu+1}^{1} & & & 0 & \hspace{-3mm} \cdots & 0 & 0 & 0 & 0 & \hspace{-5mm} -\Lambda_{\imath_\nu+\kappa_\nu+1}^- \chi_{\imath_\nu+\kappa_\nu+1}^{-w_{\imath_\nu+\kappa_\nu+1}'}
\\
};



\draw[blue] ($0.5*(MatrixPhi-2-6.north east)+0.5*(MatrixPhi-2-6.north east)$) -- ($0.5*(MatrixPhi-2-6.south east)+0.5*(MatrixPhi-2-6.south east)$);
\draw[blue] ($0.5*(MatrixPhi-2-6.south west)+0.5*(MatrixPhi-2-6.south west)$) -- ($0.5*(MatrixPhi-2-6.south east)+0.5*(MatrixPhi-2-6.south east)$);
\draw[red] ($0.5*(MatrixPhi-6-8.north west)+0.5*(MatrixPhi-6-8.north west)$) -- ($-4.55*(MatrixPhi-6-8.north west)+5.55*(MatrixPhi-6-8.north east)$);
\draw[red] ($0.5*(MatrixPhi-6-8.north west)+0.5*(MatrixPhi-6-8.north west)$) -- ($-9.5*(MatrixPhi-6-8.north west)+10.5*(MatrixPhi-6-8.south west)$);

\node[fit=(MatrixPhi-2-7)(MatrixPhi-2-7)]{$\color{blue} \left(\varphi_{{1}}\right)_{\nu}^+$};
\node[fit=(MatrixPhi-11-6)(MatrixPhi-12-7)]{\hspace{2mm}$\color{red} \left(\varphi_{{1}}\right)_{\nu}^-$};

\node[fit=(MatrixPhi-1-10)(MatrixPhi-1-11)]{\hspace{0mm}$\color{colorR} \cdots$};
\node[fit=(MatrixPhi-9-1)(MatrixPhi-10-1)]{$\color{colorG} \vdots$};

\end{tikzpicture}
$}
\caption{Submatrix $\varphi_{{2}} \left[\left(G^{\imath_\nu-1}\right)'':\left(G^{\imath_\nu+\kappa_\nu+2}\right)';Q_{\imath_\nu-1}:R_{\imath_\nu+\kappa_\nu+2}'\right]$
of $\varphi_{{2}}$}
\label{fig:Step4AfterAdding}
\end{figure}

The construction of $\varphi_{{2}}$ is based on the formulas \ref{eqn:RelationsOnF}, which makes
the matrix equation $\pi_{{2}}\varphi_{{2}} = 0$\ hold. It is obvious that any element $s\in S$
satisfying $s F_{\imath_\nu} \in \left(M_{{1}}\right)_S = \tilde{M}_S$ is in the ideal of $S$\ generated by three
elements $x$, $y$ and $z$ in $S$, namely the maximal ideal of $S$. Since these elements coincide with the nonzero entries $\chi_{\imath_\nu}$, $\chi_{\imath_\nu+1}$ and $\chi_{\imath_\nu+2}$ in the newly added row of $\varphi_{{2}}$, we can apply Lemma \ref{lem:MatrixExpansionLemma} for each
$\nu\in\left\{1,\dots,\xi\right\}$ to get a part of free resolution of $\left(M_{{2}}\right)_S$:
\begin{equation}\label{eqn:SESforM2}
\begin{tikzcd}[ = 20pt] 
S^{3\tau+5\xi} \arrow[r, "\varphi_{{2}}"] & S^{3\tau+2\xi} \arrow[r, "\pi_{{2}}"] & \left(M_{{2}}\right)_S \arrow[r] & 0.
\end{tikzcd}
\end{equation}

Then some appropriate coordinate changes on the codomain, i.e. elementary row operations, enable us to remove the terms $-\zeta_{\imath_\nu,-1,b}^c$ in $\varphi_{{2}}$. The detailed recipe for each $\nu\in\left\{1, \dots,
\xi\right\}$ is given as follows:
\begin{itemize}
\item
Set
$\left(G^{\jmath_\nu+1}\right)^* := \left(G^{\jmath_\nu+1}\right)'$,
$\left(G^{\jmath_\nu}\right)^* := \left(G^{\jmath_\nu}\right)'$, \dots ,
$\left(G^{\imath_\nu+\kappa_\nu+3}\right)^* := \left(G^{\imath_\nu+\kappa_\nu+3}\right)'$,
\item
Add $-\zeta_{\imath_\nu,b-1,b+2}^{b-1}$ times the row $\left(G^{\imath_\nu+b}\right)'$ to the row $\left(G^{\imath_\nu+b+3}\right)'$
and rename the changed row $\left(G^{\imath_\nu+b+3}\right)^*$ for each $b\in\left\{\kappa_\nu-1, \kappa_\nu-2, \dots, 1, 0\right\}$ in order, and then
\item
Set
$\left(G^{\imath_\nu+2}\right)^* := \left(G^{\imath_\nu+2}\right)'$,
$\left(G^{\imath_\nu+1}\right)^* := G^{\imath_\nu+1}$
 and 
$\left(G^{\imath_\nu}\right)^* := G^{\imath_\nu}$.
\end{itemize}
We perform those for all $\nu$ and denote the resulting matrix by $\varphi_{{3}} := \varphi_{{3}} \left(w,\lambda,1\right)\ \in S^{\left(3\tau+2\xi\right)\times\left(3\tau+5\xi\right)}$. We keep the original names for columns, although they may also have changed. For each $\nu$, the procedure eliminates the all terms $-\zeta_{\imath_\nu,-1,b+2}^c$ $\left(\text{in the row }\left(G^{\imath_\nu+b+3}\right)'\right)$
in $\varphi_{{2}}$ in the order from bottom to top, using the identities
$
\zeta_{\imath_\nu,-1,b-1}^c \zeta_{\imath_\nu,b-1,b+2}^{b-1}
=
\zeta_{\imath_\nu,-1,b+2}^c
$
observed in equation \ref{eqn:RelationsOnZeta}. At the same time, it yields some `\emph{unwanted terms}'. They
are described in Figure \ref{fig:Step4AfterRowOperations} in pink color, where we calculated as, for example,
$$
\left(-\zeta_{\imath_\nu,b-1,b+2}^{b-1}\right) \left(-\Lambda_{\imath_\nu+b-1}^- \chi_{\imath_\nu+b-1}^{-w_{\imath_\nu+b-1}'}\right)
=
\zeta_{\imath_\nu,b-1,b+2}^{b-1} \zeta_{\imath_\nu,b-2,b-1}^{b-1} \chi_{\imath_\nu+b-1}
=
\zeta_{\imath_\nu,b-2,b+2}^{b-1} \chi_{\imath_\nu+b+2},
$$
for $b\in\left\{3, \dots, \kappa_\nu-1\right\}$ using the definition of $\zeta_{\imath_\nu,b-2,b-1}^{b-1}$ in \ref{eqn:DefinitionOfZeta} and the identity \ref{eqn:RelationsOnZeta} again.

\begin{figure}[h]
\adjustbox{scale=0.67,center}{$
\begin{tikzpicture}[ = -10pt]
\setcounter{MaxMatrixCols}{100}
\matrix(MatrixPhi)[matrix of math nodes, nodes in empty cells]
{
 & \hspace{5mm} & \color{gray} Q_{\imath_\nu-1} \hspace{7mm} & \color{gray} Q_{\imath_\nu} & \color{colorQ} Q_{\imath_\nu+1} & \color{gray} R_{\imath_\nu}'' & \quad \color{gray} R_{\imath_\nu+1} & \color{gray} R_{\imath_\nu+2}' & & & & & \color{colorR} R_{\imath_\nu+\kappa_\nu+1}' & \hspace{-5mm} \color{colorR} R_{\imath_\nu+\kappa_\nu+2}'
\\[5mm]
\color{colorG} \left(G^{\imath_{\nu} - 1}\right)'' & & \color{gray} 0 \hspace{7mm} & \color{gray} 0 & \hspace{-10mm} \color{black} -\Lambda_{\imath_\nu}^+ \chi_{\imath_\nu}^{w_{\imath_\nu}' -1} \hspace{-10mm} & \color{gray} -\Lambda_{\imath_{\nu}}^+ \chi_{\imath_{\nu}}^{w_{\imath_{\nu}}'} & & & & & & & &
\\[5mm]
\color{colorF} F^{\imath_\nu} & & \color{gray} \chi_{\imath_\nu} \hspace{7mm} & \color{gray} \chi_{\imath_\nu+1} & \color{black} \chi_{\imath_\nu+2} & & & & & & & & &
\\[5mm]
\color{gray} \left(G^{\imath_{\nu}}\right)^* & & \color{gray} -1 \hspace{7mm} & \color{gray} 0 & \color{black} \color{gray} 0 & \color{gray} \chi_{{\imath_\nu}-1} & \quad \color{gray} \chi_{\imath_\nu+1} & & & & & & &
\\[5mm]
\color{gray} \left(G^{\imath_{\nu} + 1}\right)^* & & \color{gray} 0 \hspace{7mm} & \color{gray} -1 & \color{black} \color{gray} 0 & & \quad \color{gray} -\chi_{\imath_\nu} & \color{gray} \chi_{\imath_\nu+2} & & & & & &
\\[5mm]
\color{colorG} \left(G^{\imath_{\nu} + 2}\right)^* & & \color{gray} 0 \hspace{7mm} & \color{gray} 0 & \color{black} -\Lambda_{\imath_\nu+1}^- \chi_{\imath_\nu+1}^{-w_{\imath_\nu+1}'} & \hspace{3mm}
\color{gray} 0^{(*)} & & \color{gray} -\Lambda_{\imath_\nu+1}^- \chi_{\imath_\nu+1}^{-w_{\imath_\nu+1}'+1} & \hspace{-3mm} \chi_{\imath_{\nu}} & 0 & 0 & \cdots & 0 & \hspace{-5mm} 0
\\[5mm]
\color{colorG} \left(G^{\imath_{\nu} +3}\right)^* & & \color{gray} 0 \hspace{7mm} & \color{gray} 0 & \color{black} 0 & \color{rosepink}
\begin{array}{c} -\zeta_{\imath_\nu,-1,2}^{-1} \\ \chi_{\imath_\nu+2} \end{array} & \hspace{1mm} \color{rosepink}
\begin{array}{c} -\zeta_{\imath_\nu,-1,2}^{-1} \\ \chi_{\imath_\nu+1} \end{array} \hspace{1mm} & \color{gray} 0 & \hspace{-5mm} -\Lambda_{\imath_{\nu}+2}^- \chi_{\imath_{\nu}+2}^{-w_{\imath_\nu+2}'} & \chi_{\imath_\nu+1} & 0 & \cdots & 0 & \hspace{-5mm} 0
\\[5mm]
\color{colorG} \left(G^{\imath_{\nu} + 4}\right)^* & & \color{gray} 0 \hspace{7mm} & \color{gray} 0 & \color{black} 0 & & \color{rosepink}
\begin{array}{c} \zeta_{\imath_\nu,-1,3}^0 \\ \chi_{\imath_\nu+3} \end{array} & \color{rosepink}
\begin{array}{c} -\zeta_{\imath_\nu,0,3}^0 \\ \chi_{\imath_\nu+2} \end{array} &
\hspace{-3mm} 0 & \hspace{-2mm} -\Lambda_{\imath_{\nu}+3}^- \chi_{\imath_{\nu}+3}^{-w_{\imath_\nu+3}'} & \chi_{\imath_\nu+2} & \cdots & 0 & \hspace{-5mm} 0
\\[5mm]
& & \color{gray} 0 \hspace{7mm} & \color{gray} 0 & \color{black} 0 & & & \color{rosepink}
\begin{array}{c} \zeta_{\imath_\nu,0,4}^1 \\ \chi_{\imath_\nu+4} \end{array} &
\hspace{-3mm} \color{rosepink} \begin{array}{c} -\zeta_{\imath_\nu,1,4}^1 \\ \chi_{\imath_\nu+3} \end{array} & 0 & \hspace{-2mm} -\Lambda_{\imath_{\nu}+4}^- \chi_{\imath_{\nu}+4}^{-w_{\imath_\nu+4}'} & \ddots & \vdots & \hspace{-5mm} \vdots
\\[5mm]
& & \color{gray} \vdots \hspace{7mm} & \color{gray} \vdots & \color{black} \vdots & & & \color{gray} \vdots & \hspace{-3mm} \color{rosepink} \ddots & \color{rosepink} \ddots & \ddots & \ddots & \chi_{\imath_\nu+\kappa_\nu-2} & \hspace{-5mm} 0
\\[5mm]
\color{colorG} \left(G^{\imath_\nu+\kappa_\nu+1}\right)^* & & \color{gray} 0 \hspace{7mm} & \color{gray} 0 & \color{black} 0 & & & \color{gray} 0 & \hspace{-3mm} \cdots & \color{rosepink} \begin{array}{c}
\zeta_{\imath_\nu,\kappa_\nu-4,\kappa_\nu}^{\kappa_\nu-3} \\ \chi_{\imath_\nu+\kappa_\nu} \end{array} & \color{rosepink} \begin{array}{c} -\zeta_{\imath_\nu,\kappa_\nu-3,\kappa_\nu}^{\kappa_\nu-3} \\ \chi_{\imath_\nu+\kappa_\nu-1}
\end{array} & 0 & -\Lambda_{\imath_\nu+\kappa_\nu}^- \chi_{\imath_\nu+\kappa_\nu}^{-w_{\imath_\nu+\kappa_\nu}'} & \hspace{-5mm} \chi_{\imath_\nu+\kappa_\nu-1}
\\[5mm]
\color{colorG} \left(G^{\imath_\nu+\kappa_\nu+2}\right)^* & & \color{gray} 0 \hspace{7mm} & \color{gray} 0 & \color{black} 0 & & & \color{gray} 0 & \hspace{-3mm} \cdots & 0 & \color{rosepink} \begin{array}{c}
\zeta_{\imath_\nu,\kappa_\nu-3,\kappa_\nu+1}^{\kappa_\nu-2} \\ \chi_{\imath_\nu+\kappa_\nu+1} \end{array} & \color{rosepink} \begin{array}{c} -\zeta_{\imath_\nu,\kappa_\nu-2,\kappa_\nu+1}^{\kappa_\nu-2} \\ \chi_{\imath_\nu+\kappa_\nu}
\end{array} & 0 & \hspace{-5mm} -\Lambda_{\imath_\nu+\kappa_\nu+1}^- \chi_{\imath_\nu+\kappa_\nu+1}^{-w_{\imath_\nu+\kappa_\nu+1}'}
\\
};



\draw[blue] ($0.5*(MatrixPhi-2-6.north east)+0.5*(MatrixPhi-2-6.north east)$) -- ($0.5*(MatrixPhi-2-6.south east)+0.5*(MatrixPhi-2-6.south east)$);
\draw[blue] ($0.5*(MatrixPhi-2-6.south west)+0.5*(MatrixPhi-2-6.south west)$) -- ($0.5*(MatrixPhi-2-6.south east)+0.5*(MatrixPhi-2-6.south east)$);
\draw[red] ($0.5*(MatrixPhi-6-8.north west)+0.5*(MatrixPhi-6-8.north west)$) -- ($-4.55*(MatrixPhi-6-8.north west)+5.55*(MatrixPhi-6-8.north east)$);
\draw[red] ($0.5*(MatrixPhi-6-8.north west)+0.5*(MatrixPhi-6-8.north west)$) -- ($-11*(MatrixPhi-6-8.north west)+12*(MatrixPhi-6-8.south west)$);

\node[fit=(MatrixPhi-2-7)(MatrixPhi-2-7)]{$\color{blue} \left(\varphi_{{1}}\right)_{\nu}^+$};
\node[fit=(MatrixPhi-11-6)(MatrixPhi-12-7)]{\hspace{7mm}$\color{red} \left(\varphi_{{1}}\right)_{\nu}^-$};

\node[fit=(MatrixPhi-1-10)(MatrixPhi-1-11)]{\hspace{0mm}$\color{colorR} \cdots$};
\node[fit=(MatrixPhi-9-1)(MatrixPhi-10-1)]{$\color{colorG} \vdots$};

\end{tikzpicture}
$}
\caption{Submatrix $\varphi_{{3}} \left[\left(G^{\imath_\nu-1}\right)'':\left(G^{\imath_\nu+\kappa_\nu+2}\right)^*;Q_{\imath_\nu-1}:R_{\imath_\nu+\kappa_\nu+2}'\right]$
of $\varphi_{{3}}$}
\label{fig:Step4AfterRowOperations}
\end{figure}

We should carry out the coordinate changes on the domain, i.e. elementary column operations on $\pi_{{2}} \in A^{\tau\times(3\tau+2\xi)}$ corresponding to the row operations taken to $\varphi_{{2}}$, to maintain the exactness of the sequence that has been obtained. We proceed
for each $\nu\in\left\{1, \dots, \xi\right\}$ as follows:
\begin{itemize}
\item
Set $G_{\jmath_\nu+1}^* := G_{\jmath_\nu+1}'$,
$G_{\jmath_\nu}^* := G_{\jmath_\nu}'$
, \dots,
$G_{\imath_\nu+\kappa_\nu}^* := G_{\imath_\nu+\kappa_\nu}'$, and then
\item
Add $\zeta_{\imath_\nu,b-1,b+2}^{b-1}$ times the column $G_{\imath_\nu+b+3}^*$ to the column $G_{\imath_\nu+b}'$
and rename the changed column $G_{\imath_\nu+b}^*$ for each $b\in\left\{\kappa_\nu-1, \kappa_\nu-2, \dots, 1, 0\right\}$ in order.
\end{itemize}
Denoting the transformed matrix by $\pi_{{3}}:= \pi_{{3}}\left(w,\lambda,1\right) \in A^{\tau\times\left(3\tau+2\xi\right)}$,
we can express it as
$$
\pi_{{3}}
:=
\left(
\arraycolsep=2pt\def\arraystretch{1}
\begin{array}{c|c|c|c|c|c|c|c|c|c|c|c|c|c|c|c|c|c|c}
\color{colorG} G_1 & \color{colorG} \cdots & \color{colorG} G_{\imath_1-1} & \color{colorF} F_{\imath_1} & \color{colorG} G_{\imath_1}^* & \color{colorG} \cdots & \color{colorG} G_{\jmath_1}^* & \color{colorH} H_{\jmath_1} & \color{colorG} G_{\jmath_1+1} & \color{colorG} \cdots & \color{colorG} G_{\imath_\xi-1} & \color{colorF} F_{\imath_\xi} & \color{colorG} G_{\imath_\xi}^* & \color{colorG} \cdots & \color{colorG} G_{\jmath_\xi}^* & \color{colorH} H_{\jmath_\xi} & \color{colorG} G_{\jmath_\xi+1} & \color{colorG} \cdots & \color{colorG} G_{3\tau}
\end{array}
\right)_{\tau\times \left(3\tau+2\xi\right)}
$$
where the explicit form of columns are given by
$$
G_{\imath_\nu+b}^*
=
\begin{cases}
G_{\imath_\nu+b} + \displaystyle\sum_{c=b+3}^{\kappa_\nu+2} \zeta_{\imath_\nu,b-1,c-1}^{b-1} G_{\imath_\nu+c}'
\qquad
\text{if } 0 \le b \le \kappa_\nu-1,
\\[5mm]
G_{\imath_\nu+b}'
=
\begin{cases}
G_{\imath_\nu+b}
&
\hspace{24mm} \text{if } \kappa_\nu \le b \le \jmath_\nu-\imath_\nu,
\\[2mm]
H_{\jmath_\nu}
&
\hspace{24mm} \text{if } b = \jmath_\nu-\imath_\nu+1.
\end{cases}
\end{cases}
$$

Then we may check
$
\operatorname{im}\varphi_{{3}}
=
\operatorname{ker}\pi_{{3}}
$
from the fact that
$
\operatorname{im}\varphi_{{2}}
=
\operatorname{ker}\pi_{{2}}
$
and the above constructions of $\varphi_{{3}}$ and $\pi_{{3}}$.
Moreover, the construction of $\pi_{{3}}$ yields $\operatorname{im}\pi_{{3}} = \operatorname{im}\pi_{{2}}$.
Thus, setting $M_{{3}} := M_{{2}}$, we get the following modified exact sequence of $S$-modules from
\ref{eqn:SESforM2}:
\begin{equation}\label{eqn:SESforM3}
\begin{tikzcd}[ = 20pt] 
S^{3\tau+5\xi} \arrow[r, "\varphi_{{3}}"] & S^{3\tau+2\xi} \arrow[r, "\pi_{{3}}"] & \left(M_{{3}}\right)_S \arrow[r] & 0.
\end{tikzcd}
\end{equation}

To get rid of the `\emph{unwanted terms}' in the matrix $\varphi_{{3}}$, we execute coordinate changes
on its domain, i.e. elementary column operations on it. First, to remove the bottom-right entry, consider
the column
$$
\zeta_{\imath_\nu,\kappa_\nu-2,\kappa_\nu+1}^{\kappa_\nu-2} \chi_{\imath_\nu+\kappa_\nu} \mathbf{e}_{\left(G^{\imath_\nu+\kappa_\nu+2}\right)'}
=
\left(\displaystyle\prod_{d=\kappa_\nu-1}^{\kappa_\nu+1}\Lambda_{\imath+d}^-\right)
\chi_{\imath_\nu+\kappa_\nu+1}^{-w_{\imath_\nu+\kappa_\nu+1}'-1} \chi_{\imath_\nu+\kappa_\nu} \mathbf{e}_{\left(G^{\imath_\nu+\kappa_\nu+2}\right)'}.
$$
When $\imath_\nu+\kappa_\nu+1=\jmath_\nu$, it is just an $S$-multiple of the column $T_{\jmath_\nu+1} = \chi_{\jmath_\nu-1}
\mathbf{e}_{H^{\jmath_\nu}}$ in $\varphi_{{3}}$. Otherwise, when $\imath_\nu+\kappa_\nu+2\le\jmath_\nu$,
it is an $S$-multiple of the column $\chi_{\imath_\nu+\kappa_\nu} \chi_{\imath_\nu+\kappa_\nu+1} \mathbf{e}_{G^{\imath_\nu+\kappa_\nu+2}}$,
which is an $S$-linear combination of columns $R_{\imath_\nu+\kappa_\nu+3}'$, \dots, $R_{\jmath_\nu+1}'$
and $T_{\jmath_\nu+1}$ in $\varphi_{{3}}$, by the same argument used in equations \ref{eqn:GeneratingRSharp}. In any cases, therefore,
the mentioned column is an $S$-linear combination of columns in $\varphi_{{3}}$. 

Using this fact, we provide the explicit manual to remove all `\emph{unwanted terms}' in $\varphi_{{3}}$
as follows: For each $\nu\in\left\{1, \dots, \xi\right\}$,
\begin{itemize}
\item
Add the column $\zeta_{\imath_\nu,\kappa_\nu-2,\kappa_\nu+1}^{\kappa_\nu-2} \chi_{\imath_\nu+\kappa_\nu} \mathbf{e}_{\left(G^{\imath_\nu+\kappa_\nu+2}\right)'}$ to the column $R_{\imath_\nu+\kappa_\nu}'$ using elementary
column operations,
\item
Add $\zeta_{\imath_\nu,b-2,b+1}^{b-2}$ times the column $R_{\imath_\nu+b+3}'$ to the column $R_{\imath_\nu+b}'$
for each $b\in\left\{\kappa_\nu-1, \kappa_\nu-2, \dots, 2, 1\right\}$, and then
\item
Add $\Lambda_{\imath_\nu+1}^- \chi_{\imath_\nu+1}^{-w_{\imath_\nu+1}'}$ times the column $R_{\imath_\nu+3}'$
to the column $R_{\imath_\nu}''$.
\end{itemize}
We denote the transformed matrix by $\varphi_{{4}} := \varphi_{{4}} \left(w,\lambda,1\right)\ \in S^{\left(3\tau+2\xi\right)\times\left(3\tau+5\xi\right)}$ and keep the original names for its rows and columns.

One can check indeed that the process eliminates the entries of $\varphi_{{3}}$ which are colored
in pink in Figure \ref{fig:Step4AfterRowOperations}. The bottom-right entry disappears in the first stage.
The entries other than this and the top-left entry are removed in the second stage in order from right
to left, using identities such as
$$
\zeta_{\imath_\nu,b-2,b+1}^{b-2} \left(-\Lambda_{\imath_\nu+b+2}^- \chi_{\imath_\nu+b+2}^{-w_{\imath_\nu+b+2}'}\right)
=
\zeta_{\imath_\nu,b-2,b+1}^{b-2} \zeta_{\imath_\nu,b+1,b+2}^{b+2} \chi_{\imath_\nu+b+2}
=
\zeta_{\imath_\nu,b-2,b+2}^{b-1} \chi_{\imath_\nu+b+2}
$$
for $b\in\left\{1,\dots,\kappa_\nu-1\right\}$, which follow from the definition of $\zeta_{\imath_\nu,b+1,b+2}^{b+2}$
in \ref{eqn:DefinitionOfZeta} and the identity \ref{eqn:RelationsOnZeta}. Then in the final stage the top-left one goes
away, but a new term $\Lambda_{\imath_\nu+1}^- \chi_{\imath_\nu+1}^{-w_{\imath_\nu+1}'}
\chi_{\imath_\nu}$ is created instead in the place marked with $(*)$.

As coordinate changes on the domain of $\varphi_{{3}}$ do not affect the exactness of sequence \ref{eqn:SESforM3},
just setting $\pi_{{4}}:=\pi_{{3}} \in A^{\tau\times\left(3\tau+2\xi\right)}$ and $M_{{4}}:=M_{{3}}$, we get the following exact sequence:
$$
\begin{tikzcd}[column sep = 20pt] 
S^{3\tau+5\xi} \arrow[r, "\varphi_{{4}}"] & S^{3\tau+2\xi} \arrow[r, "\pi_{{4}}"] & \left(M_{{4}}\right)_S \arrow[r] & 0.
\end{tikzcd}
$$

Here we may abandon three columns $R_{\imath_\nu}''$, $R_{\imath_\nu+1}$ and $R_{\imath_\nu+2}'$ in  $\varphi_{{4}}$ because they can be expressed as $S$-linear combinations of other columns, not contributing
to $\operatorname{im}\varphi_{{4}}$. To be specific,
we have
\begin{align*}
R_{\imath_\nu}''
&=
\chi_{\imath_\nu} Q_{\imath_\nu+1} - \chi_{\imath_\nu-1} Q_{\imath_\nu-1},
\\[2mm]
R_{\imath_\nu+1}
&=
-\chi_{\imath_\nu+1} Q_{\imath_\nu-1} + \chi_{\imath_\nu} Q_{\imath_\nu},
\\[0mm]
R_{\imath_\nu+2}'
&=
-\chi_{\imath_\nu+2} Q_{\imath_\nu} + \chi_{\imath_\nu+1} Q_{\imath_\nu+1} + \Lambda_{\imath_\nu}^+ \chi_{\imath_\nu}^{w_{\imath_\nu}'-1}
\chi_{\imath_\nu+1} \mathbf{e}_{\left(G^{\imath_\nu-1}\right)''}.
\end{align*}
Note that the column $\Lambda_{\imath_\nu}^+ \chi_{\imath_\nu}^{w_{\imath_\nu}'-1} \chi_{\imath_\nu+1} \mathbf{e}_{\left(G^{\imath_\nu-1}\right)''}$
is just an $S$-multiple of the column $T_{\jmath_{\nu-1}+1} = \chi_{\jmath_{\nu-1}-1} \mathbf{e}_{H^{\jmath_{\nu-1}}}$
in $\varphi_{{4}}$ if $\jmath_{\nu-1}+1 = \imath_\nu$, and an $S$-multiple of the column $\chi_{\imath_\nu-3} \chi_{\imath_\nu-2}
\mathbf{e}_{G^{\imath_\nu-1}}$ otherwise, which is an $S$-linear combination of columns $T_{\jmath_{\nu-1}+1}$,
$T_{\jmath_{\nu-1}+2}$, $R_{\jmath_{\nu-1}+2}$, \dots, $R_{\imath_\nu-1}$ in $\varphi_{{4}}$ by the same argument used in equations \ref{eqn:GeneratingRSharp}.

Then we further delete rows $\left(G^{\imath_\nu}\right)^*$, $\left(G^{\imath_\nu+1}\right)^*$ and columns
$Q_{\imath_\nu-1}$, $Q_{\imath_\nu}$ in $\varphi_{{4}}$. Denote the surviving matrix as $\varphi_{{5}} := \varphi_{{5}} \left(w,\lambda,1\right) \in S^{3\tau \times 3\tau}$ after removing two rows
and five columns in $\varphi_{{4}} \in S^{\left(3\tau+2\xi\right)\times\left(3\tau+5\xi\right)}$ for each $\nu\in\left\{1,\dots,\xi\right\}$. Then, finally, we find that $\varphi_{{5}}$ is in fact exactly
equal to $\varphi$, comparing the submatrix $\varphi_{{5}} \left[H^{\jmath_{\nu-1}}:H^{\jmath_\nu};T_{\jmath_{\nu-1}+2}:T_{\jmath_\nu+2}\right]$ of $\varphi_{{5}}$
described in Figure \ref{fig:Step4AfterDeleting} with the corresponding part of $\varphi$. We name
the rows and columns of $\varphi$ the same as the corresponding rows and columns of $\varphi_{{5}}$,
and define submatrices $\varphi_\nu^+$ and $\varphi_\nu^-$ of $\varphi$
as boxed in the figure.

\begin{figure}[h]
\adjustbox{scale=0.7,center}{$
\begin{tikzpicture}
\setcounter{MaxMatrixCols}{100}
\matrix(MatrixPhi)[matrix of math nodes, nodes in empty cells]
{
& \hspace{5mm} & \color{colorT} T_{\jmath_{\nu-1}+2} & \hspace{-9mm} \color{colorR} R_{\jmath_{\nu-1}+2} & \color{colorR} \cdots & \color{colorR} R_{\imath_\nu-1} \hspace{-7mm} & \color{colorQ} Q_{\imath_\nu+1} & \hspace{-5mm} \color{colorR} R_{\imath_\nu+3} & \color{colorR} \cdots & \color{colorR} R_{\jmath_\nu}
\hspace{-2mm} & \hspace{-3mm} \color{colorT} T_{\jmath_\nu} & \color{colorT} T_{\jmath_\nu+1} & \color{colorT}
T_{\jmath_\nu+2}
\\[5mm]
\color{colorH} H^{\jmath_{\nu-1}} & & -\Lambda_{\jmath_{\nu-1}+1}^+ \chi_{\jmath_{\nu-1}+1}^{w_{\jmath_{\nu-1}+1}' -1} & \hspace{-9mm} 0 & \cdots & 0 \hspace{-7mm} & 0 & & & & & &
\\[5mm]
\color{colorG} G^{\jmath_{\nu-1} + 1} & & \chi_{\jmath_{\nu-1}} & \hspace{-9mm} - \Lambda_{\jmath_{\nu-1}+2}^+ \chi_{\jmath_{\nu-1}+2}^{w_{\jmath_{\nu-1}+2}' -1} & \cdots & 0 \hspace{-7mm} & 0 & & & & & &
\\[5mm]
& & 0 & \hspace{-9mm} \chi_{\jmath_{\nu-1}+1} & \ddots & \vdots  \hspace{-7mm}& \vdots & & & & & &
\\[5mm]
& & \vdots & \hspace{-9mm} \vdots & \ddots & -\Lambda_{\imath_{\nu}-1}^+ \chi_{\imath_{\nu}-1}^{w_{\imath_{\nu}-1}' -1} \hspace{-7mm} & 0 & & & & & &
\\[5mm]
\color{colorG} G^{\imath_{\nu} - 1} & & 0 & \hspace{-9mm} 0 & \cdots & \chi_{\imath_{\nu}-2} \hspace{-7mm} & -\Lambda_{\imath_\nu}^+ \chi_{\imath_\nu}^{w_{\imath_\nu}' -1} & & & & & &
\\[5mm]
\color{colorF} F^{\imath_\nu} & & & & & & \chi_{\imath_\nu-1} & & & & & &
\\[5mm]
\color{colorG} \left(G^{\imath_{\nu} + 2}\right)^* & & & & & & -\Lambda_{\imath_\nu+1}^- \chi_{\imath_\nu+1}^{-w_{\imath_\nu+1}'} & \hspace{-5mm} \chi_{\imath_{\nu}} & \cdots & 0 \hspace{-2mm} & 0 & &
\\[5mm]
& & & & & & 0 & \hspace{-5mm} -\Lambda_{\imath_{\nu}+2}^- \chi_{\imath_{\nu}+2}^{-w_{\imath_{\nu}+2}'} & \ddots & \vdots  \hspace{-2mm}& \vdots & &
\\[5mm]
& & & & & & \vdots & \hspace{-5mm} \vdots & \ddots & \chi_{\jmath_\nu-3 \hspace{-2mm}} & 0 & &
\\[5mm]
\color{colorG} \left(G^{\jmath_\nu}\right)^* & & & & & & 0 & \hspace{-5mm} 0 & \cdots & -\Lambda_{\jmath_\nu-1}^- \chi_{\jmath_\nu-1}^{-w_{\jmath_\nu-1}'} \hspace{-2mm} & \chi_{\jmath_\nu-2} & &
\\[5mm]
\color{colorH} H^{\jmath_\nu} & & & & & & 0 & \hspace{-5mm} 0 & \cdots & 0 \hspace{-2mm} & -\Lambda_{\jmath_\nu}^- \chi_{\jmath_\nu}^{-w_{\jmath_\nu}'} & \hspace{5mm}\chi_{\jmath_\nu-1}\hspace{3mm}  &  -\Lambda_{\jmath_\nu+1}^+ \chi_{\jmath_\nu+1}^{w_{\jmath_\nu+1}'-1}
\\
};

\draw[color=blue] (MatrixPhi-2-3.north west) rectangle (MatrixPhi-6-7.south east);
\draw[color=red] (MatrixPhi-8-7.north west) rectangle (MatrixPhi-12-11.south east);

\draw[blue] ($0.5*(MatrixPhi-12-13.north west)+0.5*(MatrixPhi-12-13.north west)$) -- ($0.5*(MatrixPhi-12-13.south west)+0.5*(MatrixPhi-12-13.south west)$);
\draw[blue] ($0.5*(MatrixPhi-12-13.north west)+0.5*(MatrixPhi-12-13.north west)$) -- ($0.5*(MatrixPhi-12-13.north east)+0.5*(MatrixPhi-12-13.north east)$);

\node[fit=($0.4*(MatrixPhi-11-12.north west)+0.6*(MatrixPhi-12-13.north west)$)]{$\hspace{10mm} \color{blue} \varphi_{\nu+1}^+$};


\node[fit=(MatrixPhi-4-7)(MatrixPhi-4-9)]{$\color{blue} \varphi_{\nu}^+$};
\node[fit=(MatrixPhi-10-6)(MatrixPhi-10-6)]{\hspace{0mm}$\color{red} \varphi_{\nu}^-$};

\node[fit=(MatrixPhi-4-1)(MatrixPhi-5-1)]{$\color{colorG} \vdots$};
\node[fit=(MatrixPhi-9-1)(MatrixPhi-10-1)]{$\color{colorG} \vdots$};

\end{tikzpicture}
$}
\caption{Submatrix $\varphi\left[H^{\jmath_{\nu-1}}:H^{\jmath_\nu};T_{\jmath_{\nu-1}+2}:T_{\jmath_\nu+2}\right]$ of $\varphi$}
\label{fig:Step4AfterDeleting}
\end{figure}

To get an exact sequence containing $\varphi=\varphi_{{5}}$, we also have to erase columns $G_{\imath_\nu}^*$ and $G_{\imath_\nu+1}^*$ in $\pi_{{4}}\in A^{\tau\times\left(3\tau+2\xi\right)}$, corresponding to the removed rows in $\varphi_{{4}}$. As a result, we get the reduced matrix $\pi_{{5}}:=\pi_{{5}}\left(w,\lambda,1\right)
\in A^{\tau\times 3\tau}$ given by
$$
\pi_{{5}}
:=
\left(
\arraycolsep=2pt\def\arraystretch{1}
\begin{array}{c|c|c|c|c|c|c|c|c|c|c|c|c|c|c|c|c|c|c}
\color{colorG} G_1 & \color{colorG} \cdots & \color{colorG} G_{\imath_1-1} & \color{colorF} F_{\imath_1} & \color{colorG} G_{\imath_1+2}^* & \color{colorG} \cdots & \color{colorG} G_{\jmath_1}^* & \color{colorH} H_{\jmath_1} & \color{colorG} G_{\jmath_1+1} & \color{colorG} \cdots & \color{colorG} G_{\imath_\xi-1} & \color{colorF} F_{\imath_\xi} & \color{colorG} G_{\imath_\xi+2}^* & \color{colorG} \cdots & \color{colorG} G_{\jmath_\xi}^* & \color{colorH} H_{\jmath_\xi} & \color{colorG} G_{\jmath_\xi+1} & \color{colorG} \cdots & \color{colorG} G_{3\tau}
\end{array}
\right)_{\tau\times 3\tau}
$$
Consequently, setting $M_{{5}}:=M_{{4}}$, we get the following (part of) free resolution of $\left(M_{{5}}\right)_S$ by Lemma \ref{lem:MatrixReductionLemma}:
$$
\begin{tikzcd}[column sep = 20pt] 
S^{3\tau} \arrow[rr, "\varphi = \varphi_{{5}}"] & & S^{3\tau} \arrow[r, "\pi_{{5}}"] & \left(M_{{5}}\right)_S \arrow[r] & 0.
\end{tikzcd}
$$

To finish the proof, we recall in Proposition \ref{prop:DetAndAdjOfPhiS}.(1) that $\det\varphi$ is
a unit in $S$ when $\left(w,\lambda,1\right)$ is nondegenerate. This implies that the map $\varphi
: S^{3\tau}\rightarrow S^{3\tau}$ is injective. Indeed, if $\varphi u = 0$ for some $u \in S^{3\tau}$, we multiply
$\operatorname{adj}\varphi$ to both sides, then  we have $\operatorname{adj}\varphi\cdot \varphi
u = \left(\det\varphi\right)u = 0$ by equation \ref{eqn:AdjointEquation}, yielding $u=0$ when $\det\varphi$
is a unit. Therefore, by Theorem \ref{thm:Eisenbud}, we know that $M_{{5}}$ is isomorphic to $\operatorname{coker}\mathunderbar{\varphi}$ as an $A$-module and it is actually a maximal Cohen-Macaulay $A$-module. But we also have
$M_{{5}} = M_{{4}} = M_{{3}} = M_{{2}}$ and $M = \left(M_{{2}}\right)^\dagger$ by construction. Taking these together, we conclude that $M = M_{{5}}$ and thus establish the free resolution of $M_S$ with $\pi:=\pi\left(w,\lambda,1\right)
:= \pi_{{5}} \in
A^{\tau\times 3\tau}$:
$$
\begin{tikzcd}[column sep = 20pt] 
  0 \arrow[r] & S^{3\tau} \arrow[r, "\varphi"] & S^{3\tau} \arrow[r, "\pi"] & M_S \arrow[r] & 0.
\end{tikzcd}
$$
\end{proof}


\section{Matrix Factorizations Arising from Lagrangians}\label{sec:MfFromLag}

In this section, we prove a general version of Proposition \ref{prop:LagToMFRankOne}.

Let $\mathcal{P}$ be the pair of pants and $\mathbb{L}$ the Seidel Lagrangian in $\mathcal{P}$. For an immersed loop $L$ in $\mathcal{P}$, we need the concept of `\emph{regularity}'
to define its mirror image $\LocalF\left(L\right) = \left(\LocalPhi\left(L\right),\LocalPsi\left(L\right)\right)\in
\operatorname{MF}\left(xyz\right)$ under the localized mirror functor $\LocalF$.

\begin{defn}\label{def:regular}
        Consider $L$ and $\mathbb{L}$ as immersed loops $L$, $\mathbb{L} : S^1 \rightarrow \mathcal{P}$.
        \begin{enumerate}
                \item $L$ is said to bound an immersed `\emph{fish-tale}' if there is an immersion $i:D^2\rightarrow\mathcal{P}$ which satisfies $i\left(e^{2\pi it}\right)=L\left( \imath (t) \right)$ for some immersion $\imath:[0,1]\rightarrow S^1$.
                \item $L$ and $\mathbb{L}$ are said to bound an immersed `\emph{cylinder}' if there is a continuous map $j:S^1\times[0,1]\rightarrow\mathcal{P}$ which is an immersion on $S^1\times(0,1)$ and satisfies $j\left(\left(e^{2\pi it},0\right)\right) = L(\imath(t))$ and $j\left(\left(e^{2\pi it},1\right)\right) = \mathbb{L}(\jmath(t))$ for some immersions $\imath$, $\jmath : S^1\rightarrow S^1$.
                \item $L$ is called \emph{unobstructed} if it satisfies the following conditions:
                \begin{itemize}
                        \item $L$ does not have a triple intersection,
                        \item $L$ and $\mathbb{L}$ meet transversally,
                        \item $L$ does not pass through a self-intersection of $\mathbb{L}$, and
                        \item $L$ does not bound any immersed disks or fish-tales.
                \end{itemize}
                \item $L$ is called \emph{regular} if it is unobstructed and satisfies the following additional condition:
                \begin{itemize}
                        \item $L$ and $\mathbb{L}$ do not bound any immersed cylinders.
                \end{itemize}
        \end{enumerate}
\end{defn}

The following lemma and theorem justify that we only care about regular loops.

\begin{lemma}
        Every free homotopy class of loops in $\mathcal{P}$ which is not null-homotopic has its regular representatives.
\end{lemma}
\begin{proof}
For the classes $\alpha^n, \beta^n$, and $\gamma^n$ in $\pi_1(\cpp)$, one can easily construct regular loops around each punctures. For other classes, by Proposition \ref{prop:normalnormal}, it is sufficient to consider normal representatives. We will soon construct an explicit regular loop $L\left(w',\lambda'\right)$ for each normal loop word $w'$. See Definition \ref{defn:RepresentingLoops} and Lemma \ref{lem:Regular}.
\end{proof}

\begin{thm}\label{thm:T=1Homotopy}
        Any regular immersed loop $L$ with holonomy has its mirror matrix factorization $\LocalF (L) = \left(\LocalPhi\left(L\right),\LocalPsi\left(L\right)\right)\in\operatorname{MF}(xyz)$
        over $\mathbb{C}$.
        The homotopy class of $\LocalF\left(L\right)$ depends only on the free homotopy class of $L$ and the total holonomy on $L$.
\end{thm}

\begin{proof}
We give the proof in Appendix \ref{sec:T=1}. See Propositions \ref{prop:T=1} and \ref{prop:Homotopy} for the first and second statement, respectively. See  remark \ref{rmk:Regularity} for the reason why
we should consider only regular loops.
\end{proof}

Thus $\LocalF$ takes any element in $\left[S^1,\mathcal{P}\right]$ with a holonomy to a homotopy class in $\operatorname{MF}\left(xyz\right)$.
But note that $\mathcal{P}$ has $3$ boundary components each of which can be represented by an embedded loop, and
their mirror images are the zero object in $\operatorname{MF}\left(xyz\right)$ because they don't meet
$\mathbb{L}$. Similarly, any loops that just wind one of the boundary components have zero mirror image
and therefore their free homotopy classes also have null-homotopic mirror images.
This suggests that we exclude those kinds of loops.

\begin{defn}\label{defn:RepresentingLoops}
        Consider $L$ as an immersed loop $L:S^1\rightarrow \mathcal{P}$.
        \begin{enumerate} 
                \item For $k \in \mathbb{Z}$, a ($k$-)\emph{multiple} of $L$ is defined by the immersed loop $L^k:S^1\rightarrow \mathcal{P}$, $e^{2\pi it}\mapsto L\left(e^{2k\pi i t} \right)$.
                \item $L$ is called \emph{essential} if it is not freely homotopic to a multiple of a
                boundary component of $\mathcal{P}$.
                \item If $L$ is essential, we say the same for the free homotopy class $[L]$ of $L$. This is well-defined because whether a loop is essential or not is an invariant of its free homotopy class, obviously.
        \end{enumerate}
\end{defn}

\begin{example}\label{example:MultipleOfSeidelLagerangian}
        A non-regular loop is homotopic to a multiple of the Seidel Lagrangian $\bll$. Let $L$ be a non-regular loop with a cyclically reduced homotopy class $\alpha^{a_1}\beta^{b_1}\cdots\alpha^{a_\nu}\beta^{\nu}$. Since an immersed cylinder gives a homotopy between a multiple of $L$ and a multiple of the Seidel Lagrangian $\bll$, permuting cyclically if it is necessary, we have $$(\alpha^{a_1}\beta^{b_1}\cdots\alpha^{a_\nu}\beta^{b_\nu})^n =  (\alpha\beta\alpha^{-1}\beta^{-1})^m$$ for some integers $n, m$.  One could conclude that the homotopy class of $L$ is $(\alpha\beta\alpha^{-1}\beta^{-1})^k$ for some $k$  by comparing exponents, which is followed by $L$ is freely homotopic to a multiple of $\bll$.    
\end{example} 

From now on we will only discuss essential loops. To compute their mirror images, it is enough to consider only one regular representative in each of their free homotopy class. Recall in Proposition \ref{prop:normalnormal} that the free homotopy classes of essential loops are indeed represented by essential loop words, so they can be uniquely represented by the normal loop words up to shifting.


Let us define special types of Lagrangians analogously as in Section \ref{sec:MFFromRank1Lagrangian}. Let $\tau$ be a positive integer and put $2\tau$ points on each segment $l_x, l_y$ ,and $l_z$. We call $u_1, u_2, \cdots, u_{\tau}, p_{\tau}, \cdots, p_2, p_1$ the points on $l_x$ in order. Similarly, call $s_1, s_2, \cdots, s_\tau, q_\tau, \cdots, q_2, q_1$ the points on $l_y$, and $t_1, t_2, \cdots, t_\tau, r_\tau, \cdots, r_2, r_1$ on $l_z$. Then define a path $\delta_{\square}$ from the base point $g$ to the point $\square\in\{p_i, q_i, r_i, s_i, t_i, u_i\}$ as in Section \ref{sec:MFFromRank1Lagrangian}. Also, define a path $\delta_{x, i}^\nu$, for each $i=1, \cdots, \tau$ and $\nu\in\bzz$, from $u_{i-1}$ to $p_i$, whose interior is contained in $D_x$ and the homotopy class of the concatenated loop $[\delta_{u_{i-1}}\cdot \delta^{\nu}_{x, i}\cdot \overline{\delta_{p_i}}]\in\pi_1(\cpp, g)$ represents $\alpha^\nu$. Define $\delta^\nu_{y, i}$ from $s_i$ to $q_i$ and $\delta^\nu_{z, i}$ from $t_i$ to $r_i$ in the same way. Also let $\Lambda_{\square_1\square_2}$ be a path in $\overline{A}$ from $\square_1$ to $\square_2$ whose interior contained in $A$, where $\square_1, \square_2 \in \{p_i, q_i, r_i, s_i, t_i, u_i : i=1, \cdots, \tau\}$. Now define the Lagrangian $L((l'_1, m'_1, n'_1, \cdots, l'_\tau, m'_\tau, n'_\tau), \lambda')$.

\begin{defn}
        Let $\tau$ be a positive integer and $w' = (l'_1, m'_1, n'_1, \cdots, l'_\tau, m'_\tau, n'_\tau)$ be a loop word. Then we define the closed immersed Lagrangian curve $L(w', \lambda') = L((l'_1, m'_1, n'_1, \cdots, l'_\tau, m'_\tau, n'_\tau), \lambda')$ to be a smoothing of the loop $$\delta^{l'_1}_{x, 1}\cdot \Delta_{p_1, s_1} \cdot \delta^{m'_1}_{y, 1} \cdot \Delta_{q_1, t_1} \cdot \delta^{n'_1}_{z, 1} \cdot \Delta_{r_1, u_1} \cdot \cdots \cdot \delta^{l'_\tau}_{x, \tau}\cdot \Delta_{p_\tau, s_\tau} \cdot \delta^{m'_\tau}_{y, \tau} \cdot \Delta_{q_\tau, t_\tau} \cdot \delta^{n'_\tau}_{z, 1} \cdot \Delta_{r_\tau, u_\tau}$$ whose holonomy $\lambda'$ is concentrated at a point in $\delta^{m'_1}_{x, 1}$.
\end{defn}

\begin{remark}
        \begin{enumerate}\label{rem:ImmersedLagrangians}
                \item Since each paths are not defined very concretely, the Lagrangian may seem ambiguous also. However, the mirror image is only depends on the intersection patter with the Seidel Lagrangian as long as it is regular. To guarantee regularity, there should be no triple intersections or non-transversal intersections of paths $\delta, \Delta$. Also, each path $\delta, \Delta$ should have no self-intersection. These can be generically achieved. \label{rem:UnobstructedLagrangian}
                \item\label{rem:DegenerateLagrangian} As in Remark \ref{rem:DegenerateLoopLength1}, the loop word $w' = (2, 2, \cdots, 2)$ gives a non-regular Lagrangian. So we have to perturb it to make it regular illustrated as in Figure \ref{fig:Perturbed222} for $\tau=1$ and Figure \ref{fig:StronglyAdmissibleL(2)} for $\tau\geq2$. Since the perturbed Lagrangian is strongly admissible, it is regular(c.f. \ref{ex:DegenerateLagrnagianStronglyAdmissible}) and so one can compute its mirror matrix factorization. As in Remark \ref{rem:DegenerateLoopLength1}, it also turns out that the resulting matrix is homotopy equivalent to the canonical form. 
        \end{enumerate}
\end{remark}

\begin{lemma}\label{lem:Regular}
For a normal loop word $w'$ that is not of type $(2,2,\cdots,2)$ and any holonomy $\lambda'$, the corresponding loop $L\left(w',\lambda'\right)$ is a regular representative in the free homotopy class $L\left[w'\right]\in\left[S^1,\mathcal{P}\right]$ associated to $w'$.
For the case of $w'=(2,2,\cdots,2)$, we can perturb  the loop $L\left(w',\lambda'\right)$ to be the regular one.
\end{lemma}
\begin{proof}
By construction, the free homotopy class of $L\left(w',\lambda'\right)$ is $L\left[w'\right]$.
The first three conditions of unobstructedness of  $L\left(w',\lambda'\right)$ is in Remark \ref{rem:ImmersedLagrangians}.(\ref{rem:UnobstructedLagrangian}).

For the fourth condition, we have to prove that $L$ doesn't bound immersed discs or fish-tales. Since $L$ is not null-homotopic, we only have to prove for the fish-tales. Let $i : D^2 \rightarrow \cpp$ be a fish-tale with a boundary condition $\imath : [0, 1]\rightarrow S^1$. By Remark \ref{rem:ImmersedLagrangians}.(\ref{rem:UnobstructedLagrangian}) again, the fish-tale should not be totally contained a path $\delta$ or $\Delta$. This implies that the loop word $w'$ has a subword which can be reduced which contradicts to minimality of normal words, see Remark \ref{rem:MinimalityOfNormalLoopWord}.


Finally, by Example \ref{example:MultipleOfSeidelLagerangian}, for $w' \neq (2, 2, \cdots, 2)$, $L\left(w',\lambda'\right)$ and Seidel Lagrangian cannot bound a cylinder. The last statement will be explained in Remark \ref{rem:ImmersedLagrangians}.(\ref{rem:DegenerateLagrangian}).
\end{proof}

Let us compute $\cff^\bll(L)$ for these Lagrangians $L$. The matrix component of $\cff^\bll(L((l'_1, \cdots, n'_\tau), \lambda'))$ is $3\tau \times 3\tau$-matrix since $L\cap \bll$ consists of $6\tau$ points. Denote by $(\square_i, \triangle_j)$ the $(\square_i, \triangle_j)$-entry of $\cff^\bll(L((l'_1, \cdots, n'_\tau), \lambda'))$, for $\square\in\{p, q, r\}, \triangle\in \{s, t, u\}$. Then we can compute the following.

\begin{prop}\label{prop:DiscCounting}
        Let $w' = (l'_1, m'_1, n'_1, \cdots, l'_{\tau}, m'_{\tau}, n'_{\tau})$ be a normal loop word. Then we have, up to sign and holonomy, 
        \begin{align*}
                (p_1, s_j) &= z &\text{ if } \quad j=1\\
                (p_1, t_j) &= y^{-m'_1} & \text{ if } \quad j=1\\
                (p_1, u_j) &= x^{l'_1-1} & \text{ if } \quad j=\tau,
        \end{align*}and $0$ otherwise. We are using the notation that $x^a, y^a$ and $z^a$ is considered as zero if $a<0$.
\end{prop}
\begin{remark}
The entry $(p_1, s_1)$ comes from the small triangle, similar to the $\tau=1$ case that is  illustrated  in Figure \ref{fig:L322} (a).
The entry $(p_1,t_1)$ is zero if $m_1' >0$ and this was illustrated in Figure \ref{fig:L322} (b) where the strip could not extend beyond the puncture.
If $m_1'<0$, the immersed curve will wind around the puncture in the opposite direction, and we can find an immersed polygon now.
 The entry $(p_1,j_j)$ for $j=\tau$ comes from the counting of the polygon similar to the one  described in Figure \ref{fig:L32222}.        
\end{remark}

\begin{proof}
        Denote by $L$ the Lagrangian $L(w', \lambda')$. Let $D^2$ be the unit disc in the complex plane and let $\eta^+ : [0, 1]\rightarrow D^2$ be a path defined as $\eta^+(t) = e^{\sqrt{-1}\pi (1-t)}$ and $\eta^-$ be a path defined as $\eta^-(t) = e^{\sqrt{-1}\pi (1+t)}$. Also let $z_1, \cdots, z_r$ be marking points on $\eta^-$ in counterclockwise order.
        
        Now let $\phi : D^2\rightarrow \cpp$ be a polygon such that $\phi(-1) = p_1$, $\phi(1) = s_i, t_i, $or $u_i$, $\phi\circ\eta^-\subseteq \bll$, $\phi\circ\eta^+\subseteq L$, and $\phi(z_i)$ is one of $X, Y, Z$. Then we have two paths $\phi\circ\eta^+$ and $\phi\circ\eta^-$ from $p_1$ to $\phi(1)$ which are path homotopic to each other. Hence two loops $\delta_{p_1}\cdot (\phi\circ\eta^+) \cdot \overline{\delta_{\phi(1)}}$ and $\delta_{p_1}\cdot (\phi\circ\eta^-) \cdot \overline{\delta_{\phi(1)}}$ should have the same homotopy type.
        
        Note that there are two cases : the orientation of $\bll$ and $\phi\circ\eta^-$ are the same or opposite. First consider the case that two orientations are opposite. In this case, the homotopy type of $\delta_{p_1}\cdot \phi\circ\eta^- \cdot \overline{\delta_{\phi(1)}}$ is a former part of
        $\alpha^{-l'_1}\gamma^{-n'_\tau}\beta^{-m'_\tau}\alpha^{-l'_\tau}\cdots\gamma^{-n'_1}\beta^{-m'_1}\alpha^{-n'_1}\cdots$, that is, $e, \alpha^{-l'_1}, \alpha^{-l'_1}\gamma^{-n'_\tau}, \cdots$.
        
        Now find the possible homotopy type of $\delta_{p_1}\cdot \phi\circ\eta^- \cdot \overline{\delta_{\phi(1)}}$. The path $\phi\circ\eta^-$ can be identified with a sequence $(\ell_i)$ of $l_x, l_y, l_z, \tilde{l_x}, \tilde{l_y}, \tilde{l_z}$ satisfying the following rule : 
        \begin{itemize}
                \item $\ell_1=l_x$.
                \item For each $i$, $(\ell_i, \ell_{i+1})$ is one of
                $$(\tilde{l_x}, l_z), (\tilde{l_x}, \tilde{l_z}), (\tilde{l_y}, l_x), (\tilde{l_y}, \tilde{l_x}), (\tilde{l_z}, l_y), (\tilde{l_z}, \tilde{l_y}).$$
                \item There is no $i$ such that $(\ell_i, \ell_{i+1}, \ell_{i+2})$ is one of $$(\tilde{l_x}, \tilde{l_z}, \tilde{l_y}), (\tilde{l_y}, \tilde{l_x}, \tilde{l_z}), (\tilde{l_z}, \tilde{l_y}, \tilde{l_x}).$$
        \end{itemize}
        Then the sequence is of the following form : 
        $$l_x(\tilde{l_z}\tilde{l_y}l_x)^{a_1}\tilde{l_z}l_y(\tilde{l_x}\tilde{l_z}l_y)^{b_1}\tilde{l_x}l_z(\tilde{l_y}\tilde{l_x}l_z)^{c_1}\tilde{l_y}l_x(\tilde{l_z}\tilde{l_y}l_x)^{a_2}\tilde{l_z}l_y(\tilde{l_x}\tilde{l_z}l_y)^{b_2}\tilde{l_x}l_z(\tilde{l_y}\tilde{l_x}l_z)^{c_2}\cdots,$$ where the sequence ends at $l_x, l_y$ or $l_z$ and $a_i, b_i, c_i$ are nonnegative integers. The homotopy class of the loop $\delta_{p_1}\cdot \phi\circ\eta^- \cdot \overline{\delta_{\phi(1)}}$ is then given by $$\alpha^{-a_1-1}\beta^{-b_1-1}\gamma^{-c_1-1}\alpha^{-a_2-1}\beta^{-b_2-1}\gamma^{-c_2-1}\cdots,$$ where the word breaks at $\alpha^{-a_i-1}, \beta^{-b_i-1}, \gamma^{-c_i-1}$ if the sequence ends at $l_x, l_y, l_z$, respectively. Note that the word is already reduced since $\gamma^{-c_i-1}=(\alpha\beta)^{c_i+1}$. Thus, by comparing two classes and using the normal condition,
        $$\alpha^{-l'_1}\gamma^{-n'_\tau}\beta^{-m'_\tau}\alpha^{-l'_\tau}\cdots\gamma^{-n'_1}\beta^{-m'_1}\alpha^{-n'_1}\cdots, \alpha^{-a_1-1}\beta^{-b_1-1}\gamma^{-c_1-1}\alpha^{-a_2-1}\beta^{-b_2-1}\gamma^{-c_2-1}\cdots,$$ 
        one could conclude that they are the same only when $\phi(1) = u_\tau$, $l'_\tau>0$ and the path $\phi\circ\eta^- $ corresponds to $l_x(\tilde{l_z}\tilde{l_y}l_x)^{l'_1-1}$.
         Hence, since the variable $x$ contribute once for each $\tilde{l_z}\tilde{l_y}$, $(p_1, u_\tau) = x^{l'_1-1}$. We give an example for the disc in Figure \ref{fig:L32222}, which corresponds the path  $l_x \tilde{l_z} \tilde{l_y} l_x$. The domain is illustrated in Figure \ref{fig:bdysq}.
         
\begin{figure}[h]
\includegraphics[scale=0.6]{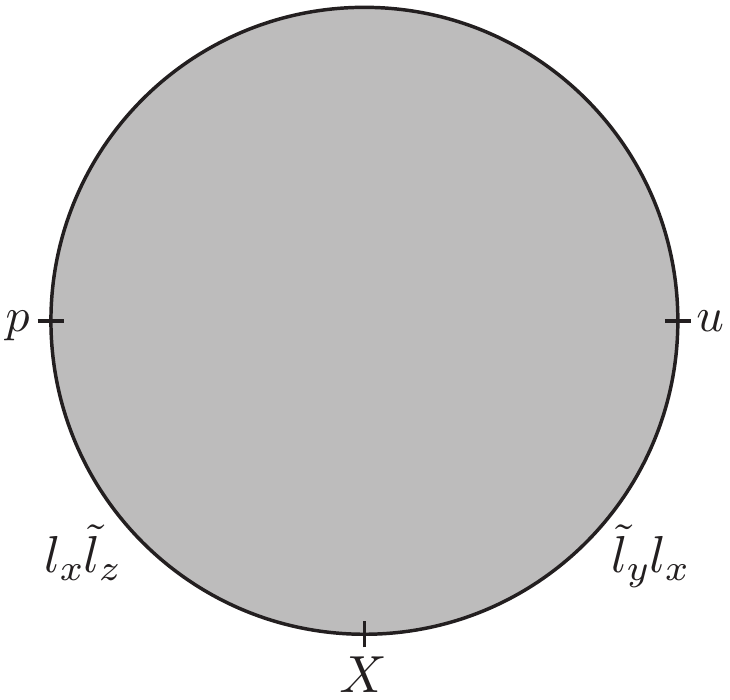}
\centering
\caption{The sequence from the disc in Figure \ref{fig:L32222}}
\label{fig:bdysq}
\end{figure}
        
        When the orientations of $\bll$ and the path $\delta_{p_1}\cdot \phi\circ\eta^- \cdot \overline{\delta_{\phi(1)}}$ are matched, one could see that there are at most two polygons in the same way. One is a triangle with vertices $p_1, Z, s_1$, which gives $(p_1, s_1) = z$. The other exists only when $m'_1\leq 0$ and it is a $(-m'_1+2)$-gon. In this case, the path $\delta_{p_1}\cdot \phi\circ\eta^- \cdot \overline{\delta_{\phi(1)}}$ is given by $l_x\tilde{l_y}l_z(l_x\tilde{l_y}l_z)^{-m'_1}$. The variable $y$ contributes once for every $l_zl_x$, thus we have $(p_1, t_1) = y^{-m'_1}$.    
\end{proof}

\begin{remark}\label{ex:nonnormal}
        The normal assumption is necessary since if the word is not normal, there can be some extra polygons. For example, consider the case where $(m'_1, n'_1, l'_1) = (-2, 0, -3)$. In this case, $(p_1, s_1)$ is not $z$ but $z+x^2y$.The full matrix is computed as follows, up to holonomy and sign,
        $$\begin{pmatrix}
                z+x^2y & 0 & x^3 \\ y^2 & x & 0 \\ 0 & 1 & y
        \end{pmatrix}.$$
        Also, if we consider the holonomy and the sign, then the matrix becomes
        $$\begin{pmatrix} z-\lambda x^2y & 0 & \lambda x^3 \\ y^2 & x & 0 \\ 0 & 1 & y \end{pmatrix}.$$
        
        The normal form of the given loop word is $(-1, 1, -2)$ and its corresponding matrix is $$\begin{pmatrix} z & 0 & \lambda x \\ y & x & 1 \\ 0 & 0 & x\end{pmatrix}.$$ Two matrix factorizations are homotopy equivalent.
        \begin{figure}[h]
                \centering
                \begin{subfigure}[b]{0.45\textwidth}
                        \includegraphics[scale=0.35]{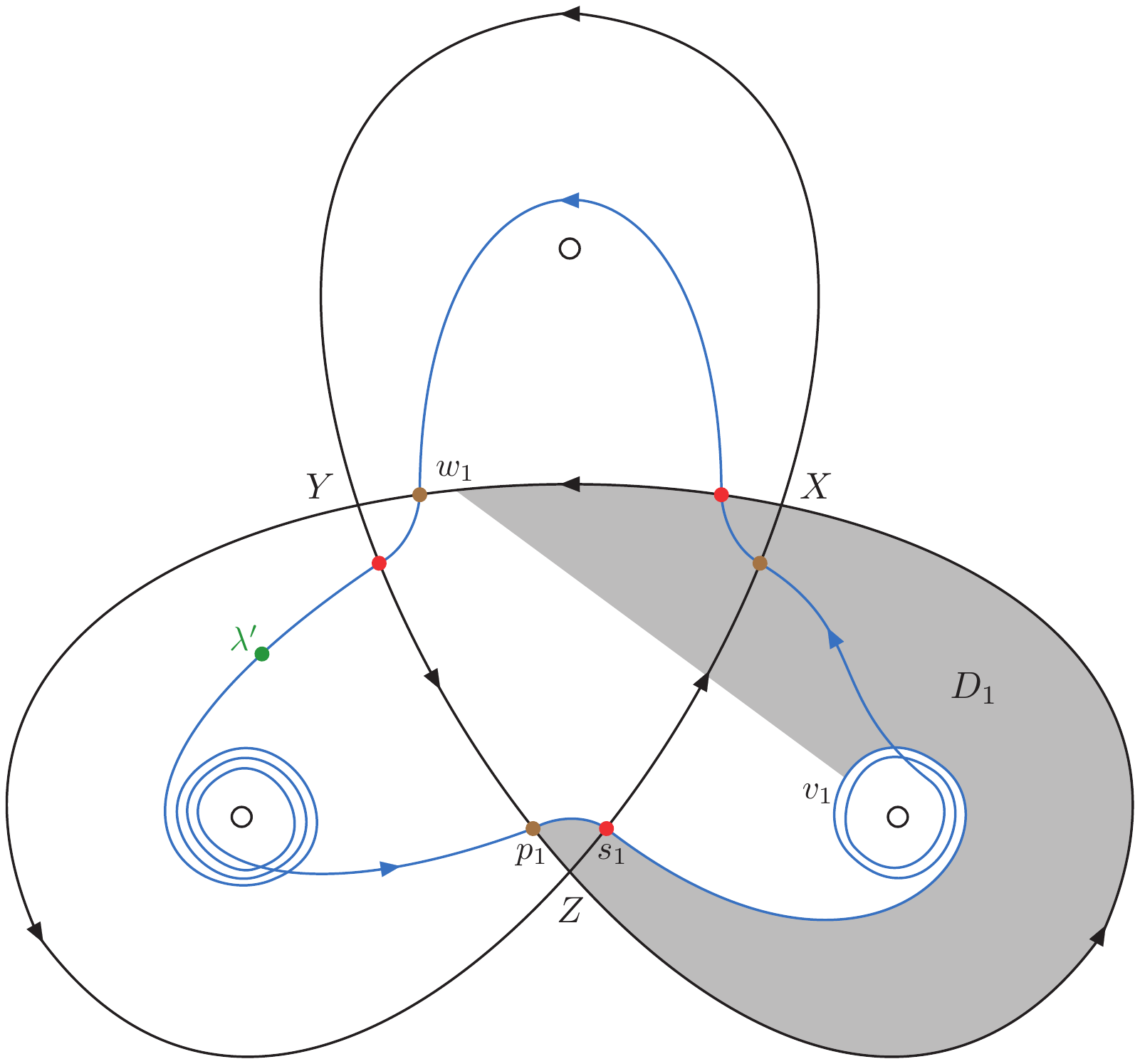}
                        \centering
                        \caption{The first piece of an extra polygon from $p_1$ to $s_1$}
                        \label{fig:extrapolygonps1}
                \end{subfigure}
                \qquad
                \begin{subfigure}[b]{0.45\textwidth}
                        \includegraphics[scale=0.35]{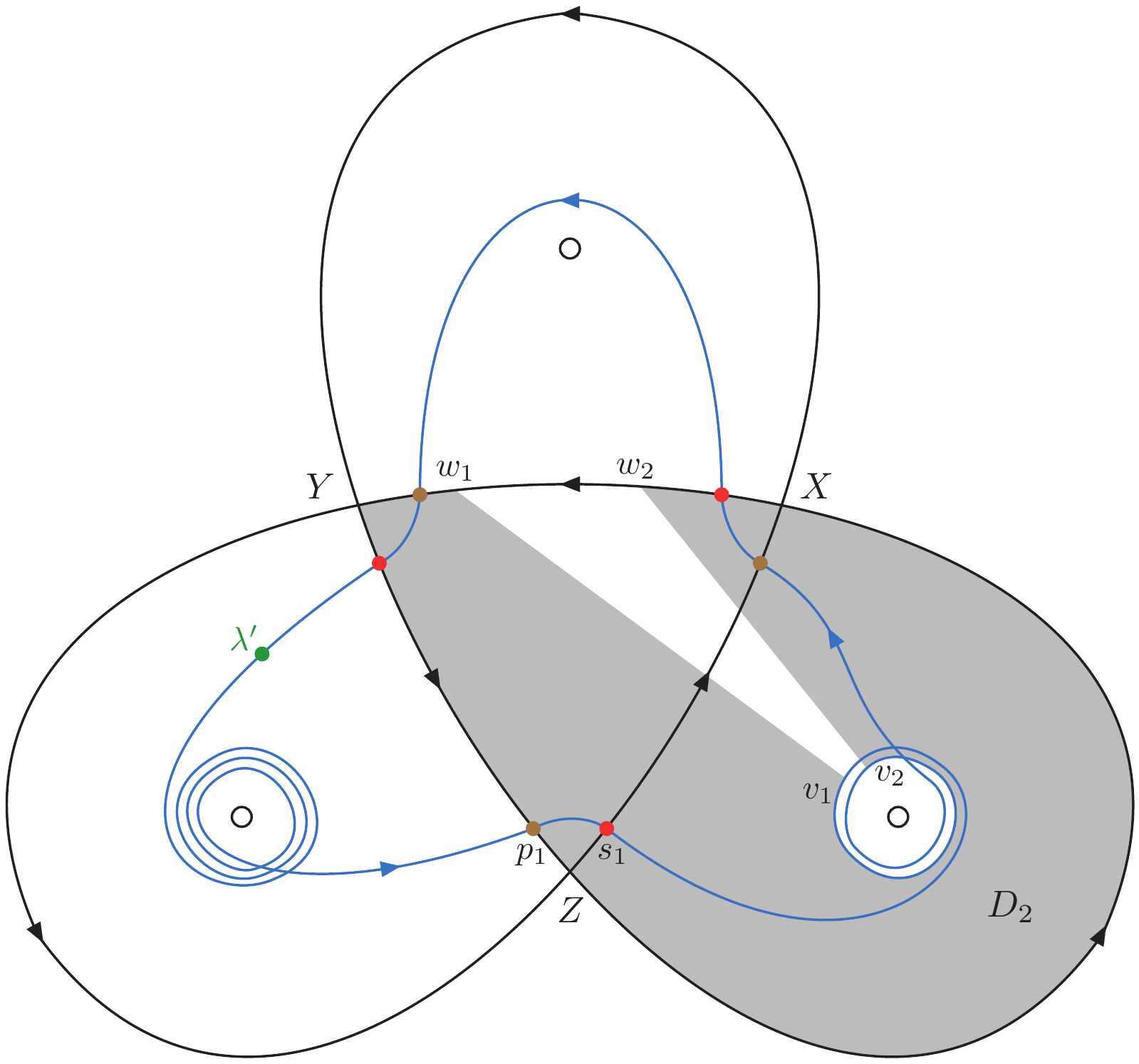}
                        \centering
                        \caption{The second piece of an extra polygon from $p_1$ to $s_1$}
                        \label{fig:extrapolygonps2}
                \end{subfigure}
                \centering
                \begin{subfigure}[b]{0.45\textwidth}
                        \includegraphics[scale=0.35]{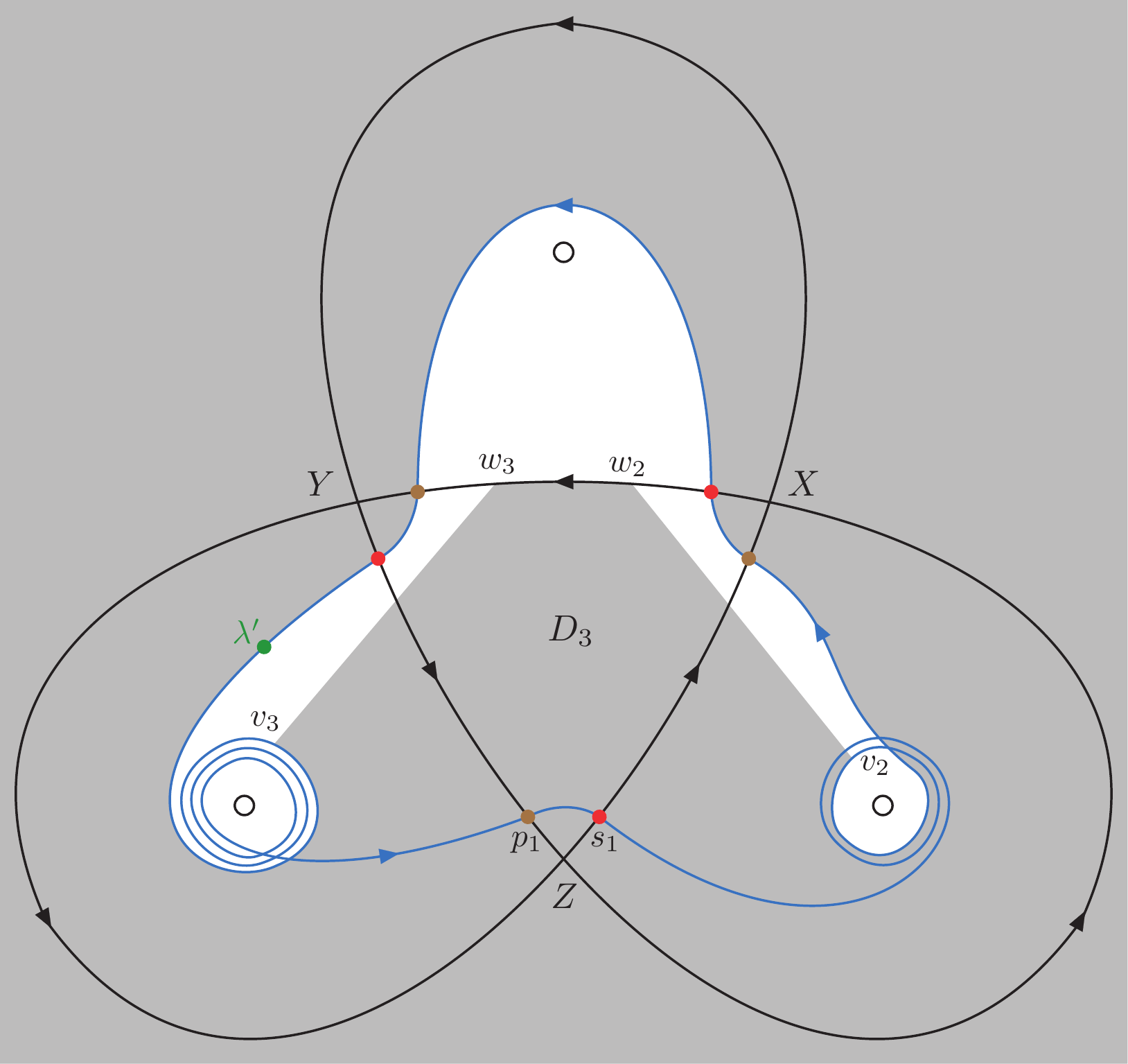}
                        \centering
                        \caption{The third piece of an extra polygon from $p_1$ to $s_1$}
                        \label{fig:extrapolygonps3}
                \end{subfigure}
                \qquad
                \begin{subfigure}[b]{0.45\textwidth}
                        \includegraphics[scale=0.35]{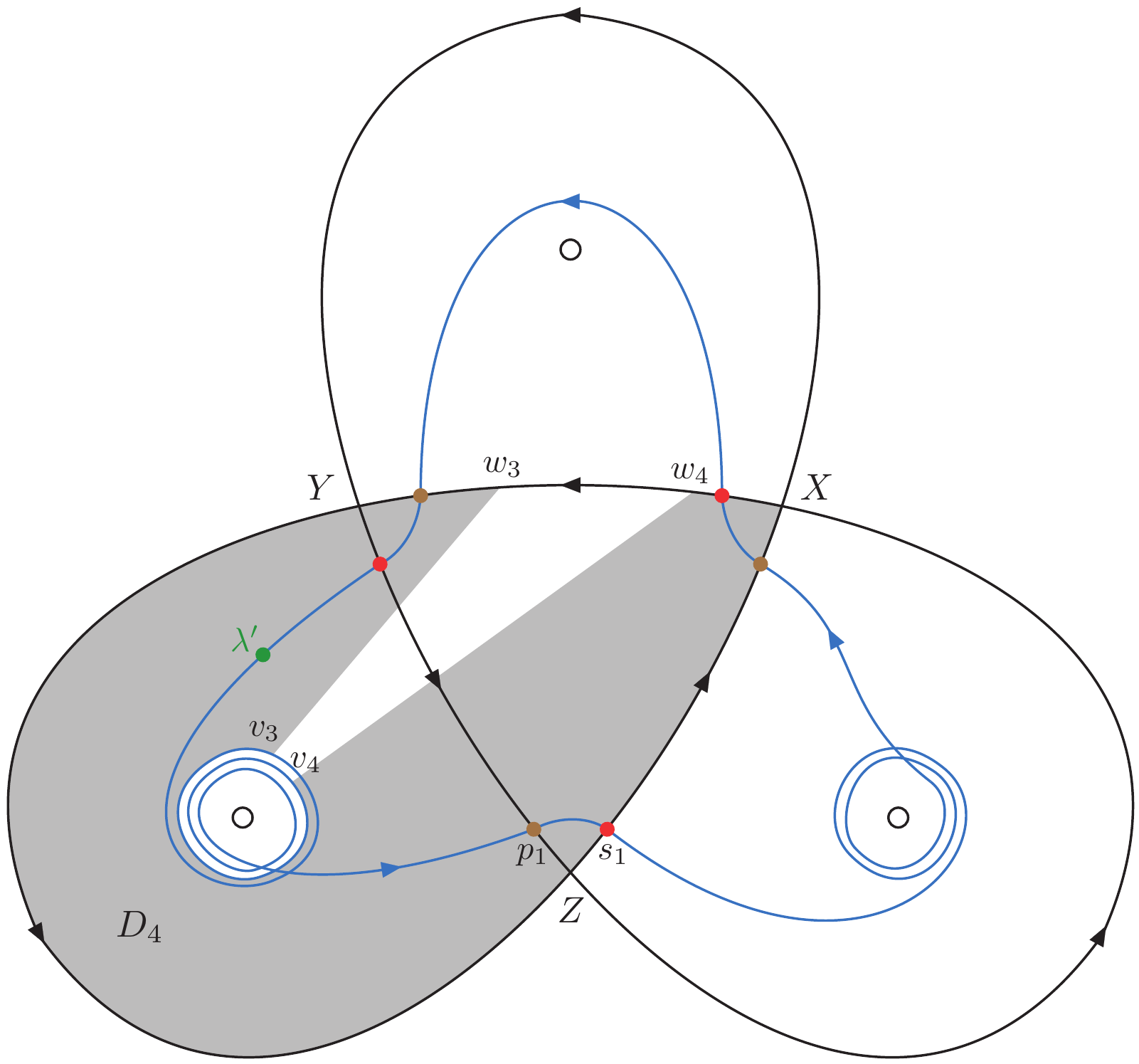}
                        \centering
                        \caption{The fourth piece of an extra polygon from $p_1$ to $s_1$}
                        \label{fig:extrapolygonps4}
                \end{subfigure}
                \centering
                \begin{subfigure}[b]{0.45\textwidth}
                        \includegraphics[scale=0.35]{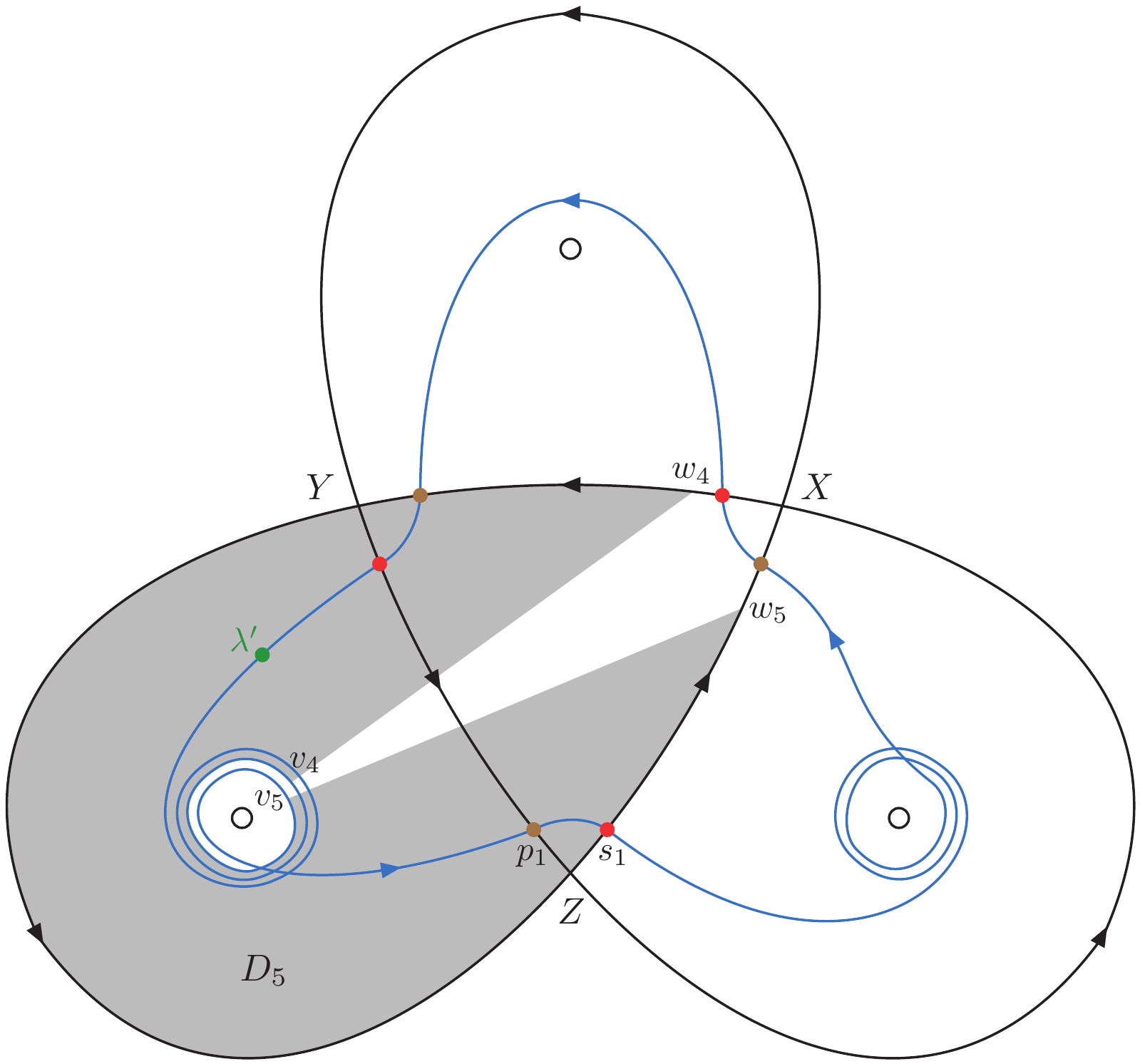}
                        \centering
                        \caption{The fifth piece of an extra polygon from $p_1$ to $s_1$}
                        \label{fig:extrapolygonps5}
                \end{subfigure}
                \qquad
                \begin{subfigure}[b]{0.45\textwidth}
                        \includegraphics[scale=0.35]{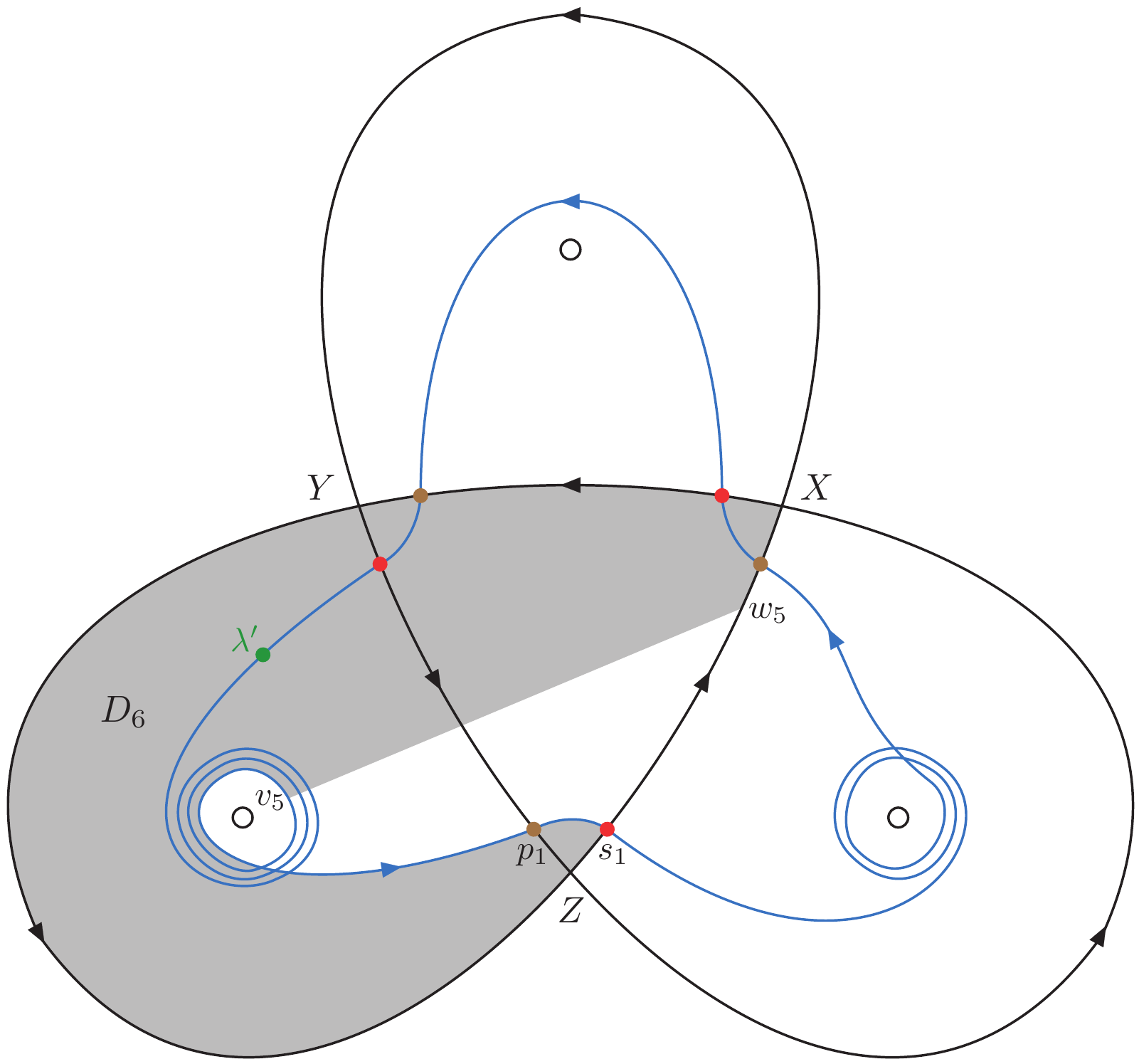}
                        \centering
                        \caption{The sixth piece of an extra polygon from $p_1$ to $s_1$}
                        \label{fig:extrapolygonps6}
                \end{subfigure}
                \centering
                \caption{Extra polygon from $p_1$ to $s_1$}
                \label{fig:nonnormalps}
        \end{figure}
        The domain of the polygon for the $x^2y$ term is divided into six pieces as in Figure \ref{fig:extrapolygondomain}.
        \begin{figure}[h]
                \centering
                \includegraphics[scale=0.4]{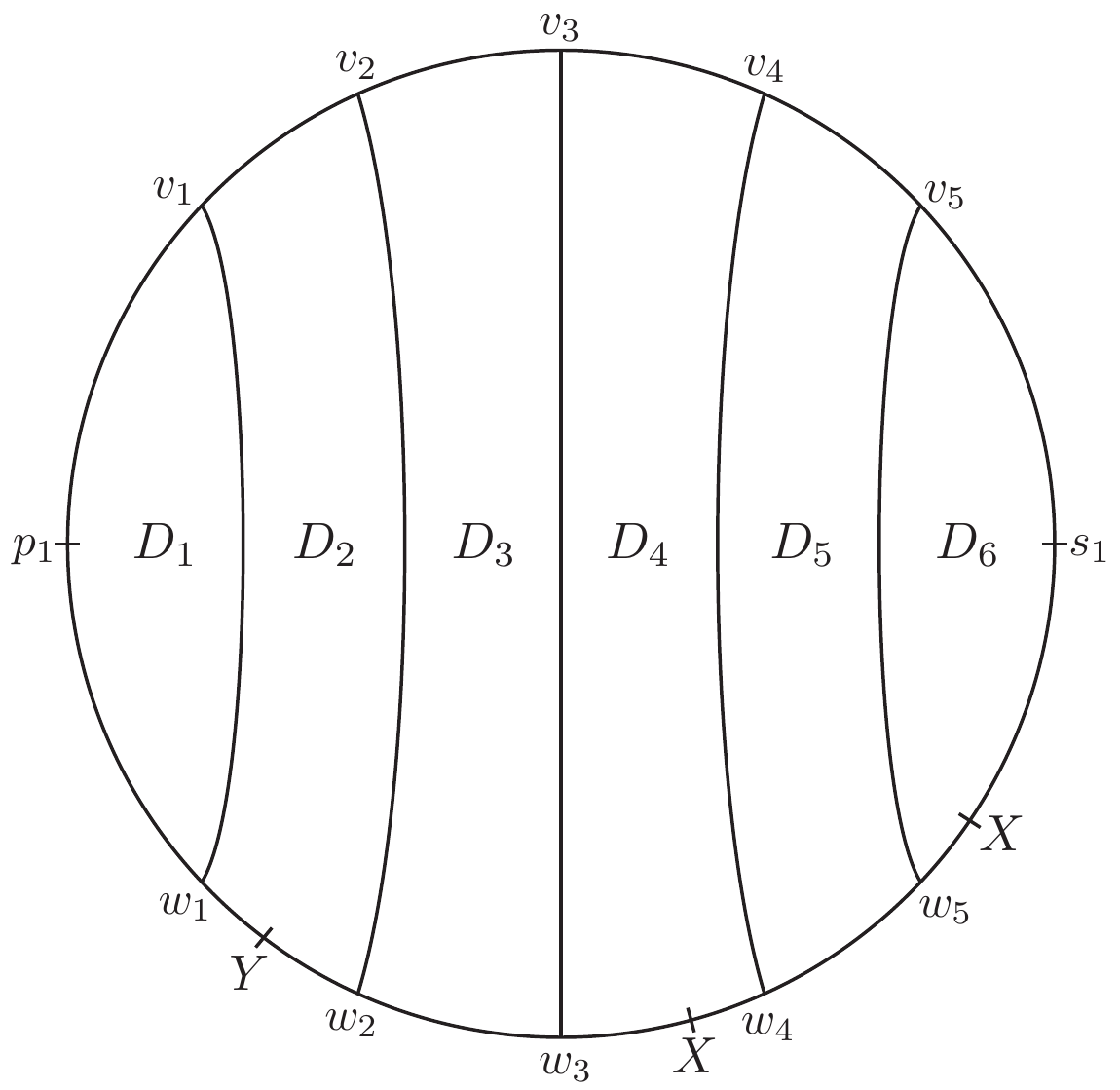}
                \caption{The domain of the extra polygon fron $p_1$ to $s_1$}
                \label{fig:extrapolygondomain}
        \end{figure}
        
\end{remark}

By the symmetry, we have the following theorem.
\begin{thm}\label{thm:MFFromLag}
        Let $w'=(l'_1, m'_1, n'_1, \cdots, l'_\tau, m'_\tau, n'_\tau)$ be a normal loop word. Then we have, up to holonomy, 
         \begin{alignat*}{3}
                &(p_i, s_i)= z, \quad
                &&(p_i, t_i) = y^{-m'_i}, \quad
                &&(p_i, u_{i-1}) = -(-x)^{l'_{i-1}-1}\\
                &(q_i, s_i) = -y^{m'_i-1}, 
                &&(q_i, t_i) = x, 
                &&(q_i, u_i) = -z^{-n'_i}\\
                &(r_i, s_{i+1}) = -(-x)^{-l'_i},
                &&(r_i, t_i) = z^{n'_i-1}, 
                &&(r_i, u_i) = y
        \end{alignat*}
        and $0$ otherwise. We are using the notation that $x^a, y^a$ and $z^a$ is considered as zero if $a<0$. The holonomy contribution occurs only at $$(p_1, u_\tau) = -\lambda'(-x)^{l'_1-1}, \quad (r_\tau, s_1) = -\lambda'^{-1}(-x)^{-l'_1}.$$
\end{thm}


Hence the matrix factor $\cff^{\bll}\left(L\left(w', \lambda'\right)\right)$ is given by
$$
M_L = \begin{blockarray}{cccccccccc}
        p_1 & q_1 & r_1 & p_2 & q_2 & r_2 & \cdots & q_\tau & r_\tau \\
        \begin{block}{(ccccccccc)c}
                z & -y^{m'_1-1} & 0 & 0 & 0 & 0 & \cdots & 0 & -\lambda'^{-1}(-x)^{-l'_1} & s_1 \\
                y^{-m'_1} & x & -z^{n'_1-1} & 0 & 0 & 0 & \cdots & 0 & 0 & t_1 \\
                0 & z^{-n'_1} & y & -(-x)^{l'_2-1} & 0 & 0 & \cdots & 0 & 0 & u_1 \\
                0 & 0 & -(-x)^{-l'_2} & z & -y^{m'_2-1} & 0 & \cdots & 0 & 0 & s_2 \\
                0 & 0 & 0 & y^{-m'_2} & x & -z^{n'_2-1} & \cdots & 0 & 0 & t_2 \\
                0 & 0 & 0 & 0 & z^{-n'_2} & y & \cdots & 0 & 0 & u_2 \\
                \vdots & \vdots & \vdots & \vdots & \vdots & \vdots & \ddots & \vdots & \vdots & \vdots \\
                0 & 0 & 0 & 0 & 0 & 0 & \cdots & x & -z^{m'_\tau-1} & t_\tau \\
                -\lambda'(-x)^{l'_1-1} & 0 & 0 & 0 & 0 & 0 & \cdots & z^{-n'_\tau} & y & u_\tau \\
        \end{block}.
\end{blockarray}
$$


\section{$T_{3,2,\infty}$-singularity}\label{sec:32}
In this section, we discuss the mirror symmetry for $x^3+y^2-xyz$ and prove Theorem \ref{thm:J}.

The Milnor fiber $M_{F_{3,2}}$ is a torus with a puncture (drawn in Figure \ref{fig:A2K} as
a hexagon whose opposite sides are identified) and
$G_{F_{3,2}}$ action is given by $\Z_3$-rotation at the center of the hexagon as well as simultaneous $\Z_2$-rotation at the
centers of three rhombuses. The quotient space is $\mathbb{P}^1_{3,2,\infty}$.

\begin{figure}[h]
\begin{subfigure}[t]{0.43\textwidth}
\includegraphics[scale=0.55]{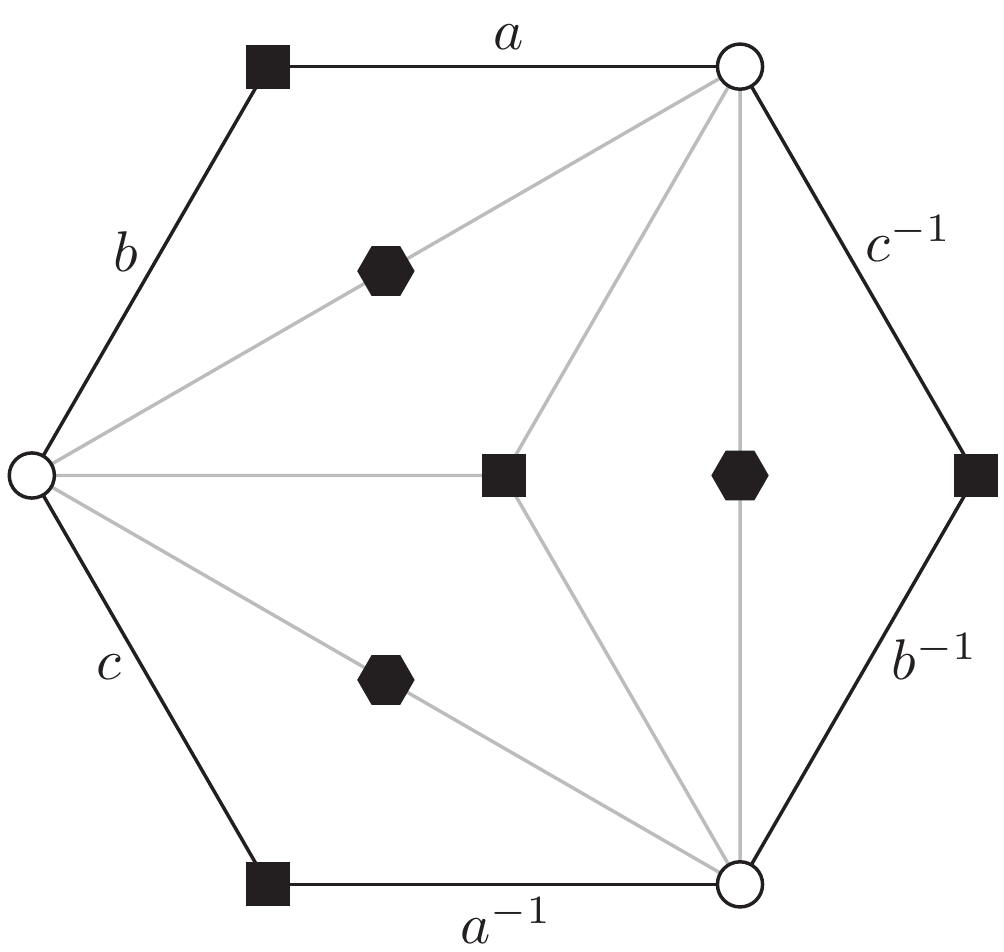}
\centering
\caption{Milnor fiber of $F_{3,2}$}
\end{subfigure}
\begin{subfigure}[t]{0.43\textwidth}
\includegraphics[scale=0.55]{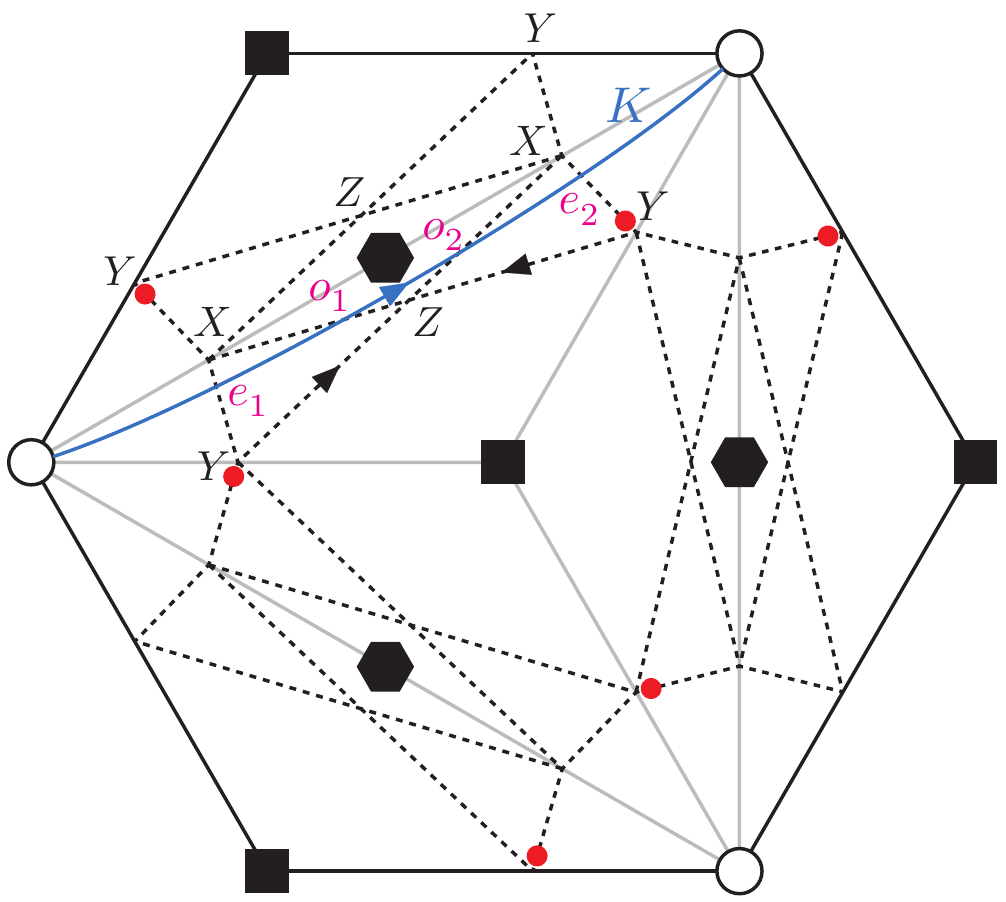}
\centering
\caption{Lift of $\bL$ (dotted lines) and the noncompact Lagrangian $K$}
\end{subfigure}
\caption{}
\label{fig:A2K}
\end{figure}
 
Let us first discuss the rank one module which is not locally free  in Proposition \ref{BD23},
namely the one corresponding to the conductor ideal $I = \langle x, y \rangle$.
 The conductor ideal $I$ gives the following  matrix factorization
\begin{equation*}
\begin{pmatrix}
x^{2} - yz& y\\
y& -x\\
\end{pmatrix}
\begin{pmatrix}
x& y\\
y& -x^{2} + yz\\
\end{pmatrix}
\end{equation*}
In \cite{CCJ}, it was observed that the non-compact Lagrangian $K$ in Figure \ref{fig:A2K} maps to the above matrix factorization under the
localized mirror functor: this can be easily done by counting decorated polygons.

\begin{figure}[h]
\includegraphics[scale=0.5]{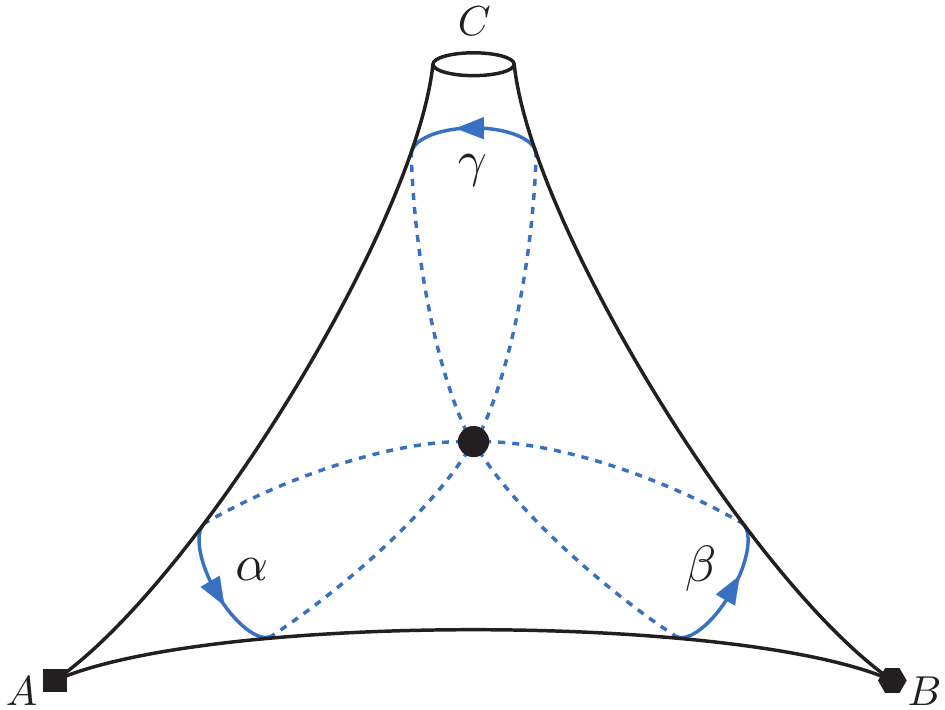}
\centering
\caption{3 generators of $\pi_1^{\mathrm{orb}}\left( M_{F_{3,2}}/G_{F_{3,2}} \right)$}
\label{fig:loops}
\end{figure}

The orbifold fundamental group of $M_{F_{3,2}}/G_{F_{3,2}}$ is given by
\begin{equation*}
 \pi_1^{\mathrm{orb}}\left( M_{F_{3,2}}/G_{F_{3,2}} \right)=\left\langle \alpha, \beta,\gamma \mid
\alpha^{3} =\beta^{2} =\alpha\beta\gamma = 1\right\rangle
\end{equation*}
where $\alpha,\beta,\gamma$ are loops in $\mathbb{P}^1_{3,2,\infty}$; see Figure \ref{fig:loops}.

Given $m \in \Z_{>0}$ and $\lambda \in \C^{*}$, we choose a compact (immersed) Lagrangian $L_{m,\lambda}$ with holonomy $\lambda$
which has the homotopy class $[ \alpha \gamma^{-m-1} ]$. See Figure \ref{fig:f32} for the case of $m=3$, or 
its lift in  the universal covering space \ref{fig:L3im} (for later computations). In Figure \ref{fig:L3im} (a), we decomposed $L_{m,\lambda}$ into
the union of (overlapping) arcs and wrote the multiplicity of each arc (negative sign indicates the reversal of orientations).

\begin{figure}[h]
\begin{subfigure}[t]{0.43\textwidth}
\includegraphics[scale=0.55]{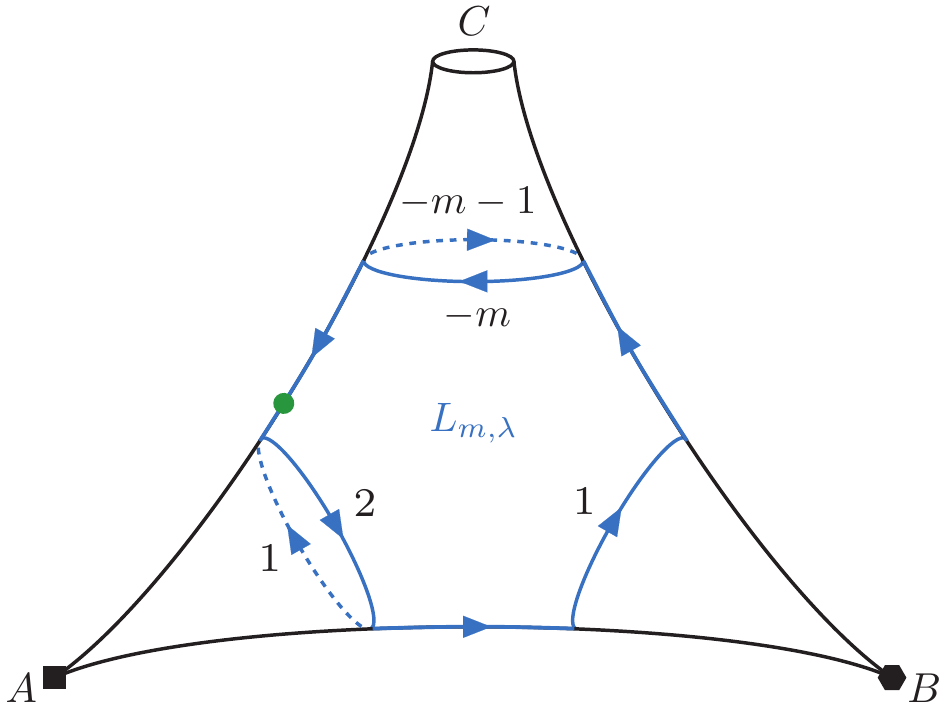}
\centering
\caption{$L_{m,\lambda}$ in $M_{F_{3,2}}/G_{F_{3,2}}$}
\label{fig:Lm}
\end{subfigure}
\begin{subfigure}[t]{0.43\textwidth}
\includegraphics[scale=0.7]{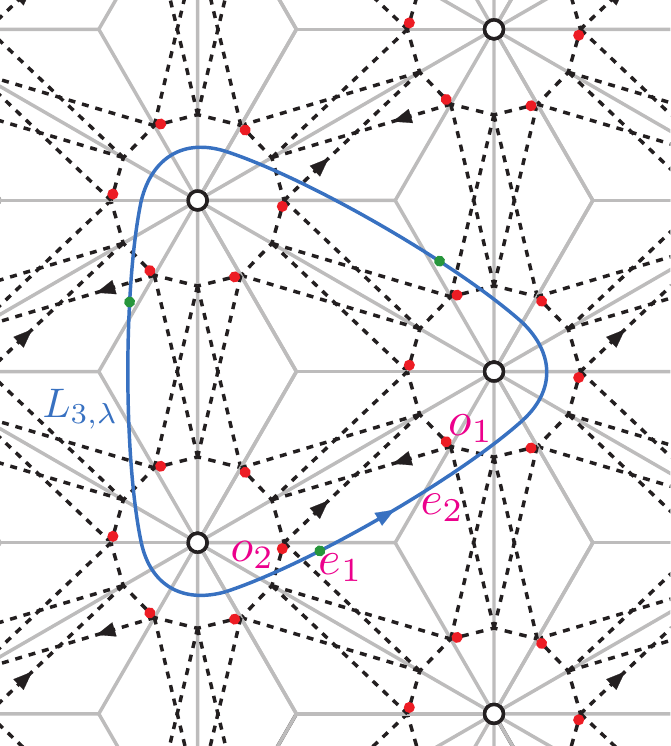}
\centering
\caption{Lift of $L_{3,\lambda}$ in the covering space}
\end{subfigure}
\centering
\caption{ }
\label{fig:L3im}
\end{figure}

\begin{lemma}\label{lem:32InfiniteGeodesic}
The Lagrangian loop $L_{m,\lambda}$ and its orientation reversal $\overline{L}_{m,\lambda}$ with $m \in \Z_{>0}$
have  homotopy classes $[\alpha\gamma^{-m-1}]$ and $[\alpha^{-1}\gamma^{m+1}]$.
These cover all homotopy classes of type $\alpha^{l^{\prime}}\beta^{m^{\prime}}\gamma^{n^{\prime}} \in  \pi_1^{\mathrm{orb}}\left( M_{F_{3,2}}/G_{F_{3,2}} \right)$
(up to conjugation) except the parabolic elements $\gamma^{n^{\prime}}$, and elliptic elements $\alpha, \alpha^2, \beta$.
Therefore  closed geodesics in $\mathbb{P}^1_{3,2,\infty}$
in the homotopy class of the form $\alpha^{l^{\prime}}\beta^{m^{\prime}}\gamma^{n^{\prime}}$ are Lagrangian isotopic to one of $L_{m,\lambda}$ or $\overline{L}_{m,\lambda}$ for some $m \in \Z_{>0}$.
\end{lemma}
\begin{proof}
Denote a class $\alpha^{l^{\prime}}\beta^{m^{\prime}}\gamma^{n^{\prime}}$ by $(l^{\prime},m^{\prime},n^{\prime})$. Then $(0,0,n^{\prime})$ type and $(1,1,n^{\prime})$ type are non-essential. $(2,0,n^{\prime})$ and $(2,1,n^{\prime})$ types can be transformed into $(0,1,n^{\prime})$ and $(1,0,n^{\prime})$ types. Hence we only consider those two types. First, since $(0,1,n^{\prime}) = (-1,0,n^{\prime}-1)$, it can be written as $[\alpha^{-1}\gamma^{m+1}]$ for $n^{\prime} \geq 3$. From $(0,1,n^{\prime}) = (1,0,n^{\prime}+1)$, it holds also for $n^{\prime} \leq -3$. For $-2 \leq n^{\prime} \leq 2$, one can check directly that they are all non-essential. Similarly, $(1,0,n^{\prime})$ is already of the form which we want for $n^{\prime} \leq -2$ and for $n^{\prime} \geq 4$, $(1,0,n^{\prime}) = (0,1,n^{\prime}-1) = (-1,0,n^{\prime}-2)$ gives a desired result. In the remaining cases, it can be also shown that all they are non-essential.
\end{proof}


Let us first outline the proof of Theorem \ref{thm:J} in the introduction.
 
\begin{thm}\label{thm:32InfiniteCorrespondence}
Let $L_{m,\lambda}$ be an immersed curve in $\mathbb{P}^1_{3,2,\infty}$ that wind the puncture $m$ times and wind the $\Z/3$-orbifold point once as in Figure \ref{fig:f32}(b), equipped complex line bundle with holonomy $\lambda$.
Then, under the localized mirror functor $\cF^{\bL}$, we have the mirror symmetry correspondences
$$\begin{cases}
L_{m,\lambda} & \mapsto J_{m,-\lambda}, \\
\overline{L}_{m,\lambda} &\mapsto I_{m,-\frac{1}{\lambda}}
\end{cases}$$
where $\overline{L}_{m,\lambda}$ is the orientation reversal of $L_{m,\lambda}$.
\end{thm}

\begin{proof}
We first find the matrix factorizations that are mirror to the Lagrangians $L_{m,\lambda}$
by explicitly computing the  image of the $\AI$-functor $\cF^{\bL}(L_{m,\lambda})$ in Proposition \ref{mirror32}.
Next, we show that these correspond to the presentations of  maximal Cohen-Macaulay modules $J_{m,-\lambda}$ in Lemma \ref{lem:J}.

This proves the theorem for $L_{m,\lambda}$ and  $J_{m,-\lambda}$.
To prove the analogous result for  $\overline{L}_{m,\lambda}$, we proceed somewhat differently
due to computational difficulty.
First, we recall from Lemma 10.2 \cite{CCJ} that orientation reversal of a Lagrangian $L_{m,\lambda}$
is mapped to Auslander-Reiten translation of the mirror $\cF^{\bL}(L_{m,\lambda})$. 
Hence, it is  enough to show that $J_{m,-\lambda}$ and $I_{m,-\frac{1}{\lambda}}$ are
Auslander-Reiten translation of each other. 
We will find different presentations of $J_{m,-\lambda}$ and $I_{m,-\frac{1}{\lambda}}$ to verify this claim.
This will be given in Proposition \ref{prop:AR}.
This will finish the proof of the theorem. 
 \end{proof}

%

%

Note that since the potential function for $W_\bL$ in \cite{CCJ} is 
$$ W = x^{3} + y^{2} +xyz,$$
we need to substitute $-z$ for $z$ to compare the results with Burban-Drozd.
 
\begin{prop}\label{mirror32}
For general $L_{m,\lambda}$ with $m \geq 2$,  its mirror matrix factorization $\cF^{\bL}(L_{m,\lambda}) = (P_1,P_2)$
where 
$$P_{1} =
\begin{pmatrix}
a_{11}& a_{12}\\
a_{21}& a_{22}\\
\end{pmatrix}, \;\; P_{2} =
\begin{pmatrix}
a^{\prime}_{11}& a^{\prime}_{12}\\
a^{\prime}_{21}& a^{\prime}_{22}\\
\end{pmatrix}.$$

%
\begingroup
\allowdisplaybreaks
\begin{align*}
&\bullet \; a_{11} = a^{\prime}_{22} = \lambda x + \sum_{i = 0}^{\lfloor \frac{m-1}{2}\rfloor} (-1)^{m-1-i} {m-1-i \choose i} x^{i+1} z^{m-1-2i} \\
&\hspace{3.5cm}+ \frac{\displaystyle \sum_{i = 0}^{\lfloor \frac{m}{2}\rfloor} (-1)^{m-i} {m-i \choose i} x^{i} z^{m+1-2i} \left(1+ \lambda^{-1} \sum_{i = 0}^{\lfloor \frac{m-2}{2}\rfloor} (-1)^{m-3-i} {m-2-i \choose i} x^{i+1} z^{m-3-2i} \right)}{1-r}
\\
&\bullet \; a_{21} = - a^{\prime}_{21} = y + \frac{\displaystyle \lambda^{-1} \sum_{i = 0}^{\lfloor \frac{m}{2}\rfloor} (-1)^{m-1-i} {m-i \choose i} x^{i+1} z^{m-2i}}{1-r}
\\
&\bullet \; a_{12} = - a^{\prime}_{12} = - y + \frac{\displaystyle -xz + \lambda^{-1} \sum_{i = 0}^{\lfloor \frac{m-2}{2}\rfloor} (-1)^{m-2-i} {m-2-i \choose i} x^{i+2} z^{m-2-2i}}{1-r}
\\
&\bullet \; a_{22} = a^{\prime}_{11} = \frac{\lambda^{-1} x^{2}}{1-r}
\end{align*}
\endgroup
where $\displaystyle r = \lambda^{-1} \sum_{i = 0}^{\lfloor \frac{m-1}{2}\rfloor} (-1)^{m-i} {m-1-i \choose i} x^{i} z^{m-1-2i}$.
\end{prop}
\begin{proof}
Let us illustrate the computation for $L_{3,\lambda}$. The general case can be handled exactly in the same way.

There are two odd degree morphisms $o_{1}$, $o_{2}$ and two even degree morphisms $e_{1}$, $e_{2}$ in $CF(L_{3,\lambda},\mathbb{L})$. We have to compute all coefficients in the following equations;
\begin{align*}
m^{0,\bb}_{1}(o_{1}) &= a_{11} \cdot e_{1} + a_{21} \cdot e_{2} &m^{0,\bb}_{1}(e_{1}) &= a^{\prime}_{11} \cdot o_{1} + a^{\prime}_{21} \cdot o_{2} \\
m^{0,\bb}_{1}(o_{2}) &= a_{12} \cdot e_{1} + a_{22} \cdot e_{2} &m^{0,\bb}_{1}(e_{2}) &= a^{\prime}_{12} \cdot o_{1} + a^{\prime}_{22} \cdot o_{2} 
\end{align*}
which give a matrix factorization $(P_1, P_2)$ of $x^3+y^2 -xyz$.
\begin{figure}
\begin{subfigure}{0.46\textwidth}
\includegraphics[scale=0.65]{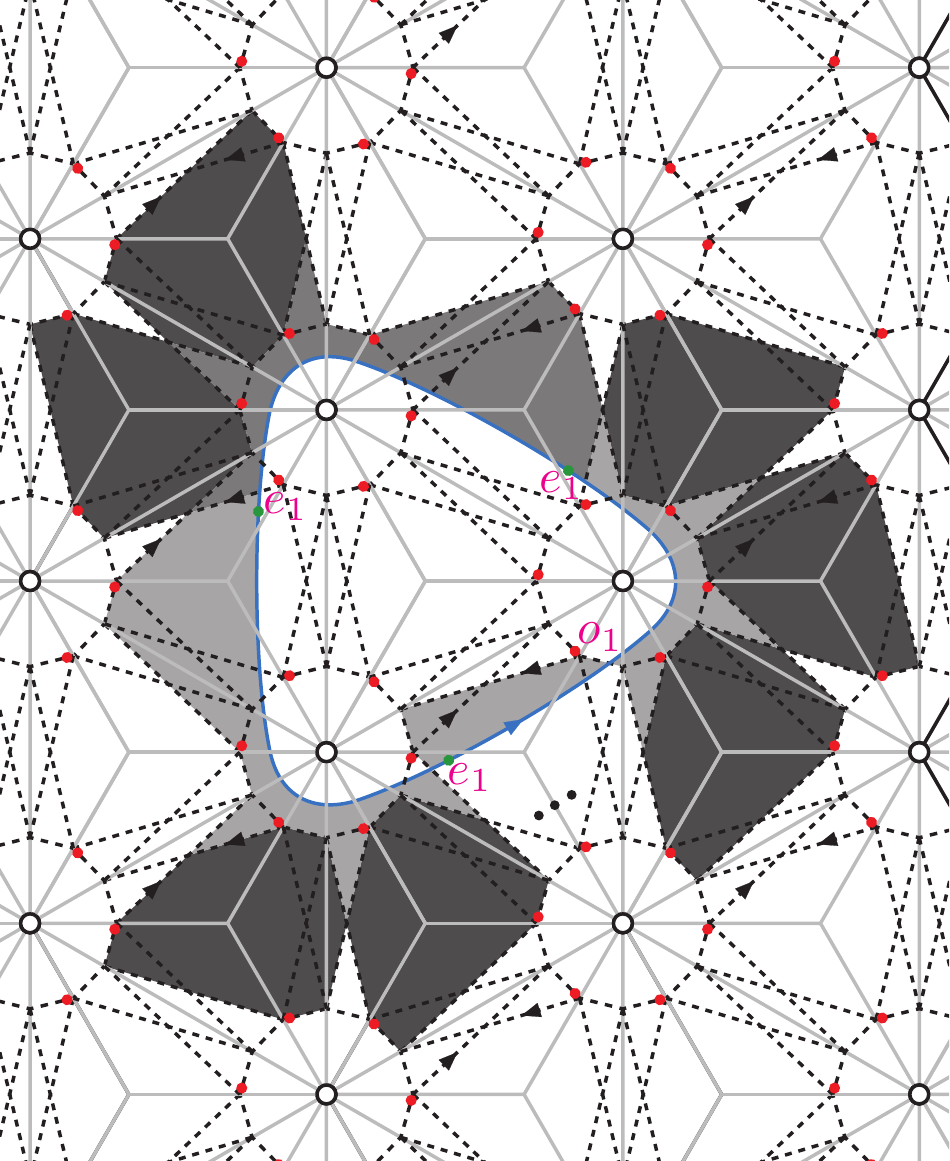}
\centering
\caption{Discs for $a_{11}$}
\end{subfigure}
\begin{subfigure}{0.46\textwidth}
\includegraphics[scale=0.65]{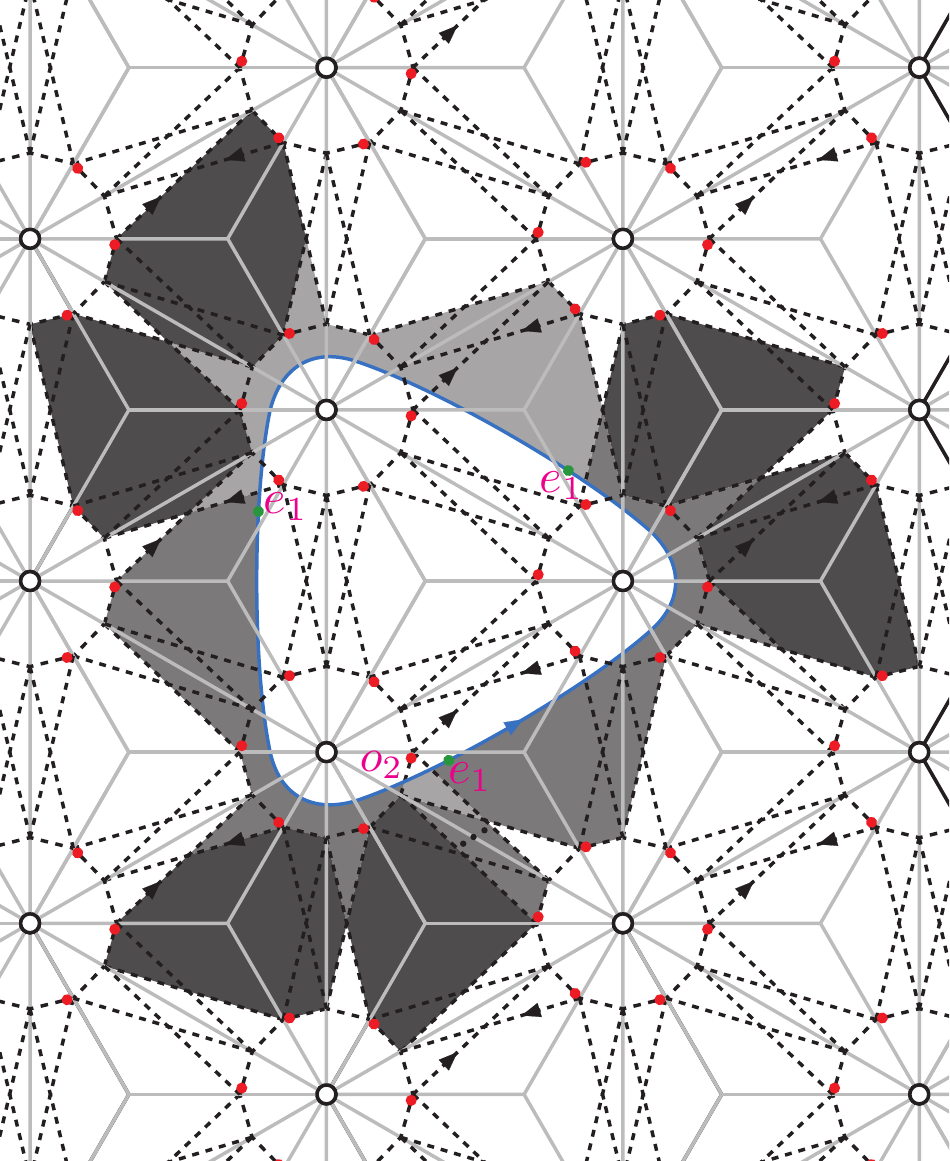}
\centering
\caption{Discs for $a_{12}$}
\end{subfigure}
\centering
\begin{subfigure}{0.46\textwidth}
\includegraphics[scale=0.65]{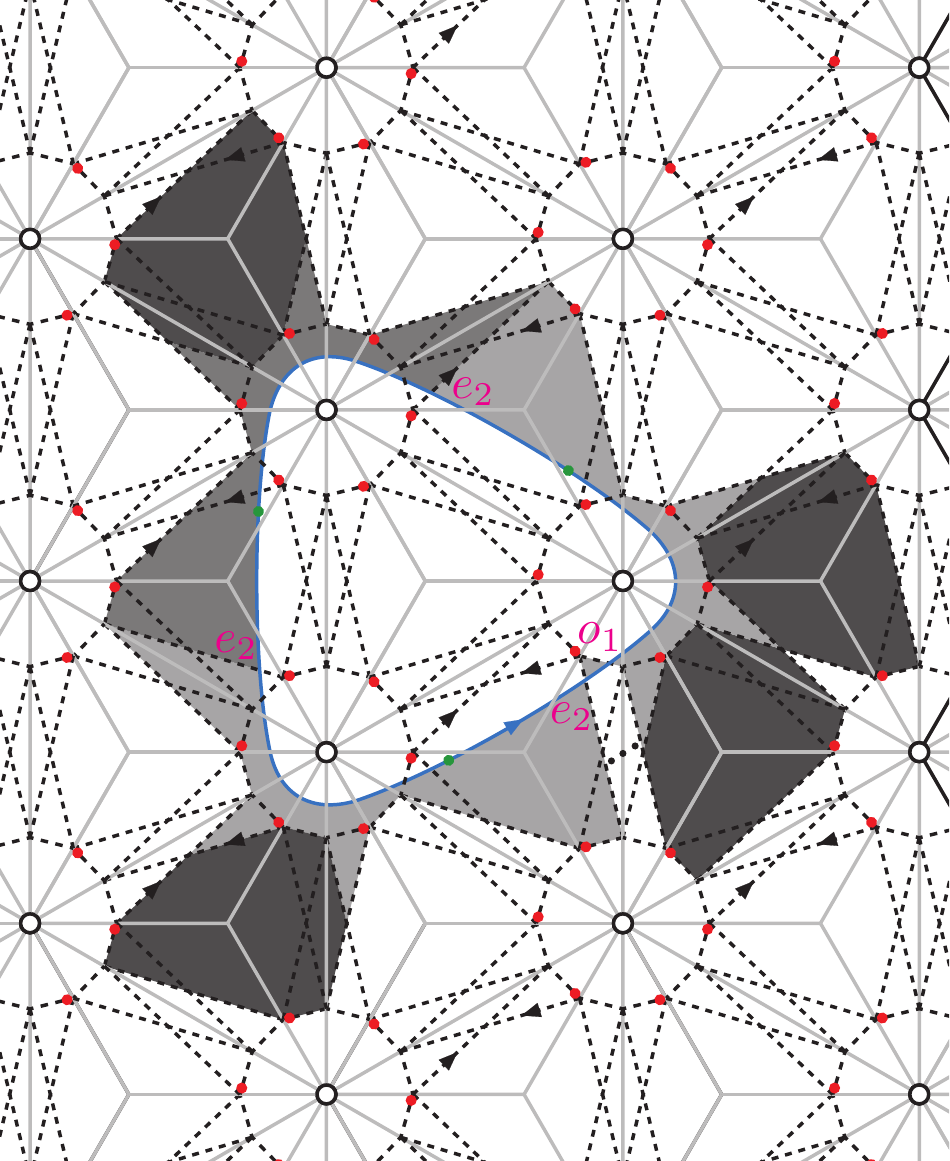}
\centering
\caption{Discs for $a_{21}$}
\end{subfigure}
\begin{subfigure}{0.46\textwidth}
\includegraphics[scale=0.65]{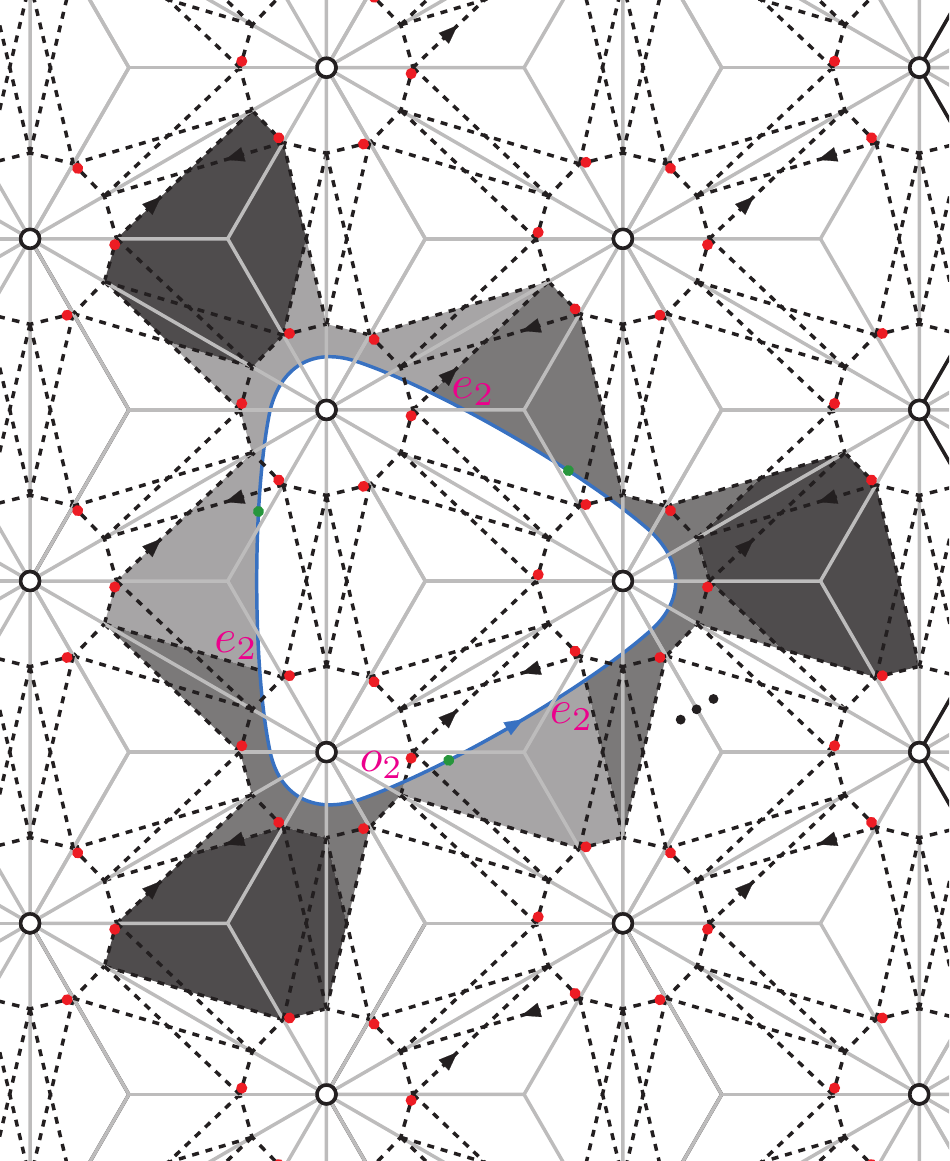}
\centering
\caption{Discs for $a_{22}$}
\end{subfigure}
\centering
\caption{ }
\label{fig:L3disc}
\end{figure}
In Figure \ref{fig:L3disc}, many pieces of polygons are given in each case. Any possible combination of these pieces gives a disc for $m^{0,\bb}_{1}$. We count (infinitely many) possible discs and get the followings.
\begingroup
\allowdisplaybreaks
\begin{align*}
a_{11} &= \lambda x - z^{4} + 3xz^{2} - x^{2} + \lambda^{-1} z^{6} - 4 \lambda^{-1} xz^{4} + 4 \lambda^{-1} x^{2}z^{2} - \lambda^{-2} z^{8} + 6 \lambda^{-2} xz^{6} - 8 \lambda^{-2} x^{2}z^{4} + 4 \lambda^{-2} x^{3}z^{2} + \cdots \\
&= \lambda x + xz^{2} - x^{2} + (2xz^{2} - z^{4} + 2 \lambda^{-1} z^{2}x^{2} - \lambda^{-1}xz^{4}) + (2xz^{2} - z^{4} + 2 \lambda^{-1} z^{2}x^{2} - \lambda^{-1}xz^{4})(-\lambda^{-1}z^{2} + \lambda^{-1} x) + \cdots \\
&= \lambda x + xz^{2} - x^{2} + \dfrac{(2xz^{2} - z^{4})(1 + \lambda^{-1}x)}{1-(-\lambda^{-1}z^{2} + \lambda^{-1} x)}\\[1ex]
a_{21} &= y + \lambda^{-1} xz^{3} - 2\lambda^{-1} x^{2}z - \lambda^{-2}xz^{5} + 3 \lambda^{-2}x^{2}z^{3} - 2 \lambda^{-2}x^{3}z + \cdots \\
& = y + (\lambda^{-1} xz^{3} - 2\lambda^{-1} x^{2}z) + (\lambda^{-1} xz^{3} - 2\lambda^{-1} x^{2}z)(-\lambda^{-1}z^{2} + \lambda^{-1} x) + \cdots \\
& = y + \dfrac{\lambda^{-1} xz^{3} - 2\lambda^{-1} x^{2}z}{1-(-\lambda^{-1}z^{2} + \lambda^{-1} x)} \\[1ex]
a_{12} &= - y -xz + \lambda^{-1} xz^{3} - 2 \lambda^{-1} x^{2}z - \lambda^{-2} xz^{5} + 3 \lambda^{-2} x^{2}z^{3} - 2 \lambda^{-2} x^{3}z + \cdots \\
&= -y + (-xz - \lambda^{-1}x^{2}z) + (-xz - \lambda^{-1}x^{2}z)(-\lambda^{-1}z^{2} + \lambda^{-1} x) + \cdots \\
&= -y + \dfrac{-xz - \lambda^{-1}x^{2}z}{1-(-\lambda^{-1}z^{2} + \lambda^{-1} x)}\\[1ex]
a_{22} &= \lambda^{-1}x^{2} - \lambda^{-2}x^{2}z^{2} + \lambda^{-2}x^{3} + \lambda^{-3}x^{2}z^{4} - 2\lambda^{-3}x^{3}z^{2} +\lambda^{-3}x^{4} + \cdots\\
&= \lambda^{-1}x^{2} + \lambda^{-1}x^{2}(-\lambda^{-1}z^{2} + \lambda^{-1} x) + \cdots \\
&= \frac{\lambda^{-1} x^{2}}{1-(-\lambda^{-1}z^{2} + \lambda^{-1} x)}
\end{align*}
\endgroup
By the symmetry of domain, there are relations between entries;
$$a_{11} = a^{\prime}_{22}, a_{21} = -a^{\prime}_{21}, a_{12} = -a^{\prime}_{12}, a_{22} = a^{\prime}_{11}.$$
 We can generalize the disc counting to the case of general $m$, which gives the results.
\end{proof}

\begin{lemma} \label{lem:J}
For any $m \geq 1$, $P_{1}$ and $P_{2}$ fit into the periodic resolution of $J_{m, -\lambda}$;
$$\begin{tikzcd}
\cdots \arrow[r, "P_{1}"] & A^{2} \arrow[r, "P_{2}"] & A_{2} \arrow[r, "P_{1}"] & A^{2} \arrow[r] & J_{m, -\lambda}\\
\end{tikzcd}$$

\begin{proof}
Since two generators of $J_{m,-\lambda} = \langle x^{m+1}, y^{m} + \lambda x^{m-1} (xz + y) \rangle$ form a regular sequence of $\C [[x,y,z]]$, we will find only one nontrivial relation in $\C [[x,y,z]] / \langle W \rangle$ between generators to compute the presentation of $J_{m,-\lambda}$. But that can be found exactly in $\cF^{\bL}(L_{m,\lambda})$. First, we give explicit forms of $\cF^{\bL}(L_{1,\lambda})$ since $m=1$ case is ruled out from the formula.
$$\cF^{\bL}(L_{1,\lambda}) = \begin{pmatrix}
\lambda x + x + \dfrac{-z^{2}}{1+\lambda^{-1}}& -y + \dfrac{-xz}{1+\lambda^{-1}}\\[3ex]
y + \dfrac{\lambda^{-1} xz}{1+\lambda^{-1}}& \dfrac{\lambda^{-1}x^{2}}{1+\lambda^{-1}}\\
\end{pmatrix}
\begin{pmatrix}
\dfrac{\lambda^{-1}x^{2}}{1+\lambda^{-1}}& y - \dfrac{-xz}{1+\lambda^{-1}}\\[3ex]
- y - \dfrac{\lambda^{-1} xz}{1+\lambda^{-1}}& \lambda x + x + \dfrac{-z^{2}}{1+\lambda^{-1}}\\
\end{pmatrix}$$
Then one can check directly that
\begin{align*}
&\begin{pmatrix}
x^2& y + \lambda (xz + y)\\
\end{pmatrix}
\begin{pmatrix}
\lambda x + x + \dfrac{-z^{2}}{1+\lambda^{-1}}& -y + \dfrac{-xz}{1+\lambda^{-1}}\\[3ex]
y + \dfrac{\lambda^{-1} xz}{1+\lambda^{-1}}& \dfrac{\lambda^{-1}x^{2}}{1+\lambda^{-1}}\\
\end{pmatrix} \\
&=\begin{pmatrix}
(1+\lambda)x^{3} + \dfrac{-x^{2}z^{2}}{1+\lambda^{-1}} + (1+\lambda) y^{2} + xyz + \lambda xyz + \dfrac{x^{2}z^{2}}{1+\lambda^{-1}}& -x^{2}y + \dfrac{-x^{3}z}{1+\lambda^{-1}} + \dfrac{\lambda^{-1}(1+\lambda)x^{2}y + x^{3}z}{1+\lambda^{-1}}
\end{pmatrix}\\
&= \begin{pmatrix}
(1+\lambda)W & 0\\
\end{pmatrix}
\end{align*}
When $m = 2$, we need the following isomorphism;
$$J_{2,-\lambda} \simeq \langle x^{2}, y(\lambda-z) + \lambda xz \rangle \simeq \langle x^{2}, y(\lambda-z) + \lambda xz - x^{2} \rangle$$
First column of
$$\begin{pmatrix}
x^{2}& y(\lambda-z) + \lambda xz - x^{2}
\end{pmatrix}
\begin{pmatrix}
\lambda x - xz + \dfrac{z^{3} -\lambda^{-1} xz^{2} -xz + \lambda^{-1}x^{2}}{1 - \lambda^{-1}z}& -y + \dfrac{-xz+\lambda^{-1} x^{2}}{1-\lambda^{-1}z}\\[3ex]
y + \dfrac{-\lambda^{-1} xz^{2} + \lambda^{-1}x^{2}}{1-\lambda^{-1}z}& \dfrac{\lambda^{-1}x^{2}}{1-\lambda^{-1}z}
\end{pmatrix}  $$
becomes a multiple of $W$;
\begin{align*}
&x^{2}(\lambda x -xz) + \dfrac{x^{2}z^{3} - \lambda^{-1}x^{3}z^{2} - x^{3}z + \lambda^{-1}x^{4}}{1-\lambda^{-1} z} + y^{2}(\lambda - z) + \lambda xy^{2} - x^{2}y \\
&\hspace{5.7cm}+ \dfrac{-xyz^{2} + x^{2}y +\lambda^{-1}xyz^{3} -\lambda^{-1}x^{2}yz - x^{2}z^{3} + x^{3}z + \lambda^{-1}x^{3}z^{2} - \lambda^{-1}x^{4}}{1 - \lambda^{-1}z}\\
&=x^{2}(\lambda x -xz) + y^{2}(\lambda - z) + \lambda xyz - x^{2}y -xyz^{2} + x^{2}y \\
&=(\lambda-z)W
\end{align*}
For general $m \geq 3$, $J_{m,-\lambda}$ is isomorphic to
$$\left\langle \lambda^{-1} x^{2}, y\left(1- \lambda^{-1} \sum_{i = 0}^{\lfloor \frac{m-1}{2}\rfloor} (-1)^{m-i} {m-1-i \choose i} x^{i} z^{m-1-2i}\right) + xz - \lambda^{-1} \sum_{i = 0}^{\lfloor \frac{m-2}{2}\rfloor} (-1)^{m-2-i} {m-2-i \choose i} x^{i+2} z^{m-2-2i} \right\rangle$$
Denote the second generator by $c_{2}$. We will show that $x^{2} a_{11} + c_{2} a_{21}$ is a multiple of $W$.
\begin{align*}
&\lambda^{-1}x^{2} a_{11} + c_{2} a_{21}\\
&=x^{3}(1-r) + \dfrac{\displaystyle \lambda^{-1} \sum_{i = 0}^{\lfloor \frac{m}{2}\rfloor} (-1)^{m-i} {m-i \choose i} x^{i+2} z^{m+1-2i}}{1-r} \\
&\hspace{4cm}+ \dfrac{\lambda^{-2} \displaystyle \sum_{i = 0}^{\lfloor \frac{m}{2}\rfloor} (-1)^{m-i} {m-i \choose i} x^{i+2} z^{m+1-2i} \cdot \sum_{i = 0}^{\lfloor \frac{m-2}{2}\rfloor} (-1)^{m-3-i} {m-2-i \choose i} x^{i+1} z^{m-3-2i}}{1-r} \\
&+ y^{2}(1-r) + \lambda^{-1} \sum_{i = 0}^{\lfloor \frac{m}{2}\rfloor} (-1)^{m-1-i} {m-i \choose i} x^{i+1}yz^{m-2i} + xyz + \dfrac{\displaystyle \lambda^{-1} \sum_{i = 0}^{\lfloor \frac{m}{2}\rfloor} (-1)^{m-1-i} {m-i \choose i} x^{i+2} z^{m+1-2i}}{1-r}\\
&- \lambda^{-1} \sum_{i = 0}^{\lfloor \frac{m-2}{2}\rfloor} (-1)^{m-2-i} {m-2-i \choose i} x^{i+2}yz^{m-2-2i} \\
&\hspace{4cm}- \dfrac{\lambda^{-2} \displaystyle \sum_{i = 0}^{\lfloor \frac{m}{2}\rfloor} (-1)^{m-1-i} {m-i \choose i} x^{i+1} z^{m-2i} \cdot \sum_{i = 0}^{\lfloor \frac{m-2}{2}\rfloor} (-1)^{m-2-i} {m-2-i \choose i} x^{i+2} z^{m-2-2i}}{1-r} \\
&= (1-r)x^{3} + (1-r)y^{2} \\
&\hspace{2.1cm} + xyz \left( 1 + \lambda^{-1} \sum_{i = 0}^{\lfloor \frac{m}{2}\rfloor} (-1)^{m-1-i} {m-i \choose i} x^{i}z^{m-1-2i} - \lambda^{-1} \sum_{i = 1}^{\lfloor \frac{m-2}{2}\rfloor + 1} (-1)^{m-1-i} {m-1-i \choose i-1} x^{i}z^{m-1-2i} \right) \\
&= (1-r)x^{3} + (1-r)y^{2} + xyz \left( 1 - \lambda^{-1} \sum_{i = 0}^{\lfloor \frac{m-1}{2}\rfloor} (-1)^{m-i} {m-1-i \choose i} x^{i}z^{m-1-2i} \right) \\
&= (1-r) W
\end{align*}
Thus, we can conclude that $\cF^{\bL}(L_{m,\lambda})$ is the stabilization of $J_{m,-\lambda}$.
\end{proof}
\end{lemma}

\begin{prop} \label{prop:AR}
For any integer $m \geq 1$, $J_{m,\lambda} = (P_{1}, P_{2})$ and $I_{m,\frac{1}{\lambda}} = (P_{2}, P_{1})$ satisfy
$$J_{m,\lambda}[1] \simeq I_{m,\frac{1}{\lambda}}$$
as matrix factorizations. Equivalently, $I_{m,\frac{1}{\lambda}}$ is the Auslander-Reiten translation of $J_{m,\lambda}$ in $CM(A)$.
\end{prop}
\begin{proof}
When $m =1$, we have $J_{1,\lambda} \simeq I_{1,\lambda} = \langle x^{2}, y(1-\lambda) + \lambda xz \rangle$ and a matrix form is given by
\begin{equation*}
\begin{pmatrix}
(1-\lambda)x + \dfrac{\lambda}{1-\lambda}z^{2}& -y - \dfrac{\lambda}{1-\lambda} xz\\[1em]
y - \dfrac{1}{1-\lambda} xz& \dfrac{1}{1-\lambda} x^{2}\\
\end{pmatrix}
\begin{pmatrix}
\dfrac{1}{1-\lambda}x^{2}& y + \dfrac{\lambda}{1-\lambda} xz\\[1em]
-y + \dfrac{1}{1-\lambda} xz& (1-\lambda)x + \dfrac{\lambda}{1-\lambda}z^{2}\\
\end{pmatrix}
\end{equation*}
As a periodic resolution, we have
$$\begin{tikzcd}[ampersand replacement=\&]
\cdots \arrow[r] \& A^{2} \arrow[r, "{\begin{scriptsize}\begin{pmatrix}
\dfrac{1}{1-\lambda}x^{2}& y + \dfrac{\lambda}{1-\lambda} xz\\[1em]
-y + \dfrac{1}{1-\lambda} xz& (1-\lambda)x + \dfrac{\lambda}{1-\lambda}z^{2}\\
\end{pmatrix}\end{scriptsize}}"] \&[10em] A_{2} \arrow[r, "{\begin{scriptsize}\begin{pmatrix}
(1-\lambda)x + \dfrac{\lambda}{1-\lambda}z^{2} & -y - \dfrac{\lambda}{1-\lambda} xz\\[1em]
y - \dfrac{1}{1-\lambda} xz & \dfrac{1}{1-\lambda} x^{2}
\end{pmatrix}\end{scriptsize}}"] \&[10em] A^{2} \arrow[r] \& J_{m,\lambda}\\
\end{tikzcd}$$
One can easily check that
$$\begin{tikzcd}[ampersand replacement=\&]
\cdots \arrow[r] \& A^{2} \arrow[r, "{\begin{scriptsize}\begin{pmatrix}
(1-\lambda)x + \dfrac{\lambda}{1-\lambda}z^{2} & -y - \dfrac{\lambda}{1-\lambda} xz\\[1em]
y - \dfrac{1}{1-\lambda} xz & \dfrac{1}{1-\lambda} x^{2}
\end{pmatrix}\end{scriptsize}}"] \&[10em] A_{2} \arrow[r, "{\begin{scriptsize}\begin{pmatrix}
\dfrac{1}{1-\lambda}x^{2}& y + \dfrac{\lambda}{1-\lambda} xz\\[1em]
-y + \dfrac{1}{1-\lambda} xz& (1-\lambda)x + \dfrac{\lambda}{1-\lambda}z^{2}\\
\end{pmatrix}\end{scriptsize}}"] \&[10em] A^{2} \arrow[r] \& I_{m,\frac{1}{\lambda}}\\
\end{tikzcd}$$
is a periodic resolution of $I_{m,\frac{1}{\lambda}}$. Hence, we can conclude that $I_{1,\frac{1}{\lambda}}$ is the shift of $J_{1,\lambda}$ in the category of matrix factorizations.

If $m = 2$, $J_{2,\lambda} = \langle x^{3}, y^{2} + \lambda x(xz - y) \rangle \simeq \langle x^{2}, y(z - \lambda) + \lambda xz \rangle$ corresponds to
\begin{equation*}
\begin{pmatrix}
(z-\lambda)^{2}x + \lambda z^{3}& -y(z-\lambda) - \lambda xz\\[1em]
y(z - \lambda) - xz^{2}& x^{2}\\
\end{pmatrix}
\begin{pmatrix}
\dfrac{1}{(z-\lambda)^{2}}x^{2}& \dfrac{1}{z - \lambda} y + \dfrac{\lambda}{(z-\lambda)^{2}} xz\\[1em]
-\dfrac{1}{z- \lambda} y + \dfrac{1}{(z-\lambda)^{2}} xz^{2}& x + \dfrac{\lambda}{(z-\lambda)^{2}}z^{3}\\
\end{pmatrix}
\end{equation*}
Similarly, we can prove that $I_{2,\frac{1}{\lambda}} \simeq \langle y(z - \lambda) - xz^{2}, x^{2} \rangle$ is the shift of $J_{2,\lambda}$

In general, for $J_{m,\lambda}$ with $m \geq 3$, it is isomorphic to the ideal
$$\langle x^{2}, y\{z^{m-1} - (m-2)xz^{m-3} - \lambda\} + \lambda xz \rangle$$
and matrix factorization is given by
\begin{equation*}
\begin{small}
\begin{pmatrix}
x\{z^{m-1} - (m-2)xz^{m-3} - \lambda\}^{2} + \lambda z^{2}\{z^{m-1} - (m-2)xz^{m-3}\} & -y\{z^{m-1} - (m-2)xz^{m-3} - \lambda\} - \lambda xz\\[1em]
y\{z^{m-1} - (m-2)xz^{m-3} - \lambda\} -xz \{z^{m-1} - (m-2)xz^{m-3}\} & x^{2}\\
\end{pmatrix}
\end{small}
\end{equation*}
\begin{equation*}
\begin{footnotesize}
\begin{pmatrix}
\dfrac{x^{2}}{\{z^{m-1} - (m-2)xz^{m-3} - \lambda\}^{2}} & \dfrac{y\{z^{m-1} - (m-2)xz^{m-3} - \lambda\} + \lambda xz}{\{z^{m-1} - (m-2)xz^{m-3} - \lambda\}^{2}}\\[2em]
\dfrac{-y\{z^{m-1} - (m-2)xz^{m-3} - \lambda\} +xz \{z^{m-1} - (m-2)xz^{m-3}\}}{\{z^{m-1} - (m-2)xz^{m-3} - \lambda\}^{2}}& \dfrac{x\{z^{m-1} - (m-2)xz^{m-3} - \lambda\}^{2} + \lambda z^{2}\{z^{m-1} - (m-2)xz^{m-3}\}}{\{z^{m-1} - (m-2)xz^{m-3} - \lambda\}^{2}}\\
\end{pmatrix}
\end{footnotesize}
\end{equation*}
Let $P_{1}$, $P_{2}$ be above matrices, respectively. Then, we can write down them as a periodic resolution
$$\begin{tikzcd}
\cdots \arrow[r, "P_{1}"] & A^{2} \arrow[r, "P_{2}"] & A_{2} \arrow[r, "P_{1}"] & A^{2} \arrow[r] & J_{m,\lambda}\\
\end{tikzcd}$$
By some direct calculations, we have an isomorphism
$$I_{m,\lambda} \simeq \langle y\{ 1 - \lambda z^{m-1} + \lambda (m-2) xz^{m-3} \} + \lambda xz^{m}, x^{2} \rangle$$
for $m \geq 3$ and another resolution
$$\begin{tikzcd}
\cdots \arrow[r, "P_{2}"] & A^{2} \arrow[r, "P_{1}"] & A_{2} \arrow[r, "P_{2}"] & A^{2} \arrow[r] & I_{m,\frac{1}{\lambda}}\\
\end{tikzcd}$$
with same matrices $P_{1}$ and $P_{2}$ of $J_{m,\lambda}$. This proves the proposition.  
\end{proof}


\appendix

\section{Normality} \label{sec:normal}

\subsection{Equivalency and Operations} \label{subsec:equiv}
In this subsection, we prove Proposition \ref{prop:equihomotopy}. If two loop words $w', w'_*$ are equivalent, we denote by $w'\cong w'_*$.
\begin{prop}\label{prop:a1}
        Two loop words $w'$ and $w_*'$ are equivalent if and only if $w_*'$ can be obtained from $w'$ by performing the operations in Lemma \ref{lem:equimove} finitely many times.
\end{prop}
\begin{proof}
        By Lemma \ref{lem:equimove}, we only have to prove the `only if' part. Let $$w' = (l'_1, m'_1, n'_1, l'_2, m'_2, n'_2, \cdots, l'_\tau, m'_\tau, n'_\tau)$$ be a loop word and let $L\left[w'\right] = \alpha^{\mu_1}\beta^{\nu_1}\alpha^{\mu_2}\beta^{\nu_2}\cdots\alpha^{\mu_\sigma}\beta^{\nu_\sigma}$ be the associated reduced word in $\left<\alpha, \beta\right>$. This word is also associated with a loop word $w'' = (\mu_1, \nu_1, 0, \mu_2, \nu_2, 0, \cdots, \mu_\sigma, \nu_\sigma, 0)$. We will show that $w''$ can be obtained by $w'$ by performing the operations. If this is done, the proposition follows  that two equivalent loop words have the same associated reduced word.
        
        Note that if we have shown that the given loop word $w'$ can be deformed by the operations to a loop word of the form $w''' = (\mu'_1, \nu'_1, 0, \mu'_2, \nu'_2, 0, \cdots, \mu'_{\sigma'}, \nu'_{\sigma'})$ where $\mu'_i, \nu'_i$ are all nonzero, then we have $$L\left[w'\right] = \alpha^{\mu'_1}\beta^{\nu'_1}\alpha^{\mu'_2}\beta^{\nu'_2}\cdots\alpha^{\mu'_\sigma}\beta^{\nu'_\sigma},$$ which implies that $w'''$ and $w''$ are the same up to cyclic permutation.

        Suppose that $n'_1<0$. Then we have, 
        \begin{align*}
                (l'_1, m'_1, n'_1, l'_2, \cdots, n'_\tau) &\cong (l'_1, m'_1, \bm{0}, \bm{0}, \bm{0}, n'_1, l'_2, \cdots, n'_\tau)\\
                &=(l'_1, m'_1, \bm{0}, \bm{1}, \bm{1}, n'_1+1, l'_2, \cdots, n'_\tau)\\
                &=(1, 1, n'_1+1, l'_2, \cdots, n'_\tau, \bm{l'_1}, \bm{m'_1}, \bm{0}).
                \end{align*}
        Similarly, if $n'_1>0$, we have,
        \begin{align*}
                (l'_1, m'_1, n'_1, l'_2, \cdots, n'_\tau) &\cong (l'_1, m'_1, \bm{1}, \bm{1}, \bm{1}, n'_1, l'_2, \cdots, n'_\tau)\\
                &\cong (l'_1, \bm{m'_1-1}, \bm{0}, \bm{0}, 1, n'_1, l'_2, \cdots, n'_\tau)\\
                &\cong (l'_1, m'_1-1, 0, \bm{-1}, \bm{0}, \bm{n'_1-1}, l'_2, \cdots, n'_\tau)\\
                &\cong (-1, 0, n'_1-1, l'_2, \cdots, n'_\tau, \bm{l'_1}, \bm{m'_1-1}, \bm{0})\\
        \end{align*}
                Repeating this, we deform the original loop word into a loop word with $n'_1=0$. Also by shifting, we could further deform the word so that $n'_i$ are all zero.
                
                Thus we may assume the given loop word $w'$ satisfies $n'_i=0$ for all $i$. By deleting subwords of the form $(0, 0, 0)$, we may also assume that $w'$ has no subword of the form $(0, 0, 0)$. Suppose that $m'_i=0$ for some $i$. Then $l'_i$ and $l'_{i+1}$ are not zero. If $l'_i>0$, we have
                \begin{align*}
                        (\cdots, l'_i, 0, 0, l'_{i+1}, \cdots) &\cong (\cdots, l'_i, \bm{1}, \bm{1}, \bm{1}, 0, 0, l'_{i+1}, \cdots)\\
                        &\cong (\cdots, \bm{l'_i-1}, \bm{0}, \bm{0}, 1, \bm{1}, \bm{1}, \bm{l'_{i+1}+1}, \cdots)
                        \\
                        &= (\cdots, l'_i-1, 0, 0, \bm{1}, \bm{1}, \bm{1}, l'_{i+1}+1, \cdots)\\
                        &\cong (\cdots, l'_i-1, 0, 0, l'_{i+1}+1, \cdots).
                \end{align*}
        Similarly if $l'_i<0$, we have
        \begin{align*}
                (\cdots, l'_i, 0, 0, l'_{i+1}, \cdots) &\cong (\cdots, \bm{l'_i+1}, \bm{1}, \bm{1}, l'_{i+1}, \cdots)\\
                &\cong (\cdots, l'_i+1, \bm{0}, \bm{0}, \bm{l'_{i+1}-1}, \cdots).
        \end{align*}
        Repeating this, we deform the word into a shorter one as $$(\cdots, 0, l'_i, 0, 0, l'_{i+1}, m'_{i+1} \cdots)\cong(\cdots,0, 0, 0, 0, l'_i+l'_{i+1}, m'_{i+1}, \cdots)\cong(\cdots, 0, l'_i+l'_{i+1}, m'_{i+1}, \cdots).$$
        This can be done for each zero $m'_i$ and for zero $l'_i$ also. This proves that the loop word $w'$ can be deformed into a word of the form $(\mu'_1, \nu'_1, 0, \mu'_2, \nu'_2, 0, \cdots, \mu'_{\sigma'}, \nu'_{\sigma'}, 0)$ as we wanted to.
\end{proof}

\subsection{Existence} \label{subsec:exist}

In this subsection, we prove the existence part of Proposition \ref{prop:normalnormal}. Here, we use the notation $(w'_1, \cdots, w'_n)$ for a loop word for convenience. Also note that throughout the section, we only care about essential loop words. We will introduce three different notions of normality, weak-normal, almost-normal, and normal. Their implication is as follows.
\begin{center}
        Weak normal $\Rightarrow$ Almost-normal $\Rightarrow$ Normal
\end{center}

Let us introduce an easy lemma which will be repeatedly used in the proof of the main proposition. One can prove the lemma by using the method in the proof the previous proposition.
\begin{lemma}\label{lem:MainLemmaForNormalLoopWord}
        The following are true.
        \begin{enumerate}
                \item $(w'_1, \cdots, w'_{i-1}, 0, 0, w'_{i+2}, \cdots, w'_n) \cong (w'_1, \cdots, w'_{i-1}+w'_{i+2}, \cdots, w'_n)$.
                \item $(w'_1, \cdots, w'_{i-1}, 1, 1, w'_{i+2}, \cdots, w'_n) \cong (w'_1, \cdots, w'_{i-1}+w'_{i+2}-1, \cdots, w'_n)$.
                \item $(w'_1, \cdots, w'_i, 0, -1, -1, \cdots, -1, 0, w'_j, \cdots, w'_n) \cong (w'_1, \cdots, w'_i+1, 2, 2, \cdots, 2, w'_j+1, \cdots, w'_n)$.
                \item $(0, -1, -1, \cdots, -1)\cong(3, 2, 2, \cdots, 2)$.
                \item $(w'_1, \cdots, w'_i, 1, 2, 2, \cdots, 2, 1, w'_j, \cdots, w'_n) \cong (w'_1, \cdots, w'_i-1, -1, -1, \cdots, -1, w'_j-1, \cdots, w'_n)$.
                \item $(1, 2, 2, \cdots, 2)\cong(-2, -1, -1, \cdots, -1)$.
        \end{enumerate}
        Note that, in all the cases, the length is shortened by 1.
\end{lemma}

\begin{example}
        For example, 
        \begin{align*}
                (2, 2, 2, 1, 2, 2)&\cong(2, 2, \bm{1}, \bm{0}, \bm{1}, 2)\\
                &\cong(2, \bm{1}, \bm{0}, \bm{-1}, 1, 2)\\
                &\cong(\bm{1}, \bm{0}, \bm{-1}, -1, 1, 2)\\
                &\cong(\bm{0}, \bm{-1}, -1, -1, 1, \bm{1})\\
                &\cong(\bm{-1}, -1, -1, -1, \bm{0}, \bm{0})\\
                &\cong(\bm{-2}, -1, -1).
        \end{align*}
\end{example}

\begin{defn}
        A loop word $w'$ is said to be {\em  weak-normal} if there is no subword of the form $(0, 0)$, $(1, 1)$, $(0, -1, \cdots, -1, 0)$ and $(1, 2, \cdots, 2, 1)$. In particular, the words $(2, 2, \cdots, 2, 1, 2, \cdots, 2, 2)$ and $(0, -1, \cdots, -1)$ are not weak-normal.
\end{defn}

\begin{prop}\label{prop:WeakNormal}
        Any essential loop word is equivalent to some weak-normal loop word.
\end{prop}
\begin{proof}
        Given a loop word $w' = (w'_1, \cdots, w'_n)$, find a subword of the form $(0, 0), (1, 1), (0, -1, \cdots, -1, 0)$ and $(1, 2, \cdots, 2, 1)$. If there is no such a subword, then the word is already weak-normal. If there is, remove or replace it by using Lemma \ref{lem:MainLemmaForNormalLoopWord}. Since this process decreases length of the word, we can apply this only a finite number of times. In particular, the resulting word is weak-normal.
\end{proof}

Note that for a weak-normal sequence $w' = (w'_1, \cdots, w'_n)$, $w'_i=0$ implies $w'_{i-1}\neq 0$ and  $w'_{i+1}\neq 0$ by definition.

\begin{defn}
        A loop word $w' = (w'_1, \cdots, w'_n)$ is said to be {\em  almost-normal} if it satisfies the following conditions : 
        \begin{itemize}
                \item the loop word $w'$ is weak-normal,
                \item any subword of the form $(a, 0, b)$ in $w'$ satisfies $a\leq -1, b\geq 1$, or $a \geq 1, b \leq -1$ or $a, b\geq 1$.
        \end{itemize}
\end{defn}

In order to show that any loop word is equivalent to some almost-normal loop word, let us introduce an operation which preserves weak-normality.
\begin{lemma}
        Suppose that $(w'_1, \cdots, w'_n)$ is a weak-normal loop word such that $w'_i=0$. Then the induced loop word $(w'_1, \cdots, w'_{i-1}+1, 1, w'_{i+1}+1, \cdots, w'_n)$ is still weak-normal.
\end{lemma}
\begin{proof}
        First note that $w'_{i-1}, w'_{i+1}$ are negative by  weak-normality. Suppose that the induced loop word has a subword of the form $(1, 1)$ or $(1, 2, \cdots, 2, 1)$. Since $w'_{i-1}+1, w'_{i+1}+1\leq 0$, the subword should be contained in $(w'_{i+2}, \cdots, w'_{i-2})$. However, this contradicts that the original word is weak-normal.
        
        Now suppose that the induced loop word has a subword of the form $(0, -1, \cdots, -1, 0)$. It cannot be contained in$(w'_{i+2}, \cdots, w'_{i-2})$. Hence either $w'_{i-1}+1=0$ or $w'_{i+1}=0$ or both. In either case, original word will contain a subword of the form $(0, -1, \cdots, -1, 0)$, which contradicts weak-normality of the original word.
        
        Finally, Suppose that the induced loop word has a subword of the form $(0, 0)$. This only happens when $(w'_{i-2}, w'_{i-1}+1)=(0, 0)$ or $(w'_{i+1}+1, w'_{i+2}) = (0, 0)$. This implies that $(w'_{i-2}, w'_{i-1}, w'_i) = (0, -1, 0)$ or $(w'_i, w'_{i+1}, w'_{i+2}) = (0, -1, 0)$. Both cases contradict weak-normality of the original loop word $w'$.
\end{proof}

\begin{prop}
        Any essential weak-normal loop word is equivalent to some almost-normal loop word.
\end{prop}
\begin{proof}
        Given a weak-normal loop word $w'=(w'_1, \cdots, w'_n)$, find a subword of the form $(a, 0, b)$ such that $a, b\leq -1$. If there is no such a subword, then the word is already almost-normal. If there is, replace the subword with $(a+1, 1, b+1)$. Let $w''$ be the resulting loop word. $w''$ is still weak-normal by the previous lemma. Also, the number of nonpositive entries of $w''$ is less than that of $w'$. Thus we can apply this procedure only a finite number of times and get an equivalent almost-normal loop word.
\end{proof}

Let us introduce an operation which preserves the almost-normality.
\begin{lemma}
        Suppose that $w' = (w'_1, \cdots, w'_n)$ is an almost-normal loop word such that $w'_i=1$. Then the induced loop word $w'' = (w'_1, \cdots, w'_{i-1}-1, 0, w'_{i+1}-1, \cdots, w'_n)$ is still almost-normal.
\end{lemma}
\begin{proof}
        First note that almost-normality implies  either $w'_{i-1}\geq 2$ or $w'_{i+1}\geq 2$. We will show that $w''$ weak-normal first. Assume that it contains $(1, 1)$. Then $(w'_{i-2}, w'_{i-1}-1) = (1, 1)$ or $(w'_{i+1}-1, w'_{i+2}) = (1, 1)$. However, this implies $(w'_{i-2}, w'_{i-1}, w'_i) = (1, 2, 1)$ or $(w'_i, w'_{i+1}, w'_{i+2}) = (1, 2, 1)$, which contradicts that $w'$ is almost-normal.
        
        Now assume $(0, 0)$ is in $w''$. Since neither $w'_{i-1}-1=0$ nor $w'_{i+1}-1=0$, this implies that $(0, 0)$ is contained in $w'$, which is impossible.
        
        Assume that $(0, -1, \cdots, -1, 0)$ is contained in $w''$. Then the first or the last $0$ should be the $i$-th entry of $w''$. This implies $w'$ contains $(1, 0, -1, \cdots, -1, 0)$ or $(0, -1, -1, \cdots, -1, 0, 1)$, which contradicts weak-normality of $w'$.
        
        Assume that $(1, 2, \cdots, 2, 1)$ is in $w''$. Then the first or the last $1$ should be $w'_{i-1}-1$ or $w'_{i+1}-1$. Then $(1, 2, \cdots, 2, 1)$ should be in the original loop word $w'$. This proves the induced loop word $w''$ is weak-normal.
        
        Now show that $w''$ is almost normal. If it has a subword of the form $(a, 0, b)$ with $a, b\leq -1$, then the $0$ should be $w'_{i-2}$ or $w'_{i+2}$. This implies that $w'$ has a subword of the form $(a+1, 0, b)$ or $(a, 0, b+1)$. Since $w'$ is almost-normal, $a+1\leq 0$ and $b+1\neq 0$. However, since $a, b\leq -1$, we have $a+1, b+1 \leq -1$. This contradicts almost-normality of $w'$. Therefore the induced loop word $w''$ is almost-normal.
\end{proof}

\begin{prop}\label{prop:NormalExistence}
        Any essential loop word is equivalent to some normal loop word.
\end{prop}
\begin{proof}
        Given an almost-normal loop word $w'=(w'_1, \cdots, w'_n)$, find a subword of the form $(a, 1, b)$ such that $a \geq 2$ or $b\geq 2$. If there is no such a subword, then the word is already normal. If there is, replace the subword $(a, 1, b)$ with $(a-1, 0, b-1)$. Let $w''$ be the resulting loop word. By the lemma, $w''$ is still almost-normal. Also, the number of positive entries of $w''$ is less than that of $w'$. Thus we can apply this procedure only a finite number of times and we get an equivalent normal loop word.
\end{proof}

\subsection{Uniqueness}\label{subsec:unique}

In this subsection, we prove the uniqueness part of Proposition \ref{prop:normalnormal}. Here, we use the notation $(l'_1, \cdots, n'_\tau)$ for a loop word.

Recall that for a loop word $w'=(l'_1, \cdots, n'_\tau)$, we defined $L\left[w'\right] = \alpha^{l'_1}\beta^{m'_1}\gamma^{n'_1}\cdots\alpha^{l'_\tau}\beta^{m'_\tau}\gamma^{n'_\tau}$ as an element of $\left[S^1, \cpp\right]$. Now denote by $L'\left[w'\right]$ an element $\alpha^{l'_1}\beta^{m'_1}\gamma^{n'_1}\cdots\alpha^{l'_\tau}\beta^{m'_\tau}\gamma^{n'_\tau}$ in the group $\left<\alpha, \beta, \gamma | \alpha\beta\gamma = 1\right>$. We show the uniqueness without regarding shifting first.
%

\begin{lemma}\label{lem:WordWord}
        For two normal sequences $w' = (l'_1, \cdots, n'_\tau)$ and $(l''_1, \cdots, n''_{\sigma})$, if $L'(l'_1, \cdots, n'_\tau)=L'(l''_1, \cdots, n''_{\sigma})$, then $l_1=l''_1$.
\end{lemma}

\begin{proof}
        We prove the lemma by comparing reduced forms of $L'\left[w'\right]$ and $L'\left[w''\right]$ in the free group $\left<\alpha, \gamma\right>$, where $\beta$ is identified with $\alpha^{-1}\gamma^{-1}$. Let $w=(l_1, m_1, n_1, \cdots, l_r, m_r, n_r)$ be a normal loop word. Then, in $\left<\alpha, \gamma\right>$, we have $$L'\left[w\right] = \alpha^{l_1}(\alpha^{-1}\gamma^{-1})^{m_1}\gamma^{n_1}\cdots\alpha^{l_r}(\alpha^{-1}\gamma^{-1})^{m_r}\gamma^{n_r}.$$ Let us compute the reduced form of the simplest word $\alpha^l\beta^m\gamma^n$ and see how the general word $L'\left[w\right]$ is reduced.
        
        The reduced form of a word $\alpha^l\beta^m\gamma^n$ is computed as follows : 
        \begin{itemize}
                \item If $m\geq 2$, then $\alpha^l\beta^m\gamma^n = \alpha^{l-1}(\gamma^{-1}\alpha^{-1})^{m-1}\gamma^{n-1}$. Note that in this case $l, n \neq 1$.
                \item If $m=1$, then $\alpha^l\beta\gamma^n = \alpha^{l-1}\gamma^{n-1}$. Note that in this case $l, n\leq 0$.
                \item If $m=0$, $\alpha^l\gamma^n$ is the reduced form and in this case, $l\leq -1, n\geq 1$ or $l\geq 1, n\leq -1$ or $l, n \geq 1$.
                \item If $m\leq -1$, $\alpha^l\beta^m\gamma^n = \alpha^l(\gamma\alpha)^{-m}\gamma^n$. Note that $l$ or $n$ can be zero, but they can't be both zero when $m=-1$.
        \end{itemize}
        
        Note that if $l_{i+1}=0$ and $m_{i+1}\leq -1$, then $n_i\geq 1$. Let define, for a reduced word $W=\alpha^{\mu_1}\gamma^{\nu_1}\cdots\alpha^{\mu_\rho}\gamma^{\nu_\rho}$, where $\mu_1$ can be zero, $\red(W)$ to be $\alpha^{\mu_1}\gamma^{\nu_1}$. Then we have the following : 
        \begin{itemize}
                \item If $m_1\geq 2$, $\red(L'\left[w\right]) = \alpha^{l_1-1}\gamma^{-1}$.
                \item If $m_1=1, l_2=0$, then $m_2\geq 0$ and, $\red(L'\left[w\right]) = \alpha^{l_1-1}\gamma^{n_1-1}$.
                \item If $m_1=1, l_2\neq 0$, $\red(L'\left[w\right]) = \alpha^{l_1-1}\gamma^{n_1-1}$.
                \item If $m_1=0, l_2=0$, then $m_2\leq -1$, $\red(L'\left[w\right]) = \alpha^{l_1}\gamma^{n_1}$.
                \item If $m_1=0, l_2\neq 0$, $\red(L'\left[w\right]) = \alpha^{l_1}\gamma^{n_1}$.
                \item If $m_1\leq -1$, $\red(L'\left[w\right]) = \alpha^{l_1}\gamma$.
        \end{itemize}
        
        Suppose that $l'_1\neq l''_1$. From the above list, we may assume $l'_1-1=l''_1$, $m'_1\geq 1$ and $m''_1\leq 0$. If $m'_1=1$, then both $l'_1-1, n'_1-1$ are negative. However, if $m''_1=0$, then $l''_1, n''_1$ cannot be both negative. Also, $m''_1\leq -1$ cannot be the case. Thus it should be $m'_1\geq 2$. In this case, we have $n''_1=-1$. If otherwise, by the list of reduced forms, it should be $l''_2=0, m''_2=-1$ and $n''_2=0$. Then $w''$ has a subword of the form $(0, -1, 0)$, which contradicts normal condition. Hence we have $n''_1=-1$. In a similar way, one can show that $w''$ has a subword of the form $(0, -1, \cdots, -1, 0)$. Thus we get $l'_1=l''_1$.
\end{proof}

\begin{cor}\label{cor:WordWordWord}
        Let $w'$ and $w''$ be normal loops. If $L'\left[w'\right]=L'\left[w''\right]$, then $w'=w''$.
\end{cor}
\begin{proof}
        Let $w' = (l'_1, m'_1, n'_1, \cdots, l'_\tau, m'_\tau, n'_\tau)$ and $w'' = (l''_1, m''_1, n''_1, \cdots, l''_\sigma, m''_\sigma, n''_\sigma)$. By the lemma, we have $l'_1=l''_1$. Now consider normal loop words $\tilde{w'} = (m'_1, n'_1, \cdots, l'_\tau, m'_\tau, n'_\tau, l'_1)$, $\tilde{w''} = (m''_1, n''_1, \cdots, l''_\sigma, m''_\sigma, n''_\sigma, l''_1)$. Since $l'_1 = l''_1$, we have $L'\left[\tilde{w'}\right] = L'\left[\tilde{w''}\right]$. By the lemma, we also have $m'_1=m''_1$. Similarly, we have $n'_1 = n''_1$ and  $(l'_i, m'_i, n'_i) = (l''_i, m''_i, n''_i)$, for each $i$, which proves the corollary.
\end{proof}

To prove the main result which involves shifting, let us define some notations and a notion of cyclic equivalence.

\setcounter{case}{0}

\begin{defn}
        Let $w$ be a reduced word in the free group $\left<\alpha, \gamma\right>$. We denote by $h(w)$ and $t(w)$ the first and the last alphabet respectively. A {\em  decomposition} of $w$ is two reduced subwords $w^{(1)}, w^{(2)}$ such that $w=w^{(1)}w^{(2)}$ and satisfy the following.
        \begin{itemize}
                \item If $t(w^{(1)})=h(w^{(2)})$, then their exponents are both positive or both negative.
        \end{itemize}
\end{defn}
\begin{example}
        Consider a word $w=\alpha^2\gamma^3\alpha^4\gamma^5$. Then $h(w)=\alpha$, $t(w)=\gamma$. Also, the following are true.
        \begin{itemize}
                \item $(w_1^{(1)} = \alpha^2\gamma, w_1^{(2)}=\gamma^2\alpha^4\gamma^5)$ is a decomposition of $w$. 
                \item $(w_2^{(1)}=\alpha^2\gamma^3, \alpha^4\gamma^5)$ is a decomposition of $w$.
                \item $(w_3^{(1)} = \alpha^3\gamma^5, w_3^{(2)} =\gamma^{-2}\alpha^4\gamma^5)$ is not a decomposition of $w$.
        \end{itemize}
\end{example}

\begin{defn}
        Let $w_1, w_2$ be reduced words in the free group $\left<\alpha, \gamma\right>$. Then $w_1$ is said to be {\em  cyclically equivalent} to $w_2$ if there is a decomposition $(w_1^{(1)}, w_1^{(2)})$ of $w_1$ such that $w_2=w_1^{(2)}w_1^{(1)}$.
\end{defn}
\begin{example}
        Let $w_1=\alpha^2\gamma^3\alpha^4\gamma^{-4}\alpha^{-2}$. Then $w_1$ is cyclically equivalent to $w_2=\alpha^4\gamma^{-1}$ with a decomposition $w_1^{(1)}=\alpha^2\gamma^3\alpha^4\gamma^{-1}$,$ w_1^{(2)}=\gamma^{-3}\alpha^{-2}$. Note that $w_2$ is not cyclically equivalent to $w_1$.\footnote{For a reduced word $w$ of length $k$ and its decomposition $(w^{(1)}, w^{(2)})$, $\text{length}(w^{(1)})+\text{length}(w^{(2)})\leq k+1$. Thus a reduced word $w'$ which is cyclically equivalent to $w$ has a length at most $k+1$.} Thus cyclical equivalency is not a equivalent relation.
\end{example}

\begin{lemma}\label{lem:CyclicAndConjugate}
        Let $w$ be a reduced word and $w'$ be the shortest reduced word among reduced words which are cyclically equivalent to $w$. Then $w'$ is also shortest among conjugacy classes of $w$.
\end{lemma}
\begin{proof}
        First we prove that a reduced word $w$ is cyclically shortest if and only if $h(w)\neq t(w)$. Suppose that $h(w) = t(w)$. We may assume $h(w)=\alpha$ without loss of generality. Then, for $w=\alpha^{m_1}\gamma^{n_1}\cdots\gamma^{n_r}\alpha^{m_{r+1}}$, $(\alpha^{m_1}, \gamma^{n_1}\cdots\gamma^{n_r}\alpha^{m_{r+1}})$ is a decomposition and $w$ is cyclically equivalent to $\gamma^{n_1}\cdots\gamma^{n_r}\alpha^{m_1+m_{r+1}}$, which has smaller length than $w$.
        
        Conversely, suppose that $h(w)\neq t(w)$. Then for any decomposition $(w^{(1)}, w^{(2)})$, $h(w^{(1)})=h(w)\neq t(w)=t(w^{(2)})$. Hence the word $w^{(2)}w^{(1)}$ is reduced. Thus $\len(w^{(1)}w^{(2)})$ is $\len(w)$ or $\len(w)+1$. Hence $w$ is cyclically shortest.
        
        Now assume $w$ is cyclically shortest. Then consider a conjugate word $gwg^{-1}$. Since $t(g)=h(g^{-1})$, $t(g)\neq h(w)$ or $h(g^{-1})\neq t(w)$. We may assume $h(g^{-1})\neq t(w)$. Also assume that $gw$ is not reduced. Then there are two possibilities.
        \begin{itemize}
                \item There are decompositions $(g^{(1)}, g^{(2)})$ of $g$ and $(w^{(1)}, w^{(2)})$ of $w$ such that $g^{(2)}=(w^{(1)})^{-1}$ and $(g^{(1)}, w^{(2)})$ is a decomposition of $g^{(1)}w^{(2)}$.
                \item There is a decomposition $(g^{(1)}, g^{(2)})$ of $g$ such that $g^{(2)}=w^{-1}$.
        \end{itemize}
        In the second case, we have \begin{align*}
                gwg^{-1}&=g^{(1)}g^{(2)}w(g^{(2)})^{-1}(g^{(1)})^{-1}\\
                & = g^{(1)}(g^{(2)})^{-1}(g^{(1)})^{-1}\\
                &=g^{(1)}w(g^{(1)})^{-1}.
        \end{align*}
        Since $\len(g^{(1)})< \len(g)$, the second case is eventually reduced to the first case. In the first case, we have
        \begin{align*}
                \text{length}(gwg^{-1}) &= \text{length}(g^{(1)}w^{(2)}(g^{(2)})^{-1}(g^{(1)})^{-1})\\
                & = \text{length}(g^{(1)}w^{(2)}w^{(1)}(g^{(1)})^{-1}).
        \end{align*}
        If $g^{(1)}=e$, then $\text{length}(gwg^{-1})=\text{length}(w^{(2)}w^{(1)})=\text{length}(w)$. If not, we have 
        \begin{align*}
                \text{length}(gwg^{-1}) & = \text{length}(g^{(1)}w^{(2)}w^{(1)}(g^{(1)})^{-1}) \\
                &\geq 2\text{length}(g^{(1)})+\text{length}(w)-2        \\
                &\geq \text{length}(w).
        \end{align*}
        Thus, $w$ is the shortest word in its conjugacy class.
\end{proof}

\begin{cor}
        Let $w_1, w_2$ be reduced words in the free group $\left<\alpha, \gamma\right>$. Suppose that $w_1$ and $w_2$ are conjugate, that is, there is some $g\in\left<\alpha, \gamma\right>$ such that $gw_1g^{-1}=w_2$. Then there is the third reduced word $w_3$ such that both $w_1$ and $w_2$ are cyclically equivalent to $w_3$.
\end{cor}
\begin{proof}
                Note that in the proof of Lemma \ref{lem:CyclicAndConjugate}, we also proved that in a conjugacy class the shortest words are cyclically equivalent. Since $w_1$ and $w_2$ are in the same conjugacy class, both are cyclically equivalent to the shortest word in the conjugate class.
\end{proof}

Now consider a loop word $w'=(l'_1, m'_1, n'_1, \cdots, l'_\tau, m'_\tau, n'_\tau)$. Then $L'\left[w'\right] = \alpha^{l'_1}\beta^{m'_1}\gamma^{n'_1}\cdots\alpha^{l'_\tau}\beta^{m'_\tau}\gamma^{n'_\tau}$ may not be cyclically minimal. There are two cases. 
\begin{itemize}
        \item $h(L'\left[w'\right])=t(L'\left[w'\right])=\alpha$. By the first list in the proof of Lemma \ref{lem:WordWord}, $t(L'\left[w'\right])=\alpha$ implies $n'_\tau=0$ and $m'_\tau\leq-1$. In this case, the last two alphabets of $L'\left[w'\right]$ is $\gamma\alpha$ and $l'_1\geq 1$. Therefore $\alpha L'\left[w'\right] \alpha^{-1}$ is cyclically minimal.
        \item $h(L'\left[w'\right])=t(L'\left[w'\right])=\gamma$. As in the first case, $h(L'\left[w'\right])=\gamma$ implies $l'_1=0$ and $m'_1\leq -1$. In this case, the first two alphabets of $L'\left[w'\right]$ is $\gamma\alpha$ and $n'_\tau\geq 1$. Thus $\gamma^{-1} L'\left[w'\right] \gamma$ is cyclically minimal.
\end{itemize}

Now assume that for two loop words $w' = (l'_1, m'_1, n'_1, \cdots, l'_\tau, m'_\tau, n'_\tau)$ and $w''= (l''_1, m''_1, n''_1, \cdots, l''_\sigma, m''_\sigma, n''_\sigma)$, $L\left[w'\right]=L\left[w''\right]$. Then, by the above argument, we know $L'\left[w'\right]$ and $L'\left[w''\right]$ are cyclically equivalent. Now we prove the main theorem.

\begin{prop}\label{prop:NormalUniqueness}
        Let $w'$ and $w''$ be loop words. If $L\left[w'\right]=L\left[w''\right]$, then $w''=(w')^{(k)}$ for some $k\in\bzz$.
\end{prop}
\begin{proof}
        Since $L'\left[w'\right]$ and $L'\left[w''\right]$ are cyclically equivalent, $$L'\left[w''\right]=(\text{a part of } \alpha^{l'_k}\beta^{m'_k}\gamma^{n'_k})^{-1}L'((w')^{(k-1)})(\text{a part of } \alpha^{l'_k}\beta^{m'_k}\gamma^{n'_k}),$$
        for some $k\in\bzz$. We may assume $k=1$. By replacing $w'$ and $w''$ by $w'^{(\frac{1}{3})}$ and $w''^{(\frac{1}{3})}$ or $w'^{(\frac{2}{3})}$ and $w''^{(\frac{2}{3})}$, we may also assume that $$L'\left[w''\right]=(\text{a part of } \alpha^{l'_1})^{-1}L'((w'))(\text{a part of } \alpha^{l'_1}).$$ Here, we define the word $w^{(\frac{j}{3})}$ for a loop word $w=(w_1, \cdots, w_\nu)$ to be $(w_{1+j}, \cdots, w_{\nu+j})$, where the indices are regarded as elements in $\bzz_\nu$.
        
        Note that for a loop word $w$, by the first list in the proof of Lemma \ref{lem:WordWord}, if $t(L'\left[w\right])=\alpha$ then the exponent of it is $1$. Thus $L'\left[w''\right] = L'\left[w'\right]$ or $L'\left[w''\right] = \alpha^{-1}L'\left[w'\right]\alpha$. Suppose that the latter holds, so that $t(L'\left[w''\right])=\alpha$. Then right end is $\gamma\alpha$ and $L'\left[w'\right]$ ends in $\gamma$. This is the case when $m'_\tau\geq 2$ and $n'_\tau = 2$ or $m'_\tau\leq 0$ and $n'_\tau=1$. If $n'_\tau=1$, then $l'_1\leq 0$. However, since $n''_\sigma=0$ and $m''_\sigma\leq -1$, $l''_1\geq 1$. Therefore $L'\left[w''\right]=\alpha^{-1}L'\left[w'\right]\alpha$ cannot hold. Hence $m'_\tau\geq 2$ and $n'_\tau = 2$.
        
        Then we get a contradiction as in the end of the proof of Lemma \ref{lem:WordWord}. For example, since $L'\left[w'\right]$ ends in $\alpha^{-1}\gamma$, it should be $m''_\sigma==l''_\sigma=-1$. In a similar way, by using normal condition and comparing two words one can deduce that $L'\left[w''\right]$ has a subword of the form $(0, -1, \cdots, -1, 0)$. This proves $L'\left[w''\right]=\alpha^{-1}L'\left[w'\right]\alpha$ cannot hold and $L'\left[w'\right]=L'\left[w''\right]$. Therefore the proposition follows from Corollary \ref{cor:WordWordWord}.
\end{proof}

\begin{remark}\label{rem:MinimalityOfNormalLoopWord}
        Since the reducing process does not increase length of loop words, together with the uniqueness, the normal representative has minimal length in its equivalent class.
\end{remark}

\begin{cor}\label{cor:NormalEssential}
        If a loop word has a normal representative in its equivalent class, then it is essential.
\end{cor}
\begin{proof}
        If the normal representative has length larger than $3$, then it cannot be equivalent to loop words of the form $(l, 0, 0), (0, m, 0)$, or $(0, 0, n)$. If the normal representative has length $3$, say $(l', m', n')$. Then that this is essential have been proved already in the Lemma \ref{lem:normal}.
\end{proof}

%
%


\section{$T=1$ Substitution and Homotopy Invariance} \label{sec:T=1}

In this section, we prove Theorem \ref{thm:T=1Homotopy} by dividing it into two Propositions \ref{prop:T=1}
and \ref{prop:Homotopy}.

The first proposition is on the `\emph{finiteness}' of the mirror matrix factorizations. Let us first
remark that for exact Lagrangians in exact symplectic manifolds, their Fukaya category can be
defined over $\C$. Punctured Riemann surfaces are exact symplectic manifolds, but the immersed loops
$L$ that we consider are not exact, hence a priori, we need to work on the Novikov field
$$\Lambda = \left\{ \left.\sum_{i=0}^\infty a_i T^{\lambda_i} \right| a_i \in \C, \lim_{i \to \infty} \lambda_i = \infty \right\}$$ 
to define
their localized mirror functor images $\LocalF\left(L\right) = \left(\LocalPhi\left(L\right),\LocalPsi\left(L\right)\right)$.
This means, for an arbitrary unobstructed loop $L$, each entry of $\LocalPhi\left(L\right)$ or $\LocalPsi\left(L\right)$
is an element of $\Lambda[[x,y,z]]$ and each coefficient of a monomial (consisting of $x$, $y$ and $z$)
in it is an element of $\Lambda$. 
In this case, the evaluation $T=1$ may not make sense for general infinite sums in $\Lambda$.

\begin{defn}
We say that $\LocalF\left(L\right) = \left(\LocalPhi\left(L\right), \LocalPsi\left(L\right)\right)$
is \emph{finite} if each coefficient (in $\Lambda$) of a monomial in any entry of $\LocalPhi\left(L\right)$ and $\LocalPsi\left(L\right)$
is a finite sum.
In this case, we can make $T=1$ substitution in $\LocalF\left(L\right)$ to get a matrix factorization
of $xyz$ over $\mathbb{C}$, or an object in $\MF\left(xyz\right)$.
\end{defn}

%

\begin{prop}\label{prop:T=1}
For a regular immersed loop $L$, its mirror matrix factorization 
$\LocalF\left(L\right)$
is finite.
\end{prop}

\begin{remark}\label{rmk:Regularity}
It is not hard to show the converse statement that an unobstructed Lagrangian $L$ with $\LocalF\left(L\right)$ finite can be
shown to be a regular immersed loop (up to Hamiltonian isotopy).
\end{remark}

Now let's discuss the second proposition on the `\emph{homotopy invariance}' of the localized mirror functor on regular loops.
It was already shown in \cite{CHL-toric} Proposition 5.4 that the localized mirror functor
takes  Hamiltonian equivalence of the Fukaya category  to the homotopy equivalence of
the DG-category  matrix factorization (over $\Lambda$).
But a homotopy between immersed loops that we consider here, are
more general than a Hamiltonian isotopy, and also we want to
establish such homotopy equivalence over $\C$.

\begin{prop}\label{prop:Homotopy}
If two regular immersed loops $L$ and $L'$ are freely homotopic to each other and have the same total
holonomy, then their mirror matrix
factorizations $\LocalF\left(L\right) = \left(\LocalPhi\left(L\right),\LocalPsi\left(L\right)\right)$
and $\LocalF\left(L'\right) = \left(\LocalPhi\left(L'\right),\LocalPsi\left(L'\right)\right)$
over $\mathbb{C}$ are homotopically equivalent to each other in $\operatorname{MF}\left(xyz\right)$.
\end{prop}

An overview of the proof of Propositions \ref{prop:T=1} and \ref{prop:Homotopy} is as follows. As we set $T=1$, we only concern about the signed count of holomorphic polygons with holonomy contributions, not about the area of such polygons.
Therefore, any homotopy of a Lagrangian $L$, which does not change the intersection and holonomy pattern with $\bL$, provides
exactly the same matrix factorization over $\mathbb{C}$.

Meanwhile, any change of intersection pattern with fixed $\mathbb{L}$ in the same free homotopy class
of $L$ can be made into a composition of $5$ types of homotopy moves described
in Figure \ref{fig:HomotopyMove} (Lemma \ref{lem:5TypesAreEnough}), with some holonomy movements.
But here we focus only on homotopy moves, as holonomy moves are easier to deal with. 
(Actually, it was already shown in \cite{CHL-toric} Lemma 5.3 that the gauge equivalence of Lagrangians yields an isomorphism in their mirror matrix factorizations.)

\begin{figure}[H]
  \centering
  \begin{subfigure}[t]{0.3\textwidth}
    \includegraphics[height=23mm]{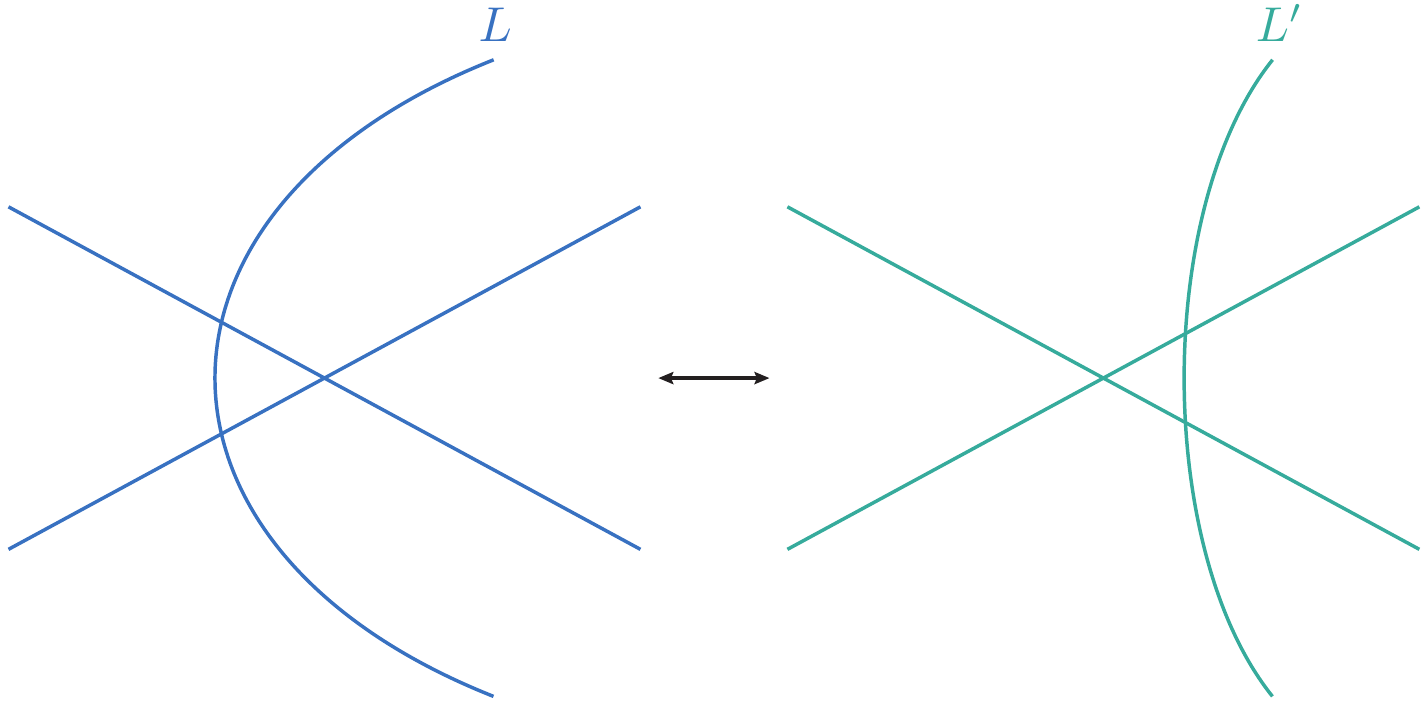}
    \centering
    \caption{Type I}
    \label{fig:HomotopyMoveI}
  \end{subfigure}
  \begin{subfigure}[t]{0.3\textwidth}
    \centering
    \includegraphics[height=23mm]{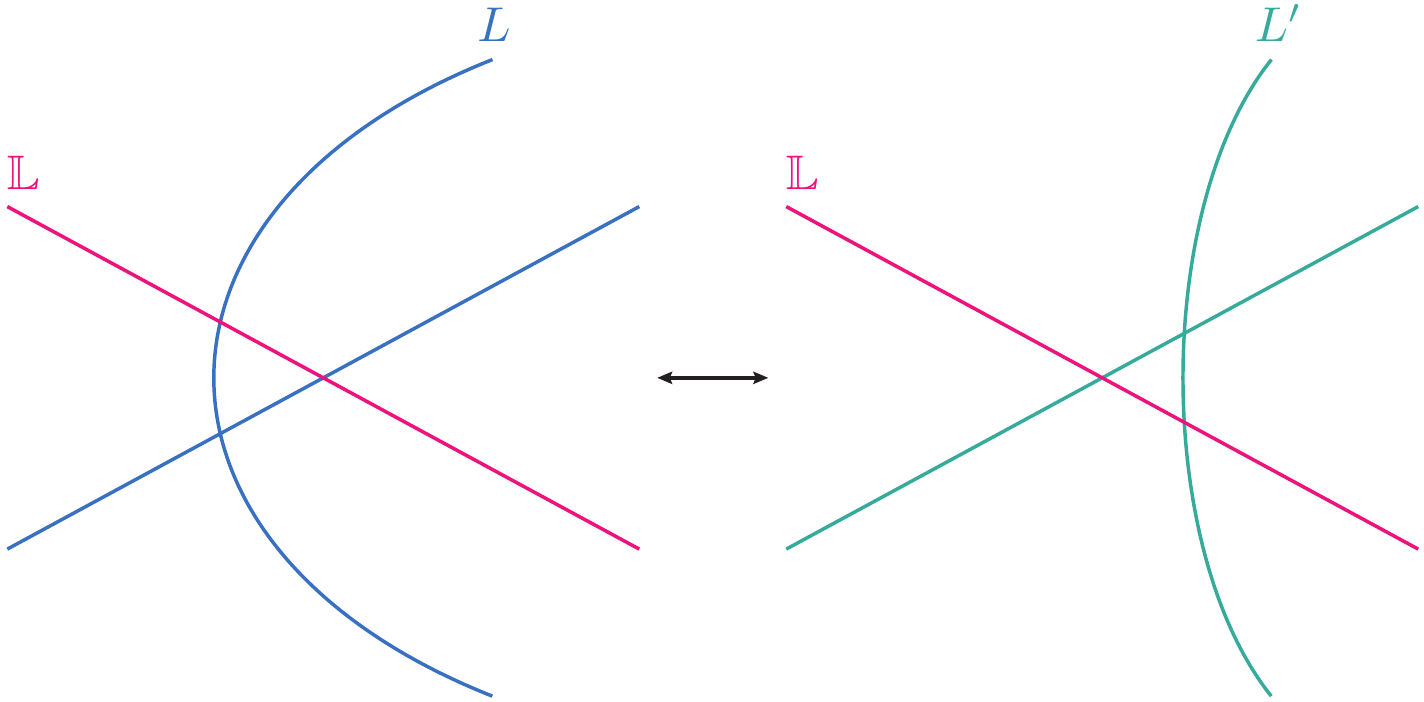}
    \caption{Type II}
    \label{fig:HomotopyMoveII}
  \end{subfigure}
  \begin{subfigure}[t]{0.3\textwidth}
    \centering
    \includegraphics[height=23mm]{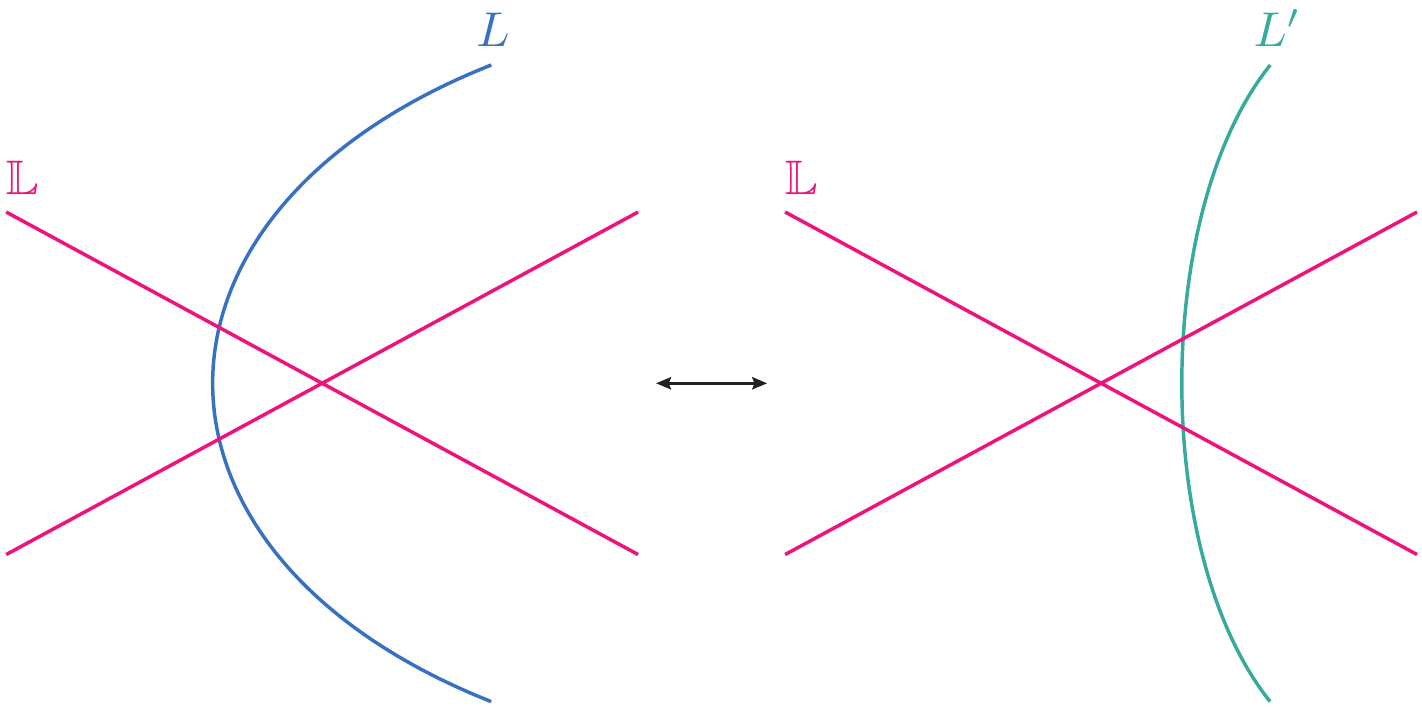}
    \caption{Type III}
    \label{fig:HomotopyMoveIII}
  \end{subfigure}
  \\[5mm]
  \begin{subfigure}[t]{0.3\textwidth}
    \centering
    \includegraphics[height=23mm]{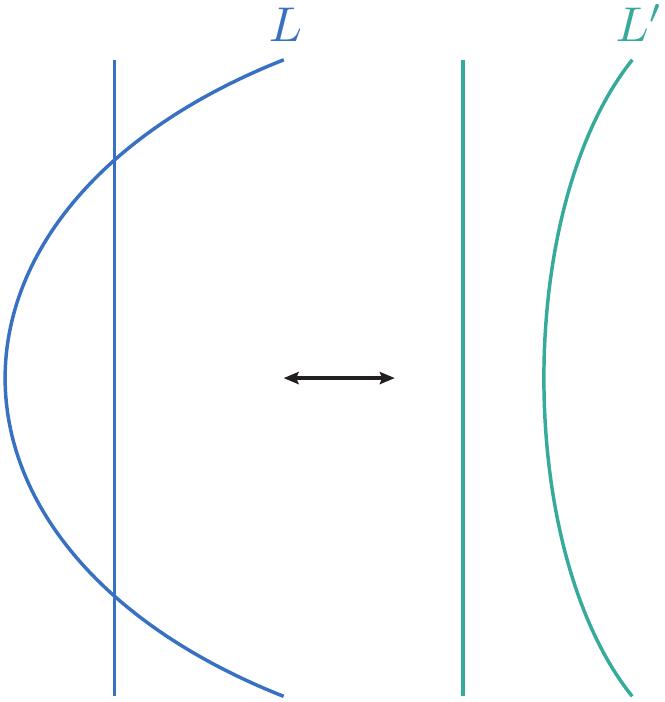}
    \caption{Type IV}
    \label{fig:HomotopyMoveIV}
  \end{subfigure}
  \begin{subfigure}[t]{0.3\textwidth}
    \centering
    \includegraphics[height=23mm]{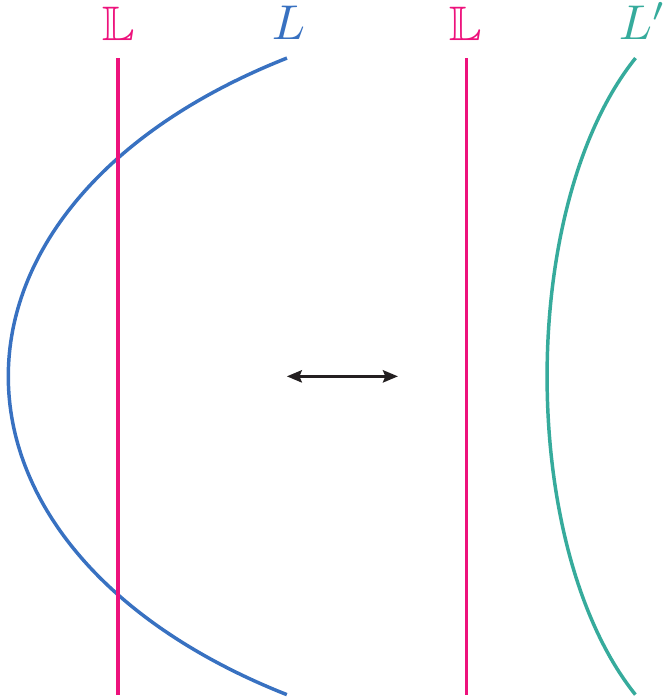}
    \caption{Type V}
    \label{fig:HomotopyMoveV}
  \end{subfigure}
  \caption{$5$ types of homotopy moves: $L$ is green and $\mathbb{L}$ is red}
  \label{fig:HomotopyMove}
\end{figure}

The type I, II or III is when $L$ slides across a self-intersection of $L$, an intersection of $L$ and $\mathbb{L}$ or a self-intersection of $\mathbb{L}$, respectively.
Later we will further divide type III into subcases III-i to III-iv as illustrated in Figure \ref{fig:HomotopyMoveIIIc},
according to the orientation patterns given to the curves, because they give different results.
The type IV or V is when $L$ unpoke out or poke in a part of $L$ or $\mathbb{L}$, respectively.
We sometimes denote by IV$^{+}$[IV$^{-}$] and V$^{+}$[V$^{-}$] for moves from the right[left] to the left[right] for convenience.

Then we look at how matrix factorizations are related under those moves.
Type I, II, III-i, III-ii and IV  moves don't change the resulting matrix factorizations at all, regardless of orientations (unless specified) on curves.
In type III-iii and III-iv moves, however, they are no longer the same, but are still isomorphic to each other by some row or column operations.
Especially, the finiteness of the matrix factorizations is preserved under type I, II, III and IV moves
(Lemma \ref{lem:HomotopyTypeIII}).

In type V moves, the resulting matrix factorizations have different sizes so they can no longer be isomorphic in $\operatorname{MF}\left(xyz\right)$.
However, we show that if $L'$ is regular and $\LocalF\left(L\right)$ is finite, then $\LocalF\left(L'\right)$ is also finite and completely determined by $\LocalF\left(L\right)$, and they are homotopic to each other in $\operatorname{MF}\left(xyz\right)$ (Lemma \ref{lem:HomotopyTypeV}).

Finally, to show Proposition \ref{prop:T=1}, we adopt the notion of `\emph{admissibility}' introduced in the paper of Azam and Blanchet \cite{AB22} and show that $\LocalF\left(L\right)$ is finite for any admissible loop $L$ (Definition \ref{defn:Admissibility}
and Corollary \ref{cor:AdmissibleT=1}).
Next we show that any regular loop $L'$ can be deformed to an admissible loop $L$ using finitely many homotopy (poking) moves of type V (Lemma \ref{lem:HomotopicToAdmissible}).
Then we deduce the finiteness of $\LocalF\left(L'\right)$ from the finiteness of $\LocalF\left(L\right)$ using Lemmas \ref{lem:HomotopyTypeIII}
and \ref{lem:HomotopyTypeV} inductively.

After the finiteness is insured for regular loops, the proof of Proposition \ref{prop:Homotopy} follows automatically from Lemmas \ref{lem:5TypesAreEnough},
\ref{lem:HomotopyTypeIII} and \ref{lem:HomotopyTypeV}.


\subsection{Matrix Transformations under Homotopy Moves}
In this subsection, we find the relations between matrix factorizations
before and after the homotopy moves.
In the process, we will have  `\emph{row/column operation}' according to    a sliding of a curve in Lemma \ref{lem:HomotopyTypeIII}.(2) and the `\emph{matrix reduction}' according to removing a bigon in Lemma \ref{lem:HomotopyTypeV}. In each type, we derive explicit isomorphisms or homotopies.
They are very useful for practical calculations as well as establishing finiteness or homotopy equivalence over $\mathbb{C}$.

\begin{lemma}\label{lem:5TypesAreEnough}
If two unobstructed loops $L$ and $L'$ are freely homotopic to each other,
the intersection pattern of $L$ with $\mathbb{L}$ can be transformed to that of $L'$ with $\mathbb{L}$
(and vice versa) by a finite composition of type I to V homotopy moves.
Moreover, if $L$ and $L'$ are regular, we can also choose all loops that appear in the process as regular loops.
\end{lemma}

\begin{remark}
        \begin{enumerate}
                \item Except the regularity condition, the assertion is just a classical Reidemeister move type proposition. So we only explain how to construct a homotopy sequence which preserves regularity.
                \item To prove the lemma, we need the notion of `strong admissibility' which will be introduced in Section \ref{sec:Admissibility}. So we will give the proof after proving Lemma \ref{lem:HomotopicToAdmissible}.
                
        \end{enumerate}
\end{remark}

Next we seek how the matrix factorizations transform under each type of homotopy moves.
We take the strategy to minimize direct disk countings and find non-trivial ones by $A_\infty$-relations.
It will be accomplished by finding isomorphisms in the Fukaya category between loops before and after the
move.

\begin{lemma}\label{lem:HomotopyTypeIII}
Let two unobstructed loops $L$ and $L'$ can be made into each other by one of type I to IV moves.
Then their mirror matrix factorizations
$\LocalF\left(L\right) = \left(\LocalPhi\left(L\right),\LocalPsi\left(L\right)\right)$
and $\LocalF\left(L'\right) = \left(\LocalPhi\left(L'\right),\LocalPsi\left(L'\right)\right)$ are determined
by each other by the following rules. Especially, if one is finite, so is the other.

(1) In type I, II, III-i, III-ii and IV cases, $\LocalF\left(L\right)$ and $\LocalF\left(L'\right)$ are
exactly the same.

(2) In a type III-iii case, let $\chi$ be one of $x$, $y$ or $z$ according to whether $b_0$ is $X$, $Y$ or $Z$ (in Figure \ref{fig:HomotopyMoveIII-iii}). Then $\LocalPhi\left(L'\right)$ can be obtained
from $\LocalPhi\left(L\right)$ by adding $\chi$ times the column corresponding to $p_{k-1}$ to the column corresponding to $p_k$, and $\LocalPsi\left(L'\right)$ can be obtained from $\LocalPsi\left(L\right)$
by subtracting $\chi$ times the row corresponding to $p_k$ from the row corresponding to $p_{k-1}$.
In a type III-iv case, exactly the same holds if we exchange $\LocalPhi \leftrightarrow \LocalPsi$ and
$p_i \leftrightarrow s_i$.

%

\end{lemma}

\begin{proof}
We restrict to a type III-iii case, but most part of the proof works also for the other cases.

We can assume that $L'$ is a $C^0$-small perturbation of $L$ as shown in Figure \ref{fig:HomotopyMoveIII-iii}.
Namely, they have two intersections near the self-intersection $b_0\in\left\{X,Y,Z\right\}\subset\Hom^1\left(\mathbb{L},\mathbb{L}\right)$
of $\mathbb{L}$, denoted by $f\in\Hom^0\left(L',L\right)$ and $g\in\Hom^0\left(L,L'\right)$, so that
$b_0$ is bounded by the `\emph{small}' embedded bigon made by a perturbation $L'$ of $L$ between vertices $f$ and $g$. We further assume
that the remaining
part of $L$ is perturbed to $L'$ in the opposite direction, which is possible as the underlying surface
$\mathcal{P}$ is orientable. Therefore, we have one more `\emph{long}' immersed bigon between $f$ and $g$, only the ends of which are shown in the figure and the middle part of which might intersect with itself.

\begin{figure}[H]
  \centering
  \begin{subfigure}[t]{0.24\textwidth}
    \centering
    \includegraphics[height=40mm]{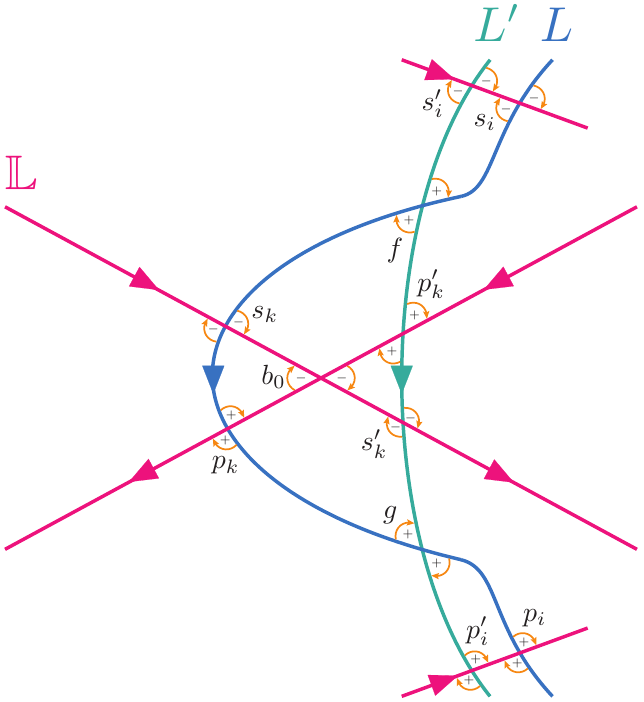}
    \caption{Type III-i}
    \label{fig:HomotopyMoveIII-i}
  \end{subfigure}
  \begin{subfigure}[t]{0.24\textwidth}
    \centering
    \includegraphics[height=40mm]{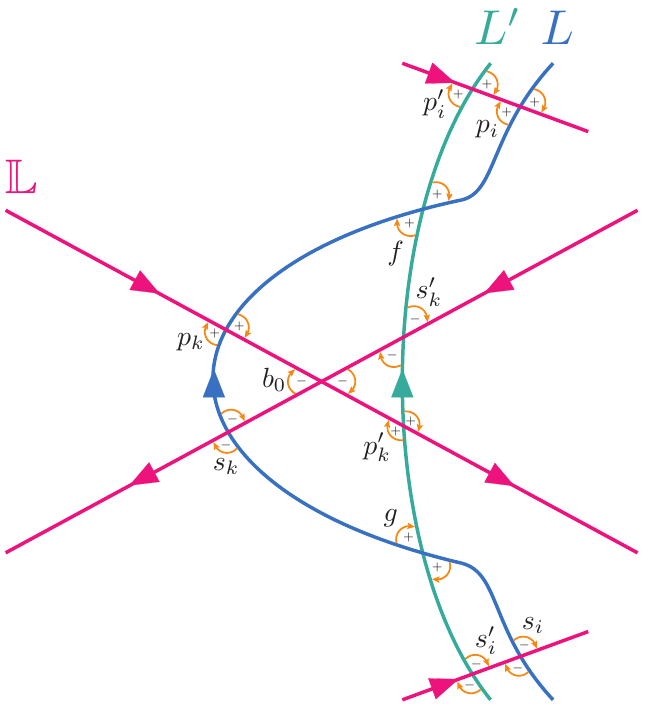}
    \caption{Type III-ii}
    \label{fig:HomotopyMoveIII-ii}
  \end{subfigure}
  \begin{subfigure}[t]{0.24\textwidth}
    \centering
    \includegraphics[height=40mm]{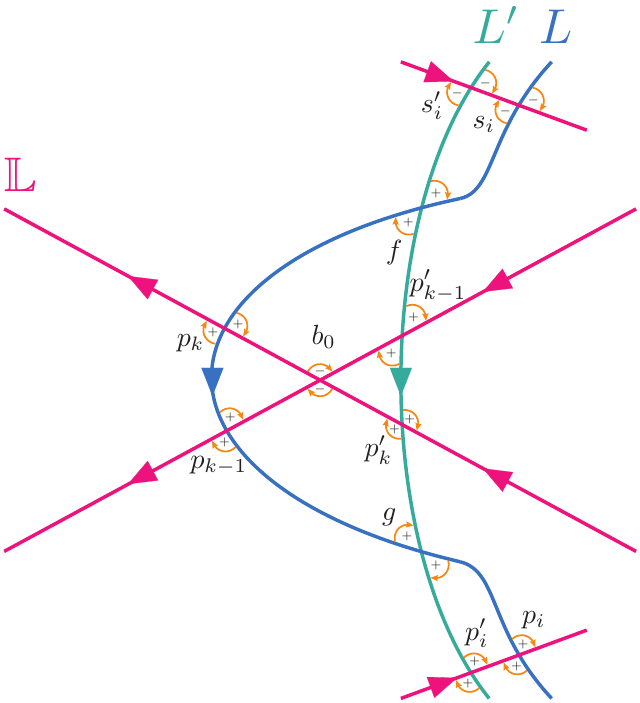}
    \caption{Type III-iii}
    \label{fig:HomotopyMoveIII-iii}
  \end{subfigure}
  \begin{subfigure}[t]{0.24\textwidth}
    \centering
    \includegraphics[height=40mm]{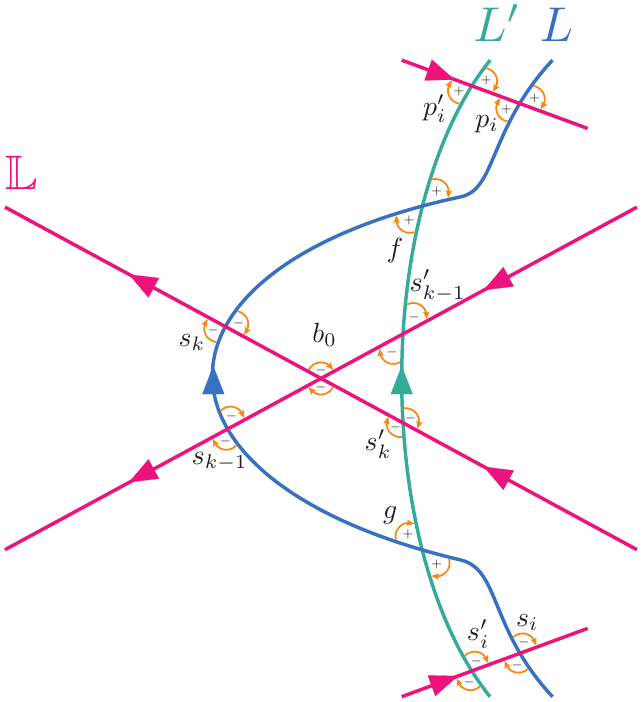}
    \caption{Type III-iv}
    \label{fig:HomotopyMoveIII-iv}
  \end{subfigure}
  \caption{Isomorphisms in type III homotopy moves}
  \label{fig:HomotopyMoveIIIc}
\end{figure}

Then the small and long bigons together give the following $4$ equations:
$$
\operatorm_1\left(f\right) =\overline{g} - \overline{g} = 0,
\quad
\operatorm_1\left(g\right) = -\overline{f} + \overline{f} = 0,
\quad
\operatorm_2\left(f,g\right) = \id_{L'}
\quad
\operatorm_2\left(g,f\right) = \id_{L}.
$$
Here $\overline{g}\in\Hom^1\left(L',L\right)$ in the first equation is the complement of $g$, where the
first and second terms come from the small and long bigons, respectively. These $4$ equations
show that $f$ and $g$ are isomorphisms between $L$ and $L'$ in the $A_\infty$-category $\operatorname{Fuk}\left(\mathcal{P}\right)$
(though we need only the first equation to prove this lemma).

Note that each intersection of $L$ and $\mathbb{L}$ is paired with an intersection of $L'$ and
$\mathbb{L}$. So we can write
$$
\arraycolsep=5pt\def\arraystretch{1}
\begin{array}{ll}
\hspace{1.2mm}
\Hom^0\left(L,\mathbb{L}\right) = \mathbb{C}\left<p_1,\dots,p_k\right>,
&
\hspace{1.2mm}
\Hom^1\left(L,\mathbb{L}\right) = \mathbb{C}\left<s_1,\dots,s_k\right>,
\\[2mm]
\Hom^0\left(L',\mathbb{L}\right) = \mathbb{C}\left<p_1',\dots,p_k'\right>,
&
\Hom^1\left(L',\mathbb{L}\right) = \mathbb{C}\left<s_1',\dots,s_k'\right>,
\end{array}
$$
where $p_{k-1}$, $p_{k-1}'$, $p_k$ and $p_k'$ denote the intersections on the boundary of the small bigon.

As $\LocalF$ is an $A_\infty$-functor, we know from Definition \ref{defn:AinftyCategory}
that
$$
\operatorm_1^{\mathcal{MF}\left(xyz\right)}\left(\LocalF_1\left(f\right)\right)
= \LocalF_1\left(\operatorm_1\left(f\right)\right)
$$
holds. Its left side is calculated as
$$
\left(\LocalPhi\left(L'\right) \circ \LocalPhi_1\left(f\right)
- \LocalPsi_1\left(f\right) \circ \LocalPhi\left(L\right)
,\ \ \
\LocalPsi\left(L'\right) \circ \LocalPsi_1\left(f\right)
- \LocalPhi_1\left(f\right) \circ \LocalPsi\left(L\right)\right)
\in
\Hom^1\left(\LocalF\left(L\right),\LocalF\left(L'\right)\right),
$$
whereas the right side is just $\left(0,0\right)$ as $\operatorm_1\left(f\right)=0$.
This means that the following diagram commutes.
$$
\begin{tikzcd}[arrow style=tikz,>=stealth,row sep=5em,column sep=5em]
\scriptstyle S\left<p_1,\dots,p_k\right>
  \arrow[r, "\LocalPhi\left(L\right)"]
  \arrow[d, "\LocalPhi_1\left(f\right)"]
&
\scriptstyle S\left<s_1,\dots,s_k\right>
  \arrow[r, "\LocalPsi\left(L\right)"]
  \arrow[d, "\LocalPsi_1\left(f\right)"]
&
\scriptstyle S\left<p_1,\dots,p_k\right>
  \arrow[d, "\LocalPhi_1\left(f\right)"]
\\
\scriptstyle S\left<p_1',\dots,p_k'\right>
  \arrow[r, "\LocalPhi\left(L'\right)"]
&
\scriptstyle S\left<s_1',\dots,s_k'\right>
  \arrow[r, "\LocalPsi\left(L'\right)"]
&
\scriptstyle S\left<p_1',\dots,p_k'\right>
\end{tikzcd}
$$

Now we perform some polygon counting containing $f$ to get
$$
\LocalPhi_1\left(f\right)\left(p_i\right)
= \operatorm_2^{0,0,b}\left(f,p_i\right)
=
\left\{
\begin{matrix}
p_i'
\ \ \ \quad\text{if }i=1,\dots,k-1,
\\
-x p_{k-1}' + p_k'
\quad
\text{if }i=k
\end{matrix}
\right.
\quad\text{and}\quad
\LocalPsi_1\left(f\right)\left(s_i\right)
= \operatorm_2^{0,0,b}\left(f,s_i\right)
= s_i' \quad\text{if }i=1,\dots,k,
$$
which are obvious from Figure \ref{fig:HomotopyMoveIII-iii}.
These yield
$$
\LocalPhi_1\left(f\right)
=
\begin{pmatrix}
I_{k-1} & -x \mathbf{e}_{k-1}
\\
0 & 1
\end{pmatrix}
\quad\text{and}\quad
\LocalPsi_1\left(f\right) = I_k
$$
as matrices of maps $S^k\rightarrow S^k$, where $\mathbf{e}_{k-1}$ is the column vector in $S^{k-1}$ whose $(k-1)$-th entry is $1$ and the rest are $0$. This proves the lemma in type III-iii case.
In type I, II, III-i, III-ii and IV cases, we have $\LocalPhi_1\left(f\right) = \LocalPsi_1\left(f\right) = I_k$, and in type III-iv
case, $\LocalPhi_1\left(f\right)$ and $\LocalPsi_1\left(f\right)$ are just swapped, which prove the lemma.
\end{proof}

We next illustrate the `\emph{matrix reduction}' process according to a type V move,
where $L'$ is obtained from $L$ by removing a bigon made by $L$ and $\mathbb{L}$.
As in Figure \ref{fig:HomotopyMoveVc}, we divided type V moves into $4$ sub-cases according to the orientation patterns.
Namely, in type V-i and V-ii cases (resp. V-iii and V-iv cases) the vertices of the bigon reverse (resp. preserve) the orientation of curves.

\begin{lemma}\label{lem:HomotopyTypeV}
Let an unobstructed loop $L'$ can be obtained from an unobstructed loop $L$ by a type V$^{-}$ homotopy
move, that is, by removing a bigon made by $L$ and $\mathbb{L}$.
Suppose that $L'$ is regular and $\LocalF\left(L\right)$ is finite so that written as
$$
\LocalF\left(L\right)
= \left(\LocalPhi\left(L\right),\LocalPsi\left(L\right):S^k\oplus S\rightarrow S^k\oplus S\right)
= \left(\begin{pmatrix} C & D \\ E^T & u \end{pmatrix}, \begin{pmatrix} F & G \\ H^T & v \end{pmatrix}\right),
$$
for some matrices $C$, $F\in S^{k\times k}$, $D$, $E$, $G$, $H\in S^{k \times 1}$ and $u$, $v\in S$,
where the last column and row of each matrix correspond to intersections $p_{k+1}$ or $s_{k+1}$ on the bigon (as in Figure \ref{fig:HomotopyMoveVc}).
Then $u$ (resp. $v$) is a unit in $S$ in type V-i or V-ii cases (resp. V-iii or V-iv cases),
and $\LocalF\left(L\right)$ reduces to the matrix factorization of $L'$ as
$$
\LocalF\left(L'\right)
= \left(\LocalPhi\left(L'\right),\LocalPsi\left(L'\right):S^k\rightarrow S^k\right)
=
\begin{cases}
\left(C-Du^{-1}E^T,\ F\right)
&
\text{in type V-i or V-ii moves, or}
\\
\left(C,\ F-Gv^{-1}H^T\right)
&
\text{in type V-iii or V-iv moves.}
\end{cases}
$$
In particular, $\LocalF\left(L'\right)$ is also finite and completely determined by $\LocalF\left(L\right)$. Moreover, they are homotopically equivalent
(over $\mathbb{C}$) to each other in $\operatorname{MF}\left(xyz\right)$.
\end{lemma}

\begin{proof}
We proceed in a similar way as in the proof of Lemma \ref{lem:HomotopyTypeIII}.
Here we only consider type V-i case.
For the V-ii type, the same can be done with a slight sign considerations.
For the other cases, we can do the same by just switching the role of $\LocalPhi$ and $\LocalPsi$.

\begin{figure}[H]
  \centering
  \begin{subfigure}[t]{0.22\textwidth}
    \centering
    \includegraphics[height=50mm]{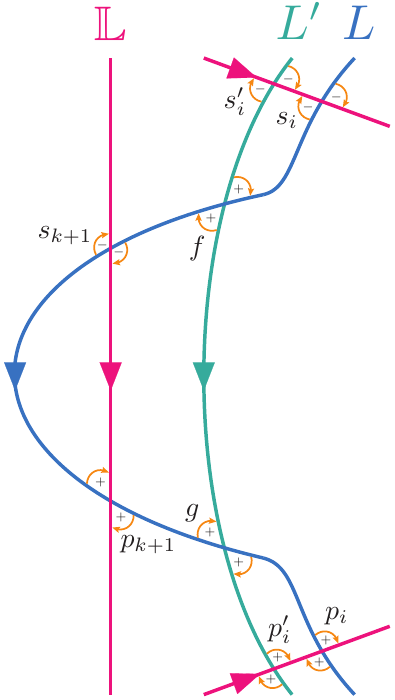}
    \caption{Type V-i}
    \label{fig:HomotopyMoveV-i}
  \end{subfigure}
  \begin{subfigure}[t]{0.22\textwidth}
    \centering
    \includegraphics[height=50mm]{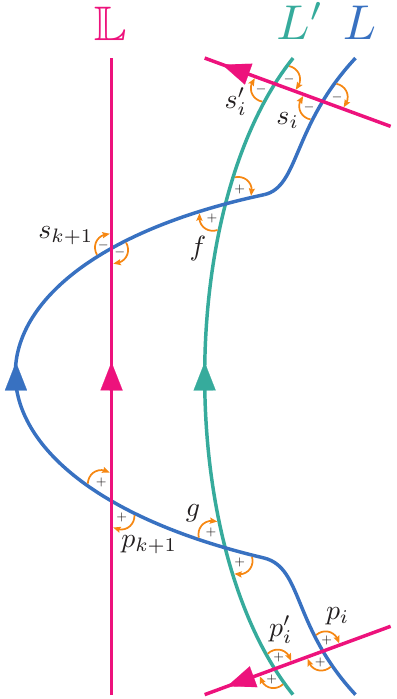}
    \caption{Type V-ii}
    \label{fig:HomotopyMoveV-ii}
  \end{subfigure}
  \begin{subfigure}[t]{0.22\textwidth}
    \centering
    \includegraphics[height=50mm]{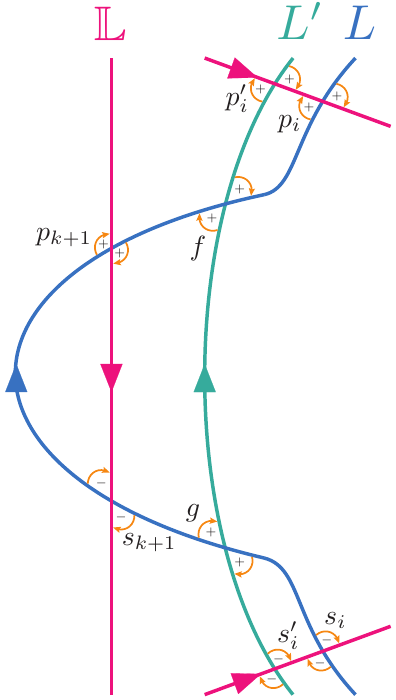}
    \caption{Type V-iii}
    \label{fig:HomotopyMoveV-iii}
  \end{subfigure}
  \begin{subfigure}[t]{0.22\textwidth}
    \centering
    \includegraphics[height=50mm]{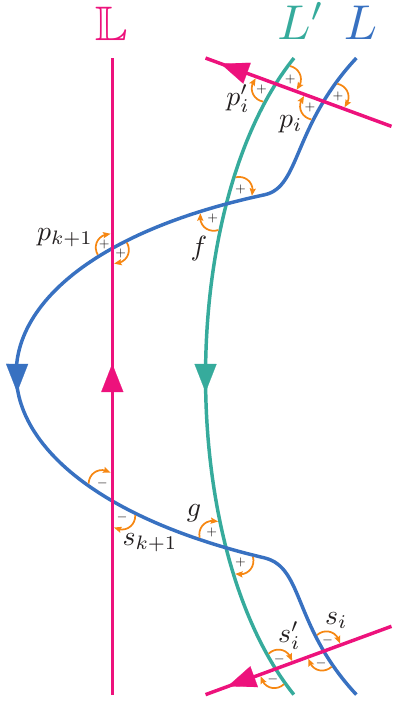}
    \caption{Type V-iv}
    \label{fig:HomotopyMoveV-iv}
  \end{subfigure}
  \caption{Isomorphisms in type V homotopy moves}
  \label{fig:HomotopyMoveVc}
\end{figure}

As in the previous lemma, $f$ and $g$ are isomorphisms between $L$ and $L'$ in $\operatorname{Fuk}\left(\mathcal{P}\right)$.
But note in this case that $L$ has $2$ more intersections with $\mathbb{L}$ than $L'$ does. So we can write
$$
\arraycolsep=5pt\def\arraystretch{1}
\begin{array}{ll}
\hspace{1.2mm}
\Hom^0\left(L,\mathbb{L}\right) = \mathbb{C}\left<p_1,\dots,p_k,p_{k+1}\right>,
&
\hspace{1.2mm}
\Hom^1\left(L,\mathbb{L}\right) = \mathbb{C}\left<s_1,\dots,s_k,s_{k+1}\right>,
\\[2mm]
\Hom^0\left(L',\mathbb{L}\right) = \mathbb{C}\left<p_1',\dots,p_k'\right>,
&
\Hom^1\left(L',\mathbb{L}\right) = \mathbb{C}\left<s_1',\dots,s_k'\right>,
\end{array}
$$
where $p_{k+1}$ and $s_{k+1}$ denote the intersections of $L$ and $\mathbb{L}$ on the boundary of the
small bigon.

Denote $\LocalPhi\left(L\right) =: \begin{pmatrix} C & D \\ E^T & u \end{pmatrix}$ and $\LocalPsi\left(L\right) =: \begin{pmatrix} F & G \\ H^T & v \end{pmatrix}$ as matrices of maps
$$
 \begin{tikzcd}[sep=40pt, every arrow/.append style={shift left}]
   S^k\oplus S \cong S\left<p_1,\dots,p_k\right>\oplus S\left<p_{k+1}\right>
     \arrow{r}{\LocalPhi\left(L\right)}
   & S^k\oplus S \cong S\left<s_1,\dots,s_k\right>\oplus S\left<s_{k+1}\right>.
     \arrow{l}{{\LocalPsi\left(L\right)\vphantom{1}}}
 \end{tikzcd}
$$
Then we have the following commuting (ignoring gray arrows) diagram from an $A_\infty$-relation on $\LocalF$
and the vanishing of $\operatorm_1\left(f\right)$ and $\operatorm_1\left(g\right)$, as in the previous lemma.

$$
\begin{tikzcd}[arrow style=tikz,>=stealth,row sep=6em,column sep=10em]
\scriptstyle S\left<p_1,\dots,p_k\right>\oplus S\left<p_{k+1}\right>
  \arrow[r, "\LocalPhi\left(L\right) = \spmat{C & D \\[1mm] E^T & u}"]
  \arrow[d, "\LocalPhi_1\left(f\right) = \spmat{I_k & 0}"]
&
\scriptstyle S\left<s_1,\dots,s_k\right>\oplus S\left<s_{k+1}\right>
  \arrow[r, "\LocalPsi\left(L\right) = \spmat{F & G \\[1mm] H^T & v}"]
  \arrow[d, "\LocalPsi_1\left(f\right) = \spmat{I_k & -Du^{-1}}"]
  \arrow[ddl, gray]
  \arrow[dl, white, swap, "\color{gray}-\LocalPsi_2\left(g{,}f\right) = \spmat{0 & 0 \\[1mm] 0 & -u^{-1}}\hspace{-18mm}"]
&
\scriptstyle S\left<p_1,\dots,p_k\right>\oplus S\left<p_{k+1}\right>
  \arrow[d, "\LocalPhi_1\left(f\right) = \spmat{I_k & 0}"]
  \arrow[ddl, gray]
  \arrow[dl, white, swap, "\color{gray}-\LocalPhi_2\left(g{,}f\right) = \spmat{0 & 0 \\[1mm] 0 & 0}\hspace{-18mm}"]
\\
\scriptstyle S\left<p_1',\dots,p_k'\right>
  \arrow[r, "\LocalPhi\left(L'\right) = C-Du^{-1}E^T"]
  \arrow[d, "\LocalPhi_1\left(g\right) = \spmat{I_k \\[1mm] -u^{-1}E^T}"]
&
\scriptstyle S\left<s_1',\dots,s_k'\right>
  \arrow[r, "\LocalPsi\left(L'\right) = F"]
  \arrow[d, "\LocalPsi_1\left(g\right) = \spmat{I_k \\[1mm] 0}"]
&
\scriptstyle S\left<p_1',\dots,p_k'\right>
  \arrow[d, "\LocalPhi_1\left(g\right) = \spmat{I_k \\[1mm] -u^{-1}E^T}"]
\\
\scriptstyle S\left<p_1,\dots,p_k\right>\oplus S\left<p_{k+1}\right>
  \arrow[r, "\LocalPhi\left(L\right) = \spmat{C & D \\[1mm] E^T & u}"]
&
\scriptstyle S\left<s_1,\dots,s_k\right>\oplus S\left<s_{k+1}\right>
  \arrow[r, "\LocalPsi\left(L\right) = \spmat{F & G \\[1mm] H^T & v}"]
&
\scriptstyle S\left<p_1,\dots,p_k\right>\oplus S\left<p_{k+1}\right>
\end{tikzcd}
$$

The second $A_\infty$-relation on $\LocalF$ from Definition \ref{defn:AinftyCategory} evaluated at $\left(g,f\right)$
is written
as
\begin{align*}
\operatorm_1^{\mathcal{MF}\left(xyz\right)}&\left(\LocalF_2\left(g,f\right)\right)
+ \operatorm_2^{\mathcal{MF}\left(xyz\right)}\left(\LocalF_1\left(g\right),\LocalF_1\left(f\right)\right)
\\
&= \LocalF_2\left(\operatorm_1\left(g\right),f\right)
+ (-1)^{\left|g\right|'}\LocalF_2\left(g,\operatorm_1\left(f\right)\right)
+ \LocalF_1\left(\operatorm_2\left(g,f\right)\right).
\end{align*}
The left hand side is calculated as
\begin{align*}
\text{(LHS)}
=
\big(&\LocalPsi\left(L\right) \circ \LocalPhi_2\left(g,f\right)
+ \LocalPsi_2\left(g,f\right) \circ \LocalPhi\left(L\right)
+ \LocalPhi_1\left(g\right) \circ \LocalPhi_1\left(f\right),
\\
&\LocalPhi\left(L\right) \circ \LocalPsi_2\left(g,f\right)
+ \LocalPhi_2\left(g,f\right) \circ \LocalPsi\left(L\right)
+ \LocalPsi_1\left(g\right) \circ \LocalPsi_1\left(f\right)
\big)
\end{align*}
and the right hand side is done as
$$
\text{(RHS)}
=
\LocalF_1\left(\id_L\right)
=
\left(\LocalPhi_1\left(\id_L\right), \LocalPsi_1\left(\id_L\right)\right)
=
\left(\operatorm_2^{0,0,b}\left(\id_L,\bullet\right),\operatorm_2^{0,0,b}\left(\id_L,\bullet\right)\right)
=
\left(I_{k+1},I_{k+1}\right),
$$
using $\operatorm_1\left(f\right)=0$, $\operatorm_1\left(g\right)=0$, $\operatorm_2\left(g,f\right)=\id_L$
and the property of the unit element $\id_L$. Thus comparing both sides yields two equations
\begin{equation}\label{eqn:HomotopyEquations}
\arraycolsep=1pt\def\arraystretch{1}
\left\{
\begin{array}{ccccccccccccc}
\LocalPhi_1\left(g\right) & \circ & \LocalPhi_1\left(f\right)
& = & I_{k+1} & + & \Psi^{\mathbb{L}}\left(L\right) & \circ & \left(-\LocalPhi_2\left(g,f\right)\right)
& + & \left(-\LocalPsi_2\left(g,f\right)\right) & \circ & \LocalPhi\left(L\right)
\\[3mm]
\LocalPsi_1\left(g\right) & \circ & \LocalPsi_1\left(f\right)
& = & I_{k+1} & +&  \LocalPhi\left(L\right) & \circ & \left(-\LocalPsi_2\left(g,f\right)\right)
& + & \left(-\LocalPhi_2\left(g,f\right)\right) & \circ & \LocalPsi\left(L\right)
\end{array}
\right..
\end{equation}

On the other hand,
by some direct polygon countings containing $f$ or $g$, we get
$$
\arraycolsep=1pt\def\arraystretch{1}
\left\{\ \
\begin{array}{ccccllrccc}
\LocalPhi_1\left(f\right)\left(p_i\right)
& = & \operatorm_2^{0,0,b}\left(f,p_i\right)
& = & \left\{ \begin{matrix*}[l] p_i' \\ 0 \end{matrix*} \right.
&
\begin{matrix*}[l]
\text{for }i=1,\dots,k, \\ \text{for }i=k+1
\end{matrix*}
& \Rightarrow &
\quad \LocalPhi_1\left(f\right) & = & \begin{pmatrix} I_k & 0 \end{pmatrix}
\\[5mm]
\LocalPsi_1\left(f\right)\left(s_i\right)
& = & \operatorm_2^{0,0,b}\left(f,s_i\right)
& = & \left. \begin{matrix*}[l] s_i' \end{matrix*} \right.
&
\text{for }i=1,\dots,k
& \Rightarrow &
\quad \LocalPsi_1\left(f\right) & = & \begin{pmatrix} I_k & * \end{pmatrix}
\\[5mm]
\LocalPhi_1\left(g\right)\left(p_i'\right)
& = & \operatorm_2^{0,0,b}\left(g,p_i'\right)
& = & \left. \begin{matrix*}[l] p_i + (*) p_{k+1} \end{matrix*} \right.
&
\text{for }i=1,\dots,k
& \Rightarrow &
\quad \LocalPhi_1\left(g\right) & = & \begin{pmatrix} I_k \\ * \end{pmatrix}
\\[5mm]
\LocalPsi_1\left(g\right)\left(s_i'\right)
& = & \operatorm_2^{0,0,b}\left(g,s_i'\right)
& = & \left. \begin{matrix*}[l] s_i \end{matrix*} \right.
&
\text{for }i=1,\dots,k
& \Rightarrow &
\quad \LocalPsi_1\left(g\right) & = & \begin{pmatrix} I_k \\ 0 \end{pmatrix}
\\[5mm]
\LocalPhi_2\left(g,f\right)\left(p_{i}\right)
& = & \operatorm_3^{0,0,b}\left(g,f,p_{i}\right)
& = & \left. \begin{matrix*}[l] 0 \end{matrix*} \right.
&
\text{for }i=1,\dots,k+1
& \quad \Rightarrow&
\quad \LocalPhi_2\left(g,f\right) & = & \begin{pmatrix} 0 & 0 \\ 0 & 0 \end{pmatrix}
\\[5mm]
\LocalPsi_2\left(g,f\right)\left(s_{i}\right)
& = & \operatorm_3^{0,0,b}\left(g,f,s_{i}\right)
& = & \left\{ \begin{matrix*}[l] 0 \\ wp_{k+1} \end{matrix*} \right.
&
\begin{matrix*}[l]
\text{for }i=1,\dots,k, \\ \text{for }i=k+1
\end{matrix*}
& \Rightarrow &
\quad \LocalPsi_2\left(g,f\right) & = & \begin{pmatrix} 0 & 0 \\ 0 & w \end{pmatrix}
\end{array}
\right.
$$
for some $w\in S$. Here, a priori, some coefficient of a monomial in $w$ might consists of an infinite sum so we cannot assume that $w$ is indeed an element of $S$. But we will demonstrate that this is not the case in the next two paragraphs, under the regularity of $L'$ and the finiteness of $\LocalF\left(L\right)$.

First, we claim that the constant term (not containing $x$, $y$ or $z$) of $w$ is just $1$ (after substituting
$T=1$).
This is equivalent to showing that 
$\operatorm_3\left(g,f,s_{k+1}\right)$ has only one term $p_{k+1}$, coming from the
obvious smallest (embedded) quadrangle $\left(\square gfs_{k+1}\overline{p_{k+1}}\right)$ in the figure with vertices $g$, $f$, $s_{k+1}$ and $\overline{p_{k+1}}$.
Indeed, if there is an another (immersed) quadrangle $\left(\tilde{\square} gfs_{k+1}\overline{p_{k+1}}\right)$ having the same vertices, it should contain $\left(\square gfs_{k+1}\overline{p_{k+1}}\right)$ at least twice at its
ends, and subtracting
one of them from one end would give an immersed cylinder bounded by $L'$ and $\mathbb{L}$, which contradicts the regularity of
$L'$.

Next, substituting the above computations into Equations \ref{eqn:HomotopyEquations} gives $wu = uw = 1$ (which should hold a priori over $\Lambda$).
Therefore, we know that the constant term of $u$ is also $1$, implying that $u$ is a unit (both in $\Lambda[[x,y,z]]$ and $S$).
Furthermore, by the finiteness of $\LocalF\left(L\right)$, each coefficient of a monomial in $u$ consists of a finite sum (in $\Lambda$ before substituting $T=1$)
and then so does $w=u^{-1}$, resulting in $w\in S$.

Equations \ref{eqn:HomotopyEquations} also fill in all other missing entries so that we have
$$
\LocalPsi_1\left(f\right) = \begin{pmatrix} I_k & -Du^{-1} \end{pmatrix},
\quad
\LocalPhi_1\left(g\right) = \begin{pmatrix} I_k \\ -u^{-1}E^T \end{pmatrix}
\quad\text{and}\quad
\LocalPsi_2\left(g,f\right) = \begin{pmatrix} 0 & 0 \\ 0 & u^{-1} \end{pmatrix}
$$
and finally, commuting of the diagram determines
$$
\LocalPhi\left(L'\right) = C-Du^{-1}E^T
\quad\text{and}\quad
\LocalPsi\left(L'\right) = F.
$$

Note that
$
\LocalF_1\left(f\right)\circ\LocalF_1\left(g\right)
$
is exactly the identity morphism of $\LocalF\left(L'\right)$,
and
$\LocalF_1\left(g\right)\circ\LocalF_1\left(f\right)$
is homotopic to the identity morphism of $\LocalF\left(L\right)$ through an explicit homotopy
$$
\left(
-\LocalPsi_2\left(g,f\right)
=
\begin{pmatrix} 0 & 0 \\ 0 & 0 \end{pmatrix},
\ \
-\LocalPhi_2\left(g,f\right)
=
\begin{pmatrix} 0 & 0 \\ 0 & -u^{-1} \end{pmatrix}\right),
$$
as described in Equations \ref{eqn:HomotopyEquations}.
That is, $\LocalF\left(L\right)$ and $\LocalF\left(L'\right)$  are homotopically equivalent.
\end{proof}

An example of the matrix reduction according to removing a bigon was given in Proposition \ref{prop:RemovingBigon}.
We remark here that actually a type III-iii or III-iv move can also be obtained by a composition of type III-i,
III-ii and V moves.
It would be a good exercise to check that the row/column operation process in Lemma \ref{lem:HomotopyTypeIII}.(2)
matches that obtained from Lemma \ref{lem:HomotopyTypeIII}.(1)
and the matrix reduction in Lemma \ref{lem:HomotopyTypeV}.

\begin{remark}\label{rmk:MatrixReduction}
According to Eisenbud's theorem, two homotopically equivalent matrix factorizations of $xyz$ should yield the same Cohen-Macaulay modules over
$A$ up to free modules. In the case of Lemma \ref{lem:HomotopyTypeV}, in particular, we can find an explicit stable isomorphism between them as below.
$$
\newcommand{\scriptverteq}{\mathrel{\rotatebox{90}{$\scriptstyle=$}}}
\tikzset{
    labl/.style={anchor=south, rotate=90, inner sep=.5mm}
}
\begin{tikzcd}[arrow style=tikz,>=stealth,row sep=5em,column sep=7em,
    execute at end picture={
    \path (\tikzcdmatrixname-1-1) -- (\tikzcdmatrixname-2-1)
    coordinate[pos=0.5] (aux1)
    (\tikzcdmatrixname-1-2) -- (\tikzcdmatrixname-2-2)
    coordinate[pos=0.5] (aux2)
    (aux1) -- (aux2) node[midway,sloped]{$\hspace{-2mm} \begin{matrix} \scriptstyle\LocalPhi\left(L'\right) \\[-1mm] \scriptverteq \end{matrix}$};
    },
    execute at end picture={
    \path (\tikzcdmatrixname-1-3) -- (\tikzcdmatrixname-2-3)
    coordinate[pos=0.5] (aux1)
    (\tikzcdmatrixname-1-2) -- (\tikzcdmatrixname-2-2)
    coordinate[pos=0.5] (aux2)
    (aux1) -- (aux2) node[midway,sloped]{$\hspace{1mm} \begin{matrix} \scriptstyle\LocalPsi\left(L'\right) \\[-1mm] \scriptverteq \end{matrix}$};
    }] 
\scriptstyle S\left<p_1,\dots,p_k,p_{k+1}\right>
  \arrow[r, "\LocalPhi\left(L\right) = \spmat{ C & D \\[1mm] E^T & u }"]
  \arrow[d, "\cong" labl, swap, "\spmat{ I_k & 0 \\[1mm] u^{-1}E^T & 1 }\hspace{3mm}"]
&
\scriptstyle S\left<s_1,\dots,s_k,s_{k+1}\right>
  \arrow[r, "\LocalPsi\left(L\right) = \spmat{F & G \\[1mm] H^T & v}"]
  \arrow[d, "\cong" labl, "\spmat{ I_k & -Du^{-1} \\[1mm] 0 & u^{-1} }"]
&
\scriptstyle S\left<p_1,\dots,p_k,p_{k+1}\right>
  \arrow[d, "\cong" labl, "\spmat{ I_k & 0 \\[1mm] u^{-1}E^T & 1 }"]
\\
\scriptstyle S\left<p_1',\dots,p_k'\right>\oplus S
  \arrow[r, "\spmat{ C-Du^{-1}E^T & 0 \\[1mm] 0 & 1 }"]
&
\scriptstyle S\left<s_1',\dots,s_k'\right>\oplus S
  \arrow[r, "\hspace{7mm}\spmat{ F & 0 \\[1mm] 0 & xyz }"]
&
\scriptstyle S\left<p_1',\dots,p_k'\right>\oplus S
\end{tikzcd}
$$
Commuting of the diagram is immediate from some matrix calculations and the fact that $\left(\LocalPhi\left(L\right),\LocalPsi\left(L\right)\right)$
is a matrix factorization of $xyz$. Also, note that the vertical maps
are all isomorphisms. From this we get isomorphisms
\begin{align*}
\cok\underline{\LocalPhi\left(L\right)}
&\cong \cok\begin{pmatrix} \underline{\LocalPhi\left(L'\right)} & 0 \\ 0 & 1 \end{pmatrix}
\cong \cok\underline{\LocalPhi\left(L'\right)}
\quad\text{and}
\\
\cok\underline{\LocalPsi\left(L\right)}
&\cong \cok\begin{pmatrix} \underline{\LocalPsi\left(L'\right)} & 0 \\ 0 & 0 \end{pmatrix}
\cong \cok\underline{\LocalPsi\left(L'\right)}\oplus A
\quad\text{in }
\operatorname{CM}\left(A\right).
\end{align*}
This is also closely related to the algebraic version of matrix reduction process used in Lemma \ref{lem:MatrixReductionLemma} which
corresponds to the removing a redundant generator of a module.
\end{remark}

\subsection{Admissibility}\label{sec:Admissibility}

In this subsection,
we recall the notion of `\emph{admissibility}' introduced in the paper of Azam and Blanchet \cite{AB22}, and prove some related lemmas that we need for our propositions. In particular, we define the notion of `{\em strong admissibility}', which is a special type of admissibility that is easier to see whether a given loop has the property.

Let $L$ be an unobstructed loop which intersects transversally with the Seidel Lagrangian $\bll$ and $L$ itself. Then $\cpp\setminus (L\cup \bll)$ and $(L\cup\bll)\setminus\{\text{intersection points}\}$ have only finitely many path connected components. Define $C_2 = C_2\left(\cpp ; \bll, L ; \bzz\right)$ as a free $\bzz$-module generated by components which do not contain punctures. Also define $C_1 = C_1\left(\cpp ; \bll, L ; \bzz\right)$ as a free $\bzz$-module generated by components of $(L\cup \bll)\setminus\{\text{intersections points}\}$. Then there is a natural boundary operator $\partial : C_2 \rightarrow C_1$. The orientations of elements in $C_1$ and $C_2$ are inherited from the orientation of $\cpp, L$ and $\bll$. Let us call an element of $C_i$ an {\em $i$-chain} and a basis element of $C_i$ an {\em $i$-basis} for $i=1, 2$. Let us denote by $\left<x, \tau\right>$ the coefficient of $\tau$ in $x$, where $\tau$ is an $i$-basis and $x$ is an $i$-chain for $i=1, 2$. Note that for any $2$-basis $\sigma$ and $1$-basis $\tau$, $\left<\partial\sigma, \tau\right>= 1, 0$, or $-1$. Also note that for each $1$-basis $\tau$, there are at most two $2$-basis $\sigma$ such that $\left<\partial\sigma, \tau\right>\neq 0$ and if there are two, then they should have opposite sign.

Note that there are exactly three component containing punctures. Let us denote them by $A, B$, and $C$. By measuring distance from the puncture, we can give a natural grading called {\em level} to each $1$- and $2$-basis as follows. First, boundary components of $A, B$ and $C$ are set to be of level $0$. If a $2$-basis has a boundary component of level $0$, then is is said to be of level $1$. For a positive integer $n$, a $1$-basis is said to be of level $n$ if it is not of level $k$ for each $k<n$, and it is a boundary of some level $n$-component. Also, a $2$-basis is said to be level $n$ if it is not of level $k$ for each $k<n$, and it has a boundary component of level $(n-1)$. Since basis is finite, level is well defined.

\begin{lemma}
        The boundary map $\partial$ is injective.
\end{lemma}
\begin{proof}
%
                
        Suppose that $x$ is a $2$-chain such that $\partial(x)=0$. Take a $2$-basis $\sigma$ of minimal level $k$ in $x$. Then, by definition, $\partial x$ has level $k$ boundary $1$-basis. To eliminate it, there should be $2$-basis of level $(k-1)$ in $x$, which contradicts minimality of $\sigma$. Thus $x=0$, which implies that $\partial$ is injective.        
\end{proof}

Now recall the definition of the Euler measure $e : C_2\rightarrow \bzz$. For a $2$-basis $\sigma$, define its Euler measure as $e(\sigma) = 1-\frac{1}{4}\#\{\text{boundary components of }\sigma\}$ and extend linearly to whole of $C_2$.\footnote{For a general surface, the Euler measure is defined as the total curvature of a metric through which its boundaries are geodesics and the corner has the right angle. See \cite{AB22}} Let $[\bll]$ and $[L]$ be $1$-chains obtained from the Lagrangians $\bll$ and $L$, respectively and define a set $H(L)$ consisting of $2$-chains with Euler measure $0$ whose boundary is a linear combination of $[\bll]$ and $[L]$.

\begin{defn}\label{defn:Admissibility}
        An unobstructed loop $L$ is said to be {\em admissible} if any nonzero $2$-chain in $H(L)$ has both positive and negative coefficients.
\end{defn}

\begin{remark}\label{rmk:RegularityAdmissibility}
The regular loops turn out to be a wider class than the admissible loops. Indeed, we can show with an obvious
argument that an admissible loop is always regular, but the converse is false.
\end{remark}

\begin{example}
        Let $L$ be an unobstructed loop. Then there are three components containing puncture, namely $A, B$, and $C$. Suppose that the boundary $\partial(A), \partial(B), \partial(C)$ have both positive and negative components of $[L]$. Since the boundary operator $\partial$ is injective and $A, B$, and $C$ are not $2$-basis, there are not positive or negative $2$-chain of which boundary is a linear combination of $[L]$ and $[\bll]$. Thus $L$ is admissible. This motivates the following definition.
\end{example}

\begin{defn}\label{def:StronglyAdmissible}
        An unobstructed loop $L$ is said to be {\em strongly admissible} if the boundary of punctured components $A, B, C$ have both positive and negative components.
\end{defn}

From the definition, the proposition below immediately follows. 
\begin{lemma}
        If an unobstructed loop is strongly admissible, then it is admissible.
\end{lemma}

\begin{example}\label{ex:DegenerateLagrnagianStronglyAdmissible}
        The Lagrangian in Remark \ref{rem:DegenerateLagrangian} is strongly admissible, hence it is regular. As an example, the Lagrangian $L((2, 2, 2, 2, 2, 2), \lambda')$ is illustrated in Figure \ref{fig:StronglyAdmissibleL(2)}. Note that the component containing puncture $a$ has a boundary with positive and negative components.
        
        \begin{figure}[H]
                \centering
                \includegraphics[height=55mm]{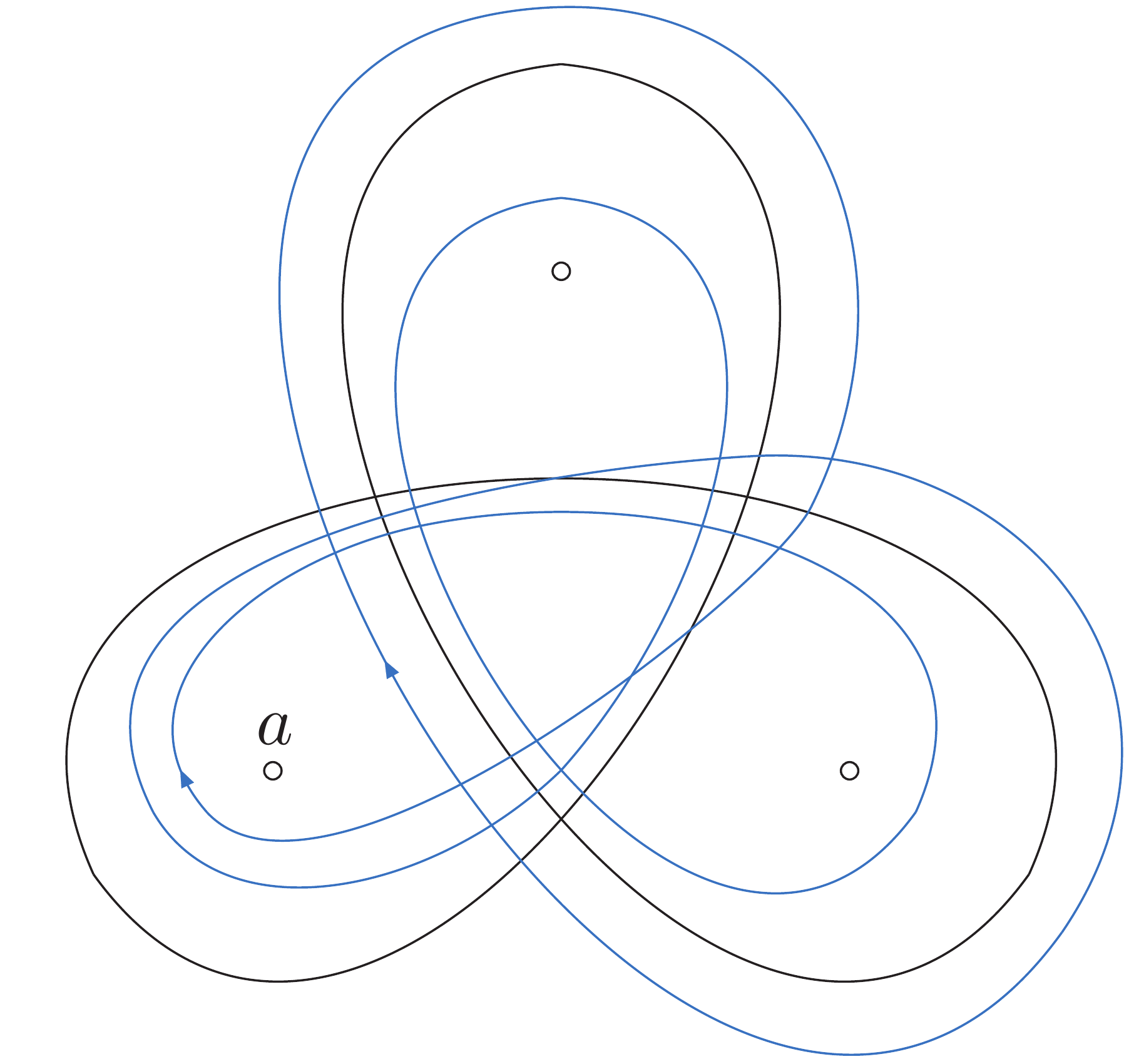}
                \caption{The Lagrangian $L(2, 2, 2, 2, 2, 2), \lambda')$}
                \label{fig:StronglyAdmissibleL(2)}
        \end{figure}
\end{example}

Let $p, q\in L\cap \bll$ and $b_1, \cdots, b_r$ be self intersection points of $\bll$, and denote the set of decorated strips whose vertices are $p, b_1, \cdots, b_r, q$ by $\mathcal{M}(p, b_1, \cdots, b_r, q)$. Note that an element of $\mathcal{M}(p, b_1, \cdots, b_r, q)$ gives a 2-chain. By mimicking the proof of Lemma $3.3$, Proposition $3.4$, and Corollary $3.5$ in \cite{AB22}, we get the following result.

\begin{lemma}\label{lem:UpperRight}
        Let $n$ be a positive integer and $\bzz_{\geq 0}$ be the set of nonnegative integers. Then for any infinite subset $S \subseteq \bzz_{\geq 0}^n$, there are distinct elements $x=(x_1, \cdots, x_n)$ and $y=(y_1, \cdots, y_n)$ such that $x_i\leq y_i$ for all $i$.
\end{lemma}
\begin{proof}
        The proof is by induction on the number of components $n$. The lemma is obviously true for $n=1$. Assume the lemma is true for $n<N$. Let $p_i : (\bzz_{\geq 0})^N \rightarrow \bzz_{\geq 0}$ be the projection to the $i$-th component. Suppose that there are some index $i$ and $m\in\bzz_{\geq 0}$ such that the fiber $p_i^{-1}(m)\cap S$ is infinite. Then, by the induction hypothesis, there are distinct $x, y\in p_i^{-1}(m)$ such that each coordinate of $x$ is not larger than that of $y$.
        
        Now assume that for each index $i$ and $m\in \bzz$ the fiber $p_i^{-1}(m)$ is finite and take an element $x=(x_1, \cdots, x_n) \in S$. Then the set
        $$\{(z_1, \cdots, z_n) : z_i\leq x_i \text{ for each } i \} = \bigcup_{i=1, \cdots, n}\bigcup_{m\leq x_i}p_i^{-1}(m)$$ is finite, so there is an element $y$ in the complement on $S$. Since for each $i$, $y_i>x_i$, so the lemma is proved.
\end{proof}
\begin{lemma}\label{lem:AdmissibleFinite}
        Let $L$ be an admissible loop and $p, q \in L\cap \bll$ be intersection points. Let also $b_1, \cdots, b_r\in\{X, Y, Z\}$ be self intersections of $\mathbb{L}$. Then the moduli space $\mathcal{M}(p, b_1, \cdots, b_r, q)$ is finite. 
\end{lemma}
\begin{proof}
        For a decorated strip $u \in \mathcal{M}(p, b_1, \cdots, b_r, q)$, we say it is {\em positively oriented} if the path $u \circ \eta^+$ follows the loop $L$ positively and {\em negatively oriented} otherwise and, denote by $\overrightarrow{\mathcal{M}}(p, b_1, \cdots, b_r, q)$ and $\overleftarrow{\mathcal{M}}(p, b_1, \cdots, b_r, q)$ the subset consisting of positively oriented decorated strips and negatively oriented one, respectively. If the moduli space $\mathcal{M}(p, b_1, \cdots, b_r, q)$ is infinite, then one of the subsets is also infinite. We may assume $\overrightarrow{\mathcal{M}}(p, b_1, \cdots, b_r, q)$ is infinite.
        
        Let $\pi : \overrightarrow{\mathcal{M}}(p, b_1, \cdots, b_r, q) \rightarrow C_2(\cpp, \bll, L ; \bzz)$ be the natural assignment. Then for $u, v \in \overrightarrow{\mathcal{M}}(p, b_1, \cdots, b_r, q)$, $\partial(\pi(u))-\partial(\pi(v))$ is a linear combination of $[L], [\bll]$. Also since the Euler measure is additive, $[u]-[v]$ has Euler measure zero. By Gromov compactness theorem, the fiber $\pi^{-1}(x)$ is finite for each $x\in C_2$. Thus, by Lemma \ref{lem:UpperRight}, there is some $u, v \in \overrightarrow{\mathcal{M}}(p, b_1, \cdots, b_r, q)$ such that $[u]-[v]$ has only nonnegative components. This contradicts admissibility of $L$. Thus the moduli space $\mathcal{M}(p, b_1, \cdots, b_r, q)$ should be finite.
\end{proof}

\begin{cor}\label{cor:AdmissibleT=1}
For an admissible loop $L$, its mirror matrix factorization 
$\LocalF\left(L\right) = \left(\LocalPhi\left(L\right), \LocalPsi\left(L\right)\right)$
is finite.

\end{cor}
\begin{proof}
        Suppose that some entry in the matrix factorization has a monomial $x^{i}y^{j}z^{k}$ with infinite coefficients, say from an odd point $p$ to an even point $q$. Since there are only finitely many reordering of $(x, \cdots, x, y, \cdots, y, z, \cdots, z)$, for some ordering $(b_1, \cdots, b_{i+j+k})$, the moduli space $\mathcal{M}(p, b_1, \cdots, b_{i+j+k}, q)$ is infinite. This implies non-admissibility of $L$ by Lemma \ref{lem:AdmissibleFinite}.
\end{proof}

As pointed out in Remark \ref{rmk:RegularityAdmissibility}, there are regular loops that are not admissible.
To show that the finiteness still holds for regular loops, by Lemma \ref{lem:HomotopyTypeV}, it is enough
to show that any regular loop can be obtained from an admissible loop by only some type V$^{-}$ moves
(Lemma \ref{lem:HomotopicToAdmissible}).
But we also require the intermediate loops to be regular, for the proof of Lemma \ref{lem:5TypesAreEnough}.
So we first show the following lemma.

\begin{lemma}\label{lem:PreservationOfRegularity}
        Let $L$ be a regular loop and $L'$ be a loop obtained from $L$ by type I, II, III, IV and V$^{+}$ moves. Then $L'$ is also regular.
\end{lemma}
\begin{proof}
        As already mentioned above, type I, II and IV moves do not change intersection patters of $L$ with $\bll$, so they do not affect regularity. We restrict to the type V$^{+}$ case, but the proof also works for the type III case. Let us prove that if a perturbed Lagrangian is not regular, then so is the original one.
        
        Let $L'$ be a Type V$^{+}$-move of $L$ as in Figure \ref{fig:Move5Regularity}, and let us denote $B_1, B_2$ the bigons with vertices $f, g$ as in Figure \ref{fig:Move5RegularityBigons}. Suppose that $L'$ is not regular so that there is an immersed cylinder bounded by $L'$ and $\bll$. There are two cases depending on the direction of the cylinder as in Figure \ref{fig:Move5RegularityLeftDisc} and \ref{fig:Move5RegularityRightDisc}. Suppose that the cylinder is on the left. Then the bigon $B_1$ is totally contained in the cylinder. Thus, by cutting out $B_1$ and attaching $B_2$ to the cylinder, we get another cylinder bounded by $L$ and $\bll$. The old and new cylinders are illustrated in Figure \ref{fig:Move5RegualrityCylinders}. Similarly, if the cylinder is on the right, one can obtain a new cylinder by cutting out $B_2$ and attaching $B_1$. Hence we prove that if $L'$ is non-regular, so is $L$.
        
        \begin{figure}[H]
                \centering
                \begin{subfigure}[t]{0.22\textwidth}
                        \centering
                        \includegraphics[height=45mm]{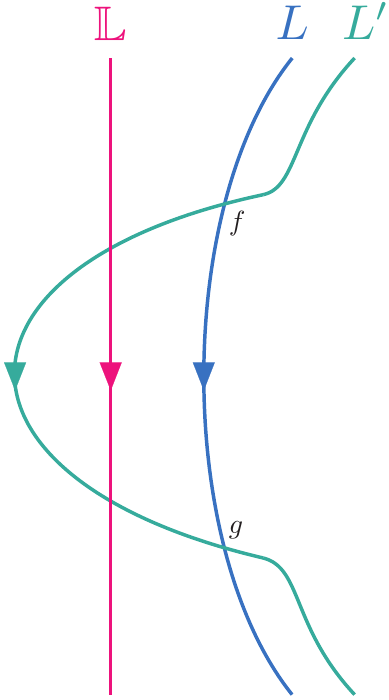}
                        \caption{Perturbation}
                        \label{fig:Move5Regularity}
                \end{subfigure}
                \begin{subfigure}[t]{0.22\textwidth}
                        \centering
                        \includegraphics[height=45mm]{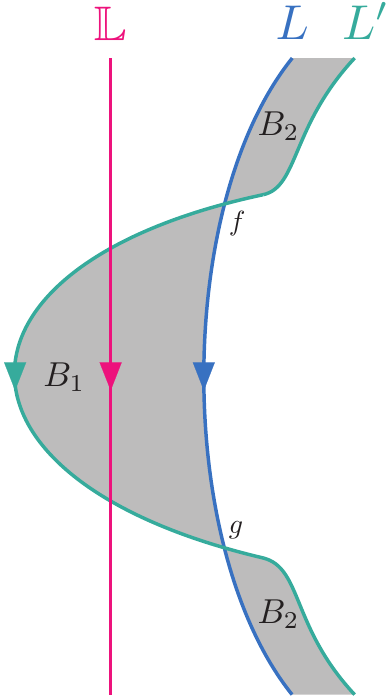}
                        \caption{Bigons}
                        \label{fig:Move5RegularityBigons}
                    \end{subfigure}
                \begin{subfigure}[t]{0.22\textwidth}
                        \centering
                        \includegraphics[height=45mm]{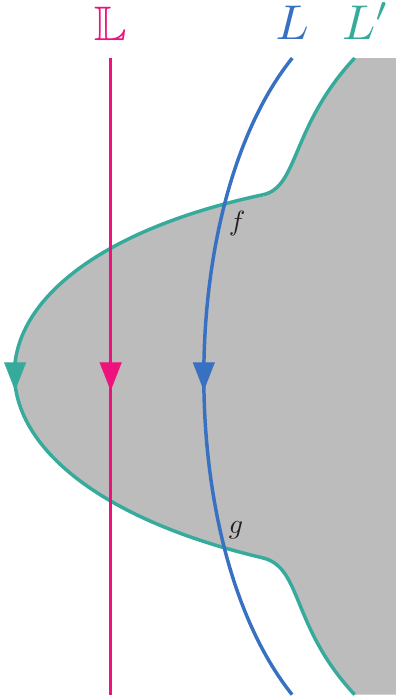}
                        \caption{Cylinder at the left side}
                        \label{fig:Move5RegularityLeftDisc}
                \end{subfigure}
                \begin{subfigure}[t]{0.22\textwidth}
                        \centering
                        \includegraphics[height=45mm]{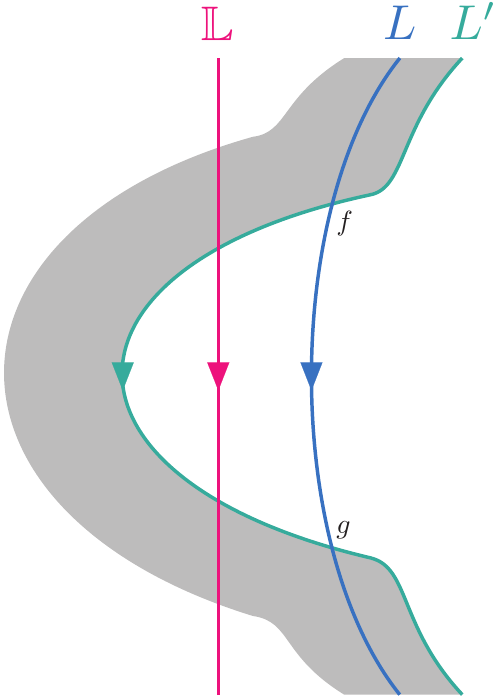}
                        \caption{Cylinder at the right disc}
                        \label{fig:Move5RegularityRightDisc}
                \end{subfigure}
                \caption{}
                \label{}
        \end{figure}
    
           \begin{figure}[H]
                \centering
                \begin{subfigure}[t]{0.30\textwidth}
                        \centering
                        \includegraphics[height=35mm]{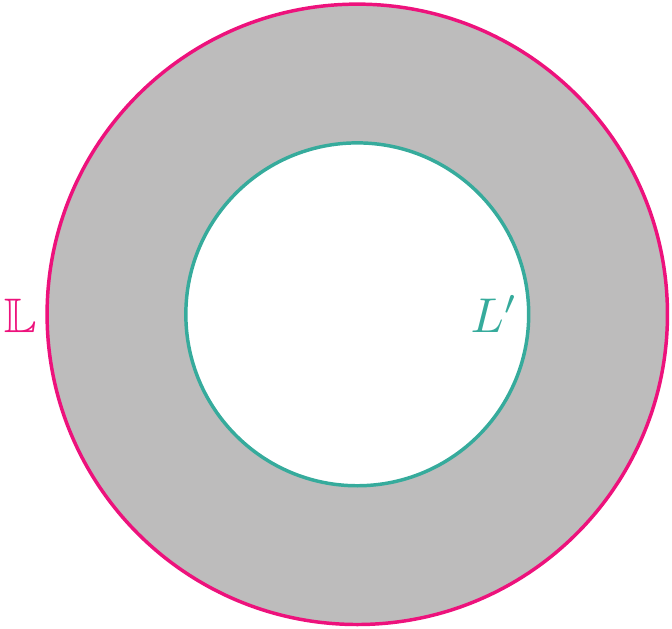}
                        \caption{Domain of the cylinder}
                \end{subfigure}
                \begin{subfigure}[t]{0.30\textwidth}
                        \centering
                        \includegraphics[height=35mm]{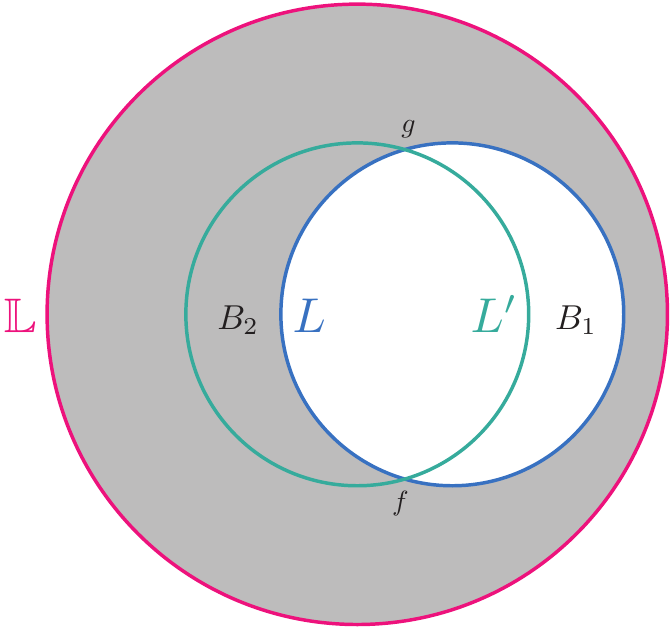}
                        \caption{Perturbed domain of the left side cylinder}
                \end{subfigure}
                \begin{subfigure}[t]{0.30\textwidth}
                        \centering
                        \includegraphics[height=35mm]{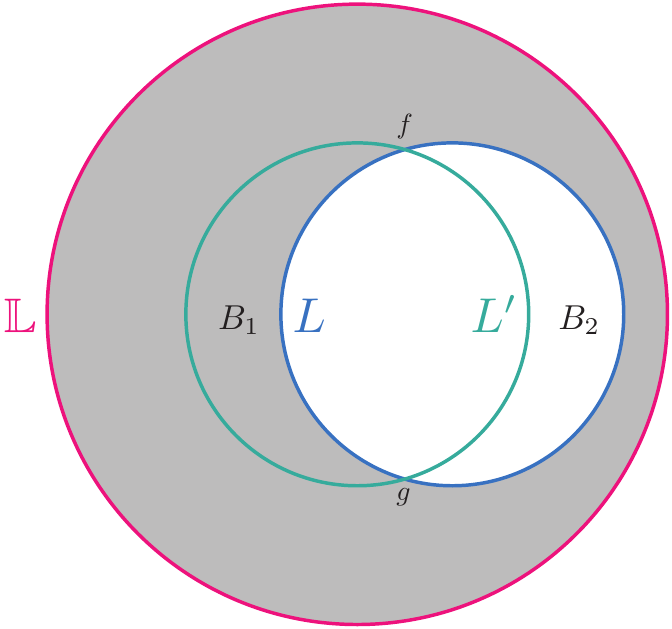}
                        \caption{Perturbed domain of the right side cylinder}
                \end{subfigure}
                \caption{Cylinders}
                \label{fig:Move5RegualrityCylinders}
        \end{figure}
\end{proof}

\begin{lemma}\label{lem:HomotopicToAdmissible}
For any regular loop $L'$, there are regular loops $L'=L^{(0)},L^{(1)},\dots,L^{(d)}=L$ satisfying the
following:
\begin{itemize}
\item
$L$ is admissible, and
\item
For $i=1,\dots,d$,
the intersection pattern of $L^{(i)}$ with $\mathbb{L}$
is obtained from that of $L^{(i-1)}$ with $\mathbb{L}$ by performing a move of type V$^{+}$.
\end{itemize}
\end{lemma}

\begin{proof}
        We show that any regular loop can be deformed into a strongly admissible one by using only type V$^{+}$ moves, which preserve regularity by Lemma \ref{lem:PreservationOfRegularity}. Choose any puncture, say $a$ and take paths from $a$ to $\ell_x$ and $\tilde{\ell}_x$. We may take the paths so that they do not go through self intersections of $L'$ and meet $L'$ transversally. Then perturb $L'$ along the path as in Figure \ref{fig:RegToAdm}, the component containing $a$ has both positive and negative boundary component of the Seidel Lagrangian, which implies that the resulting loop is strongly admissible. This proves the lemma.   
        \begin{figure}[H]
                        \centering
                        \includegraphics[height=35mm]{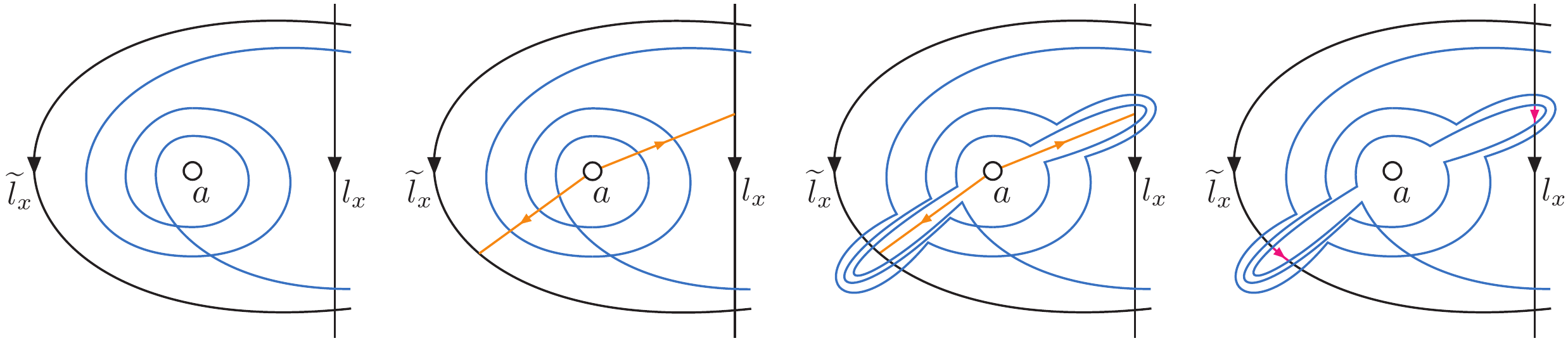}
                \caption{Deforming process}
                \label{fig:RegToAdm}
         \end{figure}
\end{proof}

Now we are ready to give the proof of Lemma \ref{lem:5TypesAreEnough}.
\begin{proof}[Proof of Lemma \ref{lem:5TypesAreEnough}]\label{proof:5TypesAreEnough}
        Suppose that homotopic unobstructed regular loops $L_0, L_1$ are given. In the proof of Lemma \ref{lem:HomotopicToAdmissible}, we may take paths from a puncture to Seidel Lagrangian so that they satisfy intersection condition for $L_0$ and $L_1$ simultaneously. Then $L_0$ and $L_1$ are deformed to strongly admissible loops $\tilde{L}_0$ and $\tilde{L_1}$, while maintaining the regularity. 
        Note that two loops $\tilde{L}_0$ and $\tilde{L}_1$ are still homotopic to each other 
        (As in Figure \ref{fig:RegToReg}, we may pretend that the puncture became larger, containing the chosen paths).  Moreover, loops appearing through the homotopy are still strongly admissible, which implies regularity. Thus we have a sequence of regular loops connecting $L_0$ and $L_1$.
                \begin{figure}[H]
                        \centering
                        \includegraphics[height=35mm]{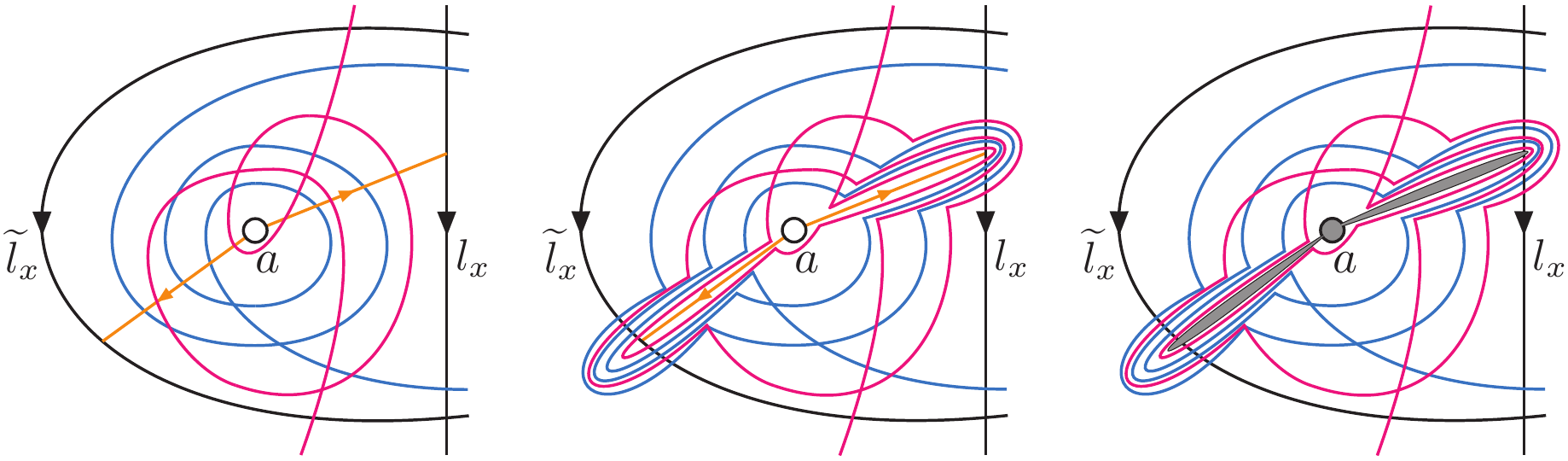}
                \caption{Deforming process}
                \label{fig:RegToReg}
        \end{figure}
\end{proof}

\bibliographystyle{amsalpha}
\bibliography{IndecomposableXYZ}

\end{document}